%% file: main_updated.tex
\documentclass[10pt]{article}
\usepackage{fullpage}
\usepackage[utf8]{inputenc}
\usepackage[bookmarks, colorlinks=true, plainpages = false, citecolor = blue,linkcolor=red,urlcolor = blue, filecolor = blue,pagebackref]{hyperref}

\makeatletter
\def\@footnotecolor{red}
\define@key{Hyp}{footnotecolor}{%
 \HyColor@HyperrefColor{#1}\@footnotecolor%
}
\def\@footnotemark{%
    \leavevmode
    \ifhmode\edef\@x@sf{\the\spacefactor}\nobreak\fi
    \stepcounter{Hfootnote}%
    \global\let\Hy@saved@currentHref\@currentHref
    \hyper@makecurrent{Hfootnote}%
    \global\let\Hy@footnote@currentHref\@currentHref
    \global\let\@currentHref\Hy@saved@currentHref
    \hyper@linkstart{footnote}{\Hy@footnote@currentHref}%
    \@makefnmark
    \hyper@linkend
    \ifhmode\spacefactor\@x@sf\fi
    \relax
  }%
\makeatother

\hypersetup{footnotecolor=black}

\usepackage{lipsum}
\usepackage{graphicx}
\usepackage{epstopdf}

\usepackage{mathtools}
\usepackage{empheq}
\usepackage{amsfonts}
\usepackage{amsgen,amsmath,amstext,amsbsy,amsopn,amssymb,stmaryrd}
\usepackage{comment}
\usepackage{enumitem}
\usepackage{booktabs}
\usepackage{blkarray}
\usepackage{algorithm}
\usepackage{algpseudocode}
\usepackage{url}
\usepackage{longtable}
\usepackage{mathtools}
\usepackage{multirow}

\ifpdf
\DeclareGraphicsExtensions{.eps,.pdf,.png,.jpg}
\else
\DeclareGraphicsExtensions{.eps}
\fi

\DeclareMathAlphabet\mathbfcal{OMS}{cmsy}{b}{n}

\newcommand{\be}{\begin{equation}}
\newcommand{\ee}{\end{equation}}
\newcommand{\bea}{\begin{eqnarray}}
\newcommand{\eea}{\end{eqnarray}}
\newcommand{\beas}{\begin{eqnarray*}}
	\newcommand{\eeas}{\end{eqnarray*}}

\newcommand{\E}{{\mathbf{E}}}

\newcommand{\R}{{\mathbf{R}}}

\newcommand{\U}{{\mathbf{U}}}

\newcommand{\diag}{{\rm diag}}

\newcommand{\polylog}{{\rm polylog}}

\newcommand{\SVD}{{\rm SVD}}

\newcommand{\argmin}{\mathop{\rm arg\min}}
\newcommand{\argmax}{\mathop{\rm arg\max}}

\newcommand{\bp}{\boldsymbol{p}}
\newcommand{\br}{\boldsymbol{r}}

\renewcommand{\t}{^{\top}}
\newcommand{\utilde}{\mathbf{\widetilde{U}}}
\newcommand{\uhat}{\mathbf{\widehat{U}}}
\usepackage{amsthm,amssymb}
\newtheorem{lemma}{Lemma}
\newtheorem{theorem}{Theorem}
\newtheorem{remark}{Remark}
\newtheorem{corollary}{Corollary}
\newtheorem{proposition}{Proposition}
\newtheorem{assumption}{Assumption}
\newcommand{\eps}{\varepsilon}
\newcommand{\p}{\mathbb{P}}
\newcommand{\inv}{^{-1}}
\newcommand{\ku}{^{(k)}}
\newcommand{\dl}{\delta_{\mathrm{L}}}

\renewcommand{\E}{\mathbb{E}}

\renewcommand{\hat}{\widehat}
\renewcommand{\tilde}{\widetilde}

\newcommand{\pure}{^{(\mathrm{pure})}}
\newcommand\numberthis{\addtocounter{equation}{1}\tag{\theequation}}
\usepackage[capitalize]{cleveref}
\usepackage{natbib}
\usepackage[dvipsnames]{xcolor}

\usepackage{multirow}
\usepackage{graphicx}
\usepackage{tikz}
\usetikzlibrary{arrows.meta,
                chains,
                positioning}
\usepackage{caption}
\usepackage{subcaption}
\usepackage{titling}

\newcommand{\red}{\color{red}}

 \makeatletter
 \newcommand*{\rom}[1]{\expandafter\@slowromancap\romannumeral #1@}
\makeatother

\allowdisplaybreaks

\begin{document}

\title{%The $\ell_2\to\ell_\infty$ Tensor Perturbation Bound with Applications to Statistics and Machine Learning\\
Estimating Higher-Order Mixed Memberships via the $\ell_{2,\infty}$ Tensor Perturbation Bound}

\author{Joshua Agterberg\thanks{Departments of Electrical and Systems Engineering and Statistics and Data Science, University of Pennsylvania, Email: jagt@seas.upenn.edu} ~ and ~ Anru R. Zhang\thanks{Departments of Biostatistics \& Bioinformatics, Computer Science, Mathematics, and Statistical Science, Duke University. Email: anru.zhang@duke.edu.}}

\date{\today}

\maketitle

%\begin{sloppypar}
\begin{abstract}
%Higher-order multiway data is ubiquitous in machine learning and statistics, and there is a need to develop new methodologies for these types of data that succinctly capture the underlying structures. For example, multiway data may exhibit community-like structures, where each component (node) along each different mode has a community membership associated with it. 
Higher-order multiway data is ubiquitous in machine learning and statistics and often exhibits community-like structures, where each component (node) along each different mode has a community membership associated with it. In this paper we propose the \emph{tensor mixed-membership blockmodel}, a generalization of the tensor blockmodel positing that memberships need not be discrete, but instead are convex combinations of latent communities. We establish the identifiability of our model and propose a computationally efficient estimation procedure based on the higher-order orthogonal iteration algorithm (HOOI) for tensor SVD composed with a simplex corner-finding algorithm. We then demonstrate the consistency of our estimation procedure by providing a per-node error bound, which showcases the effect of higher-order structures on estimation accuracy. To prove our consistency result, we develop the $\ell_{2,\infty}$ tensor perturbation bound for HOOI under independent, heteroskedastic, subgaussian noise that may be of independent interest. Our analysis uses a novel leave-one-out construction for the iterates, and our bounds depend only on spectral properties of the underlying low-rank tensor under nearly optimal signal-to-noise ratio conditions such that tensor SVD is computationally feasible. Finally, we apply our methodology to real and simulated data, demonstrating some effects not identifiable from the model with discrete community memberships.
\end{abstract}

\bigskip

%%%%%%%%%%%%%%%%
\tableofcontents
%%%%%%%%%%%%%%%%

\input{text}

 \appendix
 %\newpage

\input{appendix}

%end{sloppypar}

\bibliographystyle{tensor-perturb-2-infinity/plainnat_JA}
\bibliography{tensor-perturb-2-infinity/reference, tensor-perturb-2-infinity/tensors}

\end{document}

%% file: text.tex
\section{Introduction}\label{sec:intro}
%%%%%%%%%%%%%%%%%%%%%%%%%%%%%%%%%
Higher-order multiway data, i.e., tensor data, is ubiquitous in modern machine learning and statistics, and there is a need to develop new methodologies for these types of data that succinctly capture the underlying structures. In a variety of scenarios, tensor data may exhibit community-like structures, where each component (node) along each different mode is associated with a certain community/multiple communities. High-order clustering aims to partition each mode of a dataset in the form of a tensor into several discrete groups. In many settings, the assumption that groups are discrete, or that each node belongs to only one group, can be restrictive, particularly if there is a domain-specific reason that groups need not be distinct. \textcolor{black}{For example, in the global trade data we consider in \cref{sec:numericalresults}, one observes trading patterns for different goods between countries. Imposing the assumption that the underlying tensor has discrete communities assumes that each country can be grouped into distinct ``buckets" -- however, geography is a continuous parameter, and different countries may belong to multiple communities.  Similarly, different goods need not belong to distinct groups.}

To ameliorate this assumption of distinct communities, in this paper, we propose the \textcolor{black}{(subgaussian)} \emph{tensor mixed-membership blockmodel}, which relaxes the assumption that communities are discrete. Explicitly, we assume that each entry of the underlying tensor $\mathcal{T} \in \mathbb{R}^{p_1 \times p_2 \times p_3}$ can be written via the decomposition
\begin{align}
    \mathcal{T}_{i_1i_2i_3} &= \sum_{l_1=1}^{r_1} \sum_{l_2=1}^{r_2} \sum_{l_3=1}^{r_3} \mathcal{S}_{l_1 l_2 l_3} \big(\mathbf{\Pi}_1 \big)_{i_1l_1}\big(\mathbf{\Pi}_2 \big)_{i_2l_2} \big(\mathbf{\Pi}_3 \big)_{i_3 l_3}, \label{mmsbm}
\end{align}
where $\mathbf{\Pi}_k \in [0,1]^{p_k \times r_k}$ satisfies $\sum_{l=1}^{r_k} (\mathbf{\Pi}_k)_{il} =1$ and $\mathcal{S}\in \mathbb{R}^{r_1\times r_2\times r_3}$ is a mean tensor. In other words, the model \eqref{mmsbm} associates to each index along each mode a $[0,1]$-valued membership vector. For each index $i$ of each mode $k$, the entries of its membership vector $\big(\mathbf{\Pi}_k\big)_{i\cdot}$ correspond to one of the $r_k$ latent underlying communities, with the magnitude of the entry governing the intensity of membership within that community. The entry $i_1, i_2, i_3$ of the underlying tensor is then a weighted combination of the entries of the mean tensor $\mathcal{S} \in \mathbb{R}^{r_1 \times r_2 \times r_3}$, with weights corresponding to three different membership vectors $\big(\mathbf{\Pi}_1\big)_{i_1\cdot}, \big(\mathbf{\Pi}_2\big)_{i_2\cdot}, \big(\mathbf{\Pi}_2\big)_{i_3\cdot}$.

In the previous example of \textcolor{black}{global trade data}, % airport flights, 
considering just the mode corresponding to \textcolor{black}{country}, the mixed-membership tensor blockmodel posits that there are latent ``pure'' \textcolor{black}{countries} and each individual \textcolor{black}{country} is a convex combination of these pure \textcolor{black}{countrie}s. For a given index $i$, each entry of the $i$'th row of the membership matrix $(\mathbf{\Pi}_{\mathrm{country}})_{i\cdot}$ corresponds to how much the \textcolor{black}{country} $i$ reflects each of the latent ``pure communities.'' 
When the matrices $\mathbf{\Pi}_k$ are further assumed to be $\{0,1\}$-valued, every index is ``pure" and this model reduces to the tensor blockmodel considered in \citet{han_exact_2021,chi_provable_2020,wu_general_2016}, and \citet{wang_multiway_2019}.   \textcolor{black}{We emphasize that in our theoretical results, we do not assume that there are underlying symmetries that arise in the hypergraph or multilayer undirected network settings (e.g., \citet{jing_community_2021,ke_community_2020}).  While we make this assumption for simplicity, the model can be naturally extended to account for symmetry along certain modes of the tensor.  Our theoretical results can also be extended to this setting by modifying our constructions described in \cref{sec:proofoverview}. 
}

The factorization in \eqref{mmsbm} can be related to the so-called \emph{Tucker decomposition} of the tensor $\mathcal{T}$. A tensor $\mathcal{T}$ is said to be of Tucker rank $(r_1,r_2,r_3)$ if it can be written via
\begin{align*}
    \mathcal{T} &= \mathcal{C} \times_1 \U_1 \times_2 \U_3 \times_3 \U_3,
\end{align*}
where $\mathcal{C} \in \mathbb{R}^{r_1 \times r_2 \times r_3}$ is a \emph{core tensor} and $\U_k \in \mathbb{R}^{p_k \times r_k}$ are orthonormal \emph{loading matrices} (see \cref{sec:notation} for details).  
In this paper, %assuming that each matricization of $\mathcal{S}$ is full rank (see \cref{sec:notation}), %(Let's probably not highlight this strong assumption here)
we consider estimating $\mathbf{\Pi}_k$ (i.e., the community memberships) by considering the explicit relationship between the decomposition \eqref{mmsbm} and the loading matrices $\U_k$ in its Tucker decomposition. Our main contributions are as follows:
\begin{itemize}
    \item We provide conditions for the identifiability of the model \eqref{mmsbm} (\cref{prop:identifiability}) and we relate the decomposition in \eqref{mmsbm} to the Tucker decomposition of the underlying tensor (\cref{prop:relationship} and \cref{lem:relationship}).
    \item We propose an algorithm to estimate the membership matrices $\mathbf{\Pi}_k$ obtained by combining the \emph{higher-order orthogonal iteration} (HOOI) algorithm with the corner-finding algorithm of \citet{gillis_fast_2014}, and we demonstrate a high-probability \emph{per-node} error bound for estimation of the membership matrices $\mathbf{\Pi}_k$ in the presence of heteroskedastic, subgaussian noise (\cref{thm:estimation}).
    \item To prove our main results, we develop a new $\ell_{2,\infty}$ perturbation bound for the HOOI algorithm in the presence of heteroskedastic, subgaussian noise (\cref{thm:twoinfty}) that may be of independent interest.  Our proof uses a novel leave-one-out construction that carefully preserves independence and spectral information at each iteration.  
    %\item We then specialize our main results to the tensor blockmodel, where the communities are assumed to be discrete, and we apply our $\ell_{2,\infty}$ perturbation result to obtain perfect clustering in well-conditioned model settings (\cref{thm:perfectclustering}).  A key feature of our result is that perfect clustering can occur in well-conditioned models \emph{without} using information from the other modes besides through the HOOI algorithm.
    \item We apply our algorithm to global trading data, and we find that 
%    \begin{itemize}
%        \item In the first dataset consisting of global flights by airline, we find a ``separation phenomenon'' between American airports and airlines and Chinese airports and airlines.
 %       \item 
%In the first dataset consisting of flights between US airports over time, we obtain evidence of a  ``seasonality effect'' that vanishes at the onset of the COVID-19 pandemic.
 %       \item 
 %In the second dataset consisting of global trading for different food items, we find that 
 global food trading can be grouped by region, with European countries grouped more closely together than other regions.  We also conduct simulations and analyze two additional datasets in the supplementary materials.
 %   \end{itemize}  
\end{itemize}
Our main technical result, \cref{thm:twoinfty}, relies only on spectral properties of the underlying tensor and holds for nearly optimal signal-to-noise ratio conditions such that a polynomial-time estimator exists. %Our proof is based on a careful leave-one-out construction that requires additional ``tensorial'' considerations. In \cref{sec:proofoverview} in the supplementary materials we give further details and a high-level overview of our proof techniques. 
For ease of presentation, this paper focuses on the order-three setting. Our results and methodology naturally extend to the higher-order setting, and we give an informal statement of the extension to higher-order in  \cref{sec:discussion}.

The rest of this paper is organized as follows.  In \cref{sec:relatedwork} we review related works, and in \cref{sec:notation} we set notation and review tensor algebra.  In \cref{sec:main} we provide our main estimation algorithm and present our main theoretical results, including our per-node estimation errors \textcolor{black}{and novel $\ell_{2,\infty}$ perturbation bound}, and in \cref{sec:numericalresults} we present our data analysis results.  We provide a high-level overview of the proof of our main technical result in \cref{sec:proofoverview}, and we finish in \cref{sec:discussion} with a discussion.  The supplementary materials contain additional discussion, further data analysis and simulations,  and our full proofs.

\subsection{Related Work} \label{sec:relatedwork}

%A tensor $\sin\Theta$ perturbation bound: \cite{luo2021sharp}... $\ell_{2\to q}$ norm perturbation bound: \cite{abbe2020ell_p}, tensor SVD and sparse tensor SVD \cite{zhang2019optimal}... low-rank matrix entrywise estimation of singular vectors under heteroskedastic or dependent noise \cite{agterberg2021entrywise}; inference for singular vectors \cite{yan2021inference}... 

Tensors, or multidimensional arrays, arise in problems in the sciences and engineering, and there is a need to develop principled statistical theory and methodology for these data.  Tensor data analysis techniques are closely tied to spectral methods, which have myriad applications in high-dimensional statistics \citep{chen_spectral_2021}, including in principal component analysis, spectral clustering, and as initializations for nonconvex algorithms \citep{chi_nonconvex_2019}.  With the ubiquity of spectral methods, there has also been a development of both theory and methodology for fine-grained statistical inference with spectral methods, though the existing theory is limited to specific settings, and may not apply to tensors.

Algorithms for high-order clustering have relied on convex relaxations \citep{chi_provable_2020} or spectral relaxation \citep{wu_general_2016}. Perhaps the most closely related results for high-order clustering are in \citet{han_exact_2021}, which consider both statistical and computational thresholds for perfect cluster recovery. Their proposed algorithm HLloyd is a generalization of the classical Lloyd's algorithm for K-Means clustering to the tensor setting.  %m we also consider when the perfect clustering is achievable (\cref{thm:perfectclustering}); however, our analysis is based on spectral clustering with the output of a tensor SVD algorithm. A key difference from our clustering algorithm and the algorithm in \citet{han_exact_2021} is that our algorithm \emph{ignores} the high-order structure after estimating the singular vectors, whereas the HLloyd algorithm uses the high-order structure to refine the community estimates, though our theoretical results rely on stronger assumptions on the spectral structure of the mean tensor than in \citet{han_exact_2021}.  
Similarly, \citet{luo_tensor_2022} consider the statistical and computational limits for clustering, but they focus on expected misclustering error. Unlike these previous works, our model allows for mixed memberships, \textcolor{black}{which is a more difficult estimation problem, as the membership parameters are no longer discrete}.  \textcolor{black}{Furthermore, our main theoretical results concern the output of the HOOI algorithm that is widely used for tensor singular vector estimation and may be of independent interest, whereas these previous works have primarily focused on algorithms explicitly tailored to the tensor blockmodel setting. }

%Unlike the tensor blockmodel, the tensor mixed-membership blockmodel posits that entries comprise of convex combinations of ``pure'' nodes, rendering notions like community detection no longer appropriate. 
The tensor mixed-membership blockmodel is also closely related to and inspired by the mixed-membership stochastic blockmodel proposed by \citet{airoldi_mixed_2008}. Our estimation procedure is based on studying the relationship between singular vectors and the mixed-membership matrices. This relationship was first discovered in the matrix setting as described in the first version of \citet{jin_estimating_2017}. % and our estimation algorithm is closely related to the algorithm proposed in  \citet{mao_estimating_2021}, who propose estimating mixed memberships in networks by studying the relationship between the leading eigenvectors and the membership matrix. 
In that work, the authors also consider heterogeneous degree corrections, a significantly more difficult setting for which the homogeneous setting considered herein is a special case. Furthermore, the first version of \citet{jin_estimating_2017} was the first to coin the terms ``pure nodes'' and ``vertex hunting'' in this context, both terms that we also use. %While the work \citet{jin_estimating_2017} provides foundational analyses for this problem in the matrix setting, 
Our procedure is more similar to the work \citet{mao_estimating_2021}, which focuses on the special case with homogeneous degree corrections.

Similar to both \citet{mao_estimating_2021}  and \citet{jin_estimating_2017} \textcolor{black}{we also use a vertex hunting procedure for mixed membership estimation, and % the algorithm proposed in \citet{gillis_fast_2014} for membership estimation, and 
we obtain our main results by applying newly developed sharp $\ell_{2,\infty}$ perturbation bounds for the estimated singular vectors}. \textcolor{black}{ In both these previous works, one of the primary technical challenges is to analyze the effect of Bernoulli noise (with and without degree heterogeneity) on the $\ell_{2,\infty}$ perturbation of the empirical eigenvectors.  In contrast, the primary technical challenge of our work is to develop new arguments that yield sufficiently strong $\ell_{2,\infty}$ perturbation bounds for the HOOI algorithm, which is a nonconvex algorithm for tensor singular vector estimation. While the inferential goal in both these works and ours is to obtain estimation error rates, our major contribution is in developing new analysis tools that can provide these error rates. 
}

%However, unlike \citet{mao_estimating_2021}, our analysis requires studying the output of the higher-order orthogonal iteration (HOOI) algorithm, whereas \citet{mao_estimating_2021} need only consider the $\ell_{2,\infty}$ perturbation of the leading eigenvectors.  
%Nearly optimal perturbation bounds for the matrix mixed-membership blockmodel \textcolor{black}{without degree heterogeneity} have also been obtained in \citet{xie_entrywise_2022},  \textcolor{black}{and with degree heterogeneity in \citet{jin_estimating_2017}}, and we provide a comparison of our results to \citet{mao_estimating_2021}, \citet{xie_entrywise_2022}, and \citet{jin_estimating_2017}, demonstrating the effect of  higher-order ``tensorial'' structure on estimation accuracy. Our $\ell_{2,\infty}$ perturbation bounds are \emph{not} simply extensions of previous bounds for matrices, and instead require additional novel theoretical techniques; see \cref{sec:proofoverview} for details.  

Considering general perturbation results for tensors, \citet{cai_nonconvex_2022} focuses on symmetric tensors of low CP rank, and they consider the performance of their noisy tensor completion algorithm obtained via vanilla gradient descent, and they prove entrywise convergence guarantees and $\ell_{2,\infty}$ perturbation bounds.  Our analysis differs in a few key ways: first, we consider tensors of low Tucker rank, which generalizes the CP rank; next, our analysis holds for \emph{asymmetric} tensors under general subgaussian noise, and, perhaps most crucially, we analyze the HOOI algorithm, which can be understood as power iteration (as opposed to gradient descent).  Therefore, while the results in \citet{cai_nonconvex_2022} may be qualitatively similar, the results are not directly comparable.  Similarly, \citet{wang_implicit_2021} consider the entrywise convergence of their noiseless tensor completion algorithm for symmetric low Tucker rank tensors; our analysis is somewhat similar, but we explicitly characterize the effect of noise, which is a primary technical challenge in the analysis.

Besides \citet{cai_nonconvex_2022} and \citet{wang_implicit_2021}, entrywise perturbation bounds for tensors are still lacking in general, though there are several generalizations of classical matrix perturbation bounds to the tensor setting.    
A sharp (deterministic) $\sin\Theta$ upper bound for tensor SVD was obtained in \citet{luo_sharp_2021}, and  \citet{auddy_perturbation_2022} consider perturbation bounds for orthogonally decomposable tensors. % \citet{zhang_optimal_2019} considered tensor denoising when some of the modes have sparse factors.
\citet{zhang_tensor_2018} established statistical and computational limits for tensor SVD with Gaussian noise; our work builds off of their analysis by analyzing the tensor SVD algorithm initialized with diagonal deletion. % Finally, \citet{richard2014statistical} and \citet{auddy_estimating_2022} also consider estimating low CP-rank tensors under Gaussian and heavy-tailed noise respectively.

Our main $\ell_{2,\infty}$ bound is also closely related to a series of works developing fine-grained entrywise characterizations for eigenvectors and singular vectors, for which a general survey can be found in  %such as \citet{abbe_entrywise_2020,abbe_ell_p_2022,agterberg_entrywise_2022,agterberg_entrywise_2022-1,cape_two--infinity_2019,cape_signal-plus-noise_2019,cai_subspace_2021,koltchinskii2015perturbation,yan_inference_2021}, for example. The monograph  
\citet{chen_spectral_2021}. % gives an introduction to spectral methods from a statistical point of view, with the final chapter focusing on entrywise bounds and distributional characterizations for estimates constructed from eigenvectors.  
Several works on entrywise singular vector analyses have also applied their results to tensor data, such as \citet{xia_sup-norm_2019,cai_subspace_2021}, though these analyses often fail to take into account the additional structure arising in tensor data.

From a technical point of view, our work uses the ``leave-one-out'' analysis technique, first pioneered for entrywise eigenvector analysis in \citet{abbe_entrywise_2020}, though the method had been used previously to analyze nonconvex algorithms \citep{chi_nonconvex_2019,ma_implicit_2020},  M-estimators \citep{el_karoui_robust_2013}, among others \citep{zhong_near-optimal_2018}.  The leave-one-out technique for singular vectors and eigenvectors have been further refined to analyze large rectangular matrices \citep{cai_subspace_2021}, kernel spectral clustering \citep{abbe_ell_p_2022}, to obtain distributional guarantees for spectral methods \citep{yan_inference_2021}, and to study the performance of spectral clustering \citep{zhang_leave-one-out_2022}.  %A comprehensive survey on the use of this technique can be found in \citet{chen_spectral_2021}.  
\textcolor{black}{Unlike these previous works, since the HOOI algorithm is not equivalent to a gradient descent procedure, our analysis requires several novel considerations that bridge the gap between analyzing both nonconvex algorithms and spectral methods.}
%Our work bridges the gap between analyzing nonconvex algorithms and analyzing spectral methods: HOOI performs a low-dimensional SVD at each iteration, thereby requiring both singular vector analyses and algorithmic considerations.  %Our proof technique also demonstrates an implicit regularization effect in tensor SVD -- provided the initialization is sufficiently incoherent, tensor SVD maintains this level of incoherence at each iteration. 
Finally, our proof of the spectral initialization also slightly improves upon the bound in \citet{cai_subspace_2021} (for the singular vectors of rectangular matrices) by a factor of the condition number; see \cref{thm:spectralinit_twoinfty}.

\subsection{Notation and Preliminaries} \label{sec:notation}
For two functions $f$ and $g$ viewed as functions of some increasing index $n$, we say $f(n) \lesssim g(n)$ if there exists a uniform constant $C> 0$ such that $f(n) \leq C g(n)$, and we say $f(n) \asymp g(n)$ if $f(n) \lesssim g(n)$ and $g(n) \lesssim f(n)$.  We write $f(n) \ll g(n)$ if $f(n)/g(n) \to 0$ as the index $n$ increases.  We also write $f(n) = O(g(n))$ if $f(n) \lesssim g(n)$, and we write $f(n) = \tilde O(g(n))$ if $f(n) = O(g(n) \log^c(n))$ for some value $c$ (not depending on $n$).

We use bold letters $\mathbf{M}$ to denote matrices, we let $\mathbf{M}_{i\cdot}$ and $\mathbf{M}_{\cdot j}$ denote its $i$'th row and $j$'th column, both viewed as column vectors, and we let $\mathbf{M}\t$ denote its transpose. We denote $\|\cdot\|$ as the spectral norm for matrices and the Euclidean norm for vectors, and we let $\|\cdot\|_F$ denote the Frobenius norm.  We let $e_i$ denote the $i$'th standard basis vector and $\mathbf{I}_k$ denote the $k\times k$ identity.  For a matrix $\mathbf{M}$ we let $\|\mathbf{M}\|_{2,\infty} = \max_i \| e_i\t \mathbf{M} \|$.  For two orthonormal matrices $\U$ and $\mathbf{V}$ satisfying $\U\t \U = \mathbf{V}\t \mathbf{V} = \mathbf{I}_r$, we let $\|\sin\Theta(\U,\mathbf{V})\|$ denote their $\sin\Theta$ (spectral) distance; i.e., $\|\sin\Theta(\U,\mathbf{V})\| = \| (\mathbf{I}_r - \U \U\t) \mathbf{V} \|$. For an orthonormal matrix $\U$ we let $\U_{\perp}$ denote its orthogonal complement; that is, $\U_{\perp}$ satisfies $\U_{\perp}\t \U = 0$.  We denote the $r\times r$ orthogonal matrices as $\mathbb{O}(r)$.  

For multi-indices $\br = (r_1,r_2,r_3)$ and $\bp = (p_1,p_2,p_3)$, we let $r_{-k} = \prod_{j\neq k} r_j$, and we define $p_{-k}$ similarly.  We also denote $p_{\min} = \min p_k$ and $p_{\max} = \max p_k$, with $r_{\min}$ and $r_{\max}$ defined similarly.  A tensor $\mathcal{T} \in \mathbb{R}^{p_1 \times p_2 \times p_3}$ is a multidimensional array.  We let calligraphic letters $\mathcal{T}$ denote tensors, except for the letter $\mathcal{M}$, for which $\mathcal{M}_k(\mathcal{T})$ denotes its \emph{matricization} along the $k$'th mode; i.e., $\mathcal{M}_k(\mathcal{T})$ satisfies 
\begin{align*}
    \mathcal{M}_k(\mathcal{T}) \in \mathbb{R}^{p_k \times p_{-k}}; \qquad \big( \mathcal{M}_k(\mathcal{T}) \big)_{i_k,j} = \mathcal{T}_{i_1 i_2 i_3}; \qquad j = 1+ \sum_{\substack{l=1\\l\neq k}}^d \Bigg\{ (i_l-1) \prod_{\substack{m=1\\m\neq k}} p_m \Bigg\},
\end{align*}
for $1 \leq i_l \leq p_l$, $l=1,2,3$.  See \citet{kolda_tensor_2009} for more details on matricizations.  We also reserve the calligraphic letter $\mathcal{P}$ for either permutations or projections, as will be clear from the context.  For an orthonormal matrix $\U$, we let $\mathcal{P}_{\U}$ denote its corresponding orthogonal projection $\mathcal{P}_{\U} = \U \U\t$.

We denote the multilinear rank of a tensor $\mathcal{T}$ as a tuple $\br = (r_1,r_2,r_3)$, where $r_k$ is the rank of the $k$'th matricization of $\mathcal{T}$.  A tensor $\mathcal{T}$ of rank $\br$ has a Tucker decomposition $\mathcal{T} = \mathcal{C} \times_1 \U_1 \times_2 \U_3 \times_3 \U_3,$
where $\mathcal{C} \in \mathbb{R}^{r_1 \times r_2 \times r_3}$ is the core tensor and $\U_k$ are the $p_k \times r_k$ left singular vectors of the matrix $\mathcal{M}_k(\mathcal{T})$. Here the mode $1$ product of a tensor $\mathcal{T} \in \mathbb{R}^{p_1 \times p_2 \times p_3}$ with a matrix $\U \in \mathbb{R}^{p_1 \times r_1}$ is denoted by $\mathcal{T} \times_k \U\t \in \mathbb{R}^{r_1 \times p_2 \times p_3}$ and is given by
\begin{align*}
    (\mathcal{T} \times_1 \U\t)_{j i_2 i_3} &= \sum_{i_1=1}^{p_k} \mathcal{T}_{i_1 i_2 i_3} \U_{i_1j}.
\end{align*}
The other mode-wise multiplications are defined similarly.
For two matrices $\U$ and $\mathbf{V}$, we denote $\U \otimes \mathbf{V}$ as their Kronecker product.  For a tensor $\mathcal{S} \in \mathbb{R}^{r_1 \times r_2 \times r_3}$ and matrices $\U_k$ of appropriate sizes, the following identity holds (see e.g., \citet{kolda2006multilinear}):
\begin{align*}
    \mathcal{M}_1( \mathcal{S} \times_1 \U_1 \times_2 \U_2 \times_3 \U_3) &= \U_1 \mathcal{M}_1(\mathcal{S}) \big( \U_{2}\t \otimes \U_3\t \big),
\end{align*}
with similar identities holding for the other modes.
For a matrix $\mathbf{M}$ we write $\SVD_{r}(\mathbf{M})$ to denote the leading $r$ singular vectors of $\mathbf{M}$.  Concretely, for a tensor of Tucker rank $\br = (r_1,r_2,r_3)$, it holds that $\U_k = \SVD_{r_k}(\mathcal{M}_k(\mathcal{T}))$.

For a tensor $\mathcal{T}$ with Tucker decomposition $\mathcal{T} = \mathcal{S} \times_1 \U_1 \times_2 \U_2 \times_3 \U_3$, we denote its incoherence parameter $\mu_0$ as the smallest number such that
\begin{align*}
\max_k \sqrt{\frac{p_k}{r_k}} \|\U_k \|_{2,\infty} \leq \mu_0.
\end{align*}
For a nonsquare matrix $\mathbf{M}$ of rank $r$, we let $\lambda_{\min}(\mathbf{M})$ denote its smallest nonzero singular value, and we denote its singular values as $\lambda_k(\mathbf{M})$.  For a square matrix $\mathbf{M}$, we let $\lambda_{\min}(\mathbf{M})$ denote its smallest nonzero eigenvalue and $\sigma_{\min}(\mathbf{M})$ denote its smallest nonzero singular value, with other eigenvalues and singular values defined similarly. For a tensor $\mathcal{T}$ of rank $\br = (r_1,r_2,r_3)$, we let $\lambda_{\min}(\mathcal{T})$ denote its smallest nonzero singular value along all of its matricizations; that is $\lambda_{\min}(\mathcal{T}) = \min_{k} \lambda_{\min}(\mathcal{M}_k(\mathcal{T})).$
 We let the condition number of a tensor $\mathcal{T}$ be denoted as $\kappa$, defined as $\kappa \coloneqq \max_k \frac{\|\mathcal{M}_k(\mathcal{T})\|}{\lambda_{\min}(\mathcal{M}_k(\mathcal{T}))}.$
Finally, for a random variable $X$, we let $\|X\|_{\psi_2}$ denote its subgaussian Orlicz norm; that is, $\|X\|_{\psi_2} = \inf\{t >0: \E \exp(X^2/t^2) \leq 2\}.$
See Chapter 2 of \citet{vershynin_high-dimensional_2018} for more details on Orlicz norms and subgaussian random variables. 

%%%%%%%%%%%%%%%%%%%%%%
\section{Main Results}\label{sec:main}
%%%%%%%%%%%%%%%%%%%%%%

We now describe our model in detail.  Assume that one observes
\begin{align*}
    \mathcal{\widehat{T}} &= \mathcal{T} + \mathcal{Z} \in \mathbb{R}^{p_1 \times p_2 \times p_3},
\end{align*}
where $\mathcal{Z}$ consists of independent mean-zero subgaussian noise satisfying $\|\mathcal{Z}_{ijk}\|_{\psi_2} \leq \sigma$ (note that $\mathcal{Z}$ is not assumed to be homoskedastic).  Assume  further that the underlying tensor $\mathcal{T}$ admits the following factorization:
\begin{align}
    \mathcal{T} &= \mathcal{S} \times_1 \mathbf{\Pi}_1 \times _2 \mathbf{\Pi}_2 \times_3 \mathbf{\Pi}_3, \label{tensormmsbm}
\end{align}
where $\mathbf{\Pi}_k \in [0,1]^{p_k \times r_k}$ is a membership matrix with rows that sum to one, and $\mathcal{S} \in \mathbb{R}^{r_1 \times r_2 \times r_3}$ is a \emph{mean tensor}.  The matrices $\mathbf{\Pi}_k$ can be interpreted as follows: $(\mathbf{\Pi}_k)_{i_kl}$  denotes how much the $i_k$'th node along the $k$'th mode belongs to community $l$. %When $\mathbf{\Pi}_k$ is a $\{0,1\}$ matrix, this definition reduces to the tensor blockmodel.  
\textcolor{black}{It is possible to extend the model for symmetric cases by imposing symmetry constraints on $\mathcal{Z}$ and requiring that $\mathbf{\Pi}_k$'s are the same for some $k$.  However, we focus on the fully asymmetric setting for concreteness.}

For a node $i_k$ along mode $k$, we say $i_k$ is a \emph{pure node} if $(\mathbf{\Pi}_{k}\big)_{i_k\cdot} \in \{0,1\}^{r_k}$; that is, exactly one entry of the $i_k$'th row of $\mathbf{\Pi}_k$ is nonzero (and hence equal to one).  Intuitively, a pure node is a node that belongs to one and only one community.  Observe that if all nodes are pure nodes, then one recovers the tensor blockmodel.  As in the matrix setting \citep{mao_estimating_2021}, the existence of pure nodes is intimately related to the identifiability of the model \eqref{mmsbm}.  The following result establishes the identifiability of the tensor mixed-membership blockmodel when $\mathcal{S}$ is rank $r_k$ along each mode and there is a pure mode for each community along each direction. We note that it is also possible to establish identifiability in the case that $\mathcal{S}$ has some mode with a rank less than $r_k$, but this is beyond the scope of this paper. 

\begin{proposition}[Identifiability] \label{prop:identifiability}
Consider the model \eqref{tensormmsbm}. Assume that 
%Assume the following two conditions hold.
%Suppose that $r_k \leq r_{-k}$ for all $k$ and that: %(Comment: this condition is not needed if you assume each matricization of S is rank r_k.)
%\begin{itemize}
%    \item 
each matricization of $\mathcal{S}$ is rank $r_k$ respectively with $r_k \leq r_{-k}$, and %;
%    \item 
for each mode $k$, there is at least one pure node for each community.
%\end{itemize}
Then if there exists another set of parameters $\mathcal{S}', \mathbf{\Pi}_1', \mathbf{\Pi}_2',$ and $\mathbf{\Pi}_3'$ such that $\mathcal{T} = \mathcal{S}' \times_1 \mathbf{\Pi}_1'\times_2 \mathbf{\Pi}_2'\times_3 \mathbf{\Pi}_3'$ it must hold that $\mathbf{\Pi}_k = \mathbf{\Pi}_k' \mathcal{P}_k$, where $\mathcal{P}_k$ is an $r_k \times r_k$ permutation matrix and $\mathcal{S} = \mathcal{S}' \times_1 \mathcal{P}_1 \times_2 \mathcal{P}_2 \times_3 \mathcal{P}_3$.  

Next, suppose that each matricization of $\mathcal{S}$ is rank $r_k$ respectively with $r_k \leq r_{-k}$. Suppose that $\mathbf{\Pi}_k$ is identifiable up to permutation; i.e., any other $\mathbf{\Pi}_k'$ generating the same tensor $\mathcal{T}$ must satisfy $\mathbf{\Pi}_k' = \mathbf{\Pi}_k \mathcal{P}_k$ for some permutation $\mathcal{P}_k$ and $\mathcal{S}'= \mathcal{S} \times_k \mathcal{P}_k$.  Then there must be at least one pure node for each community along mode $k$. 
\end{proposition}

Therefore, we see that when the underlying tensor is full rank and there is at least one pure node for each community, the model will be identifiable up to permutation of the communities. % A similar assumption for the identifiability of the mixed-membership (network) blockmodel was obtained in \citet{mao_estimating_2021}.  %In contrast to the setting considered in \citet{mao_estimating_2021}, there are three separate membership matrices $\mathbf{\Pi}_k$ that we seek to estimate.  

In order to describe our estimation procedure in the following subsection, we provide the following crucial observation relating the tensor mixed-membership blockmodel to its Tucker factorization.
\begin{proposition}
 \label{prop:relationship}
 Suppose $\mathcal{T}$ is a tensor mixed-membership blockmodel of the form in \eqref{tensormmsbm}, and suppose that each matricization of $\mathcal{S}$ is rank $r_k$ respectively with $r_k \leq r_{-k}$ for each $k$.  Suppose further that there is a pure node for each community along each mode.
Let $\mathcal{T} = \mathcal{C} \times_1 \U_1 \times_2 \U_3 \times_3 \U_3$ denote its rank $(r_1,r_2,r_3)$ Tucker factorization.  Then it holds that $\U_k = \mathbf{\Pi}_k \U_k\pure$,
where $\U_k\pure \in \mathbb{R}^{r_k \times r_k}$ is rank $r_k$ and contains the rows of $\U_k$ corresponding to pure nodes. \end{proposition}

Consequently, \cref{prop:relationship} shows that the singular vectors $\U_k$ of the underlying tensor $\mathcal{T}$ belong to a simplex with vertices given by $\U_k\pure$, or the rows of $\U_k$ corresponding to pure nodes.  The connection between the membership matrix $\mathbf{\Pi}_k$ and the singular vectors $\U_k$ has previously been considered in the matrix setting in \citet{mao_estimating_2021}, \textcolor{black}{which is a special case of the results in  \citet{jin_estimating_2017}}.% and their algorithm is based on a corner-finding algorithm of \citet{gillis_fast_2014}.  Similar to \citet{mao_estimating_2021}, we will repurpose the algorithm for mixed-membership estimation.

\subsection{Estimation Procedure}
%Our estimation algorithm is based on the following crucial observation relating the tensor mixed-membership blockmodel to its Tucker factorization.  %  In what follows, recall that $\mu_0$ is the \emph{incoherence} of the tensor $\mathcal{T}$ (see \cref{sec:notation}). 
We now detail our estimation procedure. In light of \cref{prop:relationship}, the singular vectors of the tensor $\mathcal{T}$ and the matrices $\mathbf{\Pi}_k$ are intimately related via the matrix $\U_k\pure$.  Therefore, given estimated tensor singular vectors $\uhat_k$ obtained from the observed tensor $\mathcal{\hat T}$, we propose to estimate the pure nodes by applying the corner-finding algorithm of \citet{gillis_fast_2014} to the rows of $\uhat_k$ to obtain estimated  pure nodes.  Consequently, in order to run the corner-finding algorithm, we will require the estimated tensor singular vectors $\uhat_k$.

%similar to the procedure considered in \citet{mao_estimating_2021}. %similar to \citet{mao_estimating_2021}, we will repurpose the corner-finding algorithm of \citet{gillis_fast_2014} for mixed-membership estimation.  In order to 
%Our procedure will be based on applying this algorithm to the estimated singular vectors of the observed tensor $\mathcal{\widehat{T}}$.  
However, unlike the matrix SVD, tensor SVD is not well-defined in general.  For low Tucker rank tensors, a  common algorithm to estimate the singular vectors of tensors is via the higher-order orthogonal iteration (HOOI) algorithm \citep{de2000best}. Under the specific Gaussian additive model, this algorithm has been analyzed and minimax optimal error bounds in $\sin\Theta$ distances were established in \cite{zhang_tensor_2018}, %. %It was also demonstrated in \citet{zhang_tensor_2018} that there exists a statistical/computational \emph{gap} for tensor SVD; namely, when $p_k \asymp p$, $\lambda/\sigma \ll p^{3/4}$, and $r \lesssim p^{1/2}$, under a widely believed computational-hardness conjecture, no computationally feasible estimator exists. 
%However, in the regime $\lambda/\sigma \gtrsim p^{3/4}$, 
%It has been shown that in HOOI achieves the minimax lower bound \citep{zhang_tensor_2018}, 
which is the main impetus behind using HOOI to estimate the singular vectors. However, a major technical challenge in analyzing our estimator is in providing a fine-grained understanding of the output of HOOI for tensor SVD in order to ensure that the correct pure nodes are found.  \textcolor{black}{Therefore, as a major theoretical contribution of this paper, we analyze the row-wise error of HOOI, which is what allows us to demonstrate the statistical properties of our estimation procedure.} %See \cref{sec:twoinfty} for further details. 
\cref{al:tensor-power-iteration} includes full pseudo-code for HOOI. 

%Therefore, at a minimum we will require $\lambda/\sigma \gtrsim p^{3/4}$.  In fact, Assumption \ref{assumption:signalstrength} together with \cref{lem:relationship} shows that we are implicitly requiring that $\lambda/\sigma \gtrsim \kappa \sqrt{\log(p)} p^{3/4}$ in this regime, which is therefore optimal up to the factor of $\kappa \sqrt{\log(p)}$.  

In order to initialize HOOI, since we do not assume homoskedastic noise we propose initializing via diagonal-deletion; namely, we define  $\uhat_k^{(0)}$ as the leading $r_k$ eigenvectors of the matrix $    \Gamma\left[\mathcal{M}_k\big(\mathcal{\widehat{T}}\big) \mathcal{M}_k\big(\mathcal{\widehat{T}}\big)\t \right],$
where $\Gamma(\cdot)$ is the \emph{hollowing operator}: for a square matrix $\mathbf{M}$, $\Gamma(\mathbf{M})$ sets its diagonal entries to zero, i.e.,
$$[\Gamma(\mathbf{M})]_{ij} = \left\{\begin{array}{ll}
   [\mathbf{M}]_{ij}  &  i\neq j;\\
   0  & i=j.
\end{array}\right.$$
%This procedure has previously been studied; see \citet{cai_subspace_2021} for an analysis of this procedure in $\ell_{2,\infty}$ norm, and we slightly refine this result for our setting in \cref{thm:spectralinit_twoinfty}.  
%\cref{al:dd} provides the full pseudo-code for this initialization procedure.

We now have all the pieces to our estimation procedure.  First, %using the initializations $\uhat_k^{(0)}$, we plug these into \cref{al:tensor-power-iteration} to 
we estimate the tensor singular vectors via \cref{al:tensor-power-iteration}.  Next, given the estimates $\uhat_k$ for $k =1,2,3$, we obtain the index sets $J_k$ containing the estimated pure nodes via the algorithm proposed in \citet{gillis_fast_2014}, and we set $\uhat_k\pure :=\big(\uhat_k\big)_{J_k\cdot}$. Finally, we estimate $\mathbf{\widehat\Pi}_k$ via $\mathbf{\widehat\Pi}_k = \uhat_k \big(\uhat_k\pure \big)\inv$.  The full procedure is stated in \cref{al:spamm}. In practice we have found that  there are occasionally negative or very small values of $\mathbf{\hat \Pi}_k$; therefore, our actual implementation thresholds small values and re-normalizes the rows of $\mathbf{\hat \Pi}_k$, though the theory discussed in the following sections will be for the implementation without this additional step.

\textcolor{black}{\begin{remark}[Other Vertex Hunting Procedures]
Our procedure is not restricted to using the algorithm in \citet{gillis_fast_2014}.  For example, in the work \citet{jin_estimating_2017} the authors suggest several different vertex hunting algorithms that attempt to identify the matrix $\U_k^{(\mathrm{pure})}$.  For concreteness we have focused on the successive projection algorithm, but as can be seen from the proof of our main result, any algorithm that is sufficiently robust to row-wise deviations will suffice.  In general, the vertex hunting procedure can be treated as a ``plug-in" step. 
\end{remark}}

%Paragraph on estimating the singular vectors.  Since the Tucker rank of the Tensor MMSBM is of order $(r_1,r_2,r_3)$ we estiamte $\U_k$ using tensor power iteration, or higher-order orthogonal iteration \citep{de2000best}.   

%Different from the matrix singular value decomposition, there is no unified procedure for tensor singular value decomposition, while the best low-rank approximation may be NP-hard to even approximate in general \citep{hillar2013most}. We instead consider the outcome of the computationally feasible algorithm, \emph{tensor power iteration},  which can be seen as the analog of matrix singular value decomposition, as detailed in Algorithm \ref{al:tensor-power-iteration}. This algorithm can be seen as a general version of higher-order orthogonal iteration (HOOI) \cite{de2000best}. 
\begin{algorithm}[t]
	\caption{Higher-Order Orthogonal Iteration (HOOI)}
	\begin{algorithmic}[1]
		\State Input: $\mathcal{\widehat{T}}\in \mathbb{R}^{p_1 \times p_2 \times p_3}$, Tucker rank $\br = (r_1, r_2,r_3)$.
 		\State For $k \in \{2,3\},$ set $\uhat_k^{(0)}$ as the leading $r_k$ eigenvectors of the matrix %$\mathbf{\widehat{G}}$, with
		%\begin{align*}
		$\mathbf{\widehat{G}} \coloneqq 
	 \Gamma\big(\mathcal{M}_k ( \mathcal{\widehat{T}}) \mathcal{M}_k(\mathcal{\widehat{T}})\t \big)$ where $\Gamma(\cdot)$ is the \emph{hollowing operator} that sets the diagonal to zero;
		\While  {$t < t_{\max}$}
		\State  Let $t = t+1$. For {$k=1,2,3$} set $\uhat_k^{(t)} = 
\SVD_{r_k}\left(\mathcal{M}_k\left(\mathcal{\widehat{T}}\times_{k' < k} (\uhat_{k'}^{(t)})^\top \times_{k'>k} (\uhat_{k'}^{(t-1)})^\top\right)\right).$
		\EndWhile
		\State Output: $\uhat_k^{(t_{\max})}$.
	\end{algorithmic}\label{al:tensor-power-iteration}
\end{algorithm}

% \begin{algorithm}[t]
% 	\caption{Diagonal-Deletion Initialization}
% 	\begin{algorithmic}[1]
% 		\State Input: $\mathcal{\widehat{T}}\in \mathbb{R}^{p_1 \times p_2 \times p_3}$, Tucker rank $\br = (r_1, r_2,r_3)$.
		
% 		\State Output: $\uhat_k^{(0)}$.
% 	\end{algorithmic}\label{al:dd}
% \end{algorithm}

\begin{algorithm}[t]
	\caption{Successive Projection Algorithm for Tensor Mixed-Membership Estimation}
	\begin{algorithmic}[1]
		\State Input: estimated loading matrices $\{\uhat_k\}_{k=1}^{3}$ via \cref{al:tensor-power-iteration}.
		\For {$k = 1, 2, 3$}
		\State $\mathbf{R} \coloneqq \uhat_k$, $J_k = \{\}, j =1$
		 \While {$\mathbf{R} \neq 0_{n\times r_k}$ and $j \leq r_k$}
	\State Set $j^* = \argmax \| e_j\t \mathbf{R} \|^2.$ If there are ties, set $j^*$ as the smallest index.
	\State Set $\mathbf{v}_j \coloneqq e_{j^*}\t \mathbf{R}$,  $\mathbf{R} = \mathbf{R} \big( \mathbf{I}_{r_k} - \frac{\mathbf{v}_j \mathbf{v}_j\t}{\|\mathbf{v}_j\|^2} \big)$, $J_k = J_k \cup \{ j^* \}$, $j = j+1$
	\EndWhile
	\State Define $\mathbf{\widehat \Pi}_k \coloneqq \uhat_k ( \uhat_k[J_k,\cdot] )\inv$
		\EndFor
	\State Output: three membership matrices $\{\mathbf{\widehat \Pi}_k\}_{k=1}^{3}$.
	\end{algorithmic}\label{al:spamm}
\end{algorithm}

\subsection{Technical Assumptions}
To develop the theory for our estimation procedure, we will require several assumptions.  
 In light of \cref{prop:identifiability} and to induce regularity into the community memberships, we impose the following assumption.    
\begin{assumption}[Regularity and Identifiability] \label{assumption:regularity}
The community membership matrices $\mathbf{\Pi}_k$ satisfy
\begin{align*}
   \frac{p_k}{r_k} \lesssim  \lambda_{\min} \bigg( \mathbf{\Pi}_k\t \mathbf{\Pi}_k \bigg) \leq \lambda_{\max} \bigg( \mathbf{\Pi}_k\t \mathbf{\Pi}_k \bigg) \lesssim \frac{p_k}{r_k}.
\end{align*}
In addition, each matricization of $\mathcal{S}$ is rank $r_k$ respectively, and there is at least one pure node for each community for every mode.
\end{assumption}
The condition above implies that each community is approximately the same size.  When $\mathcal{T}$ is a tensor blockmodel, the matrix $\mathbf{\Pi}_k\t \mathbf{\Pi}_k$ is a diagonal matrix with diagonal entries equal to the community sizes; Assumption \ref{assumption:regularity} states then that the community sizes are each of order $p_k/r_k$, which is a widely used condition in the literature on clustering \citep{loffler_optimality_2021,han_exact_2021,hu_multiway_2022}.

Tensor SVD is feasible only with certain signal strength \citep{zhang_tensor_2018}. In order to quantify the magnitude of the signal strength, we introduce an assumption on the signal-to-noise ratio (SNR), as quantified in terms of singular values of $\mathcal{S}$ and maximum variance $\sigma$.
\begin{assumption}[Signal Strength] \label{assumption:signalstrength} The smallest singular 
value of $\mathcal{S}$,     $\Delta=\lambda_{\min}( \mathcal{S})$, satisfies %{\red (What's the definition of $p$ below? I couldn't find the meaning of $p$.)} %\textcolor{ForestGreen}{$p$ is $p_{\max}$.  It is in the notation section, but I can change it to be more clear throughout.}
\begin{align*}
    \frac{\Delta^2}{\sigma^2} \gtrsim \frac{\kappa^2 p_{\max}^2 \log(p_{\max}) r_1 r_2 r_3}{p_1 p_2 p_3 p_{\min}^{1/2}}.
\end{align*}
\end{assumption}
Here $\kappa$ denotes the condition number of $\mathcal{S}$. When $p_k \asymp p$ and $r_k = O(1)$,  Assumption \ref{assumption:signalstrength} is equivalent to the assumption $\frac{\Delta^2}{\sigma^2} \gtrsim \frac{\kappa^2 \log(p)}{p^{3/2}}.$ 

\begin{remark}[Comparison to Prior Works]
    In the \textcolor{black}{(subgaussian)} tensor blockmodel setting, \textcolor{black}{which assumes discrete memberships}, \citet{han_exact_2021} define the signal-strength parameter
\begin{align*}
    \widetilde\Delta^2 &\coloneqq \min_k \min_{i\neq j} \| \big(\mathcal{M}_k(\mathcal{S})\big)_{i\cdot} - \big(\mathcal{M}_k(\mathcal{S})\big)_{j\cdot} \|^2;
\end{align*}
i.e., the worst case row-wise difference between any two rows of each matricization of $\mathcal{S}$. % \citet{han_exact_2021} explicitly consider settings where $\mathcal{S}$ is rank degenerate; however, 
If one further assumes that $\mathcal{S}$ is rank $(r_1,r_2,r_3)$, then it is straightforward to check that both $\widetilde{\Delta}$ and $\Delta$ coincide up to a factor of the condition number. %, since 
%\begin{align*}
%     \|\big(\mathcal{M}_k(\mathcal{S})\big)_{i\cdot} - \big(\mathcal{M}_k(\mathcal{S})\big)_{j\cdot} \|^2 &= (e_i - e_j)\t \big(\mathcal{M}_k(\mathcal{S})\big) \big(\mathcal{M}_k(\mathcal{S})\big)\t   ( e_i - e_j) \geq \| e_i - e_j \|^2 \Delta^2 \asymp \Delta^2,
%\end{align*}
%and
%\begin{align*}
%      \|\big(\mathcal{M}_k(\mathcal{S})\big)_{i\cdot} - \big(\mathcal{M}_k(\mathcal{S})\big)_{j\cdot} \|^2 &= (e_i - e_j)\t \big(\mathcal{M}_k(\mathcal{S})\big)\big(\mathcal{M}_k(\mathcal{S})\big)\t   ( e_i - e_j) \leq \| e_i - e_j \|^2 \kappa^2 \Delta^2 \asymp \kappa^2 \Delta^2.
%\end{align*}
%Under the assumption that $\mathcal{S}$ is full rank along each matricization, 
\citet{han_exact_2021} demonstrate that the condition $\frac{\Delta^2}{\sigma^2} \gtrsim \frac{1}{p^{3/2}}$ is required to obtain perfect cluster recovery in polynomial time if the number of cluster centroids is assumed constant and $p_k \asymp p$.  \textcolor{black}{In contrast, our condition is slightly stronger by a factor of $\kappa^2 \log(p)$; however, our model permits mixed memberships, which is in general a more challenging estimation problem due to the ``continuous'' structure in the model.}
%Therefore, our assumption that $\frac{\Delta^2}{\sigma^2} \gtrsim \frac{\kappa^2 \log(p)}{p^{3/2}}$ is optimal up to logarithmic factors and factors of $\kappa$ in order for a polynomial time estimator to achieve exact community detection (albeit in the simple setting that $\mathcal{S}$ is full rank along each mode).  However, in contrast to \citet{han_exact_2021}, our model permits mixed memberships, so the two conditions are not directly comparable. 

\textcolor{black}{Our signal strength condition can also be compared to similar conditions that arise in the analysis of hypergraphs and multilayer networks, which, while not covered by our results, are closely related. In \citet{ke_community_2020} who study estimation in the degree-corrected hypergraph blockmodel with sparse Bernoulli noise, assuming certain regularity conditions and a sufficiently warm initialization, the authors assume that $\mathrm{SNR} \gg \frac{\log(p)}{p}$, whereas our condition translates to the stronger assumption $\mathrm{SNR} \gg \frac{\sqrt{\log(p)}}{p^{3/4}}$.  The authors further propose an initialization that only requires the condition $\mathrm{SNR} \gg \frac{\log(p)}{p}$.  However, the existence of such an initialization is due to the underlying symmetry of the hypergraph, something that does not hold in our setting.  
}

\textcolor{black}{Finally, our signal strength condition is also related to that of \citet{jing_community_2021}, who study community detection in mixture multilayer networks, which is a form of tensor blockmodel with additional symmetry along the first mode and sparse Bernoulli noise.  Assuming that the number of layers is of order $p$, provided one has a sufficiently warm initialization, their condition translates to the assumption $\mathrm{SNR} \gg \frac{\log^2(p)}{p}$, which is again weaker than our condition.  However, in order to obtain such a warm initialization they have to impose additional assumptions that leverage the symmetry along the first mode, which does not hold in our setting.  
}

\end{remark}

Finally, our analysis relies heavily on the following lemma relating the signal strength parameter $\Delta$ to the smallest singular value of the core tensor $\mathcal{C}$ in the Tucker decomposition of $\mathcal{T}$. 
\begin{lemma} \label{lem:relationship}
Let $\mathcal{T}$ be a Tensor Mixed Membership  Blockmodel  of the form \eqref{tensormmsbm}, let $\mathcal{T} = \mathcal{C} \times_1 \U_1 \times_2 \U_2 \times_3\U_3$ denote its Tucker decomposition, and let $\lambda = \lambda_{\min}(\mathcal{C})$ denote its smallest singular value.  Suppose further that  Assumptions \ref{assumption:regularity} and \ref{assumption:signalstrength} hold with $r_k \leq r_{-k}$ for each $k$. %, and that each matricization $\mathcal{S}$ is rank $r_k$ respectively.  
Then it holds that 
\begin{align*}
    \lambda &\asymp \Delta \frac{(p_1p_2p_3)^{1/2}}{(r_1 r_2 r_3)^{1/2}}; \quad
    \mu_0 = O(1),
\end{align*}
\textcolor{black}{where $\mu_0$ is the incoherence parameter of $\mathcal{T}$.}
Furthermore, $\U_k\pure (\U_k\pure)\t = \big(\mathbf{\Pi}_k\t \mathbf{\Pi}_k\big)\inv$.  
\end{lemma}

\textcolor{black}{
\begin{remark}[Relation to Tensor Subspace Estimation]
    \cref{lem:relationship} reveals that $\Delta$ is related to the smallest nonzero singular value of the core tensor $\lambda$ in the Tucker decomposition of $\mathcal{T}$.  Combining this lemma with Assumption \ref{assumption:signalstrength} shows that we require $\lambda/\sigma \gtrsim \kappa p^{3/4} \sqrt{\log(p)}$ when $p_k \asymp p$.  In \citet{zhang_tensor_2018}, it was shown that the condition $\lambda/\sigma \gtrsim p^{3/4}$ is a necessary and sufficient condition for minimax subspace estimation in polynomial time when $p_k \asymp p$. 
 Therefore, Assumption \ref{assumption:signalstrength} is only suboptimal relative to \citet{zhang_tensor_2018} by factors of $\kappa$ and $\sqrt{\log(p)}$.  However, our results are significantly different from theirs, and we discuss these further in \cref{sec:twoinfty}.
% However, our main results concern the maximum row-wise error rates, whereas \citet{zhang_tensor_2018} only consider the error rates in spectral norm, and our model permits heteroskedasticity.  The uniform row-wise error rates require novel proof techniques that we discuss in \cref{sec:proofoverview}.
 \end{remark}
 }

\subsection{Estimation Errors}

%We now have all the ingredients to establish the entrywise estimation of the matrices $\mathbf{\Pi}_k$.  
The following theorem characterizes the errors in estimating $\mathbf{\Pi}_k$.  

%establishes the estimation of the membership matrix $\mathbf{\Pi}_k$ in a uniform row-wise sense.

\begin{theorem}[Uniform Estimation Error] \label{thm:estimation}
Suppose that $r_{\max} \lesssim p_{\min}^{1/2}$, that $r_{\max} \asymp r$ with $r \lesssim r_{\min}$, and that $\kappa^2 \lesssim p_{\min}^{1/4}$.  Suppose further that Assumptions \ref{assumption:regularity} and  \ref{assumption:signalstrength} hold, \textcolor{black}{and that $\Delta/\sigma \leq \exp(c p_{\max})$ for some small constant $c$}.  Let $\mathbf{\widehat \Pi}_k$ be the output of \cref{al:spamm} with  $t$ iterations for $t \asymp  \textcolor{black}{\log\bigg( \frac{\Delta/\sigma(p_1p_2p_3)^{1/2}}{C_0\kappa \sqrt{p_{\min}\log(p_{\max})}(r_1r_2r_3)^{1/2}}\bigg)}$ %\log\bigg( \frac{\kappa p_{\max}(r_1 r_2 r_3)^{1/2}}{(\Delta/\sigma) (p_1p_2p_3)^{1/2}}\bigg) \vee 1$. 
Then with probability at least $1 - p_{\max}^{-10}$ there exists three permutation matrices $\mathcal{P}_k \in \{0,1\}^{r_k \times r_k}$ such that for each $k$ %{\red (We probably still need to specify $p$ here. Is $p$ here any positive integer, or it has to be of similar order as $p_1, p_2$ or $p_3$?)} \textcolor{ForestGreen}{fixed.}
\begin{align*}
 \max_{1\leq i\leq p_k} \| \big(\mathbf{\Pi}_k -\mathbf{\widehat \Pi}_k \mathcal{P}_k\big)_{i\cdot} \| &\lesssim \frac{\kappa \sqrt{r^3 \log(p_{\max})}}{(\Delta/\sigma) (p_{-k})^{1/2}}.
\end{align*}
Consequently, when $p_k \asymp p$, it holds that
\begin{align*}
\max_{1\leq i\leq p_k} \| \big(\mathbf{\Pi}_k -\mathbf{\widehat \Pi}_k \mathcal{P}_k\big)_{i\cdot} \| &\lesssim \frac{\kappa  \sqrt{r^3  \log(p)}}{(\Delta/\sigma) p}.
\end{align*}
\end{theorem}
%In the supplementary materials, we provide an informal discussion of the cost of ignoring the additional tensor structure on the resulting estimation error.

\cref{thm:estimation} establishes a \emph{uniform} error bound for the estimated communities; that is, the estimation error for a given node $i$.  \textcolor{black}{Unlike the tensor blockmodel considered in \citet{han_exact_2021}, in the tensor mixed-membership blockmodel has continuous community memberships, and hence estimation is a more challenging problem.  Our bound exhibits a polynomial dependence on the SNR, whereas estimation of discrete community memberships often exhibits exponential dependence on the SNR (e.g. \citep{loffler_optimality_2021}).  However, our results also demonstrate that our estimation procedure is consistent \emph{uniformly for each node}, which is a stronger result than the average-case optimality often considered in discrete community estimation.%  In essence, this stronger result is due to our main technical contribution, the $\ell_{2,\infty}$ tensor perturbation bound that we introduce in the next section.
}

%\textcolor{black}{
%}

%Consequently, uniform consistency for the $k$'th mode is guaranteed as long as $\Delta/\sigma \gg \kappa \sqrt{r\log(p)/p_{-k}}$.  
\begin{remark}[Relationship to Matrix Mixed-Membership Blockmodels]
\cref{thm:estimation} is related to similar bounds in the literature for the matrix setting. \citet{mao_estimating_2021,xie_entrywise_2022,jin_estimating_2017} consider estimating the membership matrix with the leading eigenvectors of the observed matrix.  %E%xplicitly, they define the $p\times p$ matrix 
%$\mathbf{M} = \rho_n \mathbf{\Pi} \mathbf{B}  \mathbf{\Pi}\t$,
%and they assume one observes $\mathbf{\widehat M} = \mathbf{M} + \mathbf{E},$
%where $\mathbf{E}$ consists of mean-zero Bernoulli noise with $\mathbb{E} \mathbf{E}_{ij}^2 = \mathbf{M}_{ij} (1 - \mathbf{M}_{ij})$.  The main result of \citet{mao_estimating_2021} (Theorem 3.5) demonstrates an upper bound of the form 
% \begin{align*}
%     \|\mathbf{\Pi} -  \mathbf{\widehat\Pi} \mathcal{P} \|_{2,\infty} &= \tilde O\bigg( \frac{r^{3/2}}{\sqrt{p\rho_n} \lambda_{\min}(\mathbf{B})} \bigg),
% \end{align*}
% where the $\tilde O(\cdot)$ hides logarithmic terms (we assume for simplicity that $\mathbf{B}$ is positive definite). See also \citet{xie_entrywise_2022} for a similar bound for sparse Bernoulli noise, and \citet{jin_estimating_2017} for a slightly different but much more general result allowing heterogeneous degree corrections. %See also \citet{xie_entrywise_2022} for a similar bound for sparse Bernoulli noise.  
% The term $\sqrt{\rho_n} \lambda_{\min}(B)$ can be informally understood as the signal-to-noise ratio (SNR) in the Bernoulli noise setting. %\footnote{If the noise were truly Gaussian with mean $\mathbf{M}_{ij}$ and variance bounded by $\rho_n$, then our definition of $\Delta$ would be simply $\rho_n \lambda_{\min}(\mathbf{B})$, resulting in the SNR of $\sqrt{\rho_n} \lambda_{\min}(\mathbf{B})$.}.
\textcolor{black}{Assuming certain regularity conditions, these results collectively imply that}
%Consequently, when $\kappa, r \asymp 1$, and $p_k \asymp p$, one has the bounds
\begin{align*}
  \text{(Matrix setting)}&&    \|\mathbf{\Pi} -  \mathbf{\widehat\Pi} \mathcal{P} \|_{2,\infty}  &= \tilde O\bigg(  \frac{1}{\mathrm{SNR} \times \sqrt{p}} \bigg); && \\
     \text{(Tensor setting)}&& \|\mathbf{\Pi} -  \mathbf{\widehat\Pi} \mathcal{P} \|_{2,\infty} &=\tilde O\bigg( \frac{1}{\mathrm{SNR}\times p} \bigg), && 
\end{align*}
\textcolor{black}{where $\mathrm{SNR}$ can be  understood as a form of signal-to-noise ratio taking into account the Bernoulli noise.}
Therefore, \cref{thm:estimation} can be understood as providing an estimation improvement of order  $\sqrt{p}$ compared to the matrix setting -- one may view this extra $\sqrt{p}$ factor as stemming from the higher-order tensor structure.  However, the arguments required to prove \cref{thm:estimation} require analyzing the output of HOOI, which imposes a number of nontrivial technical challenges.
\end{remark}
%\textcolor{black}{
\begin{remark}[Extension to Bernoulli Noise]
When the noise $\mathcal{Z}$ is Bernoulli, the SNR is governed by the sparsity of the Bernoulli noise.  However, our definition of SNR in \cref{thm:estimation} only concerns the subgaussian variance proxy $\sigma$ which is a constant for Bernoulli noise, and hence our results are only applicable to dense Bernoulli noise.  \textcolor{black}{While it is of theoretical and practical interest to extend our analysis to the sparse Bernoulli setting, such a result will require significant arguments beyond those already in this paper. Our proof is already quite long and highly novel  which we detail further in  \cref{sec:proofoverview}.  Therefore, in light of our already involved technical analysis, we leave this setting to future work. }
%\textcolor{black}{
%Extending our technical results to the sparse Bernoulli setting is of theoretical and practical interest, but the analysis will require significant technical contributions beyond those already in this paper; we discuss these further in \cref{sec:proofoverview}.  Therefore, we leave this setting for future work. % and focus on the general heteroskedastic subgaussian noise setting.
\end{remark}
%}

Since the rows of $\mathbf{\Pi}_k$ can be understood as weight vectors, a natural metric to use in this setting is the average $\ell_1$ norm.  \cref{thm:estimation} then implies the following corollary.
%One can also translate our bounds to an average case $\ell_1$ guarantee.

\begin{corollary}[Average $\ell_1$ Error] \label{cor:averagecase}
In the setting of \cref{thm:estimation}, with probability at least $1 - p_{\max}^{-10}$, one has
\begin{align*}
\max_k    \inf_{\mathrm{Permutations }\ \mathcal{P}} \frac{1}{p_k}\sum_{i=1}^{p_k} \|  \big(\mathbf{\widehat \Pi}_k - \mathbf{\Pi}_k \mathcal{P} \big)_{i\cdot} \|_1 &\lesssim \frac{r^2 \kappa \sqrt{\log(p_{\max})}}{(\Delta/\sigma)( p_{-k} )^{1/2}}.
\end{align*}
\end{corollary}

% Similarly, \citet{xie_entrywise_2022} established a similar bound, showing that 
% \begin{align*}
%   \|\Pi^{(\mathrm{matrix})} - \widehat \Pi^{(\mathrm{matrix})} \mathcal{P} \|_{2,\infty}  &\lesssim K \sqrt{\frac{\log(p)}{p\rho_n}}.
% \end{align*}

%  plays the role of $\mu/\sigma$ in the Bernoulli noise setting.  A similar bound was obtained in \citet{xie_entrywise_2022}, who study the setting of sparse Bernoulli networks, essentially showing that
% \begin{align*}
%     \| \widehat \Pi_k - \Pi_k \mathcal{P} \|_{2,\infty} &\lesssim K \sqrt{\frac{\log(n)}{n\rho_n}},
% \end{align*}
% where they assume that $\lambda_{\min}(B) \geq c$.  Consequently, considering $r,K \asymp 1$, our results demonstrate a gain of order $\sqrt{p}$ stemming from the higher-order structure.

%%%%%%%%%%%%%%%%%%%%%%%
\subsection{Key Tool: $\ell_{2,\infty}$ Tensor Perturbation Bound}
%%%%%%%%%%%%%%%%%%%%%%%
\label{sec:twoinfty}

In this section we introduce the new $\ell_{2,\infty}$ tensor perturbation bound, which serves as a key tool for developing the main results of this paper. Other $\ell_{2,\infty}$ bounds for HOOI in this setting have not appeared in the literature to the best of our knowledge.  Unlike the matrix SVD, HOOI (\cref{al:tensor-power-iteration}) is an iterative algorithm that proceeds by updating the estimates at each iteration. Therefore, analyzing the output of HOOI requires carefully tracking the interplay between noise and estimation error at each iteration as a function of the spectral properties of the underlying tensor.  \textcolor{black}{We further discuss our proof techniques in \cref{sec:proofoverview}.}

%In order to study the quality of our estimator, we require sharp $\ell_{2,\infty}$ concentration bounds for the outcome of \cref{al:tensor-power-iteration}, which has not appeared in the literature to the best of our knowledge.  

In what follows, recall we define the \emph{incoherence} of a tensor $\mathcal{T}$ as the smallest number $\mu_0$ such that
\begin{align*}
\max_k \sqrt{\frac{p_k}{r_k}} \|\U_k \|_{2,\infty} \leq \mu_0.
\end{align*}
By way of example, for a $p\times p \times p$ tensor $\mathcal{T}$, observe that when $\mathcal{T}$ contains only one large nonzero entry, it holds that $\mu_0 = \sqrt{p}$, whereas when $\mathcal{T}$ is the tensor with constant entries, it holds that $\mu_0 = 1$.  Consequently, $\mu_0$ can be understood as a measure of ``spikiness'' of the underlying tensor, with larger values of $\mu_0$ corresponding to more ``spiky'' $\mathcal{T}$. 

In addition, we will present bounds for the estimation of $\U_k$ up to right multiplication of an orthogonal matrix $\mathbf{W}_k$.  
The appearance of the orthogonal matrix $\mathbf{W}_k$ occurs due to the fact that we do not assume that singular values are distinct, and hence singular vectors are only identifiable up to orthogonal transformation.  %Previous results of this type also typically include an additional orthogonal matrix 
%\citep{agterberg_entrywise_2022,agterberg_entrywise_2022-1,cai_subspace_2021,abbe_entrywise_2020,cape_two--infinity_2019}.

 The following result establishes the $\ell_{2,\infty}$ perturbation bound for the estimated singular vectors from \cref{al:tensor-power-iteration} under the general tensor denoising model, which is more general than the setting considered in the previous sections.
%When $\mu_0 = O(1)$, the entries of $\U_k$ are all of similar order, which corresponds to settings in which the tensor $\mathcal{T}$ is more \emph{spread out}

%
%Now we are in position to establish the general $\ell_2\to\ell_\infty$ norm perturbation upper bound for the outcome $\hat{\U}_k$ of Algorithm \ref{al:tensor-power-iteration}. 
%
%Ideally, we would like to develop an $\ell_2\to \ell_\infty$ norm bound for $k=1,\ldots,d$ provided a reasonable initialization points $\U_1^{(0)},\ldots, \U_d^{(0)}$.
%$$\min_{\W_\U \in \mathbb{O}_{r} }\|\hat{\U}_k - \U_k\W_{U_k}\|_{2\to \infty} \leq ...  $$
%{\blue (Anru: I still need to calculate more carefully here. Ideally, we want to bound this $\ell_2\to\ell_\infty$ norm by $\|\sin\Theta(\hat{\U}_k, \U_k)\|$, $\bcX, \bcZ$... just like Theorems 3.7, 3.8 in \cite{cape2019two}. I think this should be achievable but some more time is required to work out the details.)}

\begin{theorem}\label{thm:twoinfty}
Suppose that $\mathcal{\widehat{T}} = \mathcal{T} + \mathcal{Z}$, where $\mathcal{Z}_{ijk}$ are independent mean-zero subgaussian random variables satisfying $ \|\mathcal{Z}_{ijk} \|_{\psi_2} \leq \sigma$.  Let $\mathcal{T}$  have Tucker decomposition $\mathcal{T} = \mathcal{C} \times_1 \U_1 \times_2 \U_2 \times_3 \U_3$, and suppose that $\mathcal{T}$ is incoherent with incoherence constant $\mu_0$. Suppose that $\lambda/\sigma \gtrsim \kappa \sqrt{\log(p_{\max})} p_{\max}/p_{\min}^{1/4}$, $\mu_0^2 r \lesssim p_{\min}^{1/2}$, that $\kappa^2 \lesssim p_{\min}^{1/4}$, and that $r \lesssim r_{\min}$, where $\lambda = \lambda_{\min}(\mathcal{C})$. \textcolor{black}{Suppose further that $\lambda/\sigma \leq \exp(c p)$ for some small constant $c$}. Let $\uhat^{(t)}_k$ denote the output of HOOI (\cref{al:tensor-power-iteration}) after $t$ iterations.  Then there exists orthogonal matrices $ \mathbf{W}_k  \in \mathbb{O}(r_k)$ for each $k$ such that after $t$ iterations with $t \asymp \log\big(\textcolor{black}{\frac{\lambda/\sigma}{C \kappa \sqrt{p_{\min} \log(p_{\max})}}}\big)$, %\sigma\kappa p_{\max}/\lambda) \vee 1$, 
with probability at least $1 - p_{\max}^{-10}:$
\begin{align*}
    \| \uhat_{k}^{(t)}  - \U_k \mathbf{W}_k\|_{2,\infty} &\lesssim \frac{ \kappa \mu_0 \sqrt{r_k \log(p_{\max})} }{\lambda/\sigma}.
\end{align*}
\end{theorem}

%A few remarks on this theorem may 
%be warranted. 
\begin{remark}[Signal Strength Condition]
   The condition $\lambda/\sigma \gtrsim \kappa \sqrt{\log(p_{\max})}p_{\max}/p_{\min}^{1/4}$ is only slightly stronger than the condition
$\lambda/\sigma \gtrsim p _{\max}\sqrt{r/p_{\min}}$ when $r \lesssim p_{\min}^{1/2}$.  It has been shown in \citet{luo_sharp_2021} that this second condition implies a bound of the form $\|\sin\Theta(\uhat_k,\U_k)\| \lesssim \frac{\sqrt{p_k}}{\lambda/\sigma}$, which  matches the minimax lower bound established in \citet{zhang_tensor_2018} when $p_k \asymp p$. Therefore, the condition $\lambda/\sigma \gtrsim \kappa \sqrt{\log(p_{\max})} p_{\max}/p_{\min}^{1/4}$ allows for different orders of $p_k$ without being too strong. \  %Perhaps one way to understand \cref{thm:twoinfty} is that after sufficiently many iterations, the $\sin\Theta$ upper bound is of order $\sqrt{p_k}/(\lambda/\sigma)$, and \cref{thm:twoinfty} demonstrates that these errors are approximately \emph{spread out} amongst the entries of $\uhat_k$ when $\U_k$ is incoherent.  It is perhaps of interest to study the $\ell_{2,\infty}$ errors under different signal strength conditions; however, we leave this setting to future work.
\textcolor{black}{When $p_k \asymp p$, our SNR condition translates to the condition $\lambda/\sigma \gtrsim   \kappa p^{3/4} \sqrt{\log(p)}$, which is optimal up to a factor of $\kappa\sqrt{\log(p)}$ for a polynomial-time estimator to exist \citep{zhang_tensor_2018}.}
\textcolor{black}{Furthermore, the condition $\lambda/\sigma \leq \exp(cp)$ may be an artifact of the proof strategy and not material. Our results also allow $r$ to grow with $p$ as long as $\mu_0^2 r \lesssim \sqrt{p}$. }
%
%Next for simplicity, consider the regime $p_k \asymp p$.  First, we assume that $\mu_0^2 r \lesssim \sqrt{p}$ which allows $r$ to grow. %particularly when $\mu_0 = O(1)$. 
%In addition, in this regime the SNR condition translates to the condition $\lambda/\sigma \gtrsim   \kappa p^{3/4} \sqrt{\log(p)}$, which is optimal up to a factor of $\kappa\sqrt{\log(p)}$ for a polynomial-time estimator to exist \citep{zhang_tensor_2018}.  %The work   \citet{abbe_ell_p_2022} suggests that without an additional  logarithmic factor it may not be possible to obtain $\ell_{2,\infty}$ bounds, so it is possible that this additional logarithmic factor is actually necessary.
\end{remark}

\begin{remark}[Optimality]

It was shown in \citet{zhang_tensor_2018} that the minimax rate for tensor SVD satisfies 
\begin{align*}
    \inf_{\mathbf{\bar U}_k} \sup_{\mathcal{T} \in \mathcal{F}_{p,r}(\lambda)} \E \| \sin\Theta(\mathbf{\bar U}_k,\U_k) \| \gtrsim \frac{\sqrt{p_k}}{\lambda/\sigma},
\end{align*}
where $\mathcal{F}_{p,r}(\lambda)$ is an appropriate class of low-rank signal tensors and the infimum is over all estimators of $\U_k$.  By properties of the $\sin\Theta$ distance and the $\ell_{2,\infty}$ norm, it holds that
\begin{align*}
       \inf_{\mathbf{\bar U}_k} \sup_{\mathcal{T} \in \mathcal{F}_{p,r}(\lambda)} \E \inf_{\mathbf{W} \in \mathbb{O}(r)} \| \mathbf{\bar U}_k - \U_k \mathbf{W} \|_{2,\infty} %&\geq %\frac{1}{\sqrt{p_k}} \inf_{\mathbf{\bar U}_k} \sup_{\mathcal{T} \in  \mathcal{F}_{p_k,r}(\lambda)} \E \inf_{\mathbf{W} \in \mathbb{O}(r)} \| \mathbf{\bar U}_k - \U \mathbf{W} \|_{2} \\
       &\gtrsim \frac{1}{\sqrt{p_k}}\inf_{\mathbf{\bar U}_k} \sup_{\mathcal{T} \in \mathcal{F}_{p_k,r}(\lambda)} \E \| \sin\Theta(\mathbf{\bar U}_k,\U_k) \| \gtrsim  \frac{1}{\lambda/\sigma}.
\end{align*}
Consequently, when $\kappa,\mu_0,r \asymp 1$, \cref{thm:twoinfty} shows that HOOI attains the minimax rate for the $\ell_{2,\infty}$ norm up to a logarithmic term. Such a result is new to the best of our knowledge.

\end{remark}

\begin{remark}[Intermediate Result: New $\ell_{2, \infty}$ Error Bound of Diagonal-Deletion Intialization)] Considering again $p_k \asymp p$, under the conditions of \cref{thm:twoinfty}, we prove (see \cref{thm:spectralinit_twoinfty}) that the diagonal-deletion initialization satisfies the high-probability upper bound
\begin{align}
    \| \uhat_k^S - \U_k \mathbf{W}_k \|_{2,\infty} &\lesssim \bigg( \underbrace{\frac{\kappa \sqrt{p\log(p)}}{\lambda/\sigma}}_{\text{linear error}} + \underbrace{\frac{p^{3/2} \log(p)}{(\lambda/\sigma)^2}}_{\text{quadratic error}} + \underbrace{\kappa^2 \mu_0 \sqrt{\frac{r}{p}}}_{\text{bias term}} \bigg) \mu_0 \sqrt{\frac{r}{p}}. \label{eq:inittwoinfty} %\\
%    &\asymp \bigg( \frac{\kappa \sqrt{p\log(p)}}{\lambda/\sigma} + \frac{1}{2}\bigg) \mu_0 \sqrt{\frac{r}{p}},
\end{align}
%where the second line holds when in the high noise regime, i.e., when $\lambda/\sigma \asymp \polylog(p)p^{3/4}$.
The full proof of this result is contained in \cref{sec:inittwoinftyproof}; it should be noted that this result slightly improves upon the bound of \citet{cai_subspace_2021} by a factor of $\kappa^2$. %(though we do not include missingness as in \citet{cai_subspace_2021}).
This quantity in \eqref{eq:inittwoinfty} consists of three terms: the first term is the ``linear error'' that appears in \cref{thm:twoinfty}, the second term is the ``quadratic error,'' and the third term is the error stemming from the bias induced by diagonal deletion.  In the high noise regime $\lambda/\sigma \asymp p^{3/4} \polylog(p)$, the quadratic error can dominate the linear error; and, moreover, the bias term does not scale with the noise of the problem.  
\textcolor{black}{Our results show that HOOI eliminates both the bias term and the quadratic error in $\ell_{2,\infty}$ norm.}
%In \citet{zhang_tensor_2018}, it was shown that tensor SVD removes the ``quadratic'' error for the $\sin\Theta$ distance for homoskedastic noise; \cref{thm:twoinfty} goes one step further to show that these errors are evenly spread out, particularly when $\mu_0$ is sufficiently small.  
\end{remark}

\begin{remark}[Adaptivity of HOOI to Heteroskedasticity]
    The bias term in \eqref{eq:inittwoinfty}, which does not scale with the noise $\sigma$, arises naturally due to the fact that one deletes the diagonal of both the noise and the underlying low-rank matrix.  In the setting that the noise is heteroskedastic, \citet{zhang_heteroskedastic_2022} showed that a form of bias-adjustment is necessary   for many settings; moreover, they showed that their algorithm \texttt{HeteroPCA} eliminates this bias factor in $\sin\Theta$ distance. % -- in essence, the \texttt{HeteroPCA} algorithm is a debiasing procedure for the diagonal of the Gram matrix. 
    The follow-on works \citet{agterberg_entrywise_2022} and \citet{yan_inference_2021} have shown that this algorithm also eliminates the bias term in $\ell_{2,\infty}$ distance, implying that it is possible to obtain a bound that scales with the noise.  In contrast, \cref{thm:twoinfty} shows that the HOOI algorithm does not require any additional bias-adjustment to combat heteroskedasticity in order to obtain a bound that scales with the noise.  In effect, this result demonstrates that HOOI is \emph{adaptive} to heteroskedasticity.  
\end{remark}

%\textcolor{black}{Note: the assumption that $\lambda \gtrsim \sqrt{p_k} r^{(d-2)/2}$ is satisfied when $\lambda \gtrsim p_k^{3/4}$ and $r_k \leq \sqrt{p_k}$.  Therefore, this assumption is in general weaker than the condition needed for the spectral initialization to succeed, particularly if $r_k = O(1)$.  We can also eliminate of $r_k$ on the logarithmic term if we assume the stronger condition $\lambda \gtrsim \sqrt{p_k} r^{d/2}$.}

\begin{remark}[Implicit Regularization] \cref{thm:twoinfty} and its proof also reveal an implicit regularization effect in tensor SVD with subgaussian noise -- when the underlying low-rank tensor is sufficiently incoherent (e.g., $\mu_0 = O(1)$) and the signal-to-noise ratio is sufficiently strong, all of the iterations are also incoherent with parameter $\mu_0$.  Several recent works have proposed incoherence-regularized tensor SVD \citep{ke_community_2020,jing_community_2021}, and \cref{thm:twoinfty} suggests that this regularization may not be needed. Our results do not directly apply in these settings, as their models include Bernoulli noise and some form of symmetry.  
%However, these prior works have focused on the setting of Bernoulli noise, for which the subgaussian variance proxy can be of much larger order than the standard deviation, particularly for sparse networks.  Moreover, these previous works have included some form of symmetry, whereas this work considers the completely asymmetric setting. 
Nevertheless, it may be possible to extend our work to these settings, though the analysis will likely be significantly more involved. % using other forms of concentration inequalities for certain ``tensorial terms'' arising in the analysis (see \cref{sec:proofoverview} for details on the proof techniques). 
\end{remark}

\section{Application to  Global Trade Data} \label{sec:numericalresults}
%We now consider the numerical performance of our proposed procedure. % In \cref{sec:sims} we provide simulation results for several examples, including varying levels of heteroskedasticity. 
%We  apply our algorithm to two different flight data sets -- the first %is the data described in \citet{han_exact_2021}, which measures global flights, the second 
%is USA flight data, and the second is a global trade dataset. We provide simulation results and also analyze the global flight data studied in \citet{han_exact_2021} in the supplementary materials.

%\subsection{Application to Global Trade Data}

We apply our algorithm to the global trade network dataset collected in \citet{de_domenico_structural_2015}% \footnote{\url{http://www.fao.org}} 
and further analyzed in \citet{jing_community_2021}, which consists of trading relationships for 364 different goods between 214 countries in the year 2010, with the weight corresponding to the amount traded.  Here each individual network corresponds to the trading relationships between countries for a single good. \textcolor{black}{We have also included simulations and additional data analysis in the supplementary materials.}

To preprocess, we first convert each network to undirected, and we keep the networks with the largest connected component of at least size 150, which results in a final tensor of dimension $59 \times 214 \times 214$.  Note that unlike \citet{jing_community_2021}, we do not delete or binarize edges, nor do we only use the largest connected component within each network. To select the ranks we use the same method as in the previous section, resulting in $\br = (5,4,4)$.  

In \cref{fig:worldmap} we plot each of the memberships associated to the ``country'' mode, where the pure nodes are found to be USA, Japan, Canada, and Germany.  For the communities corresponding to Germany and Japan, we see that the weight of the corresponding countries roughly corresponds to geographical location, with closer countries corresponding to higher membership intensity.  In particular, Germany's memberships are highly concentrated in Europe and Africa, with the memberships of all European countries being close to one.  \textcolor{black}{Due to space constraints, } we provide further discussion on the different trade products in the supplementary materials.

\begin{figure*}[ht!]
 \centering
        \subfloat{%
            \includegraphics[width=.5\linewidth, keepaspectratio]{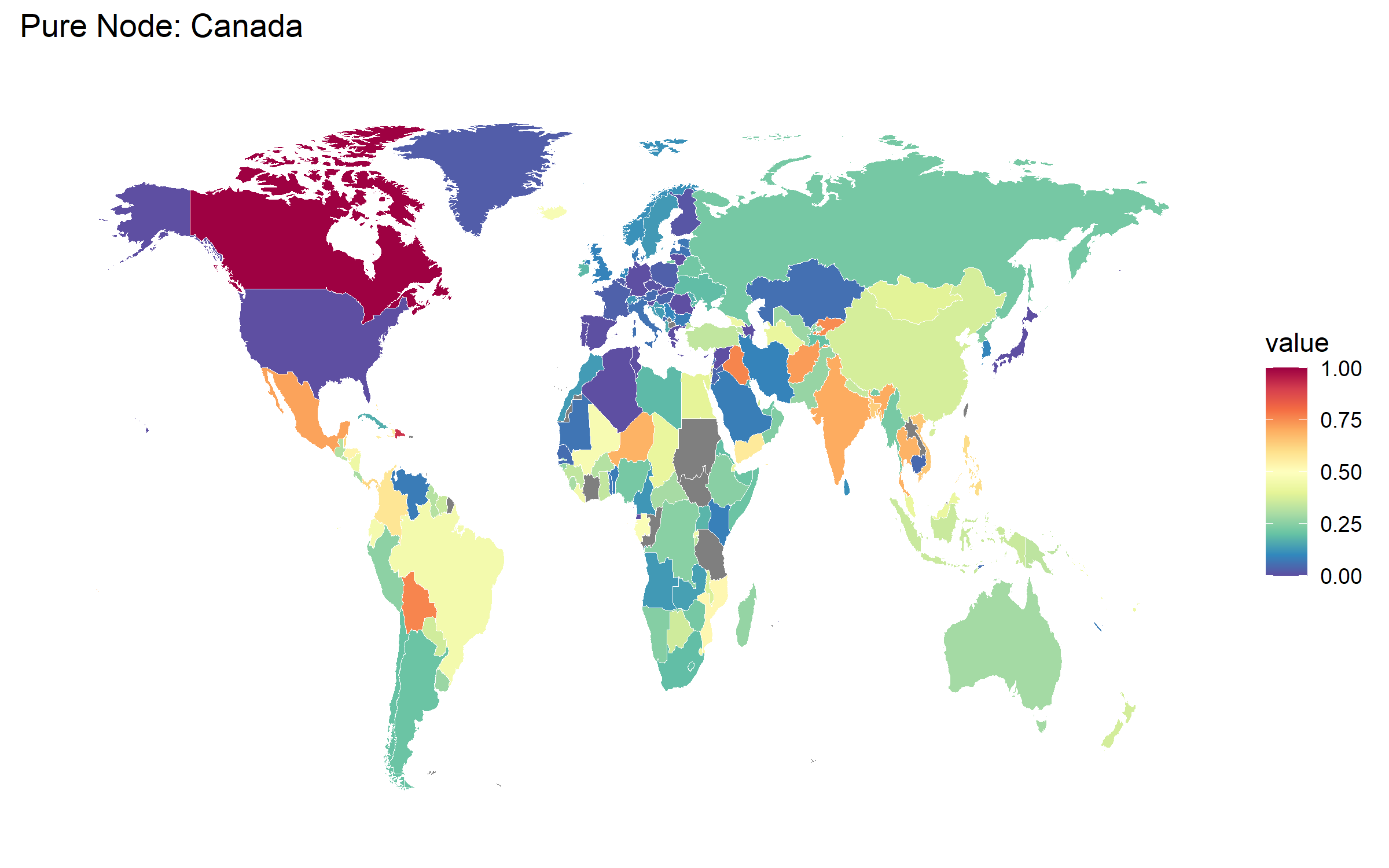}%
            \label{subfig:a}%
        }
        \subfloat{%
            \includegraphics[width=.5\linewidth, keepaspectratio]{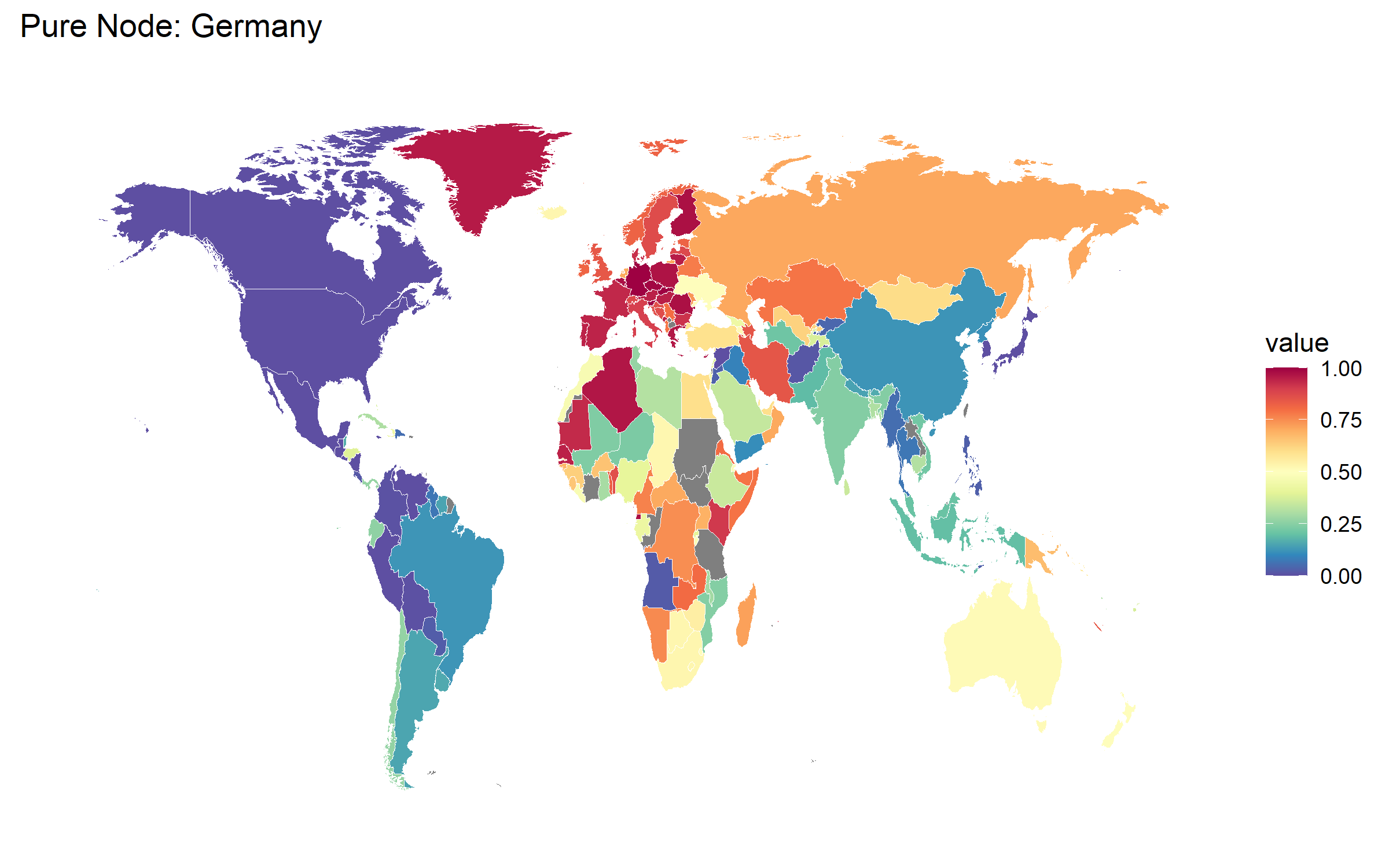}%
            \label{subfig:b}%
        } \\
      \subfloat{%
            \includegraphics[width=.5\linewidth, keepaspectratio]{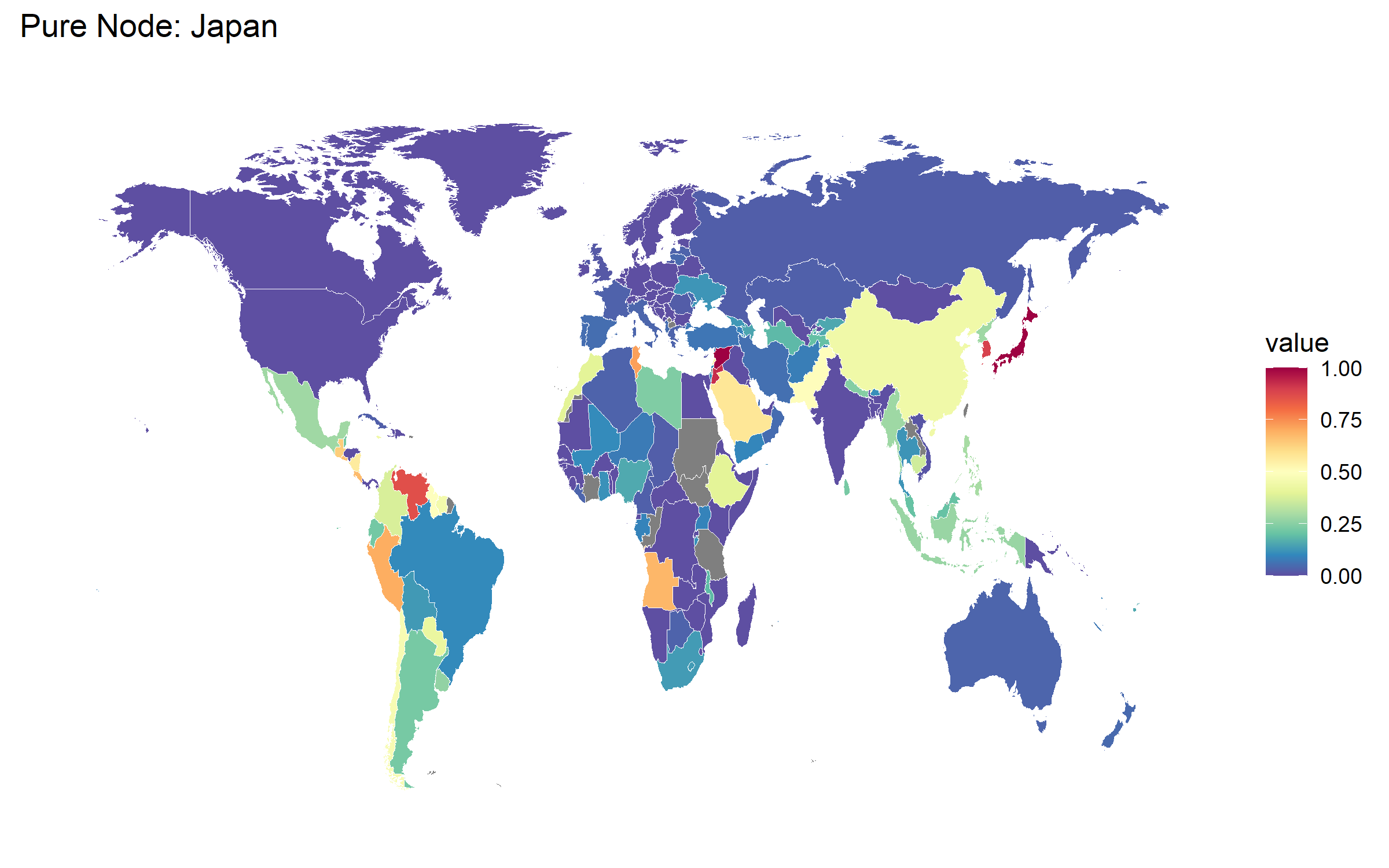}%
            \label{subfig:c}%
        }
        \subfloat{%
            \includegraphics[width=.5\linewidth, keepaspectratio]{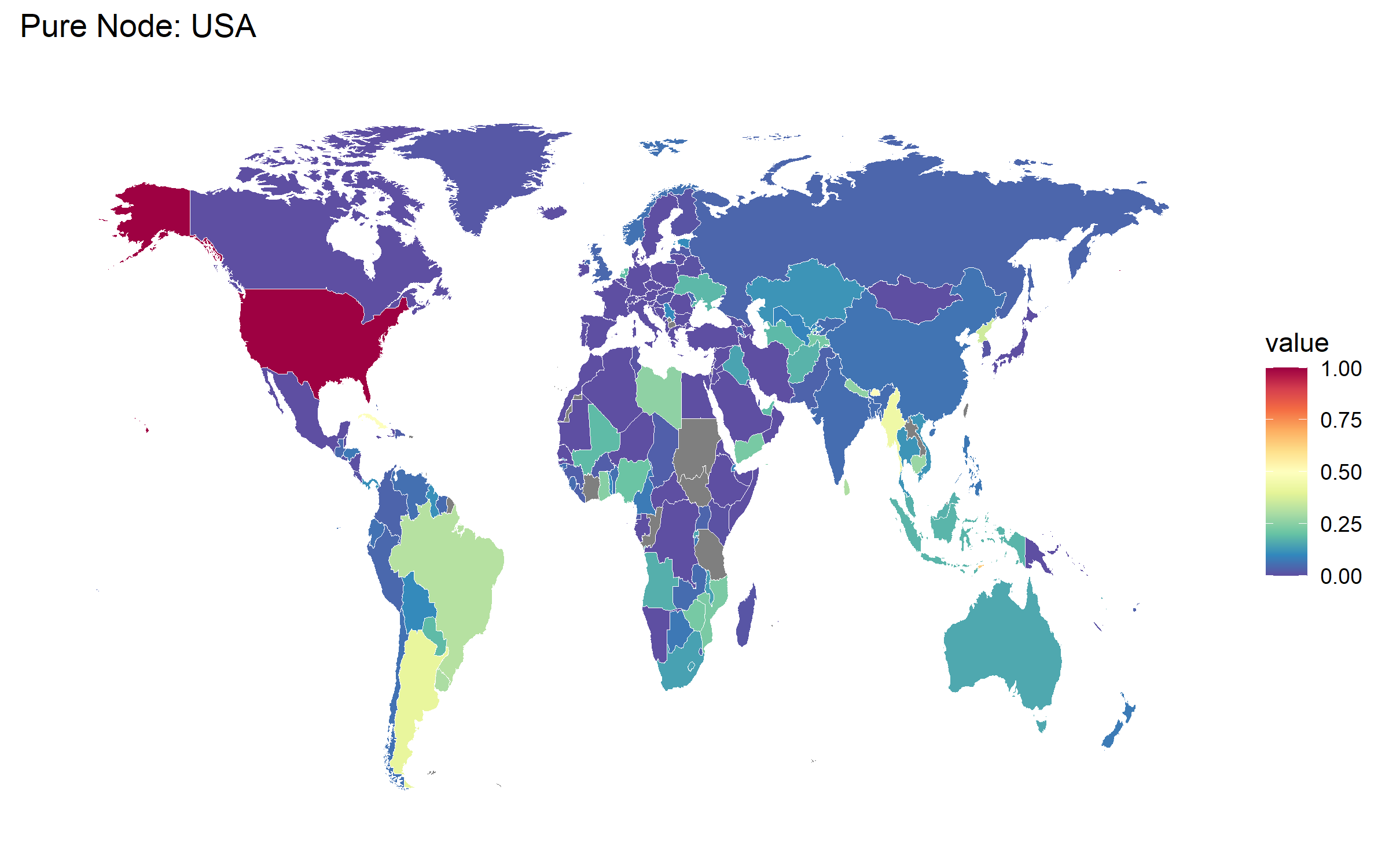}%
            \label{subfig:d}%
        }
        \caption{Pure node memberships for the countries, with red corresponding to higher membership intensity.  Grey corresponds to countries that were not included in the analysis.}
        \label{fig:worldmap}
    \end{figure*}

\section{Overview of the Proof of \cref{thm:twoinfty}} \label{sec:proofoverview}
%%%%%%%%%%%%%%%%%% 
In this section, we provide a high-level overview and highlight the novelties of the proof of \cref{thm:twoinfty}, the main \textcolor{black}{technical result} of this paper. As mentioned previously, the proof idea is based on a leave-one-out analysis. Different versions of leave-one-out analysis have been used in, for example, \cite{yan_inference_2021,abbe_ell_p_2022,abbe_entrywise_2020,chen_spectral_2021,ma_implicit_2020,cai_nonconvex_2022} \textcolor{black}{among others}.  \textcolor{black}{However, due to particular structure of tensor data, our analysis requires a number of novel considerations heretofore not used in the literature to the best of our knowledge.} % However, a key difference is that these previous works focus on the eigenvectors or singular vectors of the perturbed matrix, and subsequently do not have to analyze additional iterates.  A series of related works in nonconvex optimization have studied the iterates of algorithms using a leave-one-out sequence, but these analyses typically focus on gradient descent or similar optimization techniques \citep{ma_implicit_2020,cai_nonconvex_2022}, and hence do not need to consider additional singular vectors besides the first step. %In contrast to these works, our upper bounds rely on an inductive argument in the number of iterations, where we show that a certain upper bound holds for each iteration, so our analysis uses primarily matrix perturbation tools.

To be concrete, without loss of generality, assume that $\sigma = 1$.  For simplicity assume that $p_k \asymp p$ throughout this section. Our proof proceeds by showing that at each iteration $t$ with probability at least $1 - 3 t p^{-15}$ one has the bound
\begin{align}
    \| \uhat_k^{(t)} - \U_k \mathbf{W}_k^{(t)} \|_{2,\infty} &\leq \bigg( \frac{\dl}{\lambda} + \frac{1}{2^{t}} \bigg)\mu_0 \sqrt{\frac{r}{p}}, \label{iterativebound1}
\end{align}
where  we define $\dl$ as the \emph{linear error} $\dl \coloneqq C_0 \kappa \sqrt{p\log(p)},$
with $C_0$ being some fixed constant. Here the matrix $\mathbf{W}_k^{(t)}$ is defined via 
$%\begin{align*}
    \mathbf{W}_k^{(t)} := \argmin_{\mathbf{W}: \mathbf{W W}\t = \mathbf{I}_{r_k}} \| \uhat_k^{(t)} - \U_k \mathbf{W} \|_F;$
%\end{align*}
i.e., it is the orthogonal matrix most closely aligning $\U_k$ and $\uhat_k^{(t)}$ in Frobenius norm. 
%A key property of the matrix $\mathbf{W}_k^{(t)}$ is that it can be computed analytically from the singular value decomposition of $\U_k\t \uhat_k^{(t)}$ as follows.  Let $\U_k\t \uhat_k^{(t)}$ have singular value decomposition $\mathbf{W}_1 \Sigma \mathbf{W}_2\t$; then $\mathbf{W}_k^{(t)} = \mathbf{W}_1 \mathbf{W}_2\t$. 
The matrix $\mathbf{W}_k^{(t)}$ is also known as the \emph{matrix sign function} of $\U_k\t \uhat_k^{(t)}$, denoted 
as $\mathrm{sgn}(\U_k\t \uhat_k^{(t)})$.   

%Equivalently, $\mathbf{W}_k^{(t)}$ satisfies $\mathbf{W}_k^{(t)} = \mathrm{sgn}( \U_k\t \uhat_k^{(t)})$, where $\mathrm{sgn}(\cdot)$ is the matrix sign function defined as follows.  For a given square matrix  $\mathbf{H}$ with singular value decomposition $\mathbf{W}_1 \Sigma \mathbf{W}_2\t$, we define
%\begin{align*}
%    \mathrm{sgn}(\mathbf{H}) :&= \mathbf{W}_1 \mathbf{W}_2\t.
%\end{align*}

%Consequently, $\mathbf{W}_k^{(t)}$ can be computed analytically via the product of the left and right orthonormal factors in the SVD of $\U_k\t \uhat_k^{(t)}$.   
%minimizing the Frobenius distance between $\U_k$ and $\uhat_k^{(t)} \mathbf{W}$ over all orthogonal matrices $\mathbf{W}$ (which can be computed analytically via the SVD of $\U_k\t \uhat_k^{(t)}$).  

%To show that the bound in \eqref{iterativebound1} holds,
We first consider a fixed $m$th row to note that
\begin{align*}
  e_m\t\bigg(  \uhat_k^{(t)} - \U_k \mathbf{W}_k^{(t)} \bigg) %&= e_m\t \bigg( \uhat_k^{(t)} - \U_k \U_k\t\uhat_k^{(t)} + \U_k ( \U_k\uhat_k^{(t)} - \mathbf{W}_k^{(t)}) \bigg) \\
    &= e_m\t \bigg( (\mathbf{I} - \U_k \U_k\t) \uhat_k +\U_k ( \U_k\t\uhat_k^{(t)} - \mathbf{W}_k^{(t)}) \bigg).
\end{align*}
The second term is easily handled since $\mathbf{W}_k^{(t)}$ is close to $\U_k\t \uhat_k^{(t)}$ (see Lemma \ref{lem:orthogonalmatrixlemma}).  The first term requires additional analysis. 
For ease of exposition, consider the case $k = 1$. Recall that $\uhat_1^{(t)}$ are defined as the left singular vectors of the matrix  %\begin{align*}
    %\mathbf{\widehat T}_1 \uhat_2^{(t-1)} \otimes \uhat_3^{(t-1)} &= 
    $\mathbf{\hat T}_1^{(t)} := \mathbf{ T}_1 \uhat_2^{(t-1)} \otimes \uhat_3^{(t-1)} + \mathbf{Z}_1 \uhat_2^{(t-1)} \otimes \uhat_3^{(t-1)},$
%\end{align*}
where $\mathbf{T}_1 = \mathcal{M}_1(\mathcal{T})$ and $\mathbf{Z}_1 = \mathcal{M}_1(\mathcal{Z})$.  
%To analyze $\uhat_1^{(t)}$ more directly, we consider the corresponding Gram matrix.  Then $\uhat_1^{(t)}$ are the \emph{eigenvectors} of the matrix
%\begin{align*}
%     \mathbf{T}_1 \mathcal{P}_{\uhat_{2}^{(t-1)}} \otimes \mathcal{P}_{\uhat_3^{(t-1)}} \mathbf{T}_1\t + \mathbf{Z}_1 \mathcal{P}_{\uhat_{2}^{(t-1)}} \otimes \mathcal{P}_{\uhat_3^{(t-1)}} \mathbf{T}_1\t + \mathbf{T}_1 \mathcal{P}_{\uhat_{2}^{(t-1)}} \otimes \mathcal{P}_{\uhat_3^{(t-1)}} \mathbf{Z}_1\t + \mathbf{Z}_1 \mathcal{P}_{\uhat_{2}^{(t-1)}} \otimes \mathcal{P}_{\uhat_3^{(t-1)}} \mathbf{Z}_1\t.
% \end{align*}
% Let this matrix be denoted $\mathbf{\widehat T}_1^{(t)}$, and note that $\uhat_1^{(t)}$ satisfies $\uhat_1^{(t)} = \mathbf{\widehat T}_1^{(t)} \uhat_1^{(t)} (\mathbf{\widehat{\Lambda}}_1^{(t)})^{-2},$
% where $\mathbf{\widehat{\Lambda}}_1^{(t)}$ are the singular values of the matricizations of the iterates (and hence $(\mathbf{\widehat{\Lambda}}_1^{(t)})^2$ are the eigenvalues of $\mathbf{\widehat T}_1^{(t)}$). 
\textcolor{black}{To analyze $\mathbf{\uhat}_1^{(t)}$ further, using the fact that }
%With this identity together with the fact that 
$(\mathbf{I} - \U_1 \U_1\t) \mathbf{T}_1 = 0$ \textcolor{black}{and the eigenvector-eigenvalue equation} yields the identity
\begin{align*}
  e_m\t\bigg(  \uhat_1^{(t)} - \U_1 \mathbf{W}_1^{(t)} \bigg) %&= e_m\t (\mathbf{I} - \U_k \U_k\t) \uhat_k^{(t)} + e_m\t \U_k ( \U_k\t\uhat_k^{(t)} - \mathbf{W}_k^{(t)})  \\
  &= e_m\t (\mathbf{I} - \U_1 \U_1\t) \mathbf{Z}_1 \mathcal{P}_{\uhat_{2}^{(t-1)}} \otimes \mathcal{P}_{\uhat_3^{(t-1)}} \mathbf{T}_1\t \uhat_k^{(t)} (\mathbf{\widehat{\Lambda}}_1^{(t)})^{-2} \\
   &\qquad + e_m\t (\mathbf{I} - \U_1 \U_1\t) \mathbf{Z}_1 \mathcal{P}_{\uhat_{2}^{(t-1)}} \otimes \mathcal{P}_{\uhat_3^{(t-1)}} \mathbf{Z}_1\t \uhat_1^{(t)} (\mathbf{\widehat{\Lambda}}_1^{(t)})^{-2}; \\
   &\qquad + e_m\t \U_1 ( \U_1\t\uhat_1^{(t)} - \mathbf{W}_1^{(t)}) \\
   &\eqqcolon e_m\t \mathbf{L}_1^{(t)} + e_m\t \mathbf{Q}_1^{(t)} + e_m\t \U_1 ( \U_1\t\uhat_1^{(t)} - \mathbf{W}_1^{(t)}),
\end{align*}
where the first two terms represent the \emph{linear error} and \emph{quadratic error} respectively, \textcolor{black}{and $\mathbf{\hat \Lambda}_1^{(t)}$ is the diagonal matrix of empirical singular values of the matrix $\mathbf{\hat T}_1^{(t)}$}. For ease of exposition we will focus on the linear error.  Observe that the linear $m$'th row of the linear error is a linear combination of the random variables in the matrix $\mathbf{Z}_1$ and the previous iterates $\mathbf{\hat U}_k^{(t)}$, which depend on $\mathbf{Z}_1$, and hence we cannot appeal to standard concentration inequalities for sums of independent random variables. 

%Again for ease of exposition%, consider the linear error.  To analyze a row of the linear error, one needs to analyze the vector
%\begin{align*}
%    e_m\t (\mathbf{I} - \U_1 \U_1\t) \mathbf{Z}_1 \mathcal{P}_{\uhat_{2}^{(t-1)}} \otimes \mathcal{P}_{\uhat_3^{(t-1)}} \mathbf{T}_1\t \uhat_k^{(t)} (\mathbf{\widehat{\Lambda}}_k^{(t)})^{-2}.
%\end{align*}
%Ideally we would like to argue that this behaves like a sum of independent random variables in $\mathbf{Z}_1$, but this is not true, as there is nontrivial dependence between $\mathbf{Z}_1$ and the projection matrix $\mathcal{P}_{\uhat_{2}^{(t-1)}} \otimes \mathcal{P}_{\uhat_3^{(t-1)}}$.  

The primary argument behind the leave-one-out analysis technique is to define a sequence that is independent from the random variables in $e_m\t \mathbf{Z}_1$. \textcolor{black}{By leveraging both the independence  of the constructed sequence and its close proximity to the true sequence, it is possible to obtain sharp $\ell_{2,\infty}$ concentration.}
%Hand waving slightly, the prototypical leave-one-out analysis for eigenvector and singular vector analyses (e.g., \citet{cai_subspace_2021,abbe_entrywise_2020}) argues that
%\begin{align*}
 %   \| e_m\t \mathbf{Z}_1 \uhat \| &\leq \| e_m\t \mathbf{Z}_1 (\uhat - \mathbf{U}^{\mathrm{LOO}} ) \| + \| e_m\t \mathbf{Z}_1 \mathbf{U}^{\mathrm{LOO}} \| \lesssim \| \mathbf{Z}_1 \| \| \sin\Theta(\uhat, \mathbf{U}^{\mathrm{LOO}}) \| + \| e_m\t \mathbf{Z}_1 \mathbf{U}^{\mathrm{LOO}} \|,
%\end{align*}
%where $\mathbf{U}^{\mathrm{LOO}}$ is the matrix obtained by zeroing out the $m$'th row of $\mathbf{Z}_1$.  By independence, the second term can be handled via standard concentration inequalities, and the first term is typically handled by standard tools from matrix perturbation theory, which is sufficient as the leave-one-out sequence is extremely close to the true sequence $\uhat$.   
\textcolor{black}{As HOOI outputs estimated tensor singular vectors, one na\"ive approach is to apply the arguments in, for example, \citet{chen_spectral_2021} or \citet{cai_subspace_2021} for studying matrix eigenvectors and singular vectors.}
For tensors, this approach suffers from two drawbacks: one is that \textcolor{black}{such an argument often relies on bounds on the} spectral norm of $\mathbf{Z}_1$, \textcolor{black}{which} may be very large in the tensor setting -- of order $\sqrt{p_2 p_3} \asymp p$.  However, this difficulty can be managed by leveraging the Kronecker structure that arises from the HOOI procedure \textcolor{black}{together with the subgaussian noise}.  %\citet{zhang_tensor_2018} showed that
% \begin{align*}
%     \sup_{\|\U_1\|=1, \|\U_2\|=1, \mathrm{rank}(\U_1, \U_2) \leq 2r} \| \mathbf{Z}_1 (\U_1 \otimes \U_2) \| \lesssim \sqrt{pr}
% \end{align*}
% when $\sigma=1$, which eliminates the naive upper bound of order $p$ when $r \ll p$. % However, this bound  holds only for subgaussian noise, as it relies quite heavily on an $\eps$-net argument, which is in general suboptimal for heavy-tailed or Bernoulli noise. %\footnote{By  way of analogy, for an $n\times n$ mean-zero Bernoulli noise matrix $\mathbf{E}$ with entrywise variance at most $\rho_n$, an $\eps$-net argument only yields $\|\mathbf{E}\|\lesssim \sqrt{n}$, whereas a more refined argument as in \citet{bandeira2016sharp} yields $\|\mathbf{E}\| \lesssim \sqrt{n\rho_n}.$}.
\textcolor{black}{However, similar arguments turn out to fail for Bernoullli noise}; see e.g., \citet{jing_community_2021}, \citet{ke_community_2020}, \citet{yuan_incoherent_2017} for \textcolor{black}{alternative approaches using so-called ``tensor concentration inequalities'' that bound such terms in the sparse Bernoulli noise regime. Unfortunately, it is the lack of availability of such bounds in the literature for our setting that preclude analyzing sparse Bernoulli noise, and our proof is already extremely lengthy without introducing these results.  Nonetheless, we surmise that with the advent of such results it will be possible to extend our arguments to the sparse Bernoulli setting. }%methods to handle similar terms for other types of noise.  

The second drawback has to do with the leave-one-out sequence definition, \textcolor{black}{and highlights the main technical novelty of our proof}.  Suppose one defines the leave-one-out sequence by setting the $m$'th row of $\mathbf{Z}_1$ to zero and running HOOI with this new noise matrix. %, as one may be most tempted to do. 
It can be shown that using this na\"ive approach that the $\sin\Theta$ distance between the true sequence and the leave-one-out sequence defined in this manner will depend on a quantity that actually \emph{increases} with respect to the condition number of $\mathcal{T}$ (see the \cref{sec:extrainfo} for details).  Therefore, in order to eliminate this problem, we carefully construct a novel modified leave-one-out sequence  $\utilde_k^{(t)}$ that can eliminate the dependence on $\mathbf{T}_1$ as follows.

%First, we set the initialization $\tilde{\U}_1^{(0,1-m)}$ as the left singular vectors obtained via the diagonal deletion of $\mathcal{\widehat{T}}$, only with the $m$'th row of $\mathbf{Z}_1$ set to zero. 
First, let $\mathbf{Z}_k^{1-m}$ denote the matrix $\mathbf{Z}_k = \mathcal{M}_k(\mathcal{Z})$ with the entries associated to the $m$'th row of $\mathbf{Z}_1$ set to zero (note that in this manner $\mathbf{Z}_k - \mathbf{Z}_k^{1-m}$ will consist of sparse nonzero \emph{columns}). \textcolor{black}{We then introduce corresponding leave-one-out sequences for each other mode $k$ defined by first} setting $\utilde_k^{(0,1-m)}$  as the leading $r_k$ eigenvectors of the hollowed gram matrix of $\mathbf{T}_k + \mathbf{Z}_k^{1-m}$, 
%\begin{align*}
%\utilde_k^{(0,1-m)} = \mathrm{SVD}_{r_k} \bigg(
%    \Gamma\big( \mathbf{T}_k \mathbf{T}_k\t + \mathbf{Z}_k^{1-m} \mathbf{T}_k\t + \mathbf{T}_k (\mathbf{Z}_k^{1-m})\t + \mathbf{Z}_k^{1-m} (\mathbf{Z}_k^{1-m})\t \big) \bigg),
%\end{align*}
so that $\utilde_k^{(0,1-m)}$ is independent from the $m$'th row of $\mathbf{Z}_1$ (for each $k$). We now set $\utilde_k^{(t,1-m)}$ inductively via $\utilde_k^{(t,1-m)} \coloneqq \SVD_{r_k}\big( \mathbf{T}_k + \mathbf{Z}_k^{1-m}  \mathcal{\tilde P}_{k}^{t,1-m} \big)$, where $ \mathcal{\tilde P}_{k}^{t,1-m}$ is \textcolor{black}{a projection onto the subspace corresponding to the previous leave-one-out sequence iterates $\mathbf{\tilde U}_k^{(t,1-m)}$}.
% %defined inductively as the projection matrix
% \begin{align*}
%     \mathcal{\tilde P}_{k}^{t,1-m} &\coloneqq \begin{cases}
%     \mathcal{P}_{\utilde_{k+1}^{(t-1,1-m)} \otimes \utilde_{k+2}^{(t-1,1-m)}} & k =1; \\
%      \mathcal{P}_{\utilde_{k+1}^{(t-1,1-m)} \otimes \utilde_{k+2}^{(t,1-m)}} & k =2; \\
%       \mathcal{P}_{\utilde_{k+1}^{(t,1-m)} \otimes \utilde_{k+2}^{(t,1-m)}} & k =3,
%     \end{cases}
% \end{align*}
%which is the projection  matrix corresponding to the previous two iterates of the leave-one-out sequence.  
\textcolor{black}{In particular, %each $\mathbf{\tilde U}_k^{(t,1-m)}$ is independent from the $m$'th row of $\mathbf{Z}_1$}.  \textcolor{black}{In addition, 
this additional projection matrix $\mathcal{\tilde P}_k^{t,1-m}$ serves to denoise $\mathbf{Z}_k$ and  preserves independence while doing so. }
%preserves independence while inducing additional denoising along a subspace that is independent from the $m$'th row of $\mathbf{Z}_1$.}
%Note that this matrix is \emph{still} independent from the $m$'th row of $\mathbf{Z}_1$, and hence each sequence $\utilde_k^{(t,1-m)}$ is independent from the 
%$m$'th row of $\mathbf{Z}_1$.  

\textcolor{black}{As an additional side benefit of this construction, the \emph{projection} matrix $\utilde_1^{(t,1-m)} (\utilde_1^{(t,1-m)})\t$} is also the projection onto the dominant left singular space of the matrix $\bigg(\mathbf{T}_1 + \mathbf{Z}_1^{1-m} \mathcal{\tilde P}_{1}^{t,1-m} \bigg) \uhat_2^{(t-1)} \otimes \uhat_3^{(t-1)}$
as long as an eigengap condition is met (\cref{lem:eigengaps}).  As a consequence, the true sequence and the leave-one-out sequence depend only on the difference matrix $\mathbf{Z}_1^{1-m} \mathcal{\tilde P}_{1}^{t,1-m}  \uhat_2^{(t-1)} \otimes \uhat_3^{(t-1)} - \mathbf{Z}_1 \uhat_2^{(t-1)} \otimes \uhat_3^{(t-1)},$
which  \textcolor{black}{can be shown to depend only on the $m$'th row random matrix $\mathbf{Z}_1$ } \textcolor{black}{and the proximity of $\utilde_k^{(t-1,1-m)}$ to $\uhat_k^{(t-1)}$}.  With this novel leave-one-out construction we can obtain good bounds on the $\sin\Theta$ distance between the leave-one-out sequence and the true sequence (c.f., Lemma \ref{lem:leaveoneoutsintheta}).  

Finally, the exposition above has focused on the case $k = 1$.  Since there are three modes and we prove the result by induction, we actually need to \textcolor{black}{repeat this argument for each mode,} and we control each of these sequences simultaneously at each iteration.  Our final proof requires careful tabulation of the probabilities of the events defined by each of these separate sequences.  To ease the analysis, we first bound each term deterministically under eigengap conditions, and then further obtain probabilistic bounds by induction using the leave-one-out sequences.

\section{Discussion} \label{sec:discussion}
In this paper, we have considered the tensor mixed-membership blockmodel, which generalizes the tensor blockmodel to settings where communities are no longer discrete.  By studying the $\ell_{2,\infty}$ perturbation of the HOOI algorithm, we \textcolor{black}{obtain an estimator with convergence guarantees uniformly across the memberships} provided there are pure nodes along each mode. By applying our proposed algorithm to real data, we have identified phenomena that are not feasible to obtain in the discrete community setting.

It is natural to consider estimating the mixed memberships of the higher-order tensors. Suppose one observes a tensor $\mathcal{\hat T} \in \mathbb{R}^{p_1 \times p_2 \times \cdots  \times p_d}$.  Our algorithm and methodology naturally extend to this setting, with the only modification being the implementation of the HOOI algorithm, which is straightforward to adapt to the higher-order setting. By adapting our main arguments, we can prove the following informal result.

%\textcolor{black}{
\begin{theorem}[Estimation of mixed memberships for higher-order tensors; informal]
Suppose that $r_{\max} \lesssim p_{\min}^{1/(d-1)}$, that $r_{\max} \asymp r$ with $r \lesssim r_{\min}$, and that $\kappa^2 \lesssim p_{\min}^{1/(2(d-1))}$.   Suppose that the smallest singular value of $\mathcal{S}$ satisfies $\Delta^2/\sigma^2 \gtrsim \frac{\kappa^2 p_{\max}^2 r_1 \cdots r_d}{p_1 \cdots p_d \sqrt{p_{\min}^{(d-1)/(d-2)}}}$.   Let $\mathbf{\widehat \Pi}_k$ be the output of \cref{al:spamm} (with HOOI adapted to order $d$) with  $t$ iterations for $t \asymp \textcolor{black}{\log\bigg( \frac{\Delta/\sigma (p_1 \cdots p_d)^{1/2}}{C_0 \kappa \sqrt{p_{\min} \log(p_{\max})}(r_1r_2r_3)^{1/2}}\bigg)}$. % \log\bigg( \frac{\kappa p_{\max} (r_1 \cdots r_d)^{1/2}}{(\Delta/\sigma) (p_1 \cdots  p_d)^{1/2}} \bigg) \vee 1$. 
Then with probability at least $1 - p_{\max}^{-10}$, there exist $d$ permutation matrices $\mathcal{P}_k \in \mathbb{R}^{r_k \times r_k}$ such that for each $k$ %{\red (We probably still need to specify $p$ here. Is $p$ here any positive integer, or it has to be of similar order as $p_1, p_2$ or $p_3$?)} \textcolor{ForestGreen}{fixed.}
\begin{align*}
 \max_{1\leq i\leq p_k} \| \big(\mathbf{\Pi}_k -\mathbf{\widehat \Pi}_k \mathcal{P}_k\big)_{i\cdot} \| &\lesssim_d \frac{\kappa \sqrt{r^d\log(p_{\max})}}{(\Delta/\sigma) (p_{-k})^{1/2}}.
\end{align*}
%Consequently, when $p_k \asymp p$, it holds that
%\begin{align*}
%\max_{1\leq i\leq p_k} \| \big(\mathbf{\Pi}_k -\mathbf{\widehat \Pi}_k \mathcal{P}_k\big)_{i\cdot} \| &\lesssim_d \frac{\kappa  \sqrt{r^d  \log(p)}}{(\Delta/\sigma) p^{(d-1)/2}}.
%\end{align*}
Here $a \lesssim_d b$ means that the implicit constant depends on the number of modes $d$.
\end{theorem}
As in the order three setting, we see an improvement in the error rate of order $\sqrt{p}$ for each additional mode, albeit at the cost of a slightly stronger signal-strength condition and condition on $r$. In future work it may be interesting to determine the dependence of the implicit constants on the order $d$.
%}

In other future work, it may be natural to extend the mixed-membership tensor blockmodel to allow degree corrections as in \citet{jin_estimating_2017} or \citet{hu_multiway_2022}.   %Furthermore, the analysis in this paper focuses on subgaussian noise, whereas many multilayer network datasets have Bernoulli noise, as well as different types of symmetry, and a natural extension would encompass noise and structures of this form, and would also perhaps include missingness.  %\textcolor{black}{
 It is also possible that our results can be extended to Bernoulli noise  or correlated noise, but the analysis will be significantly more complicated to account for the interplay between tensorial structure and noise.
%Beyond these natural extensions of the model it is of interest to extend the $\ell_{2,\infty}$ perturbation theory covered herein to other regimes of signal strength, as well as provide entrywise bounds for other types of low-rank structures beyond the Tucker low-rank assumption.  
Furthermore, it may be relevant to develop distributional theory for the outputs of tensor SVD, and to obtain principled confidence intervals for the outputs of HOOI. %\textcolor{black}{
%Our analyses are not immediately amenable to developing such theory, but  it may be possible extend our leave-one-out constructions to be able to obtain distributional characterizations. 
 %}

\begin{comment}
\section*{Acknowledgements}
JA’s research was partially supported by a fellowship from the Johns
Hopkins Mathematical Institute of Data Science (MINDS) via its NSF TRIPODS award CCF-1934979, through the Charles and Catherine Counselman Fellowship, and through support from the Acheson J. Duncan Fund for the Advancement of Research in Statistics. ARZ's research was supported in part by NSF Grant CAREER-2203741 and NIH Grant R01 GM131399. JA also thanks Jes\'us Arroyo for the time series flight data studied in \cref{sec:usaflightdata}. 
\end{comment}

%% file: appendix.tex
\section{Additional Information on the Na\"ive Leave-one-out Sequence} \label{sec:extrainfo}
Let $\mathbf{\check{U}}_k^{(t)}$ denote the output of the leave-one-out sequence described in \cref{sec:proofoverview}, where one simply runs HOOI on the tensor with the entries of $\mathcal{Z}$ corresponding to the $m$'th row of $\mathbf{Z}_1$ set to zero. Then the $\sin\Theta$ distance between the true sequence and the leave-one-out sequence  for mode 1 will depend on the difference matrix
\begin{align}
   \bigg( \mathbf{T}_1 &\uhat_2^{(t-1)} \otimes \uhat_3^{(t-1)} + \mathbf{Z}_1 \uhat_2^{(t-1)} \otimes \uhat_3^{(t-1)}  \bigg) -\bigg( \mathbf{T}_1 \mathbf{\check{U}}_2^{(t-1)} \otimes \mathbf{\check{U}}_3^{(t-1)}+ \mathbf{Z}_1^{1-m} \mathbf{\check{U}}_2^{(t-1)} \otimes \mathbf{\check{U}}_3^{(t-1)}\bigg) \nonumber \\
    &= \mathbf{T}_1 \bigg( \uhat_2^{(t-1)} \otimes \uhat_3^{(t-1)} - \mathbf{\check{U}}_2^{(t-1)} \otimes \mathbf{\check{U}}_3^{(t-1)} \bigg) + \mathbf{Z}_1 \bigg(\uhat_2^{(t-1)} \otimes \uhat_3^{(t-1)} - \mathbf{\check{U}}_2^{(t-1)} \otimes \mathbf{\check{U}}_3^{(t-1)} \bigg) \nonumber \\
    &\qquad + \begin{pmatrix} \cdots 0 \cdots \\ e_m\t \mathbf{Z}_1  \\ \cdots 0 \cdots
 \end{pmatrix} \mathbf{\check{U}}_2^{(t-1)} \otimes \mathbf{\check{U}}_3^{(t-1)}, \label{baddif}
\end{align}
where we define $\mathbf{Z}_1^{1-m}$ as the matrix $\mathbf{Z}_1$ with the $m$'th row %\textcolor{ForestGreen}{
set to zero, and the final term is only nonzero in its $m$'th row.
%} {\red (Please check carefully this ``removed to zero" issue. Please double check the rest of the paper :-)}
efThe second two terms (containing $\mathbf{Z}_1$) can be shown to be quite small by appealing to spectral norm bounds together with the Kronecker structure (e.g., Lemma \ref{lem:taubound}).  However, the first term depends on both the matrix $\mathbf{T}_1$ and the proximity of the leave-one-out sequence to the true sequence.  If one simply bounds this term in the spectral norm, the $\sin\Theta$ distance may end up  \emph{increasing} with respect to the condition number $\kappa$, and hence may be much larger than the concentration for $\mathbf{Z}_1$ (and may not shrink to zero sufficiently quickly).

\section{The Cost of Ignoring Tensorial Structure}
Perhaps the simplest tensor singular vector estimation procedure for Tucker low-rank tensors is the HOSVD algorithm, which simply takes the singular vectors of each matricization of $\mathcal{\widehat{T}}$ and outputs these as the estimated singular vectors.  %which effecively ignores the tensorial structure.  The HOSVD procedure simply takes the singular vectors of each matricization of $\mathcal{\widehat{T}}$ and outputs these as the estimated singular vectors.   
We discuss briefly why HOSVD-like procedures instead of HOOI in  \cref{al:spamm} may not yield the same estimation error as in \cref{thm:estimation}, particularly in the high-noise setting. For simplicity we focus on the regime $\kappa,r,\mu_0 \asymp 1$ and $p_k \asymp p$.  

Recall that we do not assume that the noise is homoskedastic.  It has previously been demonstrated that in the presence of heteroskedastic noise, HOSVD can yield biased estimates \citep{zhang_heteroskedastic_2022}. Therefore, in order to combat bias, one could modify the HOSVD procedure and instead use a procedure that eliminates the bias term stemming from either vanilla HOSVD or diagonal-deleted SVD. It was shown in \citet{agterberg_entrywise_2022} and \citet{yan_inference_2021} that the output of the \texttt{HeteroPCA} algorithm proposed in \citet{zhang_heteroskedastic_2022}  after sufficiently many iterations yields the high-probability upper bound
\begin{align*}
    \| \uhat_k^{(\mathrm{HeteroPCA})} - \U_k \mathbf{W}_k \|_{2,\infty} &\lesssim  \frac{\sqrt{\log(p)}}{\lambda/\sigma} + \frac{p\log(p)}{(\lambda/\sigma)^2},
\end{align*}
where $\uhat_k^{(\mathrm{HeteroPCA})}$ denotes the estimated singular vectors obtained by applying the \texttt{HeteroPCA} algorithm after sufficiently many iterations.  Note that this upper bound does not suffer from any bias; in essence, this is the sharpest $\ell_{2,\infty}$ bound in the literature  for any procedure that ignores tensorial structure.

Suppose one  uses the estimate $\uhat_k^{(\mathrm{HeteroPCA})}$ to estimate $\mathbf{\Pi}_k$ via \cref{al:spamm}, and let $\mathbf{\widehat \Pi}_k^{(\mathrm{HeteroPCA})}$ denote the output of this procedure.  Arguing as in our proof of \cref{thm:estimation}, by applying the results of \citet{gillis_fast_2014} and \cref{lem:relationship}, using this bound we will obtain that
\begin{align*}
    \| \mathbf{\widehat \Pi}_k^{(\mathrm{HeteroPCA})} - \mathbf{\Pi}_k \mathcal{P} \|_{2,\infty} &\lesssim \frac{\sqrt{\log(p)}}{(\Delta/\sigma) p} + \frac{\log(p)}{(\Delta/\sigma)^2 p^{3/2}}.
\end{align*}
In the challenging regime $\Delta/\sigma \asymp \frac{\sqrt{\log(p)}}{p^{3/4}}$ (recall that by Assumption \ref{assumption:signalstrength} we must have that $\frac{\Delta^2}{\sigma^2} \gtrsim \frac{\log(p)}{p^{3/2}}$), the above bound translates to
\begin{align*}
      \| \mathbf{\widehat \Pi}_k^{(\mathrm{HeteroPCA})} - \mathbf{\Pi}_k \mathcal{P} \|_{2,\infty} &\lesssim \frac{1}{p^{1/4}} + 1 \asymp 1,
\end{align*}
which does not tend to zero as $p \to \infty$. Therefore, in this high-noise regime, the estimates obtained via \texttt{HeteroPCA} (or any similar procedure that ignores the tensorial structure) may not even be consistent.  In contrast, \cref{thm:estimation} shows that in this regime our proposed estimation procedure yields the upper bound
\begin{align*}
      \| \mathbf{\widehat \Pi}_k - \mathbf{\Pi}_k \mathcal{P} \|_{2,\infty} &\lesssim \frac{1}{p^{1/4}},
\end{align*}
which still yields consistency, even in the high-noise regime.

\section{Additional Numerical Results}
In this section we further discuss the data analysis in \cref{sec:numericalresults} as well as provide an additional analysis of the global flight data studied in \citet{han_exact_2021}.

\subsection{Simulations} \label{sec:sims}
\begin{figure*}[t!]
\centering
        \subfloat %[$\ell_{2,\infty}$ estimation error defined via $\mathrm{err} =  \min_{\mathcal{P}} \| \mathbf{\Pi}_1 - \mathbf{\widehat \Pi}_1 \mathcal{P} \|_{2,\infty}$ with varying values of $\sigma = \sigma_{\max}$ for heteroskedastic noise averaged over $10$ runs.]
                {%
            \includegraphics[width=.4\linewidth,height=50mm]{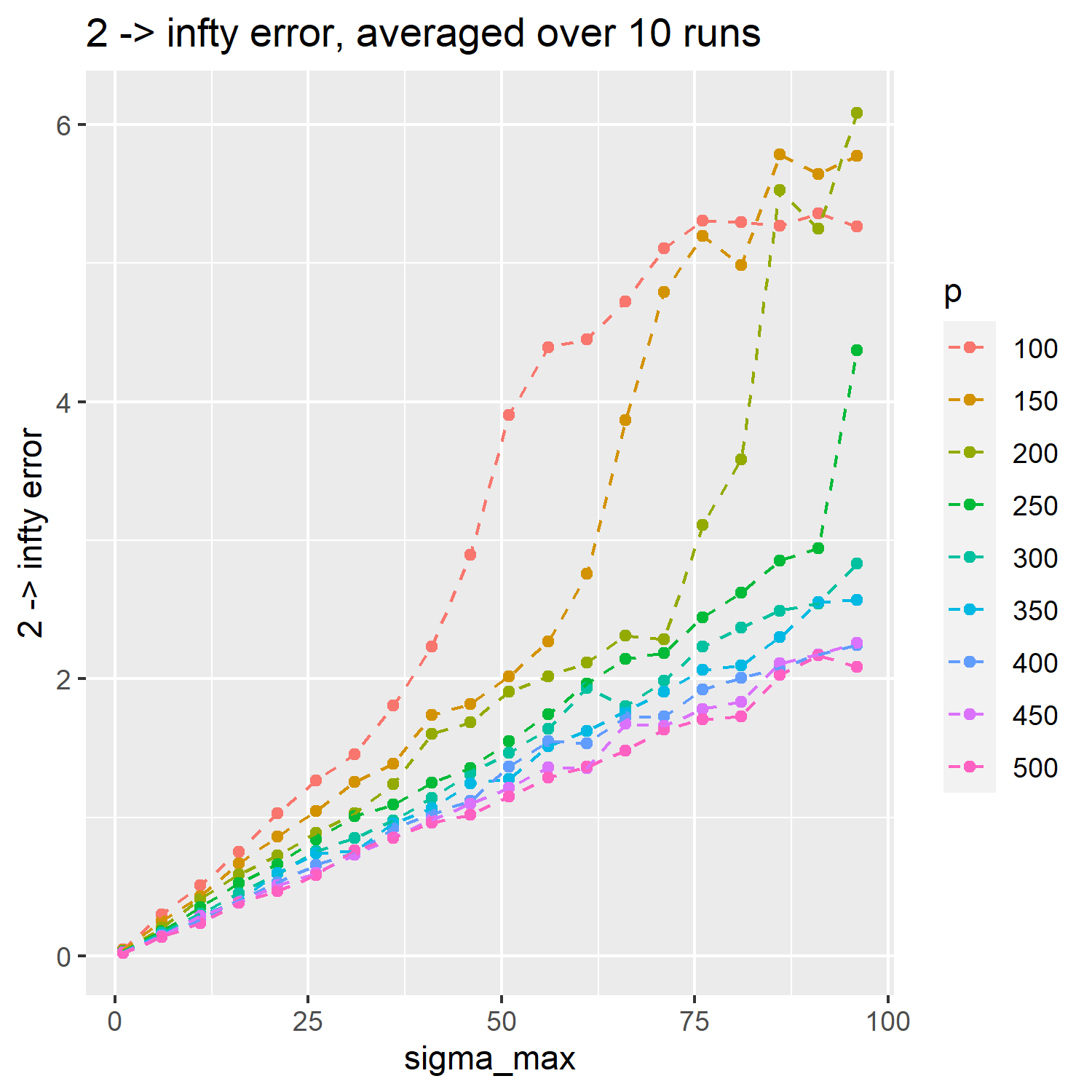}%
            \label{subfig:aveerror}%
        }\hspace{5mm}
        \subfloat%[Relative $\ell_{2,\infty}$ error defined via $\mathrm{err}/\sigma$ averaged across each $p$ for different amounts of heteroskedasticity, ranging from $1$ (least heteroskedastic) to $.25$ (most heteroskedastic).]
        { %
            \includegraphics[width=.4\linewidth,height=50mm]{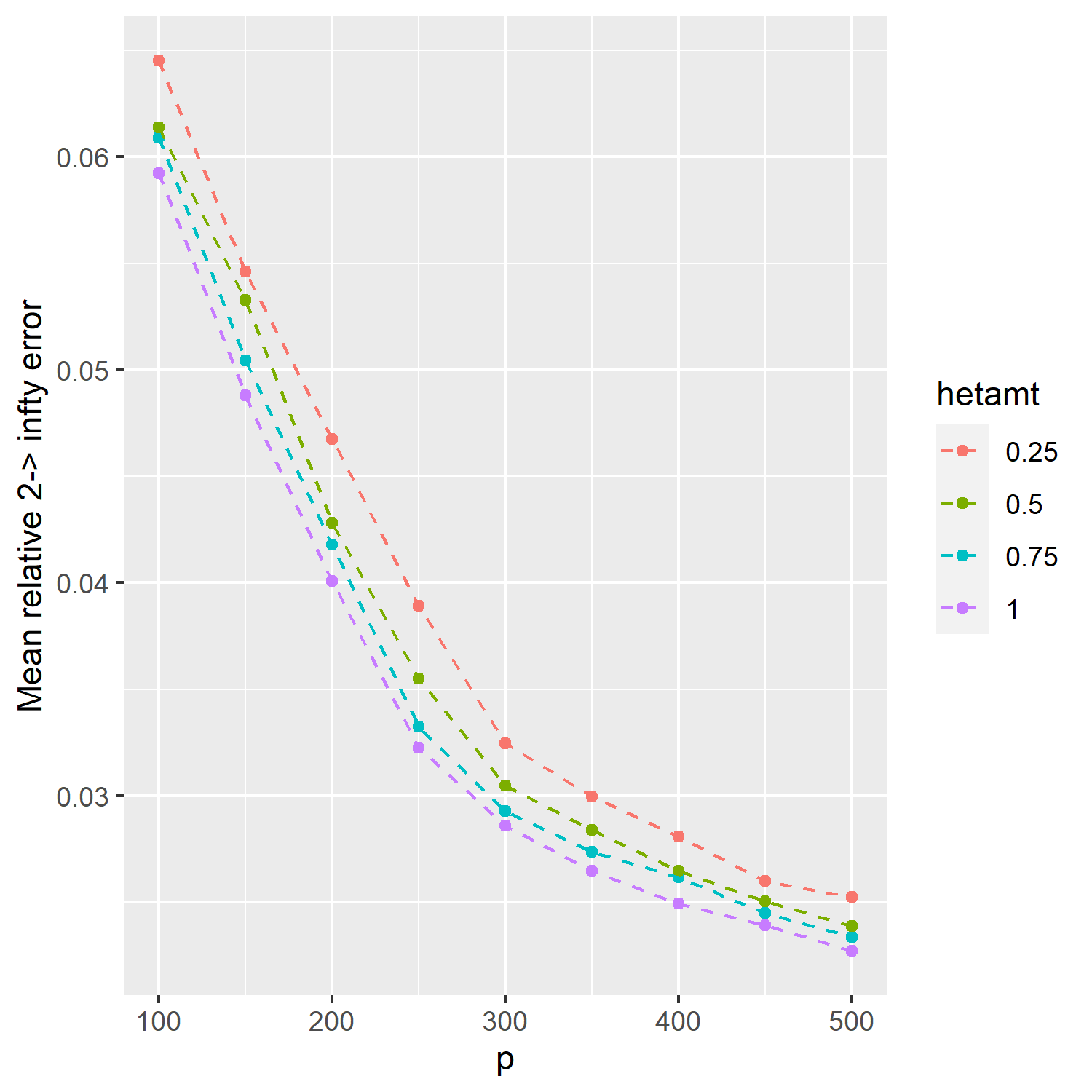}%
            \label{subfig:averror2}%
     }
        \caption{Simulated maximum node-wise errors, as described in \cref{sec:sims}. The left figure depicts relative $\ell_{2,\infty}$ estimation error with varying levels of $\sigma$ averaged over 10 runs, and the right hand figure examines relative $\ell_{2,\infty}$ estimation error averaged across $p$ with varying levels of heteroskedasticity.}
        \label{fig:usplot}
    \end{figure*}

In this section we consider the maximum row-wise estimation error for the tensor mixed membership blockmodel via \cref{al:spamm} for simulated data.  In each data setup we generate the underlying tensor by first generating the mean tensor $\mathcal{S} \in \mathbb{R}^{3\times 3 \times 3}$ with $N(0,1)$ entries and then adjusting the parameter $\Delta$ to 10.  We then draw the memberships by manually setting the first three nodes along each mode to be pure nodes, and then drawing the other vectors from a random Dirichlet distribution.  We generate the noise as follows.  First, we generate the standard deviations $\{\sigma_{ijk}\}$ via $\sigma_{ijk} \sim \sigma_{\max} \times \beta(\alpha,\alpha)$, where $\beta$ denotes a $\beta$ distribution.  The parameter $\alpha$ governs the heteroskedasticity, with $\alpha = 1$ corresponding to uniformly drawn standard deviations and $\alpha \to 0$ corresponding to ``highly heteroskedastic" standard deviations.  We then generate the noise via $\mathcal{Z}_{ijk} \sim N(0, \sigma_{ijk}^2)$.

In \cref{subfig:aveerror} we examine the $\ell_{2,\infty}$ error as a function of $\sigma = \sigma_{\max}$ for \cref{al:spamm} applied to this noisy tensor averaged over $10$ runs with $\alpha = 1$.  Here we keep the mean matrix fixed but re-draw the memberships, variances, and noise each run.  We vary $\sigma$ from $1$ to $96$ by five.  We see a clear linear relationship in the error for each value of $p$ from 100 to 500 by 50, with larger values of $\sigma_{\max}$ being significantly less accurate for smaller values of $p$.

In \cref{subfig:averror2} we consider the mean relative $\ell_{2,\infty}$ error defined as follows.  First, for each value of $\sigma_{\max}$ we obtain an estimated $\ell_{2,\infty}$ error $\mathrm{err}$ averaged over 10 runs.  We then divide this error by $\sigma_{\max}$ to put the errors on the same scale.  Finally, we average this error for all values of $\sigma_{\max}$ and plot the value as a function of $p$ for different amounts of heteroskedasticity.  We see that the error decreases in $p$ as anticipated, and slightly more heteroskedasticity results in slightly worse performance.

\subsection{Application to Global Flight Data}
\begin{figure}[t]
 \subfloat{%
            \includegraphics[width=.4\linewidth]{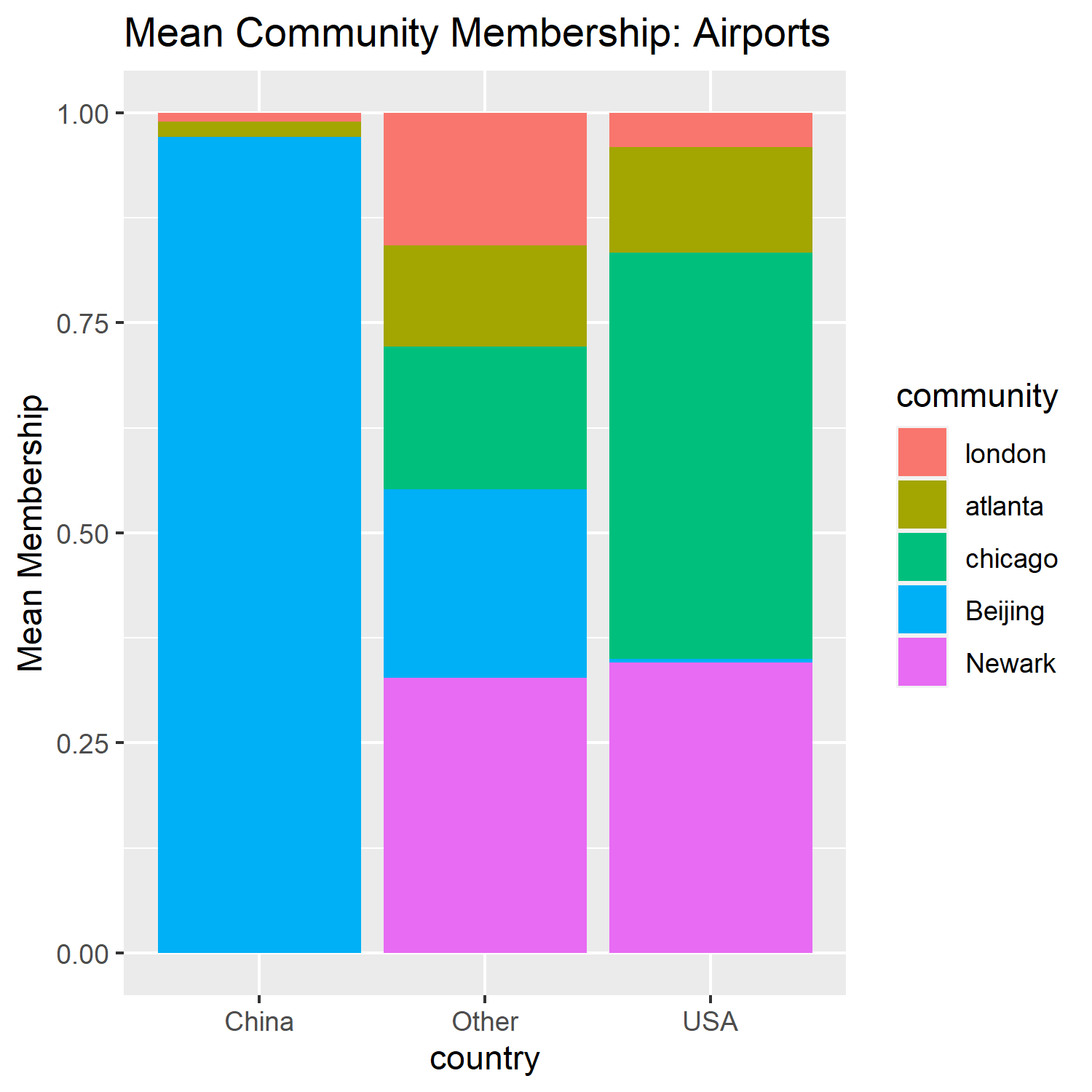}%
            \label{subfig:a}%
        }\hfill
        \subfloat{%
            \includegraphics[width=.4\linewidth]{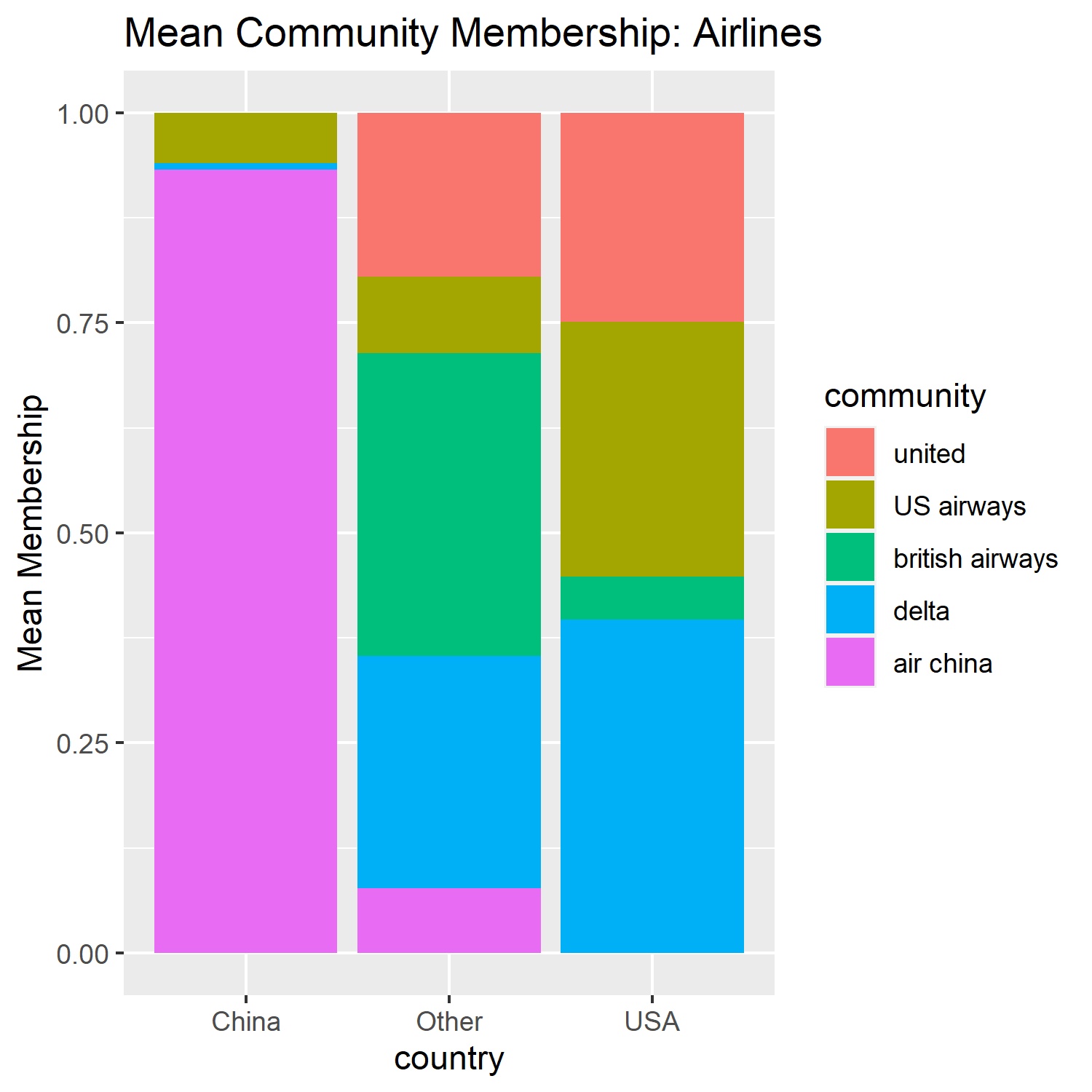}%
            \label{subfig:b}%
        }
 
    \caption{Community memberships for airports (left) and airlines (right), separated according to country to emphasize ``disconnectedness'' between Chinese airports and airlines with American airports and airlines.}   \label{fig:globalairports}
\end{figure}

We now apply our mixed-membership estimation to the flight data described in \citet{han_exact_2021}. There are initially 66,765 global flight routes from 568 airlines and 3,309 airports\footnote{\url{https://openflights.org/data.html\#route}}, 
and we preprocess similar to \citet{han_exact_2021} by considering only the top 50 airports with the highest numbers of flight routes. %in precisely the same way as \citet{han_exact_2021}, 
%$esulting in the top 50 airports with the highest numbers of flight routes.  
We end up with a tensor $\mathcal{\widehat{T}}$ of size $39\times 50 \times 50$, where each entry $\mathcal{\widehat{T}}_{i_1i_2i_3}$ is one if there is a flight route from airport $i_2$ to $i_3$ in airline $i_1$ and zero otherwise.  We use the same choice of $\br=\{5,5,5\}$ as in \citet{han_exact_2021}, chosen via the Bayesian information criterion for block models \citep{wang_multiway_2019} from candidate $r$ values ranging from 3 to 6.  When running our algorithm, occasionally there are negative or very small values of $\mathbf{\widehat \Pi}_k$; we therefore threshold and re-normalize in order to obtain our estimates.

First, our algorithm relies on identifying pure nodes along each  mode.  For the airports, the pure nodes are London, Atlanta, Chicago, Beijing, and Newark.  For airlines, we find the pure nodes to be United, US airways, British Airways, Delta, and Air China. When analyzing the output, we found that airlines and airports associated to the USA had extremely low membership in Chinese-associated pure nodes, and vice versa for Chinese airlines and airports.  Therefore, in \cref{fig:globalairports} we plot the average membership of each airport and airline associated to its home country, whether it is in China, the USA, or elsewhere. This figure demonstrates that the USA has less membership in the airline and airport communities based outside the USA; in particular almost no membership in Chinese communities, and China has almost entirely pure membership in Chinese airport and airline communities.  The other countries have nearly equal membership in each community.

Furthermore, we observe that  the USA airlines have zero membership in the ``Air China'' pure node, and the China airlines have primarily membership in the ``Air China'' pure node.  We find a similar phenomenon in the airports as well.   Interestingly, other airports (i.e., non-Chinese and non-American) do not  exhibit this phenomenon. In \citet{han_exact_2021} five clusters were found, including one that contains Beijing, which is a pure node here.  This analysis suggests that perhaps the Beijing cluster might be much more distinct from the USA cluster than the other clusters are from each other.  Airports and airlines in other countries do not exhibit such a trend -- they have memberships in all other clusters equally.  This observation is not identifiable in settings with discrete memberships, since either a node belongs to a community or does not, whereas in the tensor mixed membership blockmodel setting we can examine the strength of the membership.

\subsection{Application to USA Flight Data} \label{sec:usaflightdata}

We also apply our methods to USA flight data publicly available from the Bureau of Transportation Statistics\footnote{\url{https://transtats.bts.gov/}} and also analyzed in \citet{agterberg_joint_2022}. %and analyzed in \citet{agterberg_joint_2022}. 
We focused on the largest connected component, resulting in 343 airports with counts of flights between airports for each month from January 2016 to September 2021, resulting in 69 months of data and a $343 \times 343 \times 69$ dimensional tensor. To choose the embedding dimension, we apply the ``elbow'' method of \citet{zhu_automatic_2006}.  First, we apply the elbow procedure to the square roots of the nonnegative eigenvalues of the diagonal-deleted Gram matrix, which is the matrix we use for our initialization.  This yields $r_1  = r_2 = 3$ for the airport mode, but for the mode corresponding to time, this procedure resulted in only two nonnegative eigenvalues.  Therefore, we ran the elbow method on the vanilla singular values instead, resulting in elbows at 1 and 4.  We therefore chose $r_3 = 4$ to perform our estimation.

%with $\br = (3,3,4)$.  There are 343 airports and with counts of flights between them for each month from January 2016  
 \begin{figure*}[t!]
 \centering
        \subfloat{%
            \includegraphics[width=.25\linewidth,keepaspectratio]{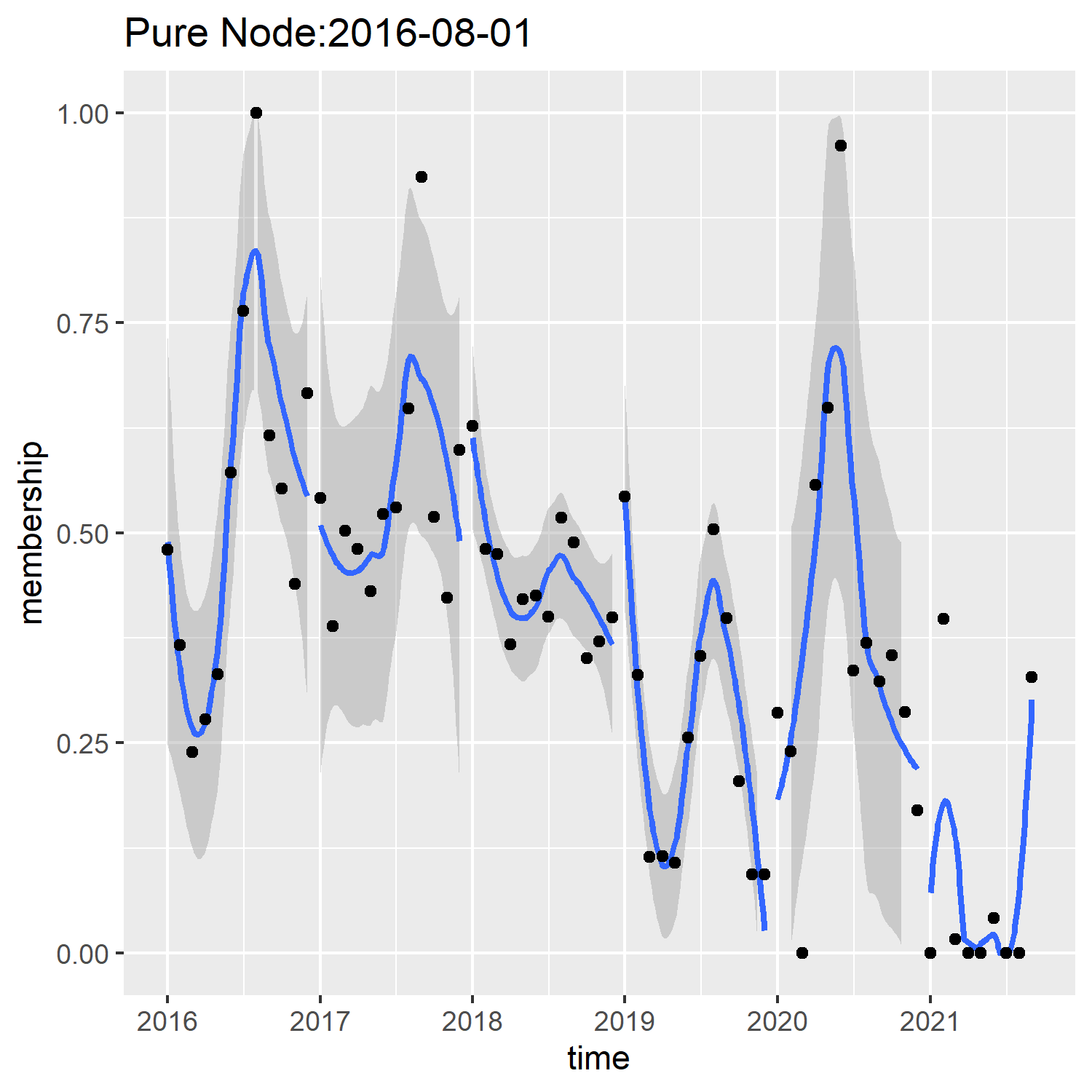}%
            \label{subfig:a}%
        }
        \subfloat{%
            \includegraphics[width=.25\linewidth,keepaspectratio]{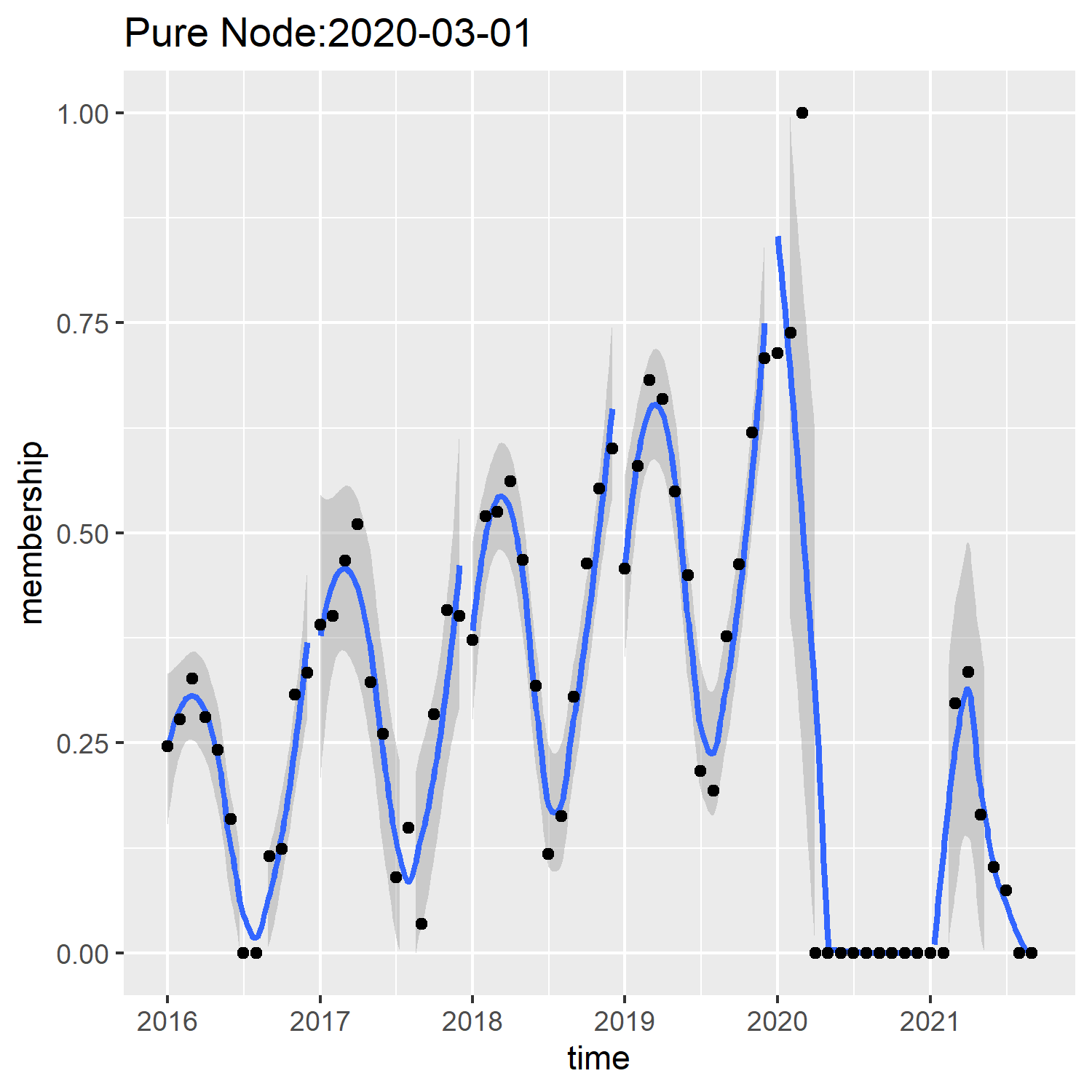}%
            \label{subfig:b}%
        }
      \subfloat{%
            \includegraphics[width=.25\linewidth,keepaspectratio]{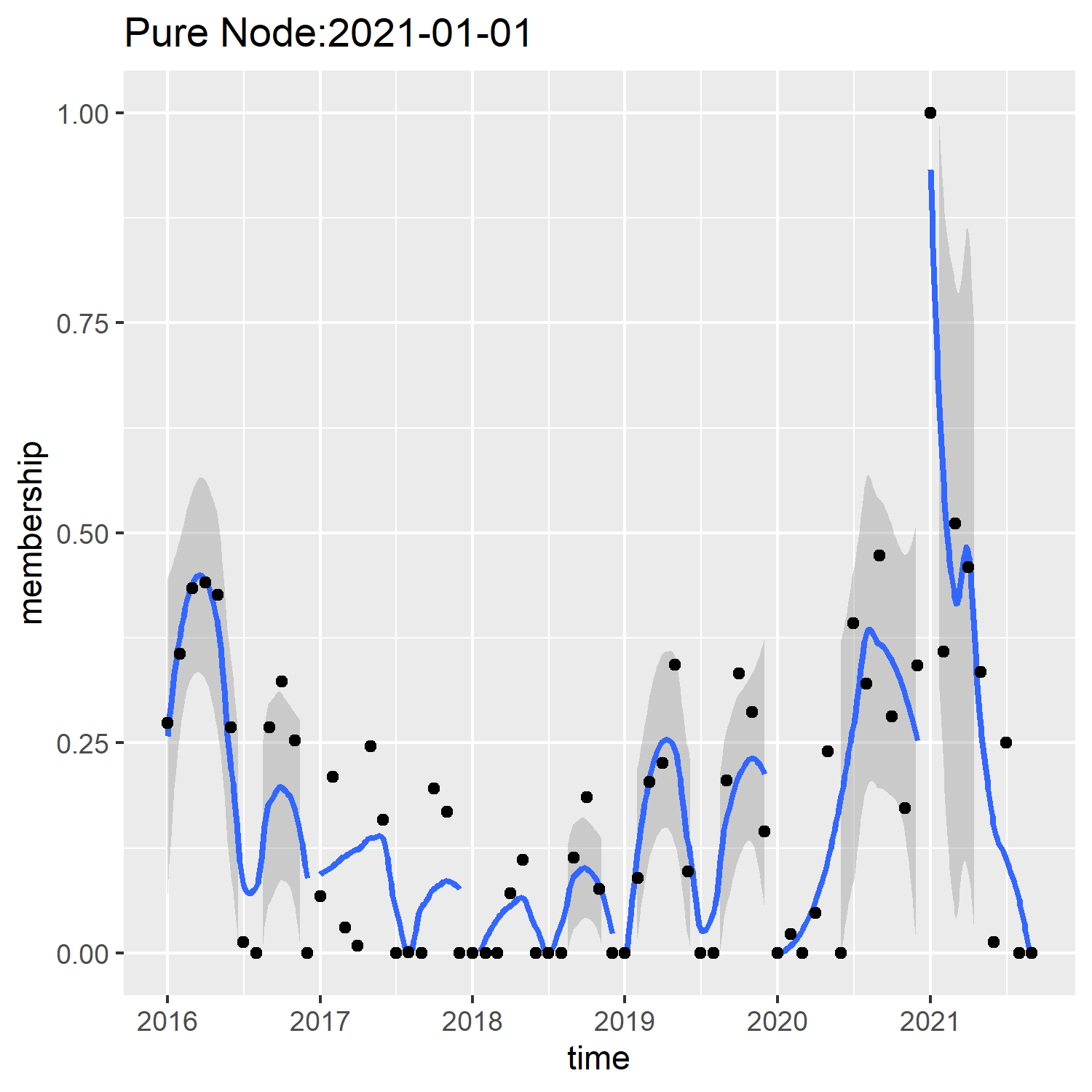}%
            \label{subfig:c}%
        }
        \subfloat{%
            \includegraphics[width=.25\linewidth,keepaspectratio]{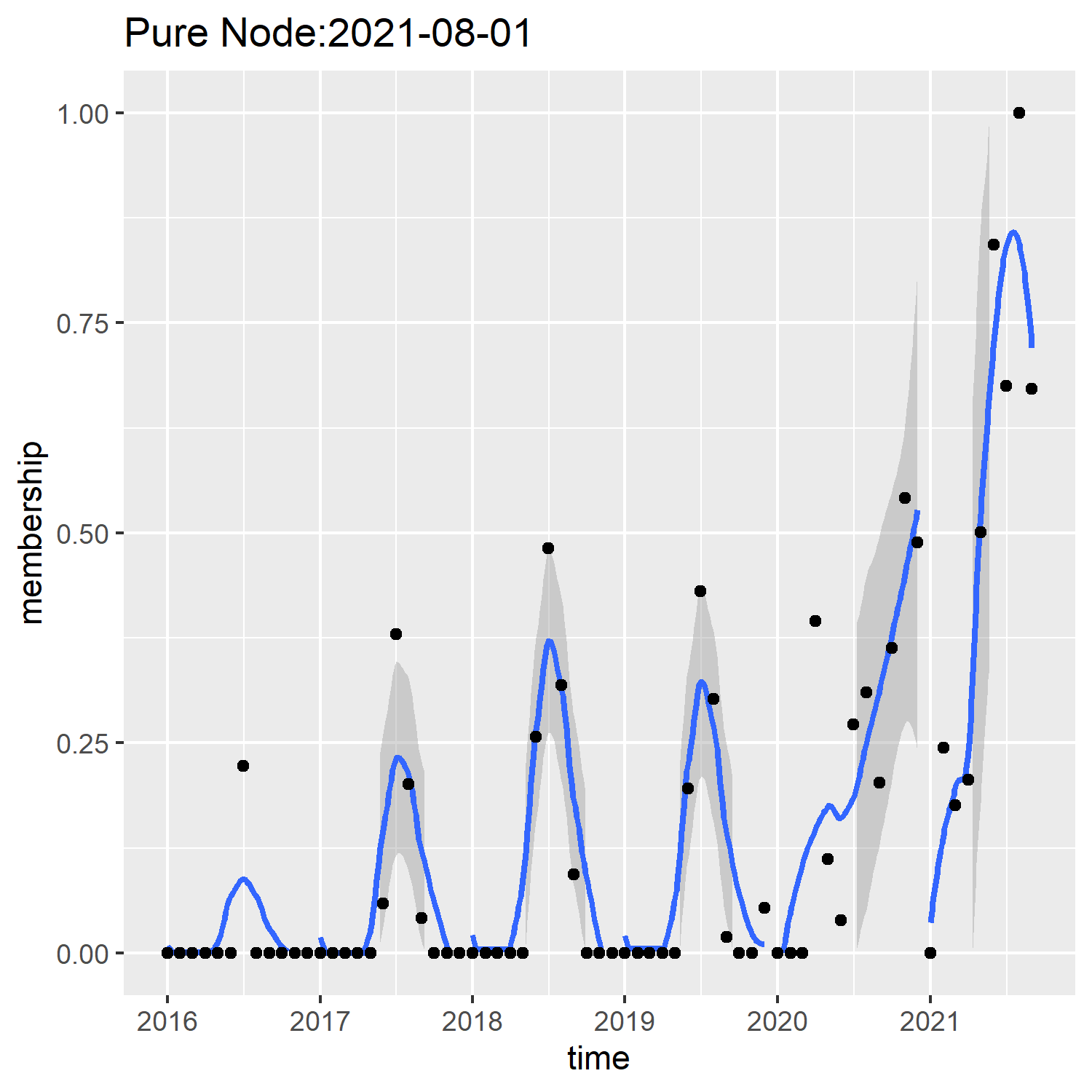}%
            \label{subfig:d}%
        }
        \caption{Pure node memberships for the time mode, with higher values corresponding to stronger membership intensity.  Data are smoothed within each year to emphasize the effect of seasonality}
        \label{fig:timeplots}
    \end{figure*}

%Note that \citet{agterberg_joint_2022} they consider a degree-corrected model, and hence they do not find a ``hub'' community  as we do here.

% \begin{figure}
%      \centering
%      \begin{subfigure}[b]{0.25\textwidth}
%          \centering
%          \includegraphics[width=\textwidth]{tensor-perturb-2-infinity/Pure_node_atl.png}
%          \caption{ATL (Atlanta)}
%          \label{fig:y equals x}
%      \end{subfigure}
%      \hfill
%      \begin{subfigure}[b]{0.25\textwidth}
%          \centering
%          \includegraphics[width=\textwidth]{tensor-perturb-2-infinity/pure_node_LAX.png}
%          \caption{LAX (Los Angeles)}
%          \label{fig:three sin x}
%      \end{subfigure}
%      \hfill
%      \begin{subfigure}[b]{0.25\textwidth}
%          \centering
%          \includegraphics[width=\textwidth]{tensor-perturb-2-infinity/pure_node_LGA.png}
%          \caption{LGA (New York)}
%          \label{fig:five over x}
%      \end{subfigure}
%      \hfill
%      \begin{subfigure}[b]{0.2\textwidth}
%          \centering
%          \includegraphics[width=
%          .3\textwidth,height=20mm]{tensor-perturb-2-infinity/legend.png}
%              \caption{Membership}
%          \label{fig:five over x}
%      \end{subfigure}
%         \caption{Plots of the pure nodes for the airport mode, with red corresponding to high membership intensity and purple corresponding to low membership intensity. }
%         \label{fig:usplot}
% \end{figure}

Plotted \cref{fig:timeplots} are the memberships in each of the four time communities, where the pure nodes were found to be August 2016, March 2020, January 2021, and August 2021.  The blue lines correspond to the yearly smoothed values (using option \texttt{loess} in the \texttt{R} programming language), and the grey regions represent confidence bands.  We chose to smooth within each year in order to emphasize seasonality.  Immediately one notices the pure node associated to March 2020 yields strong seasonality (demonstrating a ``sinusoidal'' curve within each year), only for it to vanish at the onset of the COVID-19 lockdowns in the USA, which began on March 15th, 2020.  The seasonality effect seems to mildly recover in 2021, which roughly corresponds to the reopening timeline.  The pure nodes associated to August seem to demonstrate a seasonality effect, with August 2021 also including a COVID-19 effect (as the membership in 2020 increases) -- note that vaccines in the USA became available to the general public beginning in May 2021, so the community associated to August 2021 may include some of the ``normal'' seasonal effects.  We include further discussion in the supplementary materials.

\begin{figure*}[t!]
        \subfloat{%
            \includegraphics[width=.45\linewidth]{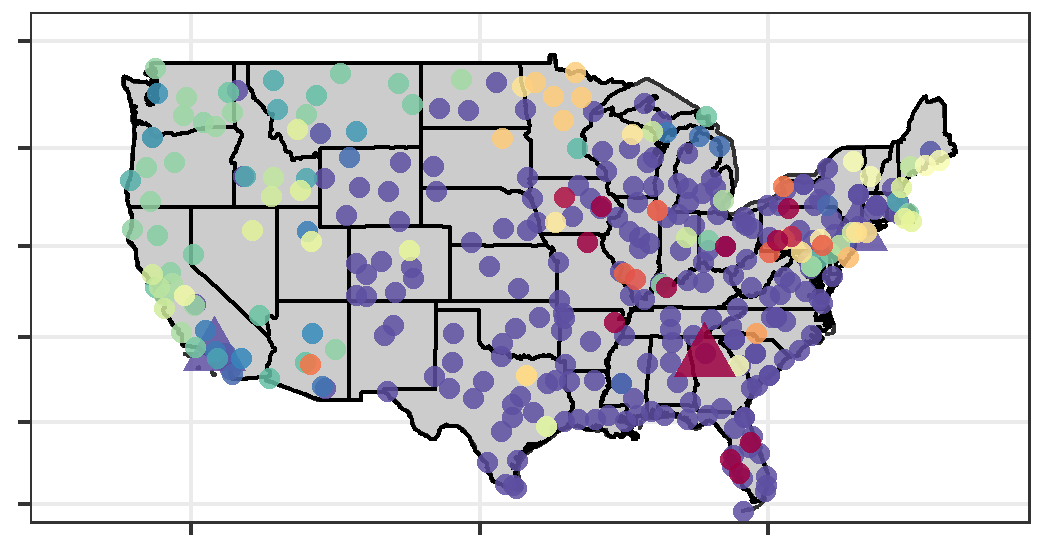}%
            \label{subfig:a}%
        }\hfill
        \subfloat{%
            \includegraphics[width=.45\linewidth]{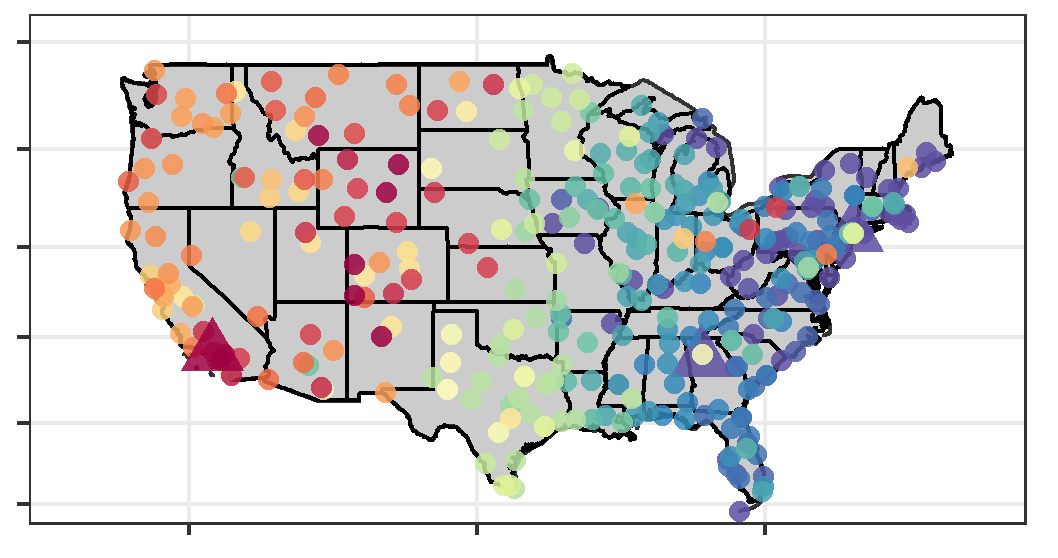}%
            \label{subfig:b}%
        }\\
        \subfloat{%
            \includegraphics[width=.45\linewidth]{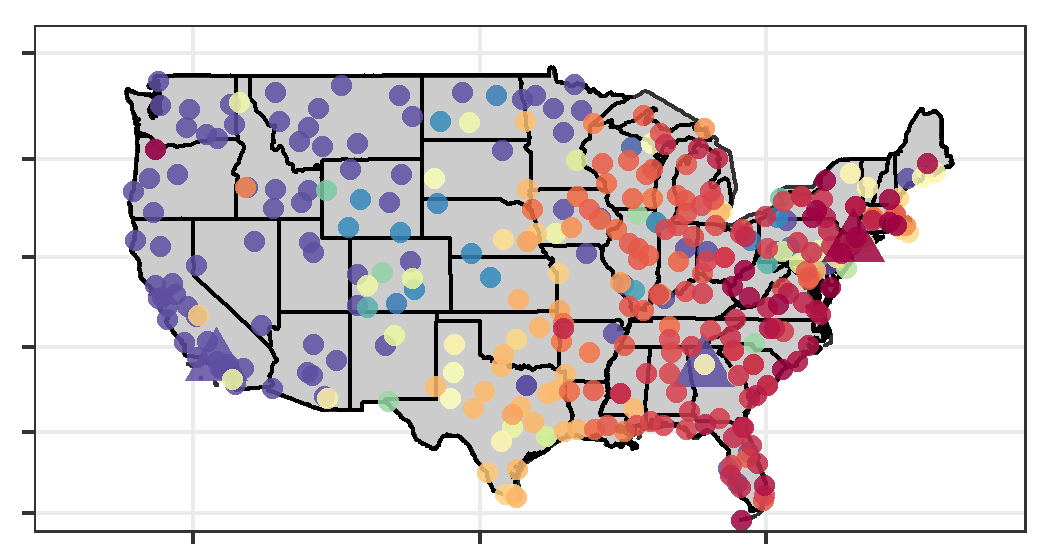}%
            \label{subfig:c}%
        }\hspace{45pt}
        \subfloat{%
            \includegraphics[width=.10\linewidth,height=40mm]{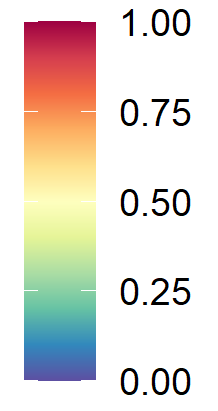}%
            \label{subfig:d}%
        }
        \caption{Pure node memberships for the airport mode, with pure nodes ATL (top left), LAX (top right), and LGA (bottom left). Red demonstrates high membership and purple demonstrates low membership within that particular community.  The pure nodes are drawn with large triangles.}
        \label{fig:usplot}
    \end{figure*}
Plotted in \cref{fig:usplot} are the membership intensities in each of the communities associated to the three different pure nodes, with red corresponding to high membership and purple corresponding to low memberships.  The three pure nodes were found to be ATL (Atlanta), LAX (Los Angeles), and LGA (New York).  From the figure it is evident that the LGA community is associated with  flights on the eastern half of the country, and LAX is associated with flights on the western half of the country.  Based on the colors, the ATL community has memberships primarily from some airports on both the east and west coasts, but less directly in central USA.  Therefore, it seems that the ATL community serves as a ``hub'' community connecting airports in the west coast to airports in the east coast -- this intuition is justified by noting that ATL has the largest number of destinations out of any airport in the USA.  

The January 2021 community seems to exhibit a combination of a form of seasonality together with COVID-19, though it is perhaps not as pronounced as the March 2020 seasonality effect, nor is it as pronounced as the August 2021 COVID-19 effect.  To emphasize these effects, we plot this mode by combining it with March 2020 (to emphasize seasonality) and August 2021 (to emphasize the COVID-19 effect) in \cref{fig:jointplot}. When combined with August 2021, the COVID-19 effect becomes more pronounced during and after 2020. When combined with March 2020, the seasonality effect becomes even more pronounced before March 2020, with larger swings within each year. Both combinations further corroborate our finding that the January 2021 community exhibits both of these effects.  
\begin{figure}[t!]
     \centering
     \begin{subfigure}[b]{0.4\textwidth}
         \centering
         \includegraphics[width=\textwidth]{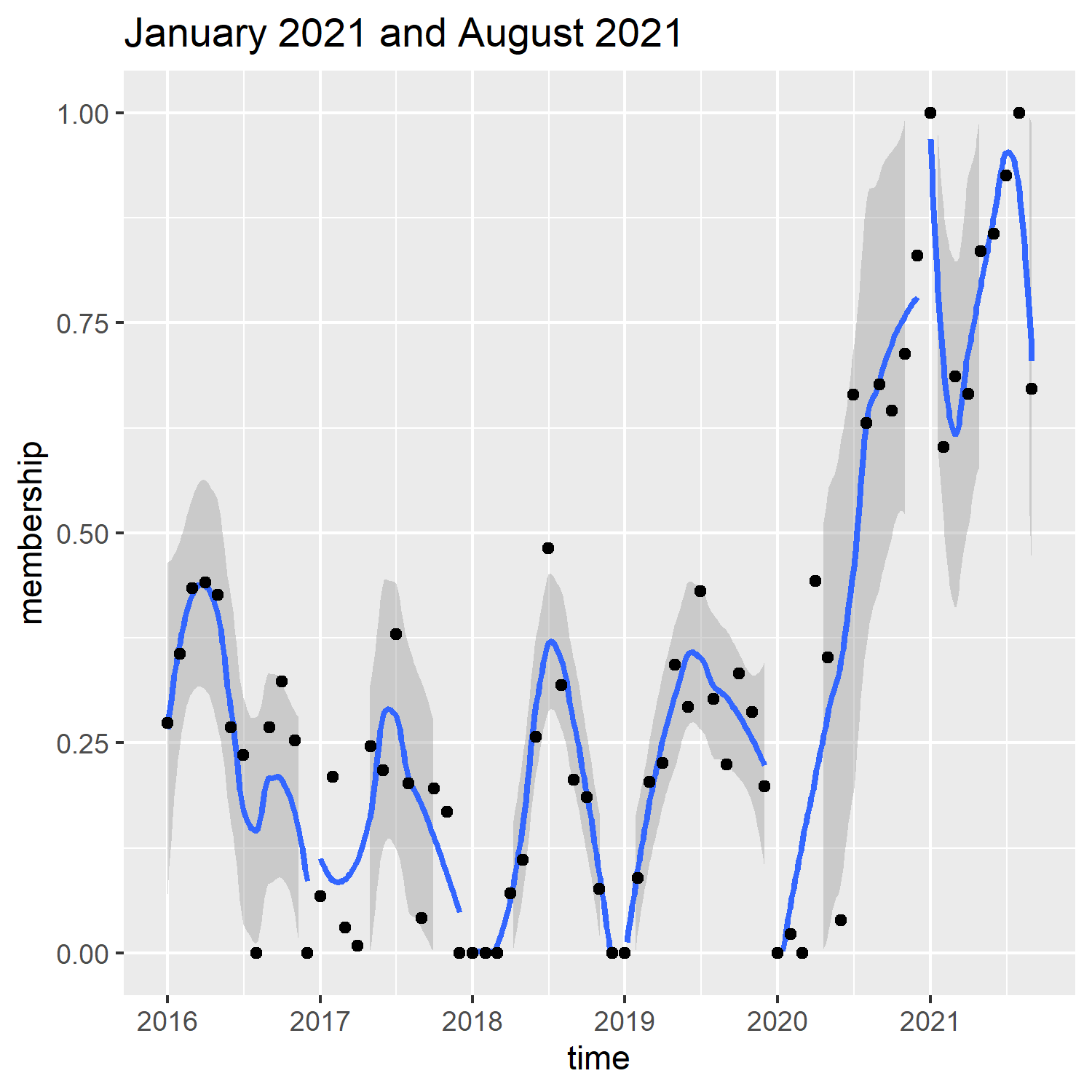}
       %  \caption{Joint plot of January 2021 and August 2021 communities to emphasize the COVID-19 effect (post 2020) of the January 2021 community.}
        % \label{fig:y equals x}
     \end{subfigure}
     \hfill
     \begin{subfigure}[b]{0.4\textwidth}
         \centering
         \includegraphics[width=\textwidth]{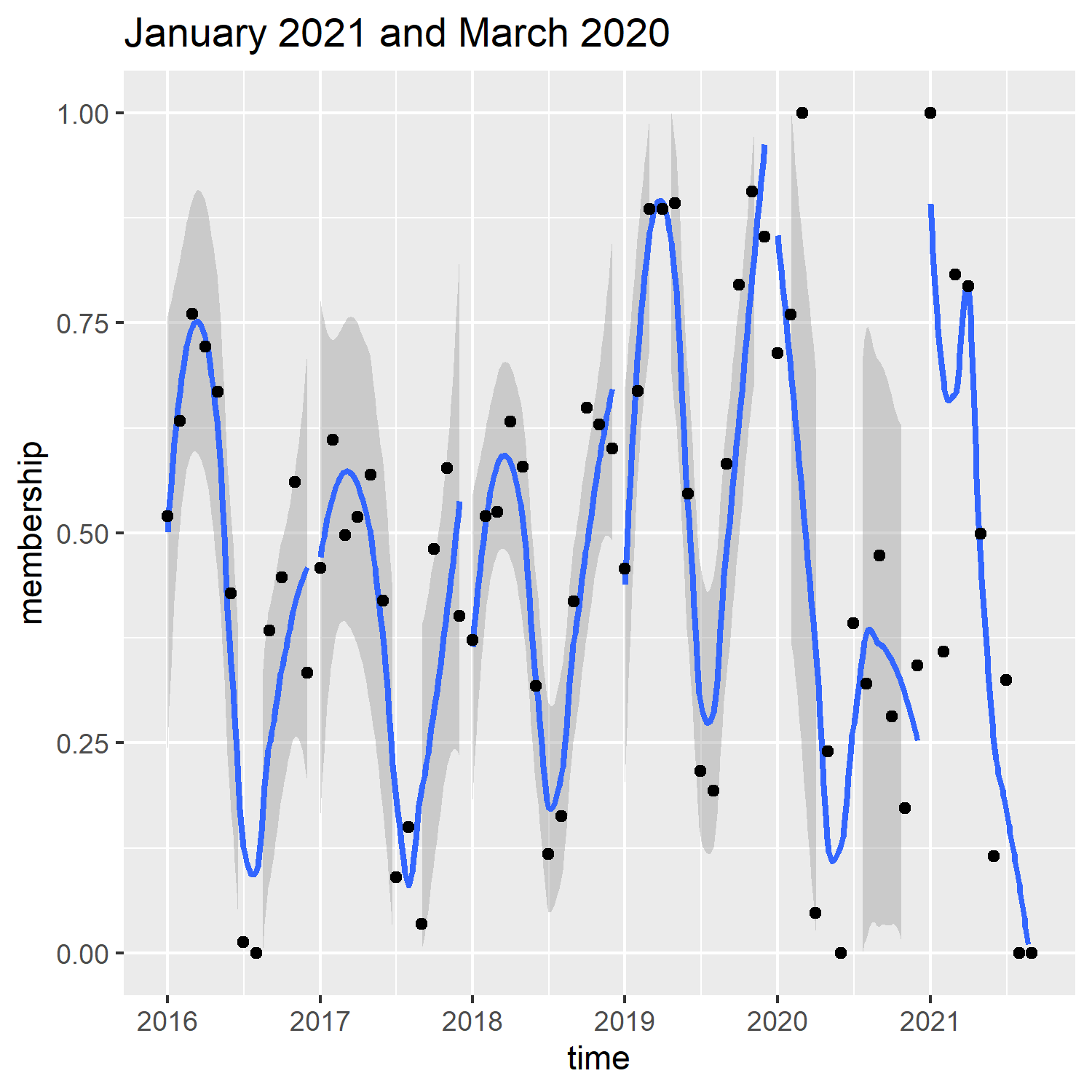}
      %  \caption{Joint plot of January 2021 and March 2020 communities to emphasize seasonality effect (pre-2020) of the January 2021 community.}
         %\label{fig:three sin x}
     \end{subfigure}
     \caption{Joint plot emphasizing COVID-19 (left) and seasonality (right) effects of the January 2021 community.}
        \label{fig:jointplot}
\end{figure}

\subsection{More Discussion on the Global Trade Data Analysis}

For the pure nodes associated to the USA and Canada, we see that the membership is relatively dispersed outside of Europe, which provides evidence that European  trade communities are ``closer-knit'' than other communities.  Since the USA and Canada likely have similar trading patterns, in \cref{fig:usacanada} we combine these two values, and we see that the memberships are fairly global besides Europe, though the intensity in any one area is not as strong as the intensities for the other pure nodes.  

\begin{figure}[t!]
        \centering
        \includegraphics[width=.7\linewidth,keepaspectratio]{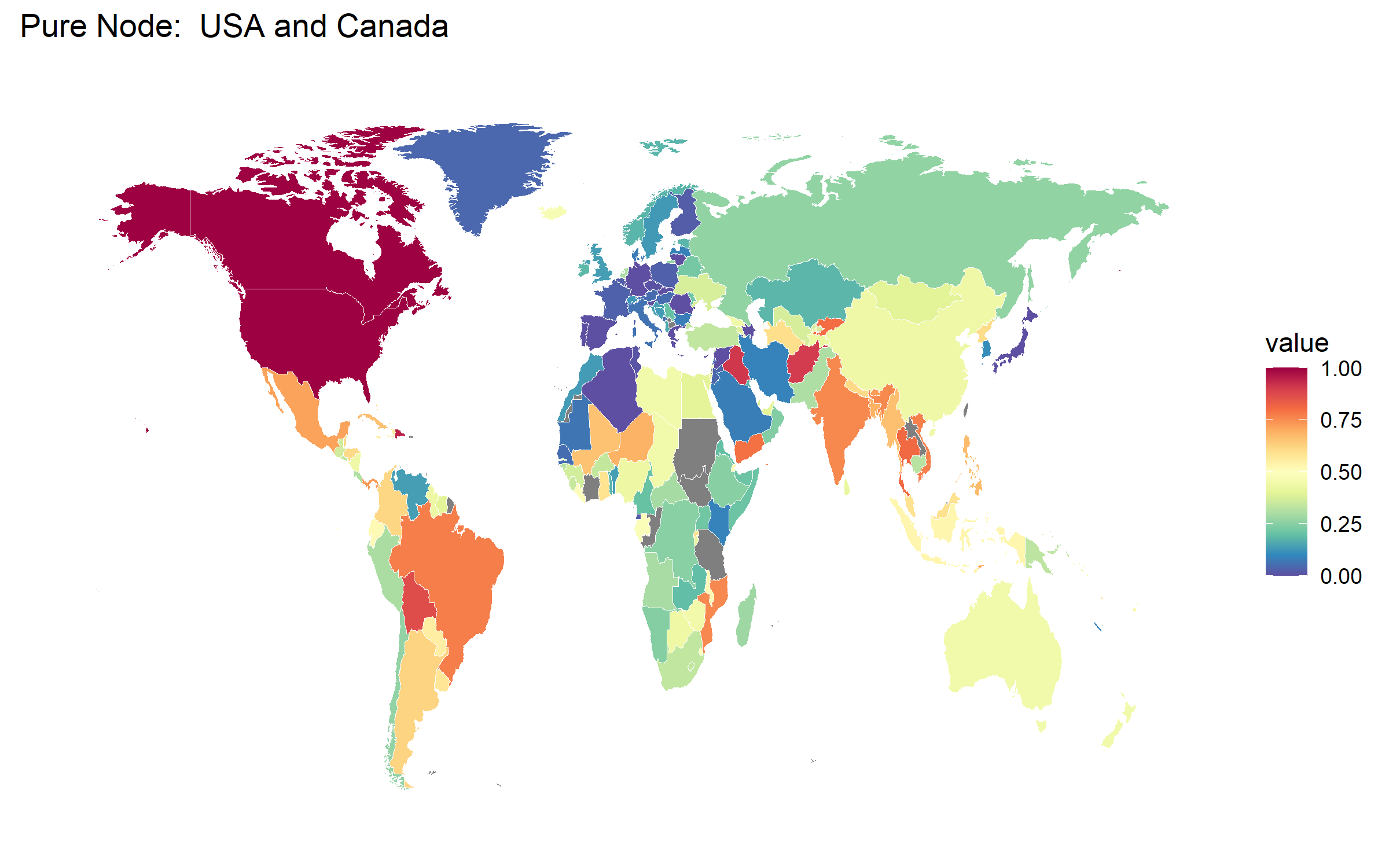}
        \caption{Combined memberships for the pure nodes associated to the USA and Canada.}
        \label{fig:usacanada}
    \end{figure}
 Next we consider the pure nodes corresponding to the different goods.  The pure nodes were found to be maize (corn), crude materials, distilled alcoholic beverages, food prep nes (not elsewhere specfied), and whole cow milk cheese.  It was found in \citet{jing_community_2021} that communities roughly correspond to either prepared or unprepared food; we also found food prep nes as one of the pure nodes, which gives further evidence to this finding. This community is also the ``largest'' community -- the mean membership in this mode is .4147.  To better understand the separation between processed and  unprocessed food, we combine the ``processed'' communities food prep nes, distilled alcoholic beverages, and whole cow milk cheese into one community and group the other two communities together. Below is a summary of the communities with greater than $.7$ membership intensity in either group, as well as those with smaller than $.7$ intensity in both communities.
    \begin{itemize}
        \item \textbf{Processed $> .7$}:  Tobacco products nes, Butter (cowmilk), Tomatoes, Milk (skimmed, dried), Tobacco (unmanufactured ), Spices (nes), Fruit (prepared nes), Cigarettes, Potatoes, non alcoholic Beverages, Vegetables (frozen), Oil (essential nes), Oil (vegetable origin nes), Nuts (prepared (exc. groundnuts)), Sugar Raw Centrifugal,  Vegetables (fresh nes), Waters (ice, etc.), Flour, wheat, Nuts nes, Tomato paste, Macaroni, Sugar refined, Food prep nes,  Cheese (whole cow milk), Chocolate products nes,  Beer of barley, Beverages (distilled alcoholic), Bread, Cereals (breakfast), Coffee extracts, Coffee (roasted),  Fruit (dried nes), Apples, Flour (maize), Pastry, Sugar confectionery, Wine, Sugar nes.
\item \textbf{Unprocessed $> .7$}: Crude materials,  Maize,  palm oil, Sesame seed, Wheat        
\item \textbf{Neither}: milled Rice, dehydrated Vegetables, Pepper (piper spp.), chicken, Infant food, Fruit (fresh nes), Tea, Beans (dry), Coffee (green), dry Chillies and peppers, orange juice (single strength), soybean oil, fruit Juice nes, Milk (whole dried), Vegetables (preserved nes), Honey (natural).
    \end{itemize}
By examining these ``communities,'' it seems that the processed foods are more similar than the unprocessed foods, since many more foods have higher memberships in communities associated to processed foods. Moreover, the ``neither'' category also contains some ``mildly processed foods" (e.g., dried milk), which shows how the mixture model here is more representative of the data.  We leave further investigations to future work.

\section{Proof of Theorem \ref{thm:twoinfty}} \label{sec:fullproof}
This section contains the full proof of Theorem \ref{thm:twoinfty}.  Without loss of generality, throughout this section we assume that $\sigma = 1$. Throughout we denote $\mathbf{T}_k = \mathcal{M}_k(\mathcal{T})$ and $\mathbf{Z}_k$ similarly. 
 We also let $p = p_{\max}$ for convenience throughout the proofs.    

Before proving our main results, we state the following results for the initialization.  The proof is contained in \cref{sec:init}. 
%The following result shows that the diagonal-deletion initialization satisfies the upper bounds needed for both the leave-one-out sequence and the true sequence.  
It is worth noting that our $\ell_{2,\infty}$ slightly sharpens the results of \citet{cai_subspace_2021} by a factor of $\kappa^2$ for the diagonal-deleted estimator; however, we do not consider missingness as they do.  In what follows, we define the leave-one-out initialization $\utilde_k^{(S,k-m)}$ as the eigenvectors of the matrix
\begin{align*}
    \Gamma\Big( \mathbf{T}_k \mathbf{T}_k\t + \mathbf{Z}_{k}^{k-m} \mathbf{T}_k\t + \mathbf{T}_k\t \mathbf{Z}_{k}^{k-m} + \mathbf{Z}_k^{k-m} (\mathbf{Z}_k^{k-m})\t \big),
\end{align*}
where $\mathbf{Z}_k^{k-m}$ denotes the matrix $\mathbf{Z}_k$ with its $m$'th row set to zero (the double appearance of the index $k$ will be useful for defining the other two leave-one-out sequences in the following subsection).  

\begin{theorem}[Initialization $\ell_{2,\infty}$ error]\label{thm:spectralinit_twoinfty}
Instate the conditions of \cref{thm:twoinfty}.  Then with probability at least $1 - O(p^{-20})$, it holds for each $k$ that
\begin{align*}
     \| \uhat_k^S  - \U_k\mathbf{W}_k^S \|_{2,\infty} &\lesssim \frac{\kappa \mu_0 \sqrt{r_1 \log(p)}}{\lambda} + \frac{\mu_0\sqrt{r _k p_{-k}}\log(p)}{\lambda^2} + \kappa^2 \mu_0^2 \frac{r_k}{p_k}; \\
     \max_m \| \uhat_k^S (\uhat_k^S)\t - \utilde_k^{(S,k-m)} (\utilde_k^{(S,k-m)})\t \| &\lesssim \frac{\kappa \mu_0 \sqrt{r_k \log(p)}}{\lambda} + \frac{\mu_0\sqrt{r_k p_{-k}} \log(p)}{\lambda^2}.
\end{align*}
\end{theorem}

In \cref{sec:loo} we describe in detail the leave-one-out sequences for the iterates of tensor SVD.   In \cref{sec:deterministicbounds} we obtain the deterministic bounds needed en route to \cref{thm:twoinfty}, and in \cref{sec:probabilisticbounds} we use these bounds to obtain high-probability guarantees on good events.  \cref{sec:twoinftyproof} contains the final proof of \cref{thm:spectralinit_twoinfty}. Throughout we rely on several self-contained probabilistic lemmas, whose statements and proofs can be found in \cref{sec:twoinftyaux}.

\subsection{The Leave-One-Out Sequence} \label{sec:loo}
In this section we formally define the leave-one-out sequence.  First, we already have defined $\uhat_k^S$ and $\utilde_k^{(S,k-m)}$ in the previous section, but we will need a few additional pieces of notation.  We define $\widehat{\U}_k^{(t)}$ as the output of tensor power iteration after $t$ iterations, with $\uhat_k^{(0)} = \uhat_k^{S}$.  It will also be useful to define
\begin{align*}
     \mathcal{\widehat P}_{k}^{(t)} &\coloneqq \begin{cases}
    \mathcal{P}_{\uhat_{k+1}^{(t-1)} \otimes \uhat_{k+2}^{(t-1)}} & k =1; \\
     \mathcal{P}_{\uhat_{k+1}^{(t-1)} \otimes \uhat_{k+2}^{(t)}} & k =2; \\
      \mathcal{P}_{\uhat_{k+1}^{(t)} \otimes \uhat_{k+2}^{(t)}} & k =3.\end{cases}
\end{align*}
The matrix $\mathcal{\widehat P}_k^{(t)}$ is simply the projection matrix corresponding to the previous two iterates.

We have already defined the matrix $\mathbf{Z}_{j}^{j-m}$ as the $j$'th matricization of $\mathcal{Z}$ with its $m$'th row set to zero.  We now define $\mathcal{Z}^{j-m}$ as the corresponding tensor $\mathcal{Z}$, where the entries corresponding to the $m$'th row of $\mathbf{Z}_j$ are set to zero.  Finally, define $\mathbf{Z}_{k}^{j-m}\coloneqq\mathcal{M}_k(\mathcal{Z}^{j-m})$.  In other words $\mathbf{Z}_k^{j-m}$ is the $k$'th matricization of the tensor $\mathcal{Z}$ with the entries corresponding to the $m$'th row of $\mathbf{Z}_j$ set to zero.  

We now define $\tilde{\U}_k^{(S,j-m)}$ as the leading $r_k$ eigenvectors of the matrix
\begin{align*}
    \Gamma\big( \mathbf{T}_k \mathbf{T}_k\t + \mathbf{Z}_k^{j-m} \mathbf{T}_k\t + \mathbf{T}_k (\mathbf{Z}_k^{j-m})\t + \mathbf{Z}_k^{j-m}(\mathbf{Z}_k^{j-m})\t \big).
\end{align*}
We now show that the other leave-one-out sequence initializations are sufficiently close to the true initialization.
\begin{lemma}[Proximity of the initialization leave-one-out sequences] \label{lem:spectralinit_leave_one_out_sintheta} Instate the conditions of \cref{thm:twoinfty}.  Then the initializations of the leave-one-out sequences satisfy for each $k$ the bound
\begin{align*}
 \max_{1\leq j \leq 3} \max_{1\leq m \leq p_j} \| \utilde_{k}^{(S,j-m)} (\utilde_{k}^{(S,j-m)})\t - \uhat_k^S (\uhat_k^S)\t \| &\lesssim \frac{\kappa \sqrt{p_k \log(p)}}{\lambda} \mu_0 \sqrt{\frac{r_1}{p_j}} + \frac{(p_1p_2p_3)^{1/2} \log(p)}{\lambda^2} \mu_0 \sqrt{\frac{r_1}{p_j}}
\end{align*}
with probability at least $1 - O(p^{-19})$.  
\end{lemma}
\cref{lem:spectralinit_leave_one_out_sintheta} is proven in \cref{sec:init} after the proof of \cref{thm:spectralinit_twoinfty}.  To define subsequent iterates, we set $\utilde_k^{(t,j-m)}$ as the outputs of tensor power iteration using these initializations, though with one modification. We now define $\tilde{\U}_k^{(t,j-m)}$ as the left singular vectors of the matrix
\begin{align*}
    \mathbf{T}_k + \mathbf{Z}_k^{j-m}  \mathcal{\tilde P}_{k}^{t,j-m} ,
\end{align*}
which is still independent from $e_m\t \mathbf{Z}_j$.  Here, we set $ \mathcal{\tilde P}_{k}^{t,j-m}$ inductively as the projection matrix
\begin{align*}
    \mathcal{\tilde P}_{k}^{t,j-m} &\coloneqq \begin{cases}
    \mathcal{P}_{\utilde_{k+1}^{(t-1,j-m)} \otimes \utilde_{k+2}^{(t-1,j-m)}} & k =1; \\
     \mathcal{P}_{\utilde_{k+1}^{(t-1,j-m)} \otimes \utilde_{k+2}^{(t,j-m)}} & k =2; \\
      \mathcal{P}_{\utilde_{k+1}^{(t,j-m)} \otimes \utilde_{k+2}^{(t,j-m)}} & k =3.
    \end{cases}
\end{align*}
Note that for each $k$ there are $3$ different leave-one-out sequences, one corresponding to each mode, by leaving out the $m$'th row of that mode (note that for convenience we use the index $m$ for each leave-one-out sequence, but we slightly abuse notation as $m$ as defined above must satisfy $1 \leq m \leq p_j$).

We now introduce some notation used for the remainder of our proofs.  Define
\begin{align*}
\mathbf{L}_{k}^{(t)} &\coloneqq \U_{k\perp} \U_{k\perp}\t \mathbf{Z}_k  \mathcal{\widehat P}_k^{(t)} \mathbf{T}_k\t \uhat^{(t)}_k ( \mathbf{\widehat{\Lambda}}_k^{(t)})^{-2}; \\
\mathbf{Q}_{k}^{(t)} &\coloneqq 
\U_{k\perp} \U_{k\perp}\t \mathbf{Z}_k  \mathcal{\widehat P}_k^{(t)} \mathbf{Z}_k\t \widehat{\U}_k^{(t)} (\mathbf{\widehat{\Lambda}}_k^{(t)})^{-2};  \\
    \tau_k &\coloneqq\sup_{\substack{ \| \mathbf{U}_1 \| = 1, \mathrm{rank}(\U_1) \leq 2 r_{k+1}\\ \|\mathbf{U}_2\| =1, \mathrm{rank}(\U_2) \leq 2 r_{k+2}} } \|  \mathbf{Z}_k \bigg( \mathcal{P}_{\mathbf{U}_1} \otimes \mathcal{P}_{\mathbf{U}_2} \bigg)\|; \\  
  %  \tilde \tau_k &\coloneqq \sup_{\substack{ \| \mathbf{U}_1 \| = 1, \mathrm{rank}(\U_1) \leq 2 r_{k+1}\\ \|\mathbf{U}_2\| =1, \mathrm{rank}(\U_2) \leq 2 r_{k+2}} } \|  \mathbf{Z}_k \bigg( \mathcal{P}_{\mathbf{U}_1} \otimes \mathcal{P}_{\mathbf{U}_2} \bigg) \mathbf{V}_k\|; \\  
    \xi_{k}^{(t,j-m)} &\coloneqq 
    \bigg\| \bigg(\mathbf{Z}_k^{j-m} - \mathbf{Z}_k \bigg) \mathcal{\tilde P}_{k}^{t,j-m}  \bigg\| \\
    \tilde \xi_{k}^{(t,j-m)} &\coloneqq   \bigg\| \bigg(\mathbf{Z}_k^{j-m} - \mathbf{Z}_k \bigg) \mathcal{\tilde P}_{k}^{t,j-m}  \mathbf{V}_k  \bigg\| \\
     \eta_{k}^{(t,j-m)} &\coloneqq 
     \begin{cases}
         \| \sin\Theta( \utilde_{k+1}^{(t-1,j-m)}, \uhat_{k+1}^{(t-1)}) \| + \| \sin\Theta( \utilde_{k+2}^{(t-1,j-m)}, \uhat_{k+2}^{(t-1)}) \| & k = 1  \\
          \| \sin\Theta( \utilde_{k+1}^{(t-1,j-m)}, \uhat_{k+1}^{(t-1)}) \| + \| \sin\Theta( \utilde_{k+2}^{(t,j-m)}, \uhat_{k+2}^{(t)}) \| & k = 2  \\
           \| \sin\Theta( \utilde_{k+1}^{(t,j-m)}, \uhat_{k+1}^{(t)}) \| + \| \sin\Theta( \utilde_{k+2}^{(t,j-m)}, \uhat_{k+2}^{(t)}) \| & k = 3  \end{cases} \\
    \eta_k^{(t)} &\coloneqq \begin{cases}\| \sin\Theta( \U_{k+1}, \uhat_{k+1}^{(t-1)}) \| + \| \sin\Theta( \U_{k+2}, \uhat_{k+2}^{(t-1)}) \| & k = 1 \\
        \| \sin\Theta( \U_{k+1}, \uhat_{k+1}^{(t-1)}) \| + \| \sin\Theta( \U_{k+2}, \uhat_{k+2}^{(t)}) \| & k = 2 \\
        \| \sin\Theta( \U_{k+1}, \uhat_{k+1}^{(t)}) \| + \| \sin\Theta( \U_{k+2}, \uhat_{k+2}^{(t)}) \| & k = 3
    \end{cases}.
\end{align*}

First we will state results deterministically with dependence on $\tau_k$, $\xi_{k}^{(t,j-m)}$ and $\eta_{k}^{(t,j-m)}$.  Note that we already have the bound $\tilde \xi_k^{(t,j-m)} \leq \xi_k^{(t,j-m)}$ since $\|\mathbf{V}_k \| = 1$, but it will turn out to be slightly more useful to have the dependence on $\mathbf{V}_k$.

\subsection{Deterministic Bounds}
\label{sec:deterministicbounds}

In this section we collect and prove deterministic bounds that we will then combine with probabilistic induction in \cref{sec:probabilisticbounds}.  

\begin{lemma}[Closeness of the orthogonal matrix] \label{lem:orthogonalmatrixlemma}
Let $\mathbf{W}_k^{(t)} = \mathrm{sgn}(\uhat_k^{(t)}, \U_k)$ be the matrix sign of $\uhat_k^{(t)}$ and $\U_k$.  Then
\begin{align*}
    \|   \U_k\mathbf{W}_k^{(t)} - \U_k \U_k\t \uhat_k^{(t)} \|_{2,\infty} &\leq\mu_0 \sqrt{\frac{ r_k}{p_k}}  \| \sin\Theta(\uhat_k^{(t)}, \U_k) \|^2.
\end{align*}
\end{lemma}

\begin{proof}[Proof of Lemma \ref{lem:orthogonalmatrixlemma}]
Observe that
\begin{align*}
     \|   \U_k\mathbf{W}_k^{(t)} - \U_k \U_k\t \uhat_k^{(t)} \|_{2,\infty} &\leq \| \U_k \|_{2,\infty} \| \mathbf{W}_k^{(t)} - \U_k\t \uhat_k^{(t)} \| \\
    &\leq \mu_0 \sqrt{\frac{r_k}{p_k}} \| \sin\Theta(\uhat_k^{(t)},\U_k ) \|^2.
\end{align*}
For details on the final inequality, see Lemma 4.6.3 of \cite{chen_spectral_2021}.  
\end{proof}

\begin{lemma}[Deterministic Bound for the Linear Term] \label{lem:linear_deterministic_bd}
Suppose $\mathbf{T}_k = \U_k \mathbf{\Lambda}_k \mathbf{V}_k\t$, and suppose that $\lambda/2 \leq  \lambda_{r_k} ( \mathbf{\hat{\Lambda}}_k^{(t)})$  Then the linear term $\mathbf{L}_k^{(t)}$ satisfies
% \begin{align*}
%     \| e_m\t \mathbf{L}_k^{(t)} \| &\leq 4 \frac{\kappa}{\lambda} \bigg( \tilde \tau_k \| \U \|_{2,\infty} + 2\tau_k ((\eta_{k+1}^{(t,j-m)} + \eta_{k+2}^{(t,j-m)} )) +\bigg\|  e_m\t  \mathbf{Z}_k \bigg[  \mathcal{P}_{\utilde_{k+1}^{(t,j-m)}}\otimes  \mathcal{P}_{\utilde_{k+2}^{(t,j-m)}}\bigg] \mathbf{V}_k  \bigg\| \bigg).
% \end{align*}
\begin{align*}
      \| e_m\t \mathbf{L}_k^{(t)} \| &\leq  \frac{8\kappa}{\lambda} \| \U_k \|_{2,\infty} \bigg( \tau_k \eta_{k}^{(t)} +  \| \U_k\t \mathbf{Z}_k \mathbf{V}_k \| \bigg)  + \frac{8\kappa}{\lambda} \bigg( \tau_k \eta_{k}^{(t,k-m)}  \bigg)  + \frac{4 \kappa}{\lambda} \tilde\xi_{k}^{t,k-m}, %\bigg\| e_m\t\mathbf{Z}_k  \bigg[  \mathcal{P}_{\utilde_{k+1}^{(t,j-m)}}\otimes  \mathcal{P}_{\utilde_{k+2}^{(t,j-m)}}\bigg] \mathbf{V}_k \bigg\|
\end{align*}
%with $\eta_{k+1}^{(t-1)}$ and $\eta_{k+1}^{(t-1,k-m)}$ replaced with $\eta_{k+1}^{(t)}$ and $\eta_{k+1}^{(t,k-m)}$ if $k = 2$ or $3$, and $\eta_{k+2}^{(t-1)}$ and $\eta_{k+2}^{(t-1,k-m)}$ replaced with $\eta_{k+2}^{(t)}$ and $\eta_{k+2}^{(t,k-m)}$ if $k = 3$.
\end{lemma}

\begin{proof}[Proof of Lemma \ref{lem:linear_deterministic_bd}]
Without loss of generality we prove the result for $k =1$; the cases for $k =2$ and $k = 3$ are similar by changing the index for $t$ using the definition of $\mathcal{\hat P}_{k}^{(t)}$.  

Recall we let $\mathbf{T}_k = \U_k \mathbf{\Lambda}_k \mathbf{V}_k\t$.  Then the $m$'th row of the linear term $\mathbf{L}_1^{(t)}$ can be written as
\begin{align*}
    e_m\t \U_{1\perp}& \U_{1\perp}\t \mathbf{Z}_1  \bigg[ \mathcal{P}_{\hat{\U}_{2}^{(t-1)}} \otimes \mathcal{P}_{\hat{\U}_{3}^{(t-1)}} \bigg] \mathbf{T}_1\t \hat{\U}_1^{(t)} (\mathbf{\hat{\Lambda}}_1^{(t)})^{-2} \\
    &=  e_m\t  \U_{1\perp} \U_{1\perp}\t \mathbf{Z}_1  \bigg[ \mathcal{P}_{\hat{\U}_{2}^{(t-1)}} \otimes \mathcal{P}_{\hat{\U}_{3}^{(t-1)}} \bigg] \mathbf{V}_1 \mathbf{\Lambda}_1 \U_1\t \hat{\U}_1^{(t)} (\mathbf{\hat{\Lambda}}_1^{(t)})^{-2}.
\end{align*}
Taking norms, we see that as long as $2\lambda\inv \geq (\hat \lambda_{r_1}^{(t)})\inv$ as in the assumptions of this lemma, we have
\begin{align}
   \bigg\|  e_m\t  &\U_{1\perp} \U_{1\perp}\t \mathbf{Z}_1  \bigg[ \mathcal{P}_{\hat{\U}_{2}^{(t-1)}} \otimes \mathcal{P}_{\hat{\U}_{3}^{(t-1)}} \bigg] \mathbf{V}_1 \mathbf{\Lambda}_1 \U_1\t \hat{\U}_1^{(t)} (\mathbf{\hat{\Lambda}}_1^{(t)})^{-2} \bigg\| \nonumber \\ %&\leq  \bigg\|  e_m\t \U_{\perp} \U_{\perp}\t \mathbf{Z}_k  \bigg[ \mathcal{P}_{\hat{\U}_{k+1}^{(t)}} \otimes \mathcal{P}_{\hat{\U}_{k+2}^{(t)}} \bigg] \mathbf{V}_k \bigg\| \| \Lambda_k \| \| \mathbf{\hat{\Lambda}}_k^{-2} \| \\
   &\leq \frac{4 \kappa  }{\lambda} \bigg\|  e_m\t \U_{1\perp} \U_{1\perp}\t \mathbf{Z}_1  \bigg[ \mathcal{P}_{\hat{\U}_{2}^{(t-1)}} \otimes \mathcal{P}_{\hat{\U}_{3}^{(t-1)}} \bigg] \mathbf{V}_1 \bigg\| \nonumber \\
   &\leq \frac{4 \kappa }{\lambda }\bigg\|  e_m\t \mathbf{Z}_1  \bigg[ \mathcal{P}_{\hat{\U}_{2}^{(t-1)}} \otimes \mathcal{P}_{\hat{\U}_{3}^{(t-1)}} \bigg] \mathbf{V}_1  \bigg\|  + \frac{4 \kappa }{\lambda }\bigg\|  e_m\t \U_1 \U_1\t \mathbf{Z}_1   \bigg[ \mathcal{P}_{\hat{\U}_{2}^{(t-1)}} \otimes \mathcal{P}_{\hat{\U}_{3}^{(t-1)}} \bigg] \mathbf{V}_1  \bigg\|. \label{eq:t1t2initbounds}
   \end{align}
%where the second line follows from the eigengap condition. % provided the initial error satisfies 
%\begin{align*}
%    \bigg\|\mathbf{Z}_k \mathbf{Z}_k\t - \mathbb{E}\mathbf{Z}_k \mathbf{Z}_k\t  + \mathbf{Z}_k \mathbf{T}_k\t +\mathbf{T}_k \mathbf{Z}_k\t\bigg\| &\lesssim \lambda_r^2,
%\end{align*}
%which follows from the assumption $\lambda_r \gtrsim \sigma p^{3/4}$. 
Thus, it suffices to analyze the two terms
\begin{align*}
T_1 &\coloneqq \bigg\|  e_m\t \mathbf{Z}_1  \bigg[ \mathcal{P}_{\hat{\U}_{2}^{(t-1)}} \otimes \mathcal{P}_{\hat{\U}_{3}^{(t-1)}} \bigg] \mathbf{V}_1 \bigg\|; \\
T_2 &\coloneqq \bigg\|  e_m\t \U_1 \U_1\t \mathbf{Z}_1   \bigg[ \mathcal{P}_{\hat{\U}_{2}^{(t-1)}} \otimes \mathcal{P}_{\hat{\U}_{3}^{(t-1)}} \bigg]\mathbf{V}_k \bigg\|; 
\end{align*}
for fixed $m$.  For the term $T_1$, we introduce the leave-one-out sequence to observe that
% \begin{align*}
%     \bigg\|  e_m\t \U_{\perp} \U_{\perp}\t \mathbf{Z}_k  \bigg[ \mathcal{P}_{\hat{\U}_{k+1}^{(t)}} \otimes \mathcal{P}_{\hat{\U}_{k+2}^{(t)}} \bigg] \mathbf{V}_k \Lambda_k \U_k\t \uhat_k^{(t-1)} \bigg\| &\leq \lambda_k
% \end{align*}
\begin{align}
     \bigg\| e_m\t  \mathbf{Z}_1   \bigg[ \mathcal{P}_{\hat{\U}_{2}^{(t-1)}} \otimes \mathcal{P}_{\hat{\U}_{3}^{(t-1)}} \bigg] \mathbf{V}_1 \bigg\| 
     &\leq   \bigg\| e_m\t  \mathbf{Z}_1  \bigg[ \bigg( \mathcal{P}_{\hat{\U}_{2}^{(t-1)}} - \mathcal{P}_{\utilde_{2}^{(t-1,1-m)}} \bigg)\otimes \mathcal{P}_{\hat{\U}_{3}^{(t-1)}}\bigg] \mathbf{V}_1  \bigg\| \nonumber \\
     &\qquad +   \bigg\| e_m\t \mathbf{Z}_1  \bigg[  \mathcal{P}_{\utilde_{2}^{(t-1,1-m)}}\otimes \mathcal{P}_{\hat{\U}_{3}^{(t-1)}} \bigg] \mathbf{V}_1 \bigg\| \nonumber\\
     &\leq \bigg\| e_m\t  \mathbf{Z}_1  \bigg[ \bigg(\mathcal{P}_{\hat{\U}_{2}^{(t-1)}} - \mathcal{P}_{\utilde_{2}^{(t-1,1-m)}} \bigg)\otimes \mathcal{P}_{\hat{\U}_{3}^{(t-1)}} \bigg] \mathbf{V}_1  \bigg\| \nonumber\\
     &\qquad +   \bigg\| e_m\t  \mathbf{Z}_1  \bigg[  \mathcal{P}_{\utilde_{2}^{(t-1,1-m)}}\otimes \bigg( \mathcal{P}_{\hat{\U}_{3}^{(t-1)}} - \mathcal{P}_{\utilde_{3}^{(t-1,1-m)}}\bigg)\bigg] \mathbf{V}_1 \bigg\|\nonumber \\
     &\qquad + \bigg\| e_m\t\mathbf{Z}_1  \bigg[  \mathcal{P}_{\utilde_{2}^{(t-1,1-m)}}\otimes  \mathcal{P}_{\utilde_{3}^{(t-1,1-m)}}\bigg] \mathbf{V}_1 \bigg\|\nonumber \\
     &\leq 2\tau_1 \| \sin\Theta( \uhat_{2}^{(t-1)}, \utilde_{2}^{(t-1,1-m)}) \| + 2\tau_1 \| \sin\Theta( \uhat_{3}^{(t-1)}, \utilde_{3}^{(t-1,1-m)}) \| \nonumber\\
     &\qquad + \bigg\| e_m\t\mathbf{Z}_1  \bigg[  \mathcal{P}_{\utilde_{2}^{(t-1,1-m)}}\otimes  \mathcal{P}_{\utilde_{3}^{(t-1,1-m)}}\bigg] \mathbf{V}_1 \bigg\|. \label{eq:t1bound}
\end{align}
As for $T_2$, we note that
\begin{align}
    & \bigg\| e_m\t \U_1 \U_1\t \mathbf{Z}_1 \bigg[ \mathcal{P}_{\uhat_{2}^{(t-1)}} \otimes \mathcal{P}_{\uhat_{3}^{(t-1)}} \bigg] \mathbf{V}_1 \bigg\| \\
    &\leq \| \U_1 \|_{2,\infty} \| \U_1\t \mathbf{Z}_1 \bigg[ \mathcal{P}_{\uhat_{2}^{(t-1)}} \otimes \mathcal{P}_{\uhat_{3}^{(t-1)}} \bigg] \mathbf{V}_1 \bigg\|\nonumber  \\
    &\leq \| \U_1 \|_{2,\infty} \| \U_1\t \mathbf{Z}_1\bigg[ \bigg( \mathcal{P}_{\uhat_{2}^{(t-1)}}  - \mathcal{P}_{\U_{2}} \bigg) \otimes \mathcal{P}_{\uhat_{3}^{(t-1)}} \bigg] \mathbf{V}_1 \bigg\| \nonumber \\
    &\qquad + \| \U_1 \|_{2,\infty} \| \U_1\t \mathbf{Z}_1 \bigg[ \mathcal{P}_{\U_{2}}  \otimes \bigg( \mathcal{P}_{\U_{3}} - \mathcal{P}_{\uhat_{3}^{(t-1)}} \bigg) \bigg] \mathbf{V}_1 \bigg\|\nonumber  \\
    &\qquad + \| \U_1 \|_{2,\infty} \| \U_1\t \mathbf{Z}_1 \bigg[  \mathcal{P}_{\U_{2}}  \otimes \mathcal{P}_{\U_{3}} \bigg] \mathbf{V}_1 \bigg\|\nonumber  \\
    &\leq 2\| \U_1 \|_{2,\infty} \tau_1 \bigg( \| \sin\Theta( \U_{2}, \uhat_{2}^{(t-1)} ) \| + \| \sin\Theta( \U_{3}, \uhat_{3}^{(t-1)} ) \| \bigg)\nonumber \\
    &\qquad + \| \U_1 \|_{2,\infty} \| \U_1\t \mathbf{Z}_1 \mathbf{V}_1 \|, \label{eq:t2bound} 
\end{align}
where the final line used the fact that $\mathcal{P}_{\U_{2}}  \otimes \mathcal{P}_{\U_{3}} \mathbf{V}_1 = \mathbf{V}_1$ by definition.

We now plug in the bound for $T_1$ in \eqref{eq:t1bound} and $T_2$ in \eqref{eq:t2bound} to the initial bound in \eqref{eq:t1t2initbounds} to obtain that
\begin{align*}
      \bigg\|  e_m\t  &\U_{1\perp} \U_{1\perp}\t \mathbf{Z}_1  \bigg[ \mathcal{P}_{\hat{\U}_{2}^{(t-1)}} \otimes \mathcal{P}_{\hat{\U}_{3}^{(t-1)}} \bigg] \mathbf{V}_1 \mathbf{\Lambda}_1 \U_1\t \hat{\U}_1^{(t)} (\mathbf{\hat{\Lambda}}_1^{(t)})^{-2} \bigg\| \\
     &\leq \frac{8 \kappa}{\lambda} \| \U_1 \|_{2,\infty} \tau_1 \bigg( \| \sin\Theta( \U_{2}, \uhat_{2}^{(t-1)} ) \| + \| \sin\Theta( \U_{3}, \uhat_{3}^{(t-1)} ) \| \bigg) \\
    &\qquad + \frac{4 \kappa}{\lambda} \| \U_1 \|_{2,\infty} \| \U_1\t \mathbf{Z}_1 \mathbf{V}_1 \| \\
    &\qquad + \frac{8 \kappa}{\lambda}  \tau_1 \| \sin\Theta( \uhat_{2}^{(t-1)}, \utilde_{2}^{(t-1,1-m)}) \| + \frac{8 \kappa}{\lambda}  \tau_1 \| \sin\Theta( \uhat_{3}^{(t-1)}, \utilde_{3}^{(t-1,1-m)}) \| \nonumber\\
     &\qquad + \frac{4 \kappa}{\lambda}  \bigg\| e_m\t\mathbf{Z}_1  \bigg[  \mathcal{P}_{\utilde_{2}^{(t-1,1-m)}}\otimes  \mathcal{P}_{\utilde_{3}^{(t-1,1-m)}}\bigg] \mathbf{V}_1 \bigg\| \\
     &\leq \frac{8\kappa}{\lambda} \| \U_1 \|_{2,\infty} \bigg(  \tau_1 \eta_{1}^{(t)}  +  \| \U_1\t \mathbf{Z}_1 \mathbf{V}_1\| \bigg) \\
     &\qquad + \frac{8\kappa}{\lambda} \bigg(  \tau_1 \eta_{1}^{(t,1-m)}   \bigg)  + \frac{4 \kappa}{\lambda} \tilde \xi_{1}^{t,1-m}
\end{align*}
as desired.
\end{proof}

%\textbf{Analyzing the linear term:} Suppose $\mathbf{T}_k = \U_k \Lambda_k \mathbf{V}_k\t$.  
%\noindent
% \textbf{Analyzing the Quadratic term}: Note that by a similar argument as above
% \begin{align*}
%     \bigg\| e_m\t \U_{\perp} \U_{\perp}\t \mathbf{Z}_k \bigg[ \mathcal{P}_{\hat{\U}_{k+1}^{(t)}} \otimes \mathcal{P}_{\hat{\U}_{k+2}^{(t)}} \bigg] \mathbf{Z}_k\t \hat{\U}_k^{(t-1)}\mathbf{\hat{\Lambda}}_k^{-2} \bigg\| &\lesssim \frac{1}{\lambda_r^2} \bigg\|  e_m\t \U_{\perp} \U_{\perp}\t \mathbf{Z}_k \bigg[ \mathcal{P}_{\hat{\U}_{k+1}^{(t)}} \otimes \mathcal{P}_{\hat{\U}_{k+2}^{(t)}} \bigg] \mathbf{Z}_k\t \hat{\U}_k^{(t-1)} \bigg\|.
% \end{align*}
% In order to bound this term, the leave-one-out argument will need to be applied earlier.  

% Note: it may be feasible to simply consider the upper bound
% \begin{align*}
%     \frac{1}{\lambda_r^2} \bigg\|  e_m\t \U_{\perp} \U_{\perp}\t \mathbf{Z}_k \bigg[ \mathcal{P}_{\hat{\U}_{k+1}^{(t)}} \otimes \mathcal{P}_{\hat{\U}_{k+2}^{(t)}} \bigg] \mathbf{Z}_k\t \hat{\U}_k^{(t-1)} \bigg\| &\leq  \frac{1}{\lambda_r^2}   \bigg\|  e_m\t \U_{\perp} \U_{\perp}\t \mathbf{Z}_k \hat{\mathbf{V}}  \bigg\| \bigg\| \hat{\mathbf{V}}\t \mathbf{Z}_k\t \bigg\|
% \end{align*}
% where I have defined
% \begin{align*}
%     \hat{\mathbf{V}}\hat{\mathbf{V}}\t \coloneqq \bigg[ \mathcal{P}_{\hat{\U}_{k+1}^{(t)}} \otimes \mathcal{P}_{\hat{\U}_{k+2}^{(t)}} \bigg] 
% \end{align*}
% for some matrix $\mathbf{\hat{V}} \in \mathbb{R}^{p_{-k}, r_{-k}}$.

\begin{lemma}[Deterministic Bound for the Quadratic Term] \label{lem:quadratic_deterministic_bd}
Suppose $\lambda/2 \leq  \lambda_{r_k} ( \mathbf{\hat{\Lambda}}_k^{(t)}$).  Then the quadratic term $\mathbf{Q}_k^{(t)}$ satisfies
\begin{align*}
\| e_m\t \mathbf{Q}_k^{(t)} \| &\leq  \frac{4}{\lambda^2} \|\U_k\|_{2,\infty} \bigg( \tau_k \eta_k^{(t)} + \bigg\| \U_k\t \mathbf{Z}_k \bigg[ \mathcal{P}_{\U_{k+1}} \otimes  \mathcal{P}_{\U_{k+1}}  \bigg] \bigg\| \bigg)  + \frac{16}{\lambda^2}\tau_k^2\bigg( \eta_{k}^{(t,k-m)}  \bigg)  \\
%&\qquad + \frac{4}{\lambda^2} \tau_k^2  \|\sin\Theta(\utilde_{k}^{k-m,t-1},\uhat_{k}^{(t-1)}) \| \\
%&\qquad +  \frac{4}{\lambda^2}\bigg\| e_m\t \mathbf{Z}_k\bigg[ \mathcal{P}_{\utilde_{k+1}^{k-m,t}}\otimes \mathcal{P}_{\utilde_{k+2}^{k-m,t}} \bigg]\bigg[\mathcal{P}_{\utilde_{k+1}^{k-m,t}} \otimes \mathcal{P}_{\utilde_{k+2}^{k-m,t}} \bigg] \bigg(\mathbf{Z}_k - \mathbf{Z}_k^{k-m} \bigg) \t \uhat_k^{(t)}\bigg\| \\
%&\qquad 
&+ \frac{4}{\lambda^2}  \xi_k^{t,k-m} \bigg( \tau_k \| \sin\Theta(\uhat_k^{(t)},\U_k) \| + \tau_k \eta^{(t-1)}_k +  \bigg\|\U_k \U_k\t \mathbf{Z}_k \mathcal{P}_{\U_{k+1}}  \otimes   \mathcal{P}_{\U_{k+2}}  \bigg\| \bigg).
%
%
% \frac{4}{\lambda^2} \tau_k^2 \| \U_k \|_{2,\infty} + \frac{8}{\lambda^2} \tau_k^2 \bigg( \eta_{k+1}^{(t,k-m)} + \eta_{k+2}^{(t,k-m)} \bigg) + \frac{4}{\lambda^2} \tau_k^2 \eta_{k}^{(t-1,k-m)} \\
%&\qquad + \frac{4}{\lambda^2} \bigg\| e_m\t \mathbf{Z}_k \bigg[ \mathcal{P}_{\utilde_{k+1}^{k-m,t}} \otimes \mathcal{P}_{\utilde_{k+2}^{k-m,t}}  \bigg] \bigg\| \bigg( \tau_k + 2 \bigg\|\bigg( \mathbf{Z}_k - \mathbf{Z}_k^{k-m} \bigg) \mathcal{P}_{\utilde_{k+1}^{k-m,t}} \otimes \mathcal{P}_{\utilde_{k+2}^{k-m,t}} \bigg\| \bigg).
% \frac{4}{\lambda^2} \tau_k^2 \| \U\|_{2,\infty} + \frac{8}{\lambda^2} \tau_k^2  \bigg( \eta_{k+1}^{(t,k-m)} + \eta_{k+2}^{(t,k-m)} \bigg) + \frac{4}{\lambda^2} \tau_k^2 \eta_{k}^{(t,k-m)} \\
% &\qquad + \frac{4}{\lambda^2} \bigg\| e_m\t \mathbf{Z}_k\bigg[ \mathcal{P}_{\utilde_{k+1}^{k-m,t}}\otimes \mathcal{P}_{\utilde_{k+2}^{k-m,t}} \bigg]\bigg[\mathcal{P}_{\utilde_{k+1}^{k-m,t}} \otimes \mathcal{P}_{\utilde_{k+2}^{k-m,t}} \bigg] \bigg(\mathbf{Z}_k - \mathbf{Z}_k^{k-m} \bigg) \t \utilde_k^{k-m,t-1}\bigg\| \\
% &\qquad + \frac{4}{\lambda^2}\bigg\| e_m\t \mathbf{Z}_k\bigg[ \mathcal{P}_{\utilde_{k+1}^{k-m,t}}\otimes \mathcal{P}_{\utilde_{k+2}^{k-m,t}} \bigg]\bigg[\mathcal{P}_{\utilde_{k+1}^{k-m,t}} \otimes \mathcal{P}_{\utilde_{k+2}^{k-m,t}} \bigg] \bigg(\mathbf{Z}_k^{k-m} \bigg)\t \utilde_k^{k-m,t-1}\bigg\|.
\end{align*}
%with $\eta_{k+1}^{(t-1)}$ and $\eta_{k+1}^{(t-1,k-m)}$ replaced with $\eta_{k+1}^{(t)}$ and $\eta_{k+1}^{(t,k-m)}$ if $k = 2$ or $3$, and $\eta_{k+2}^{(t-1)}$ and $\eta_{k+2}^{(t-1,k-m)}$ replaced with $\eta_{k+2}^{(t)}$ and $\eta_{k+2}^{(t,k-m)}$ if $k = 3$.
\end{lemma}
% Observe that the final term is a product of $e_m\t \mathbf{Z}_k$ with a matrix $\mathbf{M}$ that is \emph{independent} of $e_m\t \mathbf{Z}_k$.  Similarly, the previous term satisfies
% \begin{align*}
% \bigg\| e_m\t &\mathbf{Z}_k\bigg[ \mathcal{P}_{\utilde_{k+1}^{k-m,t}}\otimes \mathcal{P}_{\utilde_{k+2}^{k-m,t}} \bigg]\bigg[\mathcal{P}_{\utilde_{k+1}^{k-m,t}} \otimes \mathcal{P}_{\utilde_{k+2}^{k-m,t}} \bigg] \bigg(\mathbf{Z}_k - \mathbf{Z}_k^{k-m} \bigg) \t \utilde_k^{k-m,t-1}\bigg\| \\
% &\leq \bigg\| e_m\t \mathbf{Z}_k\bigg[ \mathcal{P}_{\utilde_{k+1}^{k-m,t}}\otimes \mathcal{P}_{\utilde_{k+2}^{k-m,t}} \bigg] \bigg\| \bigg\| \bigg[\mathcal{P}_{\utilde_{k+1}^{k-m,t}} \otimes \mathcal{P}_{\utilde_{k+2}^{k-m,t}} \bigg] \bigg(\mathbf{Z}_k - \mathbf{Z}_k^{k-m} \bigg) \t \utilde_k^{k-m,t-1}\bigg\| \\
% &= \bigg\| e_m\t \mathbf{Z}_k\bigg[ \mathcal{P}_{\utilde_{k+1}^{k-m,t}}\otimes \mathcal{P}_{\utilde_{k+2}^{k-m,t}} \bigg] \bigg\| \bigg\| \bigg[\mathcal{P}_{\utilde_{k+1}^{k-m,t}} \otimes \mathcal{P}_{\utilde_{k+2}^{k-m,t}} \bigg] \bigg(e_m\t \mathbf{Z}_k \bigg) \t \utilde_k^{k-m,t-1}\bigg\|.
% \end{align*}
% The first term is the product of $e_m\t \mathbf{Z}_k$ with a matrix that is independent from it, and the second term is the pre and post product of $e_m\t \mathbf{Z}_k$ with matrices that are independent from it.  Therefore, these terms should be sufficiently well-behaved.  

\begin{proof}[Proof of Lemma \ref{lem:quadratic_deterministic_bd}]
Similar to Lemma \ref{lem:linear_deterministic_bd} we prove for $k =1$; the case for $k =2$ or $k =3$ follows by modifying the index of $t$ according to the definition of $\mathcal{\hat P}_{k}^{(t)}$.  

Recall that 
\begin{align*}
\mathbf{Q}_1^{(t)} &=\U_{1\perp} \U_{1\perp}\t \mathbf{Z}_1 \bigg[ \mathcal{P}_{\uhat_{2}^{(t-1)}} \otimes \mathcal{P}_{\uhat_{3}^{(t-1)}}\bigg] \mathbf{Z}_1\t \uhat_{1}^{(t)} (\mathbf{\hat{\Lambda}}_1^{(t)})^{-2}.
\end{align*}
Observe that $\mathcal{P}_{\uhat_{2}^{(t-1)}} \otimes \mathcal{P}_{\uhat_{3}^{(t-1)}}$ is a projection matrix and hence equals its square.  Therefore, we simply decompose by noting that under the condition that $\hat \lambda_{r_k}(\mathbf{\hat{\Lambda}}_k^{(t)}) \geq \lambda/2$
\begin{align*}
\bigg\| e_m\t \mathbf{Q}_1^{(t)} \bigg\| &= \bigg\| \U_{1\perp} \U_{1\perp}\t \mathbf{Z}_1\bigg[  \mathcal{P}_{\uhat_{2}^{(t-1)}} \otimes \mathcal{P}_{\uhat_{3}^{(t-1)}} \bigg]\mathbf{Z}_1\t \uhat_{1}^{(t)}(\mathbf{\hat{\Lambda}}_1^{(t)})^{-2} \bigg\| \\
&\leq \frac{4}{\lambda^2} \bigg\| e_m\t \U_1 \U_1\t \mathbf{Z}_1 \bigg[ \mathcal{P}_{\uhat_{2}^{(t-1)}} \otimes \mathcal{P}_{\uhat_{3}^{(t-1)}} \bigg] \mathbf{Z}_1\t \uhat_{1}^{(t)} \bigg\| + \frac{4}{\lambda^2}\bigg\| e_m\t \mathbf{Z}_1 \bigg[ \mathcal{P}_{\uhat_{2}^{(t-1)}} \otimes \mathcal{P}_{\uhat_{3}^{(t-1)}}\bigg] \mathbf{Z}_1\t \uhat_{1}^{(t)} \bigg\|  \\
&\leq \frac{4}{\lambda^2} \| \U_1\|_{2,\infty}\bigg\| \U_1\t \mathbf{Z}_1 \bigg[ \mathcal{P}_{\uhat_{2}^{(t-1)}} \otimes \mathcal{P}_{\uhat_{3}^{(t-1)}}\bigg] \bigg[ \mathcal{P}_{\uhat_{2}^{(t-1)}} \otimes \mathcal{P}_{\uhat_{3}^{(t-1)}} \bigg]\mathbf{Z}_1\t \uhat_{1}^{(t)} \bigg\| \\
&\qquad + \frac{4}{\lambda^2}\bigg\| e_m\t \mathbf{Z}_1 \bigg[  \mathcal{P}_{\uhat_{2}^{(t-1)}} \otimes \mathcal{P}_{\uhat_{3}^{(t-1)}}\bigg]\bigg[  \mathcal{P}_{\uhat_{2}^{(t-1)}} \otimes \mathcal{P}_{\uhat_{3}^{(t-1)}}\bigg]\mathbf{Z}_1\t \uhat_{1}^{(t)} \bigg\|  \\
&\leq \frac{4}{\lambda^2} \|\U_1\|_{2,\infty} \tau_1 \bigg\| \U_1\t \mathbf{Z}_1 \bigg[  \mathcal{P}_{\uhat_{2}^{(t-1)}} \otimes \mathcal{P}_{\uhat_{3}^{(t-1)}}\bigg] \bigg\| \\
&\qquad + \frac{4}{\lambda^2}\bigg\| e_m\t \mathbf{Z}_1 \bigg[ \mathcal{P}_{\uhat_{2}^{(t-1)}} \otimes \mathcal{P}_{\uhat_{3}^{(t-1)}} \bigg]\bigg[ \mathcal{P}_{\uhat_{2}^{(t-1)}} \otimes \mathcal{P}_{\uhat_{3}^{(t-1)}}\bigg]\mathbf{Z}_1\t \uhat_1^{(t)} \bigg\| \\
&\leq \frac{4}{\lambda^2} \|\U_1\|_{2,\infty} \tau_1 \bigg\| \U_1\t \mathbf{Z}_1 \bigg[  \mathcal{P}_{\uhat_{2}^{(t-1)}} \otimes \mathcal{P}_{\uhat_{3}^{(t-1)}}\bigg] \bigg\|  \\
&\qquad + \frac{4}{\lambda^2}\bigg\| e_m\t \mathbf{Z}_1\bigg[ \mathcal{P}_{\utilde_{2}^{1-m,t-1} \otimes \uhat_{3}^{(t-1)}} -  \mathcal{P}_{\uhat_{2}^{(t-1)} \otimes \uhat_{3}^{(t-1)}} \bigg]\bigg[ \mathcal{P}_{\uhat_{2}^{(t-1)}} \otimes \mathcal{P}_{\uhat_{3}^{(t-1)}}\bigg]\mathbf{Z}_1\t \uhat_1^{(t)} \bigg\|\\
&\qquad + \frac{4}{\lambda^2}\bigg\| e_m\t \mathbf{Z}_1\bigg[ \mathcal{P}_{\utilde_{2}^{1-m,t-1} \otimes \uhat_{3}^{(t-1)}}  \bigg]\bigg[ \mathcal{P}_{\uhat_{2}^{(t-1)}} \otimes \mathcal{P}_{\uhat_{3}^{(t-1)}}\bigg]\mathbf{Z}_1\t \uhat_1^{(t)} \bigg\| \\
&\leq \frac{4}{\lambda^2} \|\U_1\|_{2,\infty} \tau_1 \bigg\| \U_1\t \mathbf{Z}_1 \bigg[  \mathcal{P}_{\uhat_{2}^{(t-1)}} \otimes \mathcal{P}_{\uhat_{3}^{(t-1)}}\bigg] \bigg\|  + \frac{8}{\lambda^2}\tau_1^2 \| \sin\Theta( \utilde_{2}^{1-m,t-1},\uhat_{2}^{(t-1)}) \|  \\
&\qquad + \frac{4}{\lambda^2}\bigg\| e_m\t \mathbf{Z}_1\bigg[ \mathcal{P}_{\utilde_{2}^{1-m,t-1}}\otimes \bigg( \mathcal{P}_{\uhat_{3}^{(t-1)}} - \mathcal{P}_{\utilde_{3}^{1-m,t-1}}\bigg) \bigg]\bigg[ \mathcal{P}_{\uhat_{2}^{(t-1)}} \otimes \mathcal{P}_{\uhat_{3}^{(t-1)}}\bigg]\mathbf{Z}_1\t \uhat_1^{(t)} \bigg\| \\
&\qquad + \frac{4}{\lambda^2}\bigg\| e_m\t \mathbf{Z}_1\bigg[ \mathcal{P}_{\utilde_{2}^{1-m,t-1}}\otimes \mathcal{P}_{\utilde_{3}^{1-m,t-1}} \bigg]\bigg[ \mathcal{P}_{\uhat_{2}^{(t-1)}} \otimes \mathcal{P}_{\uhat_{3}^{(t-1)}}\bigg]\mathbf{Z}_1\t \uhat_1^{(t)} \bigg\| \\
&\leq \frac{4}{\lambda^2} \|\U_1\|_{2,\infty} \tau_1 \bigg\| \U_1\t \mathbf{Z}_1 \bigg[  \mathcal{P}_{\uhat_{2}^{(t-1)}} \otimes \mathcal{P}_{\uhat_{3}^{(t-1)}}\bigg] \bigg\| \\
&\qquad + \frac{8}{\lambda^2}\tau_1^2 \bigg( \| \sin\Theta( \utilde_{2}^{1-m,t-1},\uhat_{2}^{(t-1)}) \| + \| \sin\Theta(\utilde_{3}^{1-m,t-1},\uhat_{3}^{(t-1)}) \|\bigg) \\
&\qquad + \frac{4}{\lambda^2}\bigg\| e_m\t \mathbf{Z}_1\bigg[ \mathcal{P}_{\utilde_{2}^{1-m,t-1}}\otimes \mathcal{P}_{\utilde_{3}^{1-m,t-1}} \bigg]\bigg[  \mathcal{P}_{\uhat_{2}^{(t-1)}}  \otimes \mathcal{P}_{\uhat_{3}^{(t-1)}}\bigg]\mathbf{Z}_1\t \uhat_1^{(t)} \bigg\| \\ % \\
&\leq \frac{4}{\lambda^2} \|\U_1\|_{2,\infty} \tau_1 \bigg\| \U_1\t \mathbf{Z}_1 \bigg[  \mathcal{P}_{\uhat_{2}^{(t-1)}} \otimes \mathcal{P}_{\uhat_{3}^{(t-1)}}\bigg] \bigg\|   \\
&\qquad + \frac{8}{\lambda^2}\tau_1^2\bigg(  \| \sin\Theta( \utilde_{2}^{1-m,t-1},\uhat_{2}^{(t-1)}) \| +\| \sin\Theta(\utilde_{3}^{1-m,t-1},\uhat_{3}^{(t-1)}) \| \bigg) \\
%&\qquad + \frac{4}{\lambda^2} \tau_k^2  \|\sin\Theta(\utilde_{k}^{k-m,t-1},\uhat_{k}^{(t-1)}) \| \\
%&\qquad +  \frac{4}{\lambda^2}\bigg\| e_m\t \mathbf{Z}_k\bigg[ \mathcal{P}_{\utilde_{k+1}^{k-m,t}}\otimes \mathcal{P}_{\utilde_{k+2}^{k-m,t}} \bigg]\bigg[\mathcal{P}_{\utilde_{k+1}^{k-m,t}} \otimes \mathcal{P}_{\utilde_{k+2}^{k-m,t}} \bigg] \bigg(\mathbf{Z}_k - \mathbf{Z}_k^{k-m} \bigg) \t \uhat_k^{(t)}\bigg\| \\
&\qquad + \frac{4}{\lambda^2} \bigg\| e_m\t \mathbf{Z}_1\bigg[ \mathcal{P}_{\utilde_{2}^{1-m,t-1}}\otimes \mathcal{P}_{\utilde_{3}^{1-m,t-1}} \bigg] \bigg\| \bigg\|(\uhat_1^{(t)})\t \mathbf{Z}_1 \mathcal{P}_{\uhat_2^{(t-1)}} \otimes \mathcal{P}_{\uhat_3^{1-m,t-1}} \bigg\| \\ %\bigg[\mathcal{P}_{\utilde_{k+1}^{k-m,t}} \otimes \mathcal{P}_{\utilde_{k+2}^{k-m,t}} \bigg] \bigg(\mathbf{Z}_k^{k-m} \bigg)\t \uhat_k^{(t)}\bigg\|. \\
%&= \frac{4}{\lambda^2} \tau_k^2 \| \U\|_{2,\infty} + \frac{8}{\lambda^2} \tau_k^2 \bigg( \eta_{k+1}^{(t,k-m)} + \eta_{k+2}^{(t,k-m)} \bigg) \\ %+ \frac{4}{\lambda^2} \tau_k^2 \eta_{k}^{(t,k-m)} \\
%&\qquad + \frac{4}{\lambda^2} \bigg\| e_m\t \mathbf{Z}_k\bigg[ \mathcal{P}_{\utilde_{k+1}^{k-m,t}}\otimes \mathcal{P}_{\utilde_{k+2}^{k-m,t}} \bigg]\bigg[\mathcal{P}_{\utilde_{k+1}^{k-m,t}} \otimes \mathcal{P}_{\utilde_{k+2}^{k-m,t}} \bigg] \bigg(\mathbf{Z}_k - \mathbf{Z}_k^{k-m} \bigg) \t \uhat_k^{(t)}\bigg\| \\
%&\qquad + \frac{4}{\lambda^2}\bigg\| e_m\t \mathbf{Z}_k\bigg[ \mathcal{P}_{\utilde_{k+1}^{k-m,t}}\otimes \mathcal{P}_{\utilde_{k+2}^{k-m,t}} \bigg]\bigg[\mathcal{P}_{\utilde_{k+1}^{k-m,t}} \otimes \mathcal{P}_{\utilde_{k+2}^{k-m,t}} \bigg] \bigg(\mathbf{Z}_k^{k-m} \bigg)\t \uhat_k^{(t)}\bigg\| \\
%&\leq \frac{4}{\lambda^2} \tau_k^2 \| \U_k \|_{2,\infty} + \frac{8}{\lambda^2} \tau_k^2 \bigg( \eta_{k+1}^{(t,k-m)} + \eta_{k+2}^{(t,k-m)} \bigg) \\ %+ \frac{4}{\lambda^2} \tau_k^2 \eta_{k}^{(t-1,k-m)} \\
%&\qquad + \frac{4}{\lambda^2} \bigg\| e_m\t \mathbf{Z}_k \bigg[ \mathcal{P}_{\utilde_{k+1}^{k-m,t}} \otimes \mathcal{P}_{\utilde_{k+2}^{k-m,t}}  \bigg] \bigg\| \bigg( \tau_k + 2 \bigg\|\bigg( \mathbf{Z}_k - \mathbf{Z}_k^{k-m} \bigg) \mathcal{P}_{\utilde_{k+1}^{k-m,t}} \otimes \mathcal{P}_{\utilde_{k+2}^{k-m,t}} \bigg\| \bigg) \\
&\leq \frac{4}{\lambda^2} \|\U_1\|_{2,\infty} \tau_1 \bigg\| \U_1\t \mathbf{Z}_1 \bigg[  \mathcal{P}_{\uhat_{2}^{(t-1)}} \otimes \mathcal{P}_{\uhat_{3}^{(t-1)}}\bigg] \bigg\| \\
&\qquad + \frac{8}{\lambda^2} \tau_1^2 \bigg( \eta_{1}^{(t,1-m)} \bigg) + \frac{4}{\lambda^2}  \xi_1^{t,1-m} \bigg\| (\uhat_1^{(t)})\t \mathbf{Z}_1 \mathcal{P}_{\uhat_2^{(t-1)}} \otimes \mathcal{P}_{\uhat_3^{(t-1)}} \bigg\|.
\end{align*}
Finally, we note that
\begin{align*}
    \bigg\|& (\uhat_1^{(t)})\t \mathbf{Z}_1 \mathcal{P}_{\uhat_2^{(t-1)}} \otimes \mathcal{P}_{\uhat_3^{(t-1)}} \bigg\| \\&=  \bigg\|\uhat_1^{(t)} (\uhat_1^{(t)})\t \mathbf{Z}_1 \mathcal{P}_{\uhat_2^{(t-1)}} \otimes \mathcal{P}_{\uhat_3^{(t-1)}} \bigg\| \\
    &\leq  \bigg\|\bigg( \uhat_1^{(t)} (\uhat_1^{(t)})\t - \U_1 \U_1\t \bigg) \mathbf{Z}_1 \mathcal{P}_{\uhat_2^{(t-1)}} \otimes \mathcal{P}_{\uhat_3^{(t-1)}} \bigg\| \\
    &\qquad + \bigg\|\U_1 \U_1\t \mathbf{Z}_1 \mathcal{P}_{\uhat_2^{(t-1)}} \otimes \mathcal{P}_{\uhat_3^{(t-1)}} \bigg\| \\
    &\leq \| \sin\Theta(\uhat_1^{(t)},\U_1) \| \tau_1 + \bigg\|\U_1 \U_1\t \mathbf{Z}_1 \mathcal{P}_{\uhat_2^{(t-1)}} \otimes \mathcal{P}_{\uhat_3^{(t-1)}} \bigg\| \\
    &\leq  \| \sin\Theta(\uhat_1^{(t)},\U_1) \| \tau_1 + \bigg\|\U_1 \U_1\t \mathbf{Z}_1 \bigg(\mathcal{P}_{\uhat_2^{(t-1)}} - \mathcal{P}_{\U_2} \bigg) \otimes \mathcal{P}_{\uhat_3^{(t-1)}} \bigg\| \\
    &\qquad + \bigg\|\U_1 \U_1\t \mathbf{Z}_1 \mathcal{P}_{\U_2}  \otimes  \bigg( \mathcal{P}_{\uhat_3^{(t-1)}} - \mathcal{P}_{\U_3} \bigg) \bigg\| \\
    &\qquad + \bigg\|\U_1 \U_1\t \mathbf{Z}_1 \mathcal{P}_{\U_2}  \otimes   \mathcal{P}_{\U_3} ) \bigg\| \\
    &\leq \tau_1 \| \sin\Theta(\uhat_1^{(t)},\U_1) \| + \tau_1 \| \sin\Theta(\uhat_2^{(t-1)},\U_2) \| + \tau_1 \| \sin\Theta(\uhat_3^{(t-1)},\U_2) \| \\
    &\qquad + \bigg\|\U_1 \U_1\t \mathbf{Z}_1 \mathcal{P}_{\U_2}  \otimes   \mathcal{P}_{\U_3}  \bigg\| \\
    &\leq \tau_1 \| \sin\Theta(\uhat_1^{(t)},\U_1) \| + \tau_1 \eta^{(t)}_1 +  \bigg\|\U_1 \U_1\t \mathbf{Z}_1 \mathcal{P}_{\U_2}  \otimes   \mathcal{P}_{\U_3}  \bigg\|,
\end{align*}
and, similarly, 
\begin{align*}
    \bigg\| \U_1 \t \mathbf{Z}_1 \bigg[  \mathcal{P}_{\uhat_{2}^{(t-1)}} \otimes \mathcal{P}_{\uhat_{3}^{(t-1)}}\bigg] \bigg\| &\leq \bigg\| \U_1\t \mathbf{Z}_1 \bigg[ \bigg( \mathcal{P}_{\uhat_{2}^{(t-1)}} - \mathcal{P}_{\U_2} \bigg) \otimes \mathcal{P}_{\uhat_{3}^{(t-1)}}\bigg] \bigg\| \\
    &\qquad + \bigg\| \U_1\t \mathbf{Z}_1 \bigg[ \mathcal{P}_{\U_2} \otimes \bigg( \mathcal{P}_{\uhat_{3}^{(t-1)}} - \mathcal{P}_{\U_3}  \bigg)\bigg] \bigg\| \\
    &\qquad + \bigg\| \U_1\t \mathbf{Z}_1 \bigg[ \mathcal{P}_{\U_2} \otimes  \mathcal{P}_{\U_3}  \bigg] \bigg\| \\
    &\leq \tau_1 \eta_1^{(t)} + \bigg\| \U_1\t \mathbf{Z}_1 \bigg[ \mathcal{P}_{\U_2} \otimes  \mathcal{P}_{\U_3}  \bigg] \bigg\|.
\end{align*}
Plugging in these bounds to our initial bound completes the proof.
\end{proof}

\begin{lemma}[Eigengaps]\label{lem:eigengaps}
Suppose that $\tau_k \leq \lambda/4$ and that 
\begin{align*}
   \eta_k^{(t)} \leq \frac{1}{4}.
\end{align*}
Then the following bounds hold:
\begin{align*}
    \lambda_{r_k}\bigg( \mathbf{T}_k \mathcal{\hat P}_k^{(t)} + \mathbf{Z}_k^{j-m}  \mathcal{\tilde P}_{k}^{t,j-m} \mathcal{\hat P}_k^{(t)}\bigg) &\geq \frac{3\lambda}{4}; \\
    \lambda_{r_k+1} \bigg( \mathbf{T}_k \mathcal{\hat P}_k^{(t)} + \mathbf{Z}_k \mathcal{\hat P}_k^{(t)}\bigg) &\leq \frac{\lambda}{4}.
\end{align*}
\end{lemma}
\begin{proof}   Note that since $\mathbf{Z}_k^{j-m}$ is $\mathbf{Z}_k$ with columns (or rows if $k = j$) removed, it holds that
\begin{align*}
  \mathbf{Z}_k\t \mathbf{Z}_k   \succcurlyeq  (\mathbf{Z}_k^{j-m} )\t \mathbf{Z}_k^{j-m}
\end{align*}
and hence that
\begin{align*}
   \mathcal{\tilde P}_{k}^{t,j-m} \mathbf{Z}_k\t \mathbf{Z}_k    \mathcal{\tilde P}_{k}^{t,j-m} \succcurlyeq \mathcal{\tilde P}_{k}^{t,j-m} (\mathbf{Z}_k^{j-m} )\t \mathbf{Z}_k^{j-m} \mathcal{\tilde P}_{k}^{t,j-m}.
\end{align*}
Taking norms, it holds that
\begin{align*}
\bigg\|  \mathbf{Z}_k^{j-m} \mathcal{\tilde P}_{k}^{t,j-m} \bigg\|^2 &=
 \bigg\| \mathcal{\tilde P}_{k}^{t,j-m} (\mathbf{Z}_k^{j-m} )\t \mathbf{Z}_k^{j-m} \mathcal{\tilde P}_{k}^{t,j-m} \bigg\|  \\
 &\leq  \bigg\| \mathcal{\tilde P}_{k}^{t,j-m}  \mathbf{Z}_k\t \mathbf{Z}_k    \mathcal{\tilde P}_{k}^{t,j-m} \bigg\|^2 \\
 &= \bigg\| \mathbf{Z}_k   \mathcal{\tilde P}_{k}^{t,j-m} \bigg\|\\
 &\leq  \tau_k^2,
\end{align*}
where we took the supremum in the final inequality.  Therefore, $\| \mathbf{Z}_k^{j-m}\mathcal{\tilde P}_{k}^{t,j-m} \| \leq \tau_k$.  Therefore, by Weyl's inequality, it holds that
\begin{align}
\bigg|    \lambda_{r_k} \bigg(\mathbf{T}_k\mathcal{\hat P}_k^{(t)}+ \mathbf{Z}_k^{j-m}  \mathcal{\tilde P}_{k}^{t,j-m}\mathcal{\hat P}_k^{(t)} \bigg) - \lambda_{r_k}(\mathbf{T}_k \mathcal{\hat P}_k^{(t)})\bigg| &\leq \| \mathbf{Z}_k^{j-m}\mathcal{\tilde P}_{k}^{t,j-m}\mathcal{\hat P}_k^{(t)} \| \nonumber\\
&\leq \| \mathbf{Z}_k^{j-m}\mathcal{\tilde P}_{k}^{t,j-m} \| \nonumber\\
&\leq \tau_k \nonumber\\
&\leq \frac{\lambda}{4}.  \label{eq:lambda4}
\end{align}
Next, when $\eta_k^{(t)} \leq \frac{1}{4}$, this implies that
\begin{align*}
    \max \bigg( \| \sin\Theta(\uhat_{k+1}^{(t-1)}, \U_{k+1} ) \|, \| \sin\Theta(\uhat_{k+2}^{(t-1)}, \U_{k+2} ) \| \bigg) \leq \frac{1}{4},
\end{align*}
and hence that
\begin{align}
    \lambda_{\min}( \mathbf{T}_k \mathcal{\hat P}_k^{(t)} ) &= \lambda_{\min} \bigg( \mathbf{T}_k \U_{k+1} \otimes \U_{k+2} ( \U_{k+1} \otimes \U_{k+2} )\t \mathcal{\hat P}_k^{(t)} \bigg) \nonumber \\
    &\geq \lambda \lambda_{\min} \bigg( ( \U_{k+1} \otimes \U_{k+2} )\t \mathcal{\hat P}_k^{(t)} \bigg) \nonumber\\
    &\geq \lambda \lambda_{\min} ( \U_{k+1}\t \uhat_{k+1}^{(t-1)}) \lambda_{\min} (\U_{k+2}\t \uhat_{k+2}^{(t-1)} ) \nonumber\\
    &\geq \lambda (1 - \frac{1}{16} )  \nonumber \\
    &\geq \frac{15}{16} \lambda. \label{eq:1516}
\end{align}
Combining \eqref{eq:1516} and \eqref{eq:lambda4} gives the first claim.  

For the second claim, we simply note that by Weyl's inequality,
\begin{align*}
  \bigg|  \lambda_{r_k+1}\bigg( \mathbf{T}_k  \mathcal{\hat P}_k^{(t)}+ \mathbf{Z}_k  \mathcal{\hat P}_k^{(t)} \bigg) - \lambda_{r_k+1} \bigg( \mathbf{T}_k  \mathcal{\hat P}_k^{(t)}\bigg) \bigg| &\leq \|  \mathbf{Z}_k  \mathcal{\hat P}_k^{(t)} \| \\
    &\leq \tau_k \leq \frac{\lambda}{4}.
\end{align*}
Since $\mathbf{T}_k$ is rank $r_k$, it holds that
\begin{align*}
    \lambda_{r_k+1}\bigg( \mathbf{T}_k  \mathcal{\hat P}_k^{(t)} \bigg) = 0,
\end{align*}
which proves the second assertion.  This completes the proof.
\end{proof}

\begin{lemma}[Deterministic Bound for Leave-One-Out Sequence]
\label{lem:leaveoneoutsintheta} Suppose that $\tau_k \leq \frac{\lambda}{4}$ and that $\eta_k^{(t)} \leq \frac{1}{4}$.  Then it holds that
\begin{align*}
 \| \sin\Theta(\uhat_{k}^{(t)}, \utilde_k^{t,j-m} )\| 
 &\leq \frac{16 \kappa}{\lambda} \tau_k \bigg( \eta_{k}^{(t,j-m)} \bigg) + \frac{16 \kappa}{\lambda} \xi_k^{t,j-m} \bigg( \eta_{k}^{(t,j-m)} \bigg) \\
   &\qquad + \frac{8 \kappa}{\lambda} \tilde\xi_k^{t,j-m} + \frac{16}{\lambda^2} \tau_k^2 \bigg( \eta_{k}^{(t,j-m)}  \bigg) + \frac{8}{\lambda^2} \tau_k \xi_k^{t,j-m} + \frac{4}{\lambda^2} ( \xi_k^{t,j-m})^2,
 %
 %&\leq \frac{8 \kappa}{\lambda} \left( \tau_k \left( \eta_{k+1}^{(t,j-m)} + \eta_{k+2}^{(t,j-m)} \right) + \xi_k^{j-m} \left( \eta_{k+1}^{(t,j-m)} + \eta_{k+2}^{(t,j-m)} \right) + \tilde \xi_k^{j-m} \right) \\
 %   &\qquad + \frac{4}{\lambda^2} \bigg( 2 \tau_k^2 \bigg( \eta_{k+1}^{(t,j-m)} + \eta_{k+2}^{(t,j-m)} \bigg) + 2\tau_k \xi_{k}^{j-m} + \big( \xi_k^{j-m}\big)^2 \bigg).
    % \frac{8 \kappa}{\lambda} \left( \tau_k \left( \eta_{k+1}^{(t,j-m)} + \eta_{k+2}^{(t,j-m)} \right) + \xi_k^{j-m} \left( \eta_{k+1}^{(t,j-m)} + \eta_{k+2}^{(t,j-m)} \right) + \tilde \xi_k^{j-m} \right) \\
    %&\qquad + \frac{4}{\lambda^2} \left( 2 \tau_k^2 \left( \eta_{k+1}^{(t,j-m)} + \eta_{k+2}^{(t,j-m)} \right) + \tau_k \xi_k^{j-m} + \bigg\| \mathbf{Z}_k^{j-m} \mathcal{P}_{\utilde_{k+1}^{(t,j-m)}}  \otimes \mathcal{P}_{\utilde_{k+2}^{(t,j-m)}}  \bigg\| \xi_k^{j-m} \right).
%\frac{4}{\lambda^2}\bigg( \tau_k \xi_k^{(t-1,j-m)} + (\xi_k^{(t-1,j-m)})^2 + 2 \tau_k^2  ( \eta_{k+1}^{(t-1,j-m)} + \eta_{k+2}^{(t-1,j-m)} ) \bigg) \\
% &\qquad + \frac{4\kappa}{\lambda} \bigg( \tilde \xi_k^{(t-1,j-m)}+ 2 \tau_k (  \eta_{k+1}^{(t-1,j-m)} +  \eta_{k+2}^{(t-1,j-m)}) \bigg).
% \end{align*}
% In particular, when $\tau_k \leq \lambda$ and $\xi_k^{(t-1,j-m)} \leq \lambda$, one has
% \begin{align*}
%     \| \uhat_{k}^{(t)} (\uhat_k^{(t)})\t - \utilde_k^{t,j-m}( \utilde_k^{t,j-m})\t \| &\leq \frac{8 \kappa}{\lambda} \bigg( \xi_k^{(t-1,j-m)} + 2 \tau_k \eta_{k+1}^{(t-1,j-m)} + 2\tau_k \eta_{k+2}^{(t-1,j-m)} \bigg).
\end{align*}
%with $\eta_{k+1}^{(t-1)}$ and $\eta_{k+1}^{(t-1,k-m)}$ replaced with $\eta_{k+1}^{(t)}$ and $\eta_{k+1}^{(t,k-m)}$ if $k = 2$ or $3$, and $\eta_{k+2}^{(t-1)}$ and $\eta_{k+2}^{(t-1,k-m)}$ replaced with $\eta_{k+2}^{(t)}$ and $\eta_{k+2}^{(t,k-m)}$ if $k = 3$.
\end{lemma}

\begin{proof}
We prove the result for $k = 1$; the result for $k = 2$ and $k = 3$ are similar by modifying the index on $t$.  

Recall that $\uhat_1^{(t)}$ are the singular vectors of the matrix
\begin{align*}
    \mathbf{T}_1 \uhat_{2}^{(t-1)} \otimes \uhat_{3}^{(t-1)} + \mathbf{Z}_1 \uhat_{2}^{(t-1)} \otimes \uhat_{3}^{(t-1)}
\end{align*}
and $\utilde_1^{t,j-m}$ are the singular vectors of the matrix
\begin{align*}
    \mathbf{T}_1 + \mathbf{Z}_1^{j-m} \mathcal{\tilde P}_{1}^{t,j-m}
\end{align*}
Consequently, the projection $\utilde_{1}^{t,j-m}(\utilde_1^{t,j-m})\t$ is also the projection onto the dominant left singular space of the matrix
\begin{align*}
    \bigg(\mathbf{T}_1 + \mathbf{Z}_1^{j-m}\mathcal{\tilde P}_{1}^{t,j-m} \bigg) \uhat_{2}^{(t-1)} \otimes \uhat_{3}^{(t-1)},
\end{align*}
Therefore, both projections are projections onto the dominant eigenspaces of the matrices defined via
\begin{align*}
\mathbf{\hat A} &\coloneqq     \mathbf{T}_1 \mathcal{P}_{\uhat_{2}^{(t-1)} \otimes \uhat_{3}^{(t-1)}} \mathbf{T}_1\t + \bigg[ \mathbf{Z}_1\mathcal{P}_{\uhat_{2}^{(t-1)} \otimes \uhat_{3}^{(t-1)}}  \mathbf{T}_1\t + \mathbf{T}_1 \mathcal{P}_{\uhat_{2}^{(t-1)} \otimes \uhat_{3}^{(t-1)}}  \mathbf{Z}_1\t + \mathbf{Z}_1 \mathcal{P}_{\uhat_{2}^{(t-1)} \otimes \uhat_{3}^{(t-1)}} \mathbf{Z}_1\t  \bigg]; \\
\mathbf{\tilde A} &\coloneqq \mathbf{T}_1 \mathcal{P}_{\uhat_{2}^{(t-1)} \otimes \uhat_{3}^{(t-1)}} \mathbf{T}_1\t + \bigg[ \mathbf{Z}_1^{j-m} \mathcal{P}_{\utilde_{2}^{t-1,j-m}} \otimes \mathcal{P}_{\utilde_{3}^{t-1,j-m}} \mathcal{P}_{\uhat_{2}^{(t-1)} \otimes \uhat_{3}^{(t-1)}}  \mathbf{T}_1\t \\
&\qquad \qquad \qquad \qquad \qquad + \mathbf{T}_1 \mathcal{P}_{\uhat_{2}^{(t-1)} \otimes \uhat_{3}^{(t-1)}} \mathcal{P}_{\utilde_{2}^{t-1,j-m}} \otimes \mathcal{P}_{\utilde_{3}^{t-1,j-m}} (\mathbf{Z}_1^{j-m})\t \\
&\qquad \qquad \qquad \qquad \qquad + \mathbf{Z}_1^{j-m} \mathcal{P}_{\utilde_{2}^{t-1,j-m}} \otimes \mathcal{P}_{\utilde_{3}^{t-1,j-m}} \mathcal{P}_{\uhat_{2}^{(t-1)} \otimes \uhat_{3}^{(t-1)}} \mathcal{P}_{\utilde_{2}^{t-1,j-m}} \otimes \mathcal{P}_{\utilde_{3}^{t-1,j-m}} \mathbf{Z}_1^{j-m})\t  \bigg].
\end{align*}
Therefore, the perturbation $\mathbf{\hat A - \tilde A}$ is equal to the sum of three terms, defined via
\begin{align*}
    \mathbf{P}_1 &\coloneqq \mathbf{Z}_1 \mathcal{P}_{\uhat_{2}^{(t-1)} \otimes \uhat_{3}^{(t-1)}}  \mathbf{T}_1\t - \mathbf{Z}_1^{j-m} \mathcal{P}_{\utilde_{2}^{t-1,j-m}} \otimes \mathcal{P}_{\utilde_{3}^{t-1,j-m}} \mathcal{P}_{\uhat_{2}^{(t-1)} \otimes \uhat_{3}^{(t-1)}}  \mathbf{T}_1\t \\
    &= \bigg[\mathbf{Z}_1 \mathcal{P}_{\uhat_{2}^{(t-1)} \otimes \uhat_{3}^{(t-1)}} - \mathbf{Z}_1^{j-m} \mathcal{P}_{\utilde_{2}^{t-1,j-m}} \otimes \mathcal{P}_{\utilde_{3}^{t-1,j-m}} \bigg] \mathcal{P}_{\uhat_{2}^{(t-1)} \otimes \uhat_{3}^{(t-1)}}  \mathbf{T}_1\t \\
    \mathbf{P}_2 &\coloneqq \mathbf{T}_1 \mathcal{P}_{\uhat_{2}^{(t-1)} \otimes \uhat_{3}^{(t-1)}}  \mathbf{Z}_1\t - \mathbf{T}_1 \mathcal{P}_{\uhat_{2}^{(t-1)} \otimes \uhat_{3}^{(t-1)}} \mathcal{P}_{\utilde_{2}^{t-1,j-m}} \otimes \mathcal{P}_{\utilde_{3}^{t-1,j-m}} (\mathbf{Z}_1^{j-m})\t \\
    &= \mathbf{T}_1 \mathcal{P}_{\uhat_{2}^{(t-1)} \otimes \uhat_{3}^{(t-1)}}  \bigg(\mathcal{P}_{\uhat_{2}^{(t-1)} \otimes \uhat_{3}^{(t-1)}}  \mathbf{Z}_1\t -  \mathcal{P}_{\utilde_{2}^{t-1,j-m}} \otimes \mathcal{P}_{\utilde_{3}^{t-1,j-m}} (\mathbf{Z}_1^{j-m})\t \bigg)   \\ 
    \mathbf{P}_3 &\coloneqq\mathbf{Z}_1 \mathcal{P}_{\uhat_{2}^{(t-1)} \otimes \uhat_{3}^{(t-1)}} \mathbf{Z}_1\t   - \mathbf{Z}_1^{j-m} \mathcal{P}_{\utilde_{2}^{t-1,j-m}} \otimes \mathcal{P}_{\utilde_{3}^{t-1,j-m}} \mathcal{P}_{\uhat_{2}^{(t-1)} \otimes \uhat_{3}^{(t-1)}} \mathcal{P}_{\utilde_{2}^{t-1,j-m}} \otimes \mathcal{P}_{\utilde_{3}^{t-1,j-m}} \mathbf{Z}_1^{j-m})\t 
\end{align*}
where we have used the fact that $\mathcal{P}_{\uhat_{2}^{(t-1)} \otimes \uhat_{3}^{(t-1)}}$ is a projection matrix and hence equal to its square.  We now bound each term successively.

\textbf{The term $\|\mathbf{P}_1 \|$:} Observe that
\begin{align*}
   \mathbf{P}_1 &= \bigg[\mathbf{Z}_1 \mathcal{P}_{\uhat_{2}^{(t-1)}} \otimes \mathcal{P}_{\uhat_{3}^{(t-1)}} - \mathbf{Z}_1^{j-m} \mathcal{P}_{\utilde_{2}^{t-1,j-m}} \otimes \mathcal{P}_{\utilde_{3}^{t-1,j-m}} \bigg] \mathcal{P}_{\uhat_{2}^{(t-1)}} \otimes \mathcal{P}_{\uhat_{3}^{(t-1)}} \mathbf{T}_1\t  \\
   &=  \left[ \mathbf{Z}_1 \left( \mathcal{P}_{\uhat_{2}^{(t-1)}} - \mathcal{P}_{\utilde_{2}^{t-1,j-m}} \right) \otimes \mathcal{P}_{\uhat_{3}^{(t-1)}} \right] \mathcal{P}_{\uhat_{2}^{(t-1)}} \otimes \mathcal{P}_{\uhat_{3}^{(t-1)}} \mathbf{T}_1\t \\
   &\qquad + \left[\mathbf{Z}_1 \mathcal{P}_{\utilde_{2}^{t-1,j-m}} \otimes \left( \mathcal{P}_{\uhat_{3}^{(t-1)}} - \mathcal{P}_{\utilde_{3}^{t-1,j-m}} \right) \right] \mathcal{P}_{\uhat_{2}^{(t-1)}} \otimes \mathcal{P}_{\uhat_{3}^{(t-1)}} \mathbf{T}_1\t \\
   &\qquad + \left[ \mathbf{Z}_1 - \mathbf{Z}_1^{j-m} \right] \bigg( \mathcal{P}_{\utilde_{2}^{t-1,j-m}} \otimes \mathcal{P}_{\utilde_{3}^{t-1,j-m}} \bigg) \left[ \mathcal{P}_{\uhat_{2}^{(t-1)}} - \mathcal{P}_{\utilde_{2}^{t-1,j-m}} \right] \otimes \mathcal{P}_{\uhat_{3}^{(t-1)}} \mathbf{T}_1\t \\
      &\qquad + \left[ \mathbf{Z}_1 - \mathbf{Z}_1^{j-m} \right] \bigg( \mathcal{P}_{\utilde_{2}^{t-1,j-m}} \otimes \mathcal{P}_{\utilde_{3}^{t-1,j-m}} \bigg) \mathcal{P}_{\utilde_{2}^{t-1,j-m}}  \otimes \left[ \mathcal{P}_{\uhat_{3}^{(t-1)}} - \mathcal{P}_{\utilde_{3}^{t-1,j-m}} \right]  \mathbf{T}_1\t \\
         &\qquad + \left[ \mathbf{Z}_1 - \mathbf{Z}_1^{j-m} \right] \bigg( \mathcal{P}_{\utilde_{2}^{t-1,j-m}} \otimes \mathcal{P}_{\utilde_{3}^{t-1,j-m}} \bigg) \mathbf{T}_1\t .
\end{align*}
Taking norms yields
\begin{align*}
    \| \mathbf{P}_1  \| &\leq 2\lambda_1 \tau_1 \bigg(\| \sin\Theta(\uhat_{2}^{(t-1)}, \utilde_{2}^{t-1,j-m} ) \| + \| \sin\Theta(\uhat_{3}^{(t-1)}, \utilde_{3}^{t-1,j-m} ) \| \bigg) \\
    &\qquad + 2\lambda_1\| \left[\mathbf{Z}_1 - \mathbf{Z}_1^{j-m} \right] \mathcal{P}_{\utilde_{2}^{t-1,j-m}} \otimes \mathcal{P}_{\utilde_{3}^{t-1,j-m}} \|  \| \sin\Theta(\uhat_{2}^{(t-1)}, \utilde_{2}^{t-1,j-m} ) \|  \\
    &\qquad +  2\lambda_1\| \left[\mathbf{Z}_1 - \mathbf{Z}_1^{j-m} \right] \mathcal{P}_{\utilde_{2}^{t-1,j-m}} \otimes \mathcal{P}_{\utilde_{3}^{t-1,j-m}} \|\| \sin\Theta(\uhat_{3}^{(t-1)}, \utilde_{3}^{t-1,j-m} ) \| \\
    &\qquad + \lambda_1 \|  \left[\mathbf{Z}_1 - \mathbf{Z}_1^{j-m} \right] \mathcal{P}_{\utilde_{2}^{t-1,j-m}} \otimes \mathcal{P}_{\utilde_{3}^{t-1,j-m}} \mathbf{V}_1 \| \\
    &\leq 2\lambda_1 \tau_1 \bigg( \eta_{1}^{(t,j-m)} \bigg)  + 2\lambda_1 \xi_{1}^{t,j-m} \bigg( \eta_{1}^{(t,j-m)}  \bigg) + \lambda_1 \tilde \xi_1^{t,j-m}.
\end{align*}
For $\mathbf{P}_2$ we proceed similarly.  It holds that
\begin{align*}
    \mathbf{P}_2  &= \mathbf{T}_1 \mathcal{P}_{\uhat_{2}^{(t-1)}} \otimes \mathcal{P}_{\uhat_{3}^{(t-1)}} \bigg( \mathcal{P}_{\uhat_{2}^{(t-1)}} \otimes \mathcal{P}_{\uhat_{3}^{(t-1)}} \mathbf{Z}_1 \t - \mathcal{P}_{\utilde_{2}^{t-1,j-m}} \otimes \mathcal{P}_{\utilde_{3}^{t-1,j-m}} (\mathbf{Z}_1^{j-m})\t \bigg)  \\
    &= \mathbf{T}_1 \mathcal{P}_{\uhat_{2}^{(t-1)}} \otimes \mathcal{P}_{\uhat_{3}^{(t-1)}} \bigg( \left[ \mathcal{P}_{\uhat^{(t)}_{k+1}} - \mathcal{P}_{\utilde_{2}^{t-1,j-m}} \right] \otimes \mathcal{P}_{\uhat_{3}^{(t-1)}} \mathbf{Z}_1 \t \bigg) \\
    &\qquad + \mathbf{T}_1 \mathcal{P}_{\uhat_{2}^{(t-1)}} \otimes \mathcal{P}_{\uhat_{3}^{(t-1)}} \bigg(\mathcal{P}_{\utilde_{2}^{t-1,j-m}} \otimes \left[ \mathcal{P}_{\uhat_{3}^{(t-1)}} - \mathcal{P}_{\utilde_{3}^{t-1,j-m}} \right] \mathbf{Z}_1\t \bigg) \\
    &\qquad +   \mathbf{T}_1 \bigg( \mathcal{P}_{\uhat_{2}^{(t-1)}} - \mathcal{P}_{\utilde_{2}^{t-1,j-m}} \bigg) \otimes \mathcal{P}_{\uhat_{3}^{(t-1)}} \bigg(   \mathcal{P}_{\utilde_{2}^{t-1,j-m}} \otimes \mathcal{P}_{\utilde_{3}^{t-1,j-m}} ( \mathbf{Z}_1 -  \mathbf{Z}_1^{j-m})\t \bigg)  \\
     &\qquad +   \mathbf{T}_1 \mathcal{P}_{\utilde_{2}^{t-1,j-m}}  \otimes\bigg( \mathcal{P}_{\utilde_{3}^{t-1,j-m}} -  \mathcal{P}_{\uhat_{3}^{(t-1)}} \bigg) \bigg(   \mathcal{P}_{\utilde_{2}^{t-1,j-m}} \otimes \mathcal{P}_{\utilde_{3}^{t-1,j-m}} ( \mathbf{Z}_1 -  \mathbf{Z}_1^{j-m})\t \bigg)  \\
      &\qquad +   \mathbf{T}_1   \mathcal{P}_{\utilde_{2}^{t-1,j-m}} \otimes \mathcal{P}_{\utilde_{3}^{t-1,j-m}} ( \mathbf{Z}_1 -  \mathbf{Z}_1^{j-m})\t .
\end{align*}
Taking norms yields the same upper bound as for $\|\mathbf{P}_1  \|$.  

For the the term $\mathbf{P}_3 $, we note that since $\mathcal{P}_{\uhat_{2}^{(t-1)} \otimes \uhat_{3}^{(t-1)}}$ is a projection matrix and hence equal to its cube, it holds that
\begin{align*}
    \| \mathbf{P}_3  \| &\leq \bigg\| \mathbf{Z}_1 \left[ \mathcal{P}_{\uhat_{2}^{(t-1)}} - \mathcal{P}_{\utilde_{2}^{t-1,j-m}} \right] \otimes \mathcal{P}_{\uhat_{3}^{(t-1)}} \mathcal{P}_{\uhat_{2}^{(t-1)} \otimes \uhat_{3}^{(t-1)}} \mathcal{P}_{\uhat_{2}^{(t-1)} \otimes \uhat_{3}^{(t-1)}} \mathbf{Z}_1  \bigg\| \\
    &\qquad + \bigg\|  \mathbf{Z}_1  \mathcal{P}_{\utilde_{2}^{t-1,j-m}} \otimes \left[ \mathcal{P}_{\utilde_{3}^{t-1,j-m}} -  \mathcal{P}_{\uhat_{3}^{(t-1)}} \right] \mathcal{P}_{\uhat_{2}^{(t-1)} \otimes \uhat_{3}^{(t-1)}} \mathcal{P}_{\uhat_{2}^{(t-1)} \otimes \uhat_{3}^{(t-1)}} \mathbf{Z}_1  \bigg\| \\
      &\qquad + \bigg\| \mathbf{Z}_1  \mathcal{P}_{\utilde_{2}^{t-1,j-m}}  \otimes \mathcal{P}_{\utilde_{3}^{t-1,j-m}} \mathcal{P}_{\uhat_{2}^{(t-1)} \otimes \uhat_{3}^{(t-1)}} \left[ \mathcal{P}_{\uhat_{2}^{(t-1)}} - \mathcal{P}_{\utilde_{2}^{t-1,j-m}} \right]  \otimes  \mathcal{P}_{\uhat_{3}^{(t-1)}} \mathbf{Z}_1  \bigg\| \\
     &\qquad + \bigg\| \mathbf{Z}_1  \mathcal{P}_{\utilde_{2}^{t-1,j-m}}  \otimes \mathcal{P}_{\utilde_{3}^{t-1,j-m}} \mathcal{P}_{\uhat_{2}^{(t-1)} \otimes \uhat_{3}^{(t-1)}}  \mathcal{P}_{\utilde_{2}^{t-1,j-m}}  \otimes  \left[ \mathcal{P}_{\utilde_{3}^{t-1,j-m}} - \mathcal{P}_{\uhat_{3}^{(t-1)}} \right] \mathbf{Z}_1  \bigg\| \\
     &\qquad + \bigg\| \mathbf{Z}_1 \mathcal{P}_{\utilde_{2}^{t-1,j-m}}  \otimes \mathcal{P}_{\utilde_{3}^{t-1,j-m}} \mathcal{P}_{\uhat_{2}^{(t-1)} \otimes \uhat_{3}^{(t-1)}} \mathcal{P}_{\utilde_{2}^{t-1,j-m}}  \otimes \mathcal{P}_{\utilde_{3}^{t-1,j-m}} \left[ \mathbf{Z}_1 - \mathbf{Z}_1^{j-m} \right]\t  \bigg\| \\
     &\qquad + \bigg\| \left[ \mathbf{Z}_1 - \mathbf{Z}_1^{j-m} \right] \mathcal{P}_{\utilde_{2}^{t-1,j-m}}  \otimes \mathcal{P}_{\utilde_{3}^{t-1,j-m}} \mathcal{P}_{\uhat_{2}^{(t-1)} \otimes \uhat_{3}^{(t-1)}} \mathcal{P}_{\utilde_{2}^{t-1,j-m}}  \otimes \mathcal{P}_{\utilde_{3}^{t-1,j-m}} \mathbf{Z}_1^{j-m} \bigg\| \\
     &\leq 2 \tau_1^2 \bigg( \| \sin\Theta(\uhat_{2}^{(t-1)}, \utilde_{2}^{t-1,j-m} ) \| + \| \sin\Theta(\uhat_{3}^{(t-1)}, \utilde_{3}^{t-1,j-m} ) \| \bigg) \\
     &\qquad + \tau_1 \bigg\| \left[ \mathbf{Z}_1 - \mathbf{Z}_1^{j-m} \right] \mathcal{P}_{\utilde_{2}^{t-1,j-m}}  \otimes \mathcal{P}_{\utilde_{3}^{t-1,j-m}}  \bigg\| \\
     &\qquad + \bigg\| \mathbf{Z}_1^{j-m} \mathcal{P}_{\utilde_{2}^{t-1,j-m}}  \otimes \mathcal{P}_{\utilde_{3}^{t-1,j-m}}  \bigg\| \bigg\| \left[ \mathbf{Z}_1 - \mathbf{Z}_1^{j-m} \right] \mathcal{P}_{\utilde_{2}^{t-1,j-m}}  \otimes \mathcal{P}_{\utilde_{3}^{t-1,j-m}}  \bigg\| \\
     &\leq 4 \tau_1^2 \bigg( \eta_{1}^{(t,j-m)}  \bigg) + \tau_1 \xi_{1}^{t,j-m} + \bigg\| \mathbf{Z}_1^{j-m} \mathcal{P}_{\utilde_{2}^{t-1,j-m}}  \otimes \mathcal{P}_{\utilde_{3}^{t-1,j-m}}  \bigg\| \xi_1^{t,j-m} \\
     &\leq 4 \tau_1^2 \bigg( \eta_{1}^{(t,j-m)}  \bigg) + 2\tau_1 \xi_{1}^{t,j-m} + \bigg\| \bigg(\mathbf{Z}_1 - \mathbf{Z}_1^{j-m}\bigg) \mathcal{P}_{\utilde_{2}^{t-1,j-m}}  \otimes \mathcal{P}_{\utilde_{3}^{t-1,j-m}}  \bigg\| \xi_1^{t,j-m} \\
     &\leq 4 \tau_1^2 \bigg( \eta_{1}^{(t,j-m)}  \bigg) + 2\tau_1 \xi_{1}^{t,j-m} + \big( \xi_1^{t,j-m}\big)^2.
\end{align*}
We note that by Lemma \ref{lem:eigengaps}, it holds that
\begin{align*}
    \lambda_{r_1} \big( \mathbf{\tilde A} \big) - \lambda_{r_1 + 1}\big( \mathbf{\hat A}\big) &=  \lambda_{r_1}^2 \bigg( \mathbf{T}_1 \mathcal{\hat P}_1^{(t)} + \mathbf{Z}_1^{j-m}  \mathcal{\tilde P}_1^{t,j-m}  \bigg) - \lambda_{r_1 + 1}^2 \bigg( \mathbf{T}_1 \mathcal{\hat P}_1^{(t)} + \mathbf{Z}_k \mathcal{\hat P}_1^{(t)} \bigg) \\
    &\geq \bigg( \frac{3}{4} \lambda \bigg)^2 - \bigg( \frac{\lambda}{4} \bigg)^2 \\
    &\geq \frac{\lambda^2}{4}.
\end{align*}
Consequently, by the Davis-Kahan Theorem, it holds that 
\begin{align*}
    \| \sin\Theta(\uhat_{1}^{(t)},  \utilde_1^{t,j-m}) \| &\leq \frac{4}{\lambda^2} \bigg( \| \mathbf{P}_1  \| +\| \mathbf{P}_2  \|  + \| \mathbf{P}_3  \|  \bigg),
\end{align*}
which holds under the eigengap condition by Lemma \ref{lem:eigengaps} and the assumption $\tau_k \leq \frac{\lambda}{4}$.  Therefore,
\begin{align*}
    \| \sin\Theta(\uhat_1^{(t)}, \utilde_1^{t,j-m} )\| 
   &\leq \frac{4}{\lambda^2} \bigg( \| \mathbf{P}_1  \| +\| \mathbf{P}_2  \|  + \| \mathbf{P}_3  \|  \bigg) \\
   &\leq \frac{8}{\lambda^2} \bigg( 2\lambda_1 \tau_1 \bigg( \eta_{1}^{(t,j-m)} \bigg)  + 2\lambda_1 \xi_{1}^{t,j-m} \bigg( \eta_{1}^{(t,j-m)} \bigg) + \lambda_1 \tilde \xi_1^{t,j-m} \bigg) \\
   &\qquad + \frac{4}{\lambda^2} \bigg( 4 \tau_1^2 \bigg( \eta_{1}^{(t,j-m)}  \bigg) + 2\tau_1 \xi_{1}^{t,j-m} + \big( \xi_1^{t,j-m}\big)^2 \bigg) \\
   &\leq \frac{16 \kappa}{\lambda} \tau_1 \bigg( \eta_{1}^{(t,j-m)}\bigg) + \frac{16 \kappa}{\lambda} \xi_1^{t,j-m} \bigg( \eta_{1}^{(t,j-m)} \bigg) \\
   &\qquad + \frac{8 \kappa}{\lambda} \tilde\xi_1^{t,j-m} + \frac{16}{\lambda^2} \tau_1^2 \bigg( \eta_1^{(t,j-m)}  \bigg) + \frac{8}{\lambda^2} \tau_1 \xi_1^{t,j-m} + \frac{4}{\lambda^2} ( \xi_1^{t,j-m})^2
\end{align*}
as desired.
\end{proof}

% \subsection{Probabilistic Bounds}
% \begin{lemma}[Bound on $\xi_k^{(t,j-m)}$]
% The term
% \begin{align*}
%     \xi_k^{(t,j-m)} = \bigg\| \bigg(\mathbf{Z}_k^{j-m} - \mathbf{Z}_k \bigg) \mathcal{P}_{\utilde_{k+1}^{(t-1,j-m)} \otimes \utilde_{k+2}^{(t,j-m)}} \bigg\|
% \end{align*}
% satisfies, with probability at least $1 - O(p^{-20})$,
% \begin{align*}
%     \xi_k^{(t,j-m)} &\lesssim \sigma \sqrt{p_{-j} \log(p)} \bigg(  \textcolor{black}{fill \ this \ in} \bigg)
% \end{align*}

% \end{lemma}

\subsection{Probabilistic Bounds on Good Events} \label{sec:probabilisticbounds}
This section contains high-probability bounds for the terms considered in the previous subsection.  Let $r = \max r_k$, $p = \max p_k$.  In what follows, we denote
\begin{align*}
    \delta_{\mathrm{L}}\ku &\coloneqq C_0 \kappa \sqrt{p_k\log(p)},
\end{align*}
where $C_0$ is taken to be some fixed constant.  
  
We will also recall the notation from the previous section:
\begin{align*}
      \mathcal{\hat P}_{k}^{(t)} &\coloneqq \begin{cases}
    \mathcal{P}_{\uhat_{k+1}^{(t-1)} \otimes \uhat_{k+2}^{(t-1)}} & k =1; \\
     \mathcal{P}_{\uhat_{k+1}^{(t-1)} \otimes \uhat_{k+2}^{(t)}} & k =2; \\
      \mathcal{P}_{\uhat_{k+1}^{(t)} \otimes \uhat_{k+2}^{(t)}} & k =3.\end{cases} \\
       \mathcal{\tilde P}_{k}^{t,j-m} &\coloneqq \begin{cases}
    \mathcal{P}_{\utilde_{k+1}^{(t-1,j-m)} \otimes \utilde_{k+2}^{(t-1,j-m)}} & k =1; \\
     \mathcal{P}_{\utilde_{k+1}^{(t-1,j-m)} \otimes \utilde_{k+2}^{(t,j-m)}} & k =2; \\
      \mathcal{P}_{\utilde_{k+1}^{(t,j-m)} \otimes \utilde_{k+2}^{(t,j-m)}} & k =3.
    \end{cases} \\
\mathbf{L}_{k}^{(t)} &\coloneqq \U_{k\perp} \U_{k\perp}\t \mathbf{Z}_k  \mathcal{\hat P}_k^{(t)} \mathbf{T}_k\t \uhat^{(t-1)}_k ( \mathbf{\hat{\Lambda}}_k^{(t-1)})^{-2}; \\
\mathbf{Q}_{k}^{(t)} &\coloneqq 
\U_{k\perp} \U_{k\perp}\t \mathbf{Z}_k  \mathcal{\hat P}_k^{(t)} \mathbf{Z}_k\t \hat{\U}_k^{(t-1)} (\mathbf{\hat{\Lambda}}_k^{(t-1)})^{-2}  \\
    \tau_k &\coloneqq\sup_{\substack{ \| \mathbf{U}_1 \| = 1, \mathrm{rank}(\U_1) \leq 2 r_{k+1}\\ \|\mathbf{U}_2\| =1, \mathrm{rank}(\U_2) \leq 2 r_{k+2}} } \|  \mathbf{Z}_k \bigg( \mathcal{P}_{\mathbf{U}_1} \otimes \mathcal{P}_{\mathbf{U}_2} \bigg)\|; \\  
  %  \tilde \tau_k &\coloneqq \sup_{\substack{ \| \mathbf{U}_1 \| = 1, \mathrm{rank}(\U_1) \leq 2 r_{k+1}\\ \|\mathbf{U}_2\| =1, \mathrm{rank}(\U_2) \leq 2 r_{k+2}} } \|  \mathbf{Z}_k \bigg( \mathcal{P}_{\mathbf{U}_1} \otimes \mathcal{P}_{\mathbf{U}_2} \bigg) \mathbf{V}_k\|; \\  
    \xi_{k}^{(t,j-m)} &\coloneqq 
    \bigg\| \bigg(\mathbf{Z}_k^{j-m} - \mathbf{Z}_k \bigg) \mathcal{\tilde P}_{k}^{t,j-m}  \bigg\| \\
    \tilde \xi_{k}^{(t,j-m)} &\coloneqq   \bigg\| \bigg(\mathbf{Z}_k^{j-m} - \mathbf{Z}_k \bigg) \mathcal{\tilde P}_{k}^{t,j-m}  \mathbf{V}_k  \bigg\| \\
     \eta_{k}^{(t,j-m)} &\coloneqq 
     \begin{cases}
         \| \sin\Theta( \utilde_{k+1}^{(t-1,j-m)}, \uhat_{k+1}^{(t-1)}) \| + \| \sin\Theta( \utilde_{k+2}^{(t-1,j-m)}, \uhat_{k+2}^{(t-1)}) \| & k = 1  \\
          \| \sin\Theta( \utilde_{k+1}^{(t-1,j-m)}, \uhat_{k+1}^{(t-1)}) \| + \| \sin\Theta( \utilde_{k+2}^{(t,j-m)}, \uhat_{k+2}^{(t)}) \| & k = 2  \\
           \| \sin\Theta( \utilde_{k+1}^{(t,j-m)}, \uhat_{k+1}^{(t)}) \| + \| \sin\Theta( \utilde_{k+2}^{(t,j-m)}, \uhat_{k+2}^{(t)}) \| & k = 3  \end{cases} \\
    \eta_k^{(t)} &\coloneqq \begin{cases}\| \sin\Theta( \U_{k+1}, \uhat_{k+1}^{(t-1)}) \| + \| \sin\Theta( \U_{k+2}, \uhat_{k+2}^{(t-1)}) \| & k = 1 \\
        \| \sin\Theta( \U_{k+1}, \uhat_{k+1}^{(t-1)}) \| + \| \sin\Theta( \U_{k+2}, \uhat_{k+2}^{(t)}) \| & k = 2 \\
        \| \sin\Theta( \U_{k+1}, \uhat_{k+1}^{(t)}) \| + \| \sin\Theta( \U_{k+2}, \uhat_{k+2}^{(t)}) \| & k = 3
    \end{cases}.
\end{align*}

We will also need to define several probabilistic events.
The first event $\mathcal{E}_{\mathrm{Good}}$ collects several probabilistic bounds that hold independently of $t$, provided $t_{\max} \leq \textcolor{black}{c p}$ for some constant $c$:
\begin{align*}
    \mathcal{E}_{\mathrm{Good}} &\coloneqq \bigg\{ \max_k \tau_k \leq C \sqrt{pr} \bigg\} \bigcap \bigg\{ \| \sin\Theta(\uhat_k^{(t)}, \U_k ) \| \leq \frac{\dl\ku}{\lambda} + \frac{1}{2^{t}} \text{ for all $t \leq t_{\max}$ and $1\leq k \leq 3$ } \bigg\} \\
    &\qquad \bigcap \bigg\{ \max_k \bigg\| \U_k\t \mathbf{Z}_k \mathbf{V}_k \bigg\| \leq C \left( \sqrt{r} + \sqrt{\log(p)} \right) \bigg\}; \\
    &\qquad \bigcap \bigg\{ \max_k  \bigg\| \U_k\t \mathbf{Z}_k \mathcal{P}_{\U_{k+1}}  \otimes   \mathcal{P}_{\U_{k+2}}  \bigg\| \leq C \left( r + \sqrt{\log(p)} \right) \bigg\}; \\
    &\qquad \bigcap \bigg\{ \max_k \bigg\| \mathbf{Z}_k \mathbf{V}_k \bigg\| \leq C \sqrt{p_k} \bigg\}. \numberthis \label{Egood}
    \end{align*}
\cref{lem:Egood}  demonstrates that the event $\mathcal{E}_{{\mathrm{Good}}}$ holds with probability at least $1 - O(p^{-30})$. We now define several events we use in our induction argument.  Set
\begin{align*}
    \mathcal{E}_{2,\infty}^{t,k} &\coloneqq \bigg\{  \| \uhat_k^{(t)} - \U_k \mathbf{W}_k^{(t)} \|_{2,\infty} \leq \bigg( \frac{\dl\ku}{\lambda} + \frac{1}{2^{t}} \bigg) \mu_0 \sqrt{\frac{r_k}{p_k}} \bigg\}; \\
    %\bigg\{ \max_k \| \uhat_k^{(t)} - \U_k \mathbf{W}_k^{(t)} \|_{2,\infty} \leq \bigg( \frac{\dl}{\lambda} + \frac{1}{2^{t}} \bigg) \mu_0 \sqrt{\frac{r}{p}} \text{ for all $t \leq t_0 - 1$}\bigg\}; \\
    \mathcal{E}_{j-m}^{t,k} &\coloneqq \bigg\{ \| \sin\Theta (\utilde_k^{t,j-m}, \uhat_k^{(t)}) \| \leq  \bigg( \frac{\dl\ku}{\lambda} + \frac{1}{2^{t}} \bigg) \mu_0 \sqrt{\frac{r_k}{p_j}} \bigg\}; \\
    %\bigg\{ \max_j \| \sin\Theta(\utilde_j^{t,k-m}, \uhat_j^{(t)}) \| \leq \bigg( \frac{\dl}{\lambda} + \frac{1}{2^{t}} \bigg) \mu_0 \sqrt{\frac{r}{p}} \text{ for all $t \leq t_0 - 1$} \bigg\}; \\
  \mathcal{E}_{\mathrm{main}}^{t_0-1,1} &\coloneqq \bigcap_{t=1}^{t_0-1} \Bigg\{ \bigcap_{k=1}^{3}  \mathcal{E}^{t,k}_{2,\infty} \cap  \bigcap_{j=1}^{3} \bigcap_{m=1}^{p_j} \mathcal{E}_{k-m}^{t,j} \Bigg\}; \\
       \mathcal{E}_{\mathrm{main}}^{t_0-1,2} &\coloneqq\mathcal{E}_{\mathrm{main}}^{t_0-1,1} \cap \bigg\{ \bigcap_{k=1}^{3} \bigcap_{m = 1}^{p_k} \mathcal{E}_{k-m}^{t_0,1} \bigg\} \cap \mathcal{E}_{2,\infty}^{t_0,1} \\
        \mathcal{E}_{\mathrm{main}}^{t_0-1,3} &\coloneqq\mathcal{E}_{\mathrm{main}}^{t_0-1,2} \cap \bigg\{ \bigcap_{k=1}^{3} \bigcap_{m = 1}^{p_k} \mathcal{E}_{k-m}^{t_0,2} \bigg\} \cap \mathcal{E}_{2,\infty}^{t_0,2}.
  \end{align*}
  The event $\mathcal{E}^{t,k}_{2,\infty}$ concerns the desired bound, the event $\mathcal{E}^{t,k}_{j-m}$ controls the leave one out sequences, and the other events $\mathcal{E}_{\mathrm{main}}^{t_0-1,k}$ are simply the intersection of these events, mainly introduced for convenience.
  
  Finally, the following event concerns the incoherence of our leave-one-out sequences:
  \begin{align*}
        \mathcal{\tilde E}_{j-m}^{t,k} &\coloneqq \Bigg\{ \|\mathcal{\tilde P}_k^{t_0,j-m}  \mathbf{V}_k \|_{2,\infty} \leq c \mu_0^2 \frac{\sqrt{r_{-k}}}{p_j} \bigg( \frac{\dl^{(k+1)}}{\lambda} + \frac{1}{2^{t_0-1}}\bigg)\bigg( \frac{\dl^{(k+2)}}{\lambda} + \frac{1}{2^{t_0-1}}\bigg) \\
     &\quad + c \mu_0^2 \frac{\sqrt{r_{-k}}}{\sqrt{p_jp_{k+2}}} \bigg( \frac{\dl^{(k+1)}}{\lambda} + \frac{1}{2^{t_0-1}} \bigg) + c \mu_0^2 \frac{\sqrt{r_{-k}}}{\sqrt{p_j p_{k+1}}} \bigg( \frac{\dl^{(k+2)}}{\lambda} + \frac{1}{2^{t_0-1}} \bigg)\\
     &\quad + c \mu_0^2 \frac{\sqrt{r_{-k}}}{\sqrt{p_{-k}}} \bigg( \frac{\dl^{(k+1)}}{\lambda} + \frac{1}{2^{t_0-1}} \bigg) + c\mu_0^2 \frac{\sqrt{r_{-k}}}{\sqrt{p_{-k}}} \bigg( \frac{\dl^{(k+2)}}{\lambda} + \frac{1}{2^{t_0-1}} \bigg) + c \mu_0 \sqrt{\frac{r_k}{p_{-k}}}.\Bigg\} \\
     &\bigcap \Bigg\{ \|\mathcal{\tilde P}_k^{t_0,j-m}   \|_{2,\infty} \leq c \mu_0^2 \frac{\sqrt{r_{-k}}}{p_j}  \bigg( \frac{\dl^{(k+1)}}{\lambda} + \frac{1}{2^{t_0-1}} \bigg)  \bigg( \frac{\dl^{(k+2)}}{\lambda} + \frac{1}{2^{t_0-1}} \bigg) \\
    &\quad +  c\mu_0^2 \frac{\sqrt{r_{-k}}}{\sqrt{p_jp_{k+2}}} \bigg( \frac{\dl^{(k+1)}}{\lambda} + \frac{1}{2^{t_0-1}} \bigg) + c\mu_0^2 \frac{\sqrt{r_{-k}}}{\sqrt{p_jp_{k+1}}} \bigg( \frac{\dl^{(k+2)}}{\lambda} + \frac{1}{2^{t_0-1}} \bigg) + c\mu_0^2 \frac{\sqrt{r_{-k}}}{\sqrt{p_{-k}}}.\Bigg\},
  \end{align*}
  where $c$ is some deterministic constant.  
  While the precise definition of the event $\mathcal{\tilde E}_{j-m}^{t,k}$ is complicated, it is useful to keep in mind that the event will be used as an event independent of the nonzero elements in the matrices $\mathbf{Z}_k - \mathbf{Z}_k^{j-m}$, and it simply controls the incoherence of the leave-one-out sequences.  
  
  The following lemma shows that the leave-one-out sequences are incoherent whenever there are bounds on the previous iterates in $\ell_{2,\infty}$ norm.
\begin{lemma}\label{lem:eventintersection}
For any fixed $t_0$, $j$, $k$, and $m$ with $1 \leq t_0 \leq t_{\max}$, $1 \leq j \leq 3$, $1 \leq k \leq 3$, and $1 \leq m \leq p_{j}$, it holds that the set
\begin{align*}
   \mathcal{E}_{\mathrm{Good}} \cap \mathcal{E}_{\mathrm{main}}^{t_0-1,k} \cap \left( \mathcal{\tilde E}_{j-m}^{t_0,k} \right)^c
\end{align*}
is empty.
% Suppose $t_0$, $k$, and $m$ are fixed with $t_0 \geq 1$ and $1 \leq m \leq p_k$.  Define the events
% \begin{align*}
%     \mathcal{F}^{t_0,k}_{2,\infty} &\coloneqq \begin{cases} \mathcal{E}^{t_0 - 1,1}_{2,\infty} \cap \mathcal{E}^{t_0 - 1,2}_{2,\infty} \cap \mathcal{E}^{t_0-1,3}_{2,\infty} & k = 1; \\
%     \mathcal{E}^{t_0,1}_{2,\infty} \cap  \mathcal{E}^{t_0-1,2}_{2,\infty} \cap \mathcal{E}^{t_0-1,3}_{2,\infty} & k = 2; \\
%     \mathcal{E}^{t_0,1}_{2,\infty} \cap \mathcal{E}^{t_0,2}_{2,\infty} \cap \mathcal{E}^{t_0-1,3}_{2,\infty} & k = 3; \end{cases} 
%      \end{align*}
% Then it holds that $\mathcal{E}_{\mathrm{Good}} \cap \mathcal{A}^{t_0,k} \cap \mathcal{E}_{-m}^{t_0-1,k} \cap \big( \mathcal{\tilde E}_{-m}^{t_0,k} \big)^c = \varnothing$. In other words, if $\mathcal{E}_{\mathrm{Good}} \cap \mathcal{A}^{t_0,k} \cap \mathcal{E}_{-m}^{t_0-1,k} $ holds then so must $\mathcal{\tilde E}_{-m}^{t_0,k}$.   
%The events $\mathcal{E}_{2,\infty}^{t_0}$, $  \mathcal{E}_{k-m}^{t_0}$, and $ \mathcal{\tilde E}_{k-m}^{t_0}$ satisfy $\p\bigg( \mathcal{E}_{2,\infty}^{t_0} \bigcap \mathcal{E}_{k-m}^{t_0} \bigcap \big( \mathcal{\tilde E}_{k-m}^{t_0} \big)^c \bigg) = 0$.
\end{lemma}

\begin{proof}
Without loss of generality, we prove the result for $k = 1$; the cases  $ k =2$ and $k =3$ are similar (in fact, the result can be made slightly sharper, but this is not needed for our purposes).  

Note that when $t_0 \geq 1$, it holds that on the event $\mathcal{E}_{\mathrm{main}}^{t_0-1,1}$.
\begin{align*}
    \| \uhat_1^{(t_0-1)} \|_{2,\infty} &\leq \| \uhat_{1}^{(t_0-1)} - \U_1 \mathbf{W}_1^{(t_0-1)} \|_{2,\infty} + \| \U_1 \|_{2,\infty} \leq 2 \mu_0 \sqrt{\frac{r_1}{p_1}}.  
\end{align*}
Similarly,
\begin{align*}
\| \uhat_2^{(t_0-1)} \|_{2,\infty} &\leq 2 \mu_0 \sqrt{\frac{r_2}{p_2}}; \\
    \| \uhat_3^{(t_0-1)} \|_{2,\infty} &\leq 2 \mu_0 \sqrt{\frac{r_3}{p_3}}.
\end{align*}
In addition, on this event it holds that
\begin{align*}
    \| \mathcal{P}_{\utilde_2^{(t_0-1,j-m)}} - \mathcal{P}_{\uhat_2^{(t_0-1)}} \| &\leq 2\bigg( \frac{\dl^{(2)}}{\lambda} + \frac{1}{2^{t_0-1}} \bigg) \mu_0 \sqrt{\frac{r_2}{p_j}}; \\
    \| \mathcal{P}_{\utilde_3^{(t_0-1,j-m)}} - \mathcal{P}_{\uhat_3^{(t_0-1)}} \| &\leq2 \bigg( \frac{\dl^{(3)}}{\lambda} + \frac{1}{2^{t_0-1}} \bigg) \mu_0 \sqrt{\frac{r_3}{p_j}}. \numberthis \label{loobnd}
\end{align*}
Next, observe that on the events listed,
\begin{align*}
    \| \mathcal{P}_{\utilde_{2}^{(t_0-1,j-m)}}&  \otimes \mathcal{P}_{ \utilde_{3}^{(t_0-1,j-m)}} \mathbf{V}_1 \|_{2,\infty} \\
    &\leq \bigg\|\bigg[ \mathcal{P}_{\utilde_{2}^{(t_0-1,j-m)}} - \mathcal{P}_{\uhat_{2}^{(t_0-1)}} \bigg] \otimes \mathcal{P}_{ \utilde_{3}^{(t_0-1,j-m)}} \mathbf{V}_1 \bigg\|_{2,\infty} \\
    &+ \bigg\| \mathcal{P}_{\uhat_{2}^{(t_0-1)}} \otimes \mathcal{P}_{ \utilde_{3}^{(t_0-1,j-m)}} \mathbf{V}_1 \bigg\|_{2,\infty} \\
    &\leq \bigg\|\bigg[ \mathcal{P}_{\utilde_{2}^{(t_0-1,j-m)}} - \mathcal{P}_{\uhat_{2}^{(t_0-1)}} \bigg] \otimes \bigg[ \mathcal{P}_{ \utilde_{3}^{(t_0-1,j-m)}} - \mathcal{P}_{\uhat_{3}^{(t_0-1)}} \bigg]\mathbf{V}_1 \bigg\|_{2,\infty} \\
    &+ \bigg\|\bigg[ \mathcal{P}_{\utilde_{2}^{(t_0-1,j-m)}} - \mathcal{P}_{\uhat_{2}^{(t_0-1)}} \bigg] \otimes \mathcal{P}_{\uhat_{3}^{(t_0-1)}}\mathbf{V}_1 \bigg\|_{2,\infty} \\
    &+ \bigg\| \mathcal{P}_{\uhat_{2}^{(t_0-1)}} \otimes \bigg[  \mathcal{P}_{ \utilde_{3}^{(t_0-1,j-m)}} - \mathcal{P}_{\uhat_{3}^{(t_0-1)}} \bigg] \mathbf{V}_1 \bigg\|_{2,\infty} \\
    &+ \bigg\| \mathcal{P}_{\uhat_{2}^{(t_0-1)}} \otimes  \mathcal{P}_{\uhat_{3}^{(t_0-1)}}  \mathbf{V}_1 \bigg\|_{2,\infty} \\
    &\leq \bigg\|\bigg[ \mathcal{P}_{\utilde_{2}^{(t_0-1,j-m)}} - \mathcal{P}_{\uhat_{2}^{(t_0-1)}} \bigg] \otimes \bigg[ \mathcal{P}_{ \utilde_{3}^{(t_0-1,j-m)}} - \mathcal{P}_{\uhat_{3}^{(t_0-1)}} \bigg]\mathbf{V}_1 \bigg\|_{2,\infty} \\
    &+ \bigg\|\bigg[ \mathcal{P}_{\utilde_{2}^{(t_0-1,j-m)}} - \mathcal{P}_{\uhat_{2}^{(t_0-1)}} \bigg] \otimes \mathcal{P}_{\uhat_{2}^{(t_0-1)}}\mathbf{V}_1 \bigg\|_{2,\infty} \\
    &+ \bigg\| \mathcal{P}_{\uhat_{2}^{(t_0-1)}} \otimes \bigg[  \mathcal{P}_{ \utilde_{3}^{(t_0-1,j-m)}} - \mathcal{P}_{\uhat_{3}^{(t_0-1)}} \bigg] \mathbf{V}_1 \bigg\|_{2,\infty} \\
    &+ \bigg\| \bigg[ \mathcal{P}_{\uhat_{2}^{(t_0-1)}} - \mathcal{P}_{\U_{2}} \bigg] \otimes  \mathcal{P}_{\uhat_{3}^{(t_0-1)}}  \mathbf{V}_1 \bigg\|_{2,\infty} \\
    &+  \bigg\| \mathcal{P}_{\U_{2}}  \otimes \bigg[ \mathcal{P}_{\uhat_{3}^{(t_0-1)}} - \mathcal{P}_{\U_{3}} \bigg] \mathbf{V}_1 \bigg\|_{2,\infty} \\
    &+ \bigg\| \mathcal{P}_{\U_{2}} \otimes \mathcal{P}_{\U_{3}} \mathbf{V}_1 \bigg\|_{2,\infty} \\
    &\eqqcolon (I) + (II) + (III) + (IV) + (V) + (VI),
    \end{align*}
    where
    \begin{align*}
        (I) &\coloneqq \bigg\|\bigg[ \mathcal{P}_{\utilde_{2}^{(t_0-1,j-m)}} - \mathcal{P}_{\uhat_{2}^{(t_0-1)}} \bigg] \otimes \bigg[ \mathcal{P}_{ \utilde_{3}^{(t_0-1,j-m)}} - \mathcal{P}_{\uhat_{3}^{(t_0-1)}} \bigg]\mathbf{V}_1 \bigg\|_{2,\infty}; \\
     (II) : &= \bigg\|\bigg[ \mathcal{P}_{\utilde_{2}^{(t_0-1,j-m)}} - \mathcal{P}_{\uhat_{2}^{(t_0-1)}} \bigg] \otimes \mathcal{P}_{\uhat_{3}^{(t_0-1)}}\mathbf{V}_1 \bigg\|_{2,\infty};  \\
         (III) &\coloneqq \bigg\| \mathcal{P}_{\uhat_{2}^{(t_0-1)}} \otimes \bigg[  \mathcal{P}_{ \utilde_{3}^{(t_0-1,j-m)}} - \mathcal{P}_{\uhat_{3}^{(t_0-1)}} \bigg] \mathbf{V}_1 \bigg\|_{2,\infty};  \\
         (IV)  &\coloneqq \bigg\| \bigg[ \mathcal{P}_{\uhat_{2}^{(t_0-1)}} - \mathcal{P}_{\U_{2}} \bigg] \otimes  \mathcal{P}_{\uhat_{3}^{(t_0-1)}}  \mathbf{V}_1 \bigg\|_{2,\infty}; \\
        (V)&\coloneqq\bigg\| \mathcal{P}_{\U_{2}}  \otimes \bigg[ \mathcal{P}_{\uhat_{3}^{(t_0-1)}} - \mathcal{P}_{\U_{3}} \bigg] \mathbf{V}_1 \bigg\|_{2,\infty}; \\
        (VI) &\coloneqq \bigg\| \mathcal{P}_{\U_{2}} \otimes \mathcal{P}_{\U_{3}} \mathbf{V}_1 \bigg\|_{2,\infty}.
    \end{align*}
We now bound each term in turn, where we will use \eqref{loobnd} repeatedly.  We have that
\begin{align*}
    (I) &\leq  \bigg\|\bigg[ \mathcal{P}_{\utilde_{2}^{(t_0-1,j-m)}} - \mathcal{P}_{\uhat_{2}^{(t_0-1)}} \bigg] \otimes \bigg[ \mathcal{P}_{ \utilde_{3}^{(t_0-1,j-m)}} - \mathcal{P}_{\uhat_{3}^{(t_0-1)}} \bigg] \bigg\|_{2,\infty} \\
    &\leq  \bigg\| \mathcal{P}_{\utilde_{2}^{(t_0-1,j-m)}} - \mathcal{P}_{\uhat_{2}^{(t_0-1)}} \bigg\| \bigg\| \mathcal{P}_{ \utilde_{3}^{(t_0-1,j-m)}} - \mathcal{P}_{\uhat_{3}^{(t_0-1)}}  \bigg\|\\ 
    &\leq 4 \mu_0^2 \frac{\sqrt{r_2r_3}}{p_j} \bigg( \frac{\dl^{(2)}}{\lambda} + \frac{1}{2^{t_0-1}} \bigg) \bigg( \frac{\dl^{(3)}}{\lambda} + \frac{1}{2^{t_0-1}} \bigg). \numberthis \label{onebound}
\end{align*}
Similarly,
\begin{align*}
    (II) &\leq  \bigg\|\bigg[ \mathcal{P}_{\utilde_{2}^{(t_0-1,j-m)}} - \mathcal{P}_{\uhat_{2}^{(t_0-1)}} \bigg] \otimes \mathcal{P}_{\uhat_{3}^{(t_0-1)}} \bigg\|_{2,\infty} \\
    &\leq   \bigg\| \mathcal{P}_{\utilde_{2}^{(t_0-1,j-m)}} - \mathcal{P}_{\uhat_{2}^{(t_0-1)}} \bigg\| \bigg\| \mathcal{P}_{\uhat_{3}^{(t_0-1)}} \bigg\|_{2,\infty} \\
    &\leq 4\bigg( \frac{\dl^{(2)}}{\lambda} + \frac{1}{2^{t_0-1}} \bigg) \mu_0^2 \sqrt{\frac{r_2}{p_j}}  \sqrt{\frac{r_3}{p_3}}.
\end{align*}
Next,
\begin{align*}
    (III) &\leq  \bigg\| \mathcal{P}_{\uhat_{2}^{(t_0-1)}} \otimes \bigg[  \mathcal{P}_{ \utilde_{3}^{(t_0-1,j-m)}} - \mathcal{P}_{\uhat_{3}^{(t_0-1)}} \bigg]  \bigg\|_{2,\infty} \\
    &\leq4 \bigg( \frac{\dl^{(3)}}{\lambda} + \frac{1}{2^{t_0-1}} \bigg) \mu_0^2 \sqrt{\frac{r_3}{p_j}} \sqrt{\frac{r_2}{p_2}}.
\end{align*}
For the next two terms, we note that
for any orthogonal matrix $\mathbf{W} \in \mathbb{O}(r_k) $,
\begin{align*}
    \| \uhat_{2}^{(t_0-1)} \uhat_{2}^{t_0-1\top} - \U_{2} \U_{2}\t \|_{2,\infty} &= \| \uhat_{2}^{(t_0-1)} \mathbf{W} \mathbf{W}\t \uhat_{2}^{t_0-1\top} - \U_{2} \U_{2}\t \|_{2,\infty}  \\
    &\leq \| (\uhat_{2}^{(t_0-1)} \mathbf{W} - \U_{2} ) (\uhat_{2}^{(t_0-1)} \mathbf{W})\t \|_{2,\infty} \\
    &\qquad + \| \U_{2} ( \uhat_{2}^{(t_0-1)} \mathbf{W} - \U_{2} ) \|_{2,\infty} \\
    &\leq \| \uhat_{2}^{(t_0-1)} \mathbf{W} - \U_{2} \|_{2,\infty}  \\
    &\qquad + \| \U_{2} \|_{2,\infty} \| \uhat_{2}^{(t_0-1)} \mathbf{W} - \U_{2} \|. \end{align*}% \\
By taking the infimum over $\mathbb{O}(r_{2})$, we note that by Proposition 1 of \cite{cai_rate-optimal_2018}
\begin{align*}
\inf_{\mathbf{W} \in \mathbb{O}(r_{2})} \| \uhat_{2}^{(t_0-1)} \mathbf{W} - \U_{2} \| &\leq \sqrt{2} \| \sin\Theta( \uhat_{2}^{(t_0-1)}, \U_{2}) \| \\
&\leq \sqrt{2} \bigg( \frac{\dl^{(2)}}{\lambda} + \frac{1}{2^{t_0-1}} \bigg),
\end{align*}
where the final inequality is on the event $\mathcal{E}_{\mathrm{Good}}$ since $t_0\leq t_{\max}$.  We also note that on the event $\mathcal{E}_{\mathrm{main}}^{t_0-1,1}$ and the right-invariance of $\|\cdot\|_{2,\infty}$ to orthogonal matrices,
\begin{align*}
    \inf_{\mathbf{W} \in \mathbb{O}(r_{2})} \| \uhat_{2}^{(t_0-1)}\mathbf{W} - \U_{2} \|_{2,\infty} &\leq \inf_{\mathbf{W} \in \mathbb{O}(r_{2})} \| \uhat_{2}^{(t_0-1)} - \U_{2}\mathbf{W} \|_{2,\infty} \\
    &\leq \| \uhat_{2}^{(t_0-1)} - \U_{2} \mathbf{W}_2^{(t_0-1)} \|_{2,\infty} \\
    &\leq \mu_0 \sqrt{\frac{r_2}{p_2}} \bigg( \frac{\dl^{((2)}}{\lambda} + \frac{1}{2^{t_0-1}} \bigg).
\end{align*}
Therefore,
\begin{align*}
   \| \mathcal{P}_{\uhat_{2}^{(t_0-1)}} - \mathcal{P}_{\U_{2}} \|_{2,\infty} &\leq 3 \mu_0 \sqrt{\frac{r_2}{p_2}} \bigg( \frac{\dl^{(2)}}{\lambda} + \frac{1}{2^{t_0-1}} \bigg).
\end{align*}
Similarly,
\begin{align*}
     \| \mathcal{P}_{\uhat_{3}^{(t_0-1)}} - \mathcal{P}_{\U_{3}} \|_{2,\infty} &\leq 3 \mu_0 \sqrt{\frac{r_3}{p_3}} \bigg( \frac{\dl^{(3)}}{\lambda} + \frac{1}{2^{t_0-1}} \bigg).
\end{align*}
Therefore,
\begin{align*}
         (IV)  &\leq \bigg\| \mathcal{P}_{\uhat_{2}^{(t_0-1)}} - \mathcal{P}_{\U_{2}} \bigg\|_{2,\infty} \|  \mathcal{P}_{\uhat_{3}^{(t_0-1)}}  \|_{2,\infty}; \\
         &\leq 6 \mu_0^2 \sqrt{\frac{r_3}{p_3}}\sqrt{\frac{r_2}{p_2}}\bigg( \frac{\dl^{(2)}}{\lambda} + \frac{1}{2^{t_0-1}} \bigg) ; \\
         (V) &\leq\bigg\| \mathcal{P}_{\U_{2}}  \otimes \bigg[ \mathcal{P}_{\uhat_{3}^{(t_0-1)}} - \mathcal{P}_{\U_{3}} \bigg]  \bigg\|_{2,\infty} \\
         &\leq 3\mu_0^2 \sqrt{\frac{r_2}{p_2}}  \sqrt{\frac{r_3}{p_3}} \bigg( \frac{\dl^{(3)}}{\lambda} + \frac{1}{2^{t_0-1}} \bigg).
\end{align*}
Finally,
\begin{align*}
        (VI) &\coloneqq \bigg\| \mathcal{P}_{\U_{2}} \otimes \mathcal{P}_{\U_{3}} \mathbf{V}_1 \bigg\|_{2,\infty} \\
        &= \| \mathbf{V}_1 \|_{2,\infty} \\
        &\leq \mu_0 \sqrt{\frac{r_1}{p_{-1}}}.
\end{align*}
Plugging all of these bounds in we obtain (with $c = 6$)
\begin{align*}
     \| \mathcal{P}_{\utilde_{2}^{(t_0-1,j-m)}}  \otimes \mathcal{P}_{ \utilde_{3}^{(t_0-1,j-m)}} \mathbf{V}_1 \|_{2,\infty} &\leq c \Bigg\{ \mu_0^2 \frac{\sqrt{r_2r_3}}{p_j} \bigg( \frac{\dl^{(2)}}{\lambda} + \frac{1}{2^{t_0-1}}\bigg)\bigg( \frac{\dl^{(3)}}{\lambda} + \frac{1}{2^{t_0-1}}\bigg); \\
     &\quad + \mu_0^2 \frac{\sqrt{r_2r_3}}{\sqrt{p_jp_3}} \bigg( \frac{\dl^{(2)}}{\lambda} + \frac{1}{2^{t_0-1}} \bigg) + \mu_0^2 \frac{\sqrt{r_2 r_3}}{\sqrt{p_j p_2}} \bigg( \frac{\dl^{(3)}}{\lambda} + \frac{1}{2^{t_0-1}} \bigg)\\
     &\quad +  \mu_0^2 \frac{\sqrt{r_2r_3}}{\sqrt{p_2p_3}} \bigg( \frac{\dl^{(2)}}{\lambda} + \frac{1}{2^{t_0-1}} \bigg) + \mu_0^2 \frac{\sqrt{r_2r_3}}{\sqrt{p_2p_3}} \bigg( \frac{\dl^{(3)}}{\lambda} + \frac{1}{2^{t_0-1}} \bigg) \\
     &\quad + \mu_0 \sqrt{\frac{r_1}{p_{-1}}} \Bigg\}.
\end{align*}
This shows that the first part of the event in $ \mathcal{\tilde E}_{j-m}^{t_0,1}$ must hold.  For the second part of the event, we note that
\begin{align*}
    \| \mathcal{P}_{\utilde_{2}^{(t_0-1,j-m)} \otimes \utilde_{3}^{(t_0-1,j-m)}} \|_{2,\infty} &\leq \|  \bigg( \mathcal{P}_{\utilde_{2}^{(t_0-1,j-m)}} - \mathcal{P}_{\uhat_{2}^{(t_0-1)}} \bigg) \otimes \mathcal{P}_{\utilde_{3}^{(t_0-1,j-m)}} \|_{2,\infty} \\
    &\qquad + \| \mathcal{P}_{\uhat_{2}^{(t_0-1)}}  \otimes \mathcal{P}_{\utilde_{3}^{(t_0-1,j-m)}} \|_{2,\infty} \\
    &\leq  \|  \bigg( \mathcal{P}_{\utilde_{2}^{(t_0-1,j-m)}} - \mathcal{P}_{\uhat_{2}^{(t_0-1)}} \bigg) \otimes \bigg(  \mathcal{P}_{\utilde_{3}^{(t_0-1,j-m)}} - \mathcal{P}_{\uhat_{3}^{(t_0-1)}} \bigg) \|_{2,\infty} \\
    &\qquad + \|  \bigg( \mathcal{P}_{\utilde_{2}^{(t_0-1,j-m)}} - \mathcal{P}_{\uhat_{2}^{(t_0-1)}} \bigg) \otimes  \mathcal{P}_{\uhat_{3}^{(t_0-1)}} \|_{2,\infty} \\
    &\qquad + \| \mathcal{P}_{\uhat_{2}^{(t_0-1)}}  \otimes \bigg( \mathcal{P}_{\utilde_{3}^{(t_0-1,j-m)}} - \mathcal{P}_{\uhat_{3}^{(t_0-1)}}\bigg)  \|_{2,\infty} \\
    &\qquad + \| \mathcal{P}_{\uhat_{2}^{(t_0-1)}}  \otimes \mathcal{P}_{\uhat_{3}^{(t_0-1)}} \|_{2,\infty} \\
    &\leq 4\mu_0^2 \frac{\sqrt{r_2 r_3}}{p_j}  \bigg( \frac{\dl^{(2)}}{\lambda} + \frac{1}{2^{t_0-1}} \bigg)  \bigg( \frac{\dl^{(3)}}{\lambda} + \frac{1}{2^{t_0-1}} \bigg) \\
    &\quad +  4\mu_0^2 \frac{\sqrt{r_2 r_3}}{\sqrt{p_jp_3}} \bigg( \frac{\dl^{(2)}}{\lambda} + \frac{1}{2^{t_0-1}} \bigg) + 4\mu_0^2 \frac{\sqrt{r_2r_3}}{\sqrt{p_jp_2}} \bigg( \frac{\dl^{(3)}}{\lambda} + \frac{1}{2^{t_0-1}} \bigg) \\
    &\quad + 4\mu_0^2 \frac{\sqrt{r_2 r_3}}{\sqrt{p_2p_3}}.
\end{align*}
where we used the fact that $\|\uhat_{3}^{(t_0-1)} \|_{2,\infty} \leq 2 \mu_0 \sqrt{\frac{r_3}{p_3}}$ on the events in question, and similarly for $\uhat_2^{(t_0-1)}$.  This  shows the second part of the event must hold, which completes the proof.
\end{proof}

%With Lemma \ref{lem:eventintersection} in hand, we also have the following lemma.
%The following lemma quantifies the proximity of the leave-one-out sequence on appropriately good events.

\begin{lemma}[Proximity of the Leave-one-out Sequence on a good event] \label{lem:leaveoneout_goodevent}
Let $1 \leq j \leq 3$ and $1 \leq m \leq p_j$ be fixed.  Then
\begin{align*}
    \p\Bigg\{  \bigg\{ \|  \sin\Theta(\uhat_{k}^{t_0}, \utilde_k^{t_0,j-m})\| \geq \bigg(\frac{\dl\ku}{\lambda} + \frac{1}{2^{t_0}}\bigg) \mu_0 \sqrt{\frac{r_k}{p_j}} \bigg\} \bigcap \mathcal{E}_{\mathrm{Good}} \bigcap \mathcal{E}_{\mathrm{main}}^{t_0-1,k} \Bigg\} &\leq p^{-29}.
\end{align*}
\end{lemma}

\begin{proof}
On the event $\mathcal{E}_{\mathrm{Good}}$ it holds that $\tau_k \leq C \sqrt{pr} \ll \lambda$ by assumption, since $r \leq C p_{\min}^{1/2}$ and $\lambda \gtrsim \kappa p/p_{\min}^{1/4} \sqrt{\log(p)}$.  Therefore, the eigengap assumption in Lemma \ref{lem:leaveoneoutsintheta} is met, so  on the event $\mathcal{E}_{\mathrm{Good}}$ it holds that
\begin{align*}
    \| \sin\Theta(\uhat_{k}^{(t)}, \utilde_k^{t,j-m}) \| 
 &\leq \frac{16 \kappa}{\lambda} \tau_k \bigg( \eta_{k}^{(t_0,j-m)} \bigg) + \frac{16 \kappa}{\lambda} \xi_k^{t_0,j-m} \bigg( \eta_{k}^{(t_0,j-m)} \bigg) \\
   &\qquad + \frac{8 \kappa}{\lambda} \tilde\xi_k^{t_0,j-m} + \frac{16}{\lambda^2} \tau_k^2 \bigg( \eta_{k}^{(t_0,j-m)}  \bigg) + \frac{8}{\lambda^2} \tau_k \xi_k^{t_0,j-m} + \frac{4}{\lambda^2} ( \xi_k^{t_0,j-m})^2,
\end{align*}
%Since $\tau_k \leq C_1 \sqrt{pr}$ % \begin{align*}
%     \p\bigg( \|  \sin\Theta(\uhat_{k}^{(t)},& \utilde_k^{t,j-m})\| \geq \bigg(\frac{\dl}{\lambda} + \frac{1}{2^{t_0}}\bigg) \mu_0 \sqrt{\frac{r}{p}} \bigg) \\
%     &\leq   \p\bigg( \|  \sin\Theta(\uhat_{k}^{(t)}, \utilde_k^{t,j-m})\| \geq \bigg(\frac{\dl}{\lambda} + \frac{1}{2^{t_0}}\bigg) \mu_0 \sqrt{\frac{r}{p}}  \bigcap \mathcal{E}_{\mathrm{Good}}\bigg) + O(p^{-30}) \\
%     &\leq \p\bigg\{ \|  \sin\Theta(\uhat_{k}^{(t)}, \utilde_k^{t,j-m})\| \geq \bigg(\frac{\dl}{\lambda} + \frac{1}{2^{t_0}}\bigg) \mu_0 \sqrt{\frac{r}{p}}  \bigcap \mathcal{E}_{\mathrm{Good}} \bigcap \mathcal{E}_{2,\infty}^{t_0} \bigcap \mathcal{E}_{j-m}^{t_0} \bigg\} + O(p^{-30}) + t_0 p^{-20}. % \\
%     \end{align*}
where we recall the notation
\begin{align*}
    \eta_{k}^{(t, j-m)}&\coloneqq \begin{cases}\left\|\sin \Theta\left(\tilde{\mathbf{U}}_{k+1}^{t-1, j-m}, \hat{\mathbf{U}}_{k+1}^{(t-1)}\right)\right\|+\left\|\sin \Theta\left(\tilde{\mathbf{U}}_{k+2}^{t-1, j-m}, \hat{\mathbf{U}}_{k+2}^{(t-1)}\right)\right\| & k=1; \\ \left\|\sin \Theta\left(\tilde{\mathbf{U}}_{k+1}^{t-1, j-m}, \hat{\mathbf{U}}_{k+1}^{(t-1)}\right)\right\|+\left\|\sin \Theta\left(\tilde{\mathbf{U}}_{k+2}^{t, j-m}, \hat{\mathbf{U}}_{k+2}^{(t)}\right)\right\| & k=2 \\ \left\|\sin \Theta\left(\tilde{\mathbf{U}}_{k+1}^{t, j-m}, \hat{\mathbf{U}}_{k+1}^{(t)}\right)\right\|+\left\|\sin \Theta\left(\tilde{\mathbf{U}}_{k+2}^{t, j-m}, \hat{\mathbf{U}}_{k+2}^{(t)}\right)\right\| & k=3;\end{cases}; \\
    \xi_{k}^{(t, j-m)}&\coloneqq\left\|\left(\mathbf{Z}_{k}^{j-m}-\mathbf{Z}_{k}\right) \tilde{\mathcal{P}}_{k}^{t, j-m}\right\|; \\
    \tilde{\xi}_{k}^{(t, j-m)}&\coloneqq\left\|\left(\mathbf{Z}_{k}^{j-m}-\mathbf{Z}_{k}\right) \tilde{\mathcal{P}}_{k}^{t, j-m} \mathbf{V}_{k}\right\|.
\end{align*}
Similar to Lemma \ref{lem:eventintersection} we now complete the proof for $k = 1$ without loss of generality (if $k = 2$ or $3$,  the proof is similar since slightly stronger bounds hold, but this again is not needed for our analysis).  On the event $\mathcal{E}_{\mathrm{Good}} \cap \mathcal{E}_{\mathrm{main}}^{t_0-1,1}$, we have the additional bounds 
 \begin{align*}
     \tau_1 &\leq C_1 \sqrt{pr}; \\
     \eta_{1}^{(t_0,j-m)} &\leq  \mu_0 \sqrt{\frac{r_2}{p_j}} \bigg( \frac{\dl^{(2)}}{\lambda} + \frac{1}{2^{t_0-1}} \bigg)+ \mu_0 \sqrt{\frac{r_3}{p_j}} \bigg( \frac{\dl^{(3)}}{\lambda} + \frac{1}{2^{t_0-1}} \bigg).
 \end{align*}
 Plugging this in to the deterministic bound for $\|\sin\Theta(\uhat_k^{(t)}, \utilde_k^{t,j-m})\|$ above yields 
 \begin{align*}
    \| &\sin\Theta(\uhat_{1}^{(t)}, \utilde_1^{t,j-m}) \|\\ 
 &\leq \frac{16 C_1 \kappa}{\lambda} \sqrt{pr} \bigg(\mu_0 \sqrt{\frac{r_2}{p_j}} \bigg( \frac{\dl^{(2)}}{\lambda} + \frac{1}{2^{t_0-1}} \bigg)+ \mu_0 \sqrt{\frac{r_3}{p_j}} \bigg( \frac{\dl^{(3)}}{\lambda} + \frac{1}{2^{t_0-1}} \bigg) \bigg) \\
 &\quad + \frac{16 \kappa}{\lambda} \xi_k^{t_0,j-m} \bigg( \mu_0 \sqrt{\frac{r_2}{p_j}} \bigg( \frac{\dl^{(2)}}{\lambda} + \frac{1}{2^{t_0-1}} \bigg)+ \mu_0 \sqrt{\frac{r_3}{p_j}} \bigg( \frac{\dl^{(3)}}{\lambda} + \frac{1}{2^{t_0-1}} \bigg) \bigg) \\
   &\qquad + \frac{8 \kappa}{\lambda} \tilde\xi_k^{t_0,j-m} + \frac{16 C_1^2 pr }{\lambda^2}  \bigg( \mu_0 \sqrt{\frac{r_2}{p_j}} \bigg( \frac{\dl^{(2)}}{\lambda} + \frac{1}{2^{t_0-1}} \bigg)+ \mu_0 \sqrt{\frac{r_3}{p_j}} \bigg( \frac{\dl^{(3)}}{\lambda} + \frac{1}{2^{t_0-1}} \bigg)  \bigg) \\
   &\quad + \frac{8C_1 \sqrt{pr}}{\lambda^2}  \xi_k^{t_0,j-m} +\frac{4}{\lambda^2} ( \xi_k^{t_0,j-m})^2.
 \end{align*}
Observe that 
 \begin{align*}
     \frac{16 C_1 \kappa}{\lambda} &\sqrt{pr} \bigg(\mu_0 \sqrt{\frac{r_2}{p_j}} \bigg( \frac{\dl^{(2)}}{\lambda} + \frac{1}{2^{t_0-1}} \bigg)+ \mu_0 \sqrt{\frac{r_3}{p_j}} \bigg( \frac{\dl^{(3)}}{\lambda} + \frac{1}{2^{t_0-1}} \bigg) \bigg) \\
     &= \mu_0 \sqrt{\frac{r_2}{p_j}}\frac{16 C_1 \kappa \sqrt{pr}}{\lambda}  \frac{C_0 \kappa \sqrt{p_2 \log(p)}}{\lambda} + \mu_0 \sqrt{\frac{r_3}{p_j}}\frac{16 C_1 \kappa \sqrt{pr}}{\lambda} \frac{C_0 \kappa \sqrt{p_3 \log(p)}}{\lambda} \\
     &\quad + \frac{16C_1 \kappa \sqrt{pr}}{\lambda} \bigg(\mu_0 \sqrt{\frac{r_3}{p_j}} + \mu_0 \sqrt{\frac{r_2}{p_j}}\bigg) \frac{1}{2^{t_0-1}} \\
     &= \mu_0 \sqrt{\frac{r_1}{p_j}} \frac{C_0 \kappa \sqrt{p_1 \log(p)}}{\lambda} \bigg( \frac{16 C_1\kappa \sqrt{pr} \sqrt{\frac{p_2}{p_1}} \sqrt{\frac{r_2}{r_1}}}{\lambda } + \frac{16C_1 \kappa \sqrt{pr} \sqrt{\frac{p_3}{p_1}} \sqrt{\frac{r_3}{r_1}}}{\lambda} \bigg) \\
     &\quad + \mu_0 \sqrt{\frac{r_1}{p_j}} \frac{1}{2^{t_0-1}} \bigg( \frac{16 C_1 \kappa \sqrt{pr} \sqrt{\frac{r_3}{r_1}}}{\lambda} + \frac{16 C_1 \kappa \sqrt{pr} \sqrt{\frac{r_2}{r_1}}}{\lambda} \bigg) \\
     &\leq \mu_0 \sqrt{\frac{r_1}{p_j}} \frac{\dl^{(1)}}{\lambda} \bigg( \frac{32 C_1  \kappa \sqrt{pr} \sqrt{\frac{p}{p_{\min}}}}{\lambda} \max\bigg\{\sqrt{\frac{r_3}{r_1}},\sqrt{\frac{r_2}{r_1}} \bigg\} \bigg) \\%  \sqrt{\frac{r_2}{r_1}}}{\lambda } + \frac{16C_1 \kappa \sqrt{pr} \sqrt{\frac{p_3}{p_1}} }{\lambda} \bigg) \\
     &\quad + \mu_0 \sqrt{\frac{r_1}{p_j}} \frac{1}{2^{t_0-1}} \bigg( \frac{32 C_1 \kappa \sqrt{pr} }{\lambda}\max\bigg\{\sqrt{\frac{r_3}{r_1}},\sqrt{\frac{r_2}{r_1}} \bigg\}  \bigg) \\
     &\leq \mu_0 \sqrt{\frac{r_1}{p_j}} \frac{\dl^{(1)}}{\lambda} \bigg( \frac{32 C_1 C_2  \kappa p/p_{\min}^{1/4}}{\lambda} \bigg\} \bigg) %  \sqrt{\frac{r_2}{r_1}}}{\lambda } + \frac{16C_1 \kappa \sqrt{pr} \sqrt{\frac{p_3}{p_1}} }{\lambda} \bigg) \\
      + \mu_0 \sqrt{\frac{r_1}{p_j}} \frac{1}{2^{t_0-1}} \bigg( \frac{32 C_1 C_2 \kappa \sqrt{pr} }{\lambda}\bigg) \\
      &\leq \frac{1}{8} \mu_0 \sqrt{\frac{r_1}{p_j}} \bigg( \frac{\dl^{(1)}}{\lambda} + \frac{1}{2^{t_0}}\bigg),
 \end{align*}
 which holds under the assumption $\lambda \gtrsim \kappa \sqrt{\log(p)} p/p_{\min}^{1/4}$, $r_k/r_j \leq C$, and $\mu_0^2 r \lesssim p_{\min}^{1/2}$.  By a similar argument,
 \begin{align*}
     \frac{16 C_1^2 pr }{\lambda^2}  \bigg( \mu_0 \sqrt{\frac{r_2}{p_j}} \bigg( \frac{\dl^{(2)}}{\lambda} + \frac{1}{2^{t_0-1}} \bigg) &\leq \frac{1}{8} \mu_0 \sqrt{\frac{r_1}{p_j}} \bigg( \frac{\dl^{(1)}}{\lambda} + \frac{1}{2^{t_0}}\bigg).
 \end{align*}
 Therefore,
 \begin{align*}
       \| \sin\Theta(\uhat_{1}^{(t)}, \utilde_1^{t,j-m}) \| 
 &\leq \frac{1}{8} \mu_0 \sqrt{\frac{r_1}{p_j}} \bigg( \frac{\dl^{(1)}}{\lambda} + \frac{1}{2^{t_0}}\bigg) + \frac{8 \kappa}{\lambda} \tilde\xi_1^{t_0,j-m} \\
 &\quad + \xi_1^{t_0,j-m} \frac{16 \kappa}{\lambda}  \bigg( \mu_0 \sqrt{\frac{r_2}{p_j}} \bigg( \frac{\dl^{(2)}}{\lambda} + \frac{1}{2^{t_0-1}} \bigg)+ \mu_0 \sqrt{\frac{r_3}{p_j}} \bigg( \frac{\dl^{(3)}}{\lambda} + \frac{1}{2^{t_0-1}} \bigg) \bigg) \\
   &\quad + \frac{8C_1 \sqrt{pr}}{\lambda^2}  \xi_1^{t_0,j-m} +\frac{4}{\lambda^2} ( \xi_1^{t_0,j-m})^2.
 \end{align*}
 The bound above depends only on $\xi_1^{t_0,j-m}$ and $\tilde \xi_1^{t_0,j-m}$. Define
 \begin{align*}
     (I) &\coloneqq  \xi_1^{t_0,j-m} \bigg\{ \frac{16 \kappa}{\lambda}  \bigg( \mu_0 \sqrt{\frac{r_2}{p_j}} \bigg( \frac{\dl^{(2)}}{\lambda} + \frac{1}{2^{t_0-1}} \bigg)+ \mu_0 \sqrt{\frac{r_3}{p_j}} \bigg( \frac{\dl^{(3)}}{\lambda} + \frac{1}{2^{t_0-1}} \bigg) \bigg) + \frac{8C_1 \sqrt{pr}}{\lambda^2}  \bigg\}; \\
     (II) &\coloneqq   \frac{4}{\lambda^2} ( \xi_1^{t_0,j-m})^2; \\
     (III)&\coloneqq \frac{8 \kappa}{\lambda} \tilde\xi_1^{t_0,j-m}.
 \end{align*}
 Then
 \begin{align*}
     \p\Bigg\{  \bigg\{ \| &\sin\Theta(\uhat_1^{t_0}, \utilde_1^{t_0,j-m}) \| \geq \bigg( \frac{\dl^{(1)}}{\lambda} + \frac{1}{2^{t_0}}\bigg) \mu_0 \sqrt{\frac{r_1}{p_j}} \bigg\} \bigcap \mathcal{E}_{\mathrm{Good}} \bigcap \mathcal{E}_{\mathrm{main}}^{t_0-1,1} \Bigg\} \\
     &\leq \p\Bigg\{ \bigg\{ (I) \geq \frac{1}{4} \bigg( \frac{\dl^{(1)}}{\lambda} + \frac{1}{2^{t_0}}\bigg) \mu_0 \sqrt{\frac{r_1}{p_j}} \bigg\} \bigcap \mathcal{E}_{\mathrm{Good}} \bigcap \mathcal{E}_{\mathrm{main}}^{t_0-1,1} \Bigg\} \\
     &\quad + \p\Bigg\{ \bigg\{ (II) \geq \frac{1}{4} \bigg( \frac{\dl^{(1)}}{\lambda} + \frac{1}{2^{t_0}}\bigg) \mu_0 \sqrt{\frac{r_1}{p_j}} \bigg\} \bigcap \mathcal{E}_{\mathrm{Good}} \bigcap \mathcal{E}_{\mathrm{main}}^{t_0-1,1} \Bigg\} \\
     &\quad + \p\Bigg\{ \bigg\{ (III) \geq \frac{1}{4} \bigg( \frac{\dl^{(1)}}{\lambda} + \frac{1}{2^{t_0}}\bigg) \mu_0 \sqrt{\frac{r_1}{p_j}} \bigg\} \bigcap \mathcal{E}_{\mathrm{Good}} \bigcap \mathcal{E}_{\mathrm{main}}^{t_0-1,1} \Bigg\}.
 \end{align*}

 We now will derive probabilistic bounds for each of the terms above on the event $\mathcal{E}_{\mathrm{Good}} \cap \mathcal{E}_{\mathrm{main}}^{t_0-1,1}$.  We will consider each term separately, though the strategy for each will remain the same: since there is nontrivial dependence between the events above and the random variable $\xi_1^{j-m}$, we use the auxiliary event $\mathcal{\tilde E}_{j-m}^{t_0,1}$, which is independent of the nonzero entries in the random matrix $\mathbf{Z}_1^{j-m} - \mathbf{Z}_1$.  We then use Lemma \ref{lem:eventintersection} to show that the intersection of this event with other events is empty.
 \\ \ \\ \noindent
\textbf{The term} $(I)$:
We note that 
\begin{align*}
\p \Bigg\{ &\bigg\{ (I) \geq \frac{1}{4} \bigg( \frac{\dl^{(1)}}{\lambda} + \frac{1}{2^{t_0}} \bigg) \mu_0\sqrt{\frac{r_1}{p_j
    }} \bigg\} \bigcap \mathcal{E}_{\mathrm{Good}} \bigcap \mathcal{E}_{\mathrm{main}}^{t_0-1,1} \Bigg\} \\
&\leq \p \Bigg\{ \bigg\{ (I) \geq \frac{1}{4} \bigg( \frac{\dl^{(1)}}{\lambda} + \frac{1}{2^{t_0}} \bigg) \mu_0\sqrt{\frac{r_1}{p_j
    }}  \bigg\} \bigcap \mathcal{E}_{\mathrm{Good}} \bigcap \mathcal{E}_{\mathrm{main}}^{t_0-1,1} \bigcap \mathcal{\tilde E}_{j-m}^{t_0,1} \Bigg\} \\
    &\qquad + \p \Bigg\{ \mathcal{E}_{\mathrm{Good}} \bigcap \mathcal{E}_{\mathrm{main}}^{t_0-1,1} \bigcap \big(\mathcal{\tilde E }_{j-m}^{t_0,1}\big)^c \Bigg\} \\
    &\leq \p \Bigg\{ \bigg\{ (I) \geq \frac{1}{4} \bigg( \frac{\dl^{(1)}}{\lambda} + \frac{1}{2^{t_0}} \bigg) \mu_0\sqrt{\frac{r_1}{p_j
    }}  \bigg\} \bigcap  \mathcal{\tilde E}_{j-m}^{t_0,1} \Bigg\}, % \\
%&\qquad \bigcap \mathcal{E}_{\mathrm{Good}} \bigcap \mathcal{E}_{2,\infty}^{t_0} \bigcap \mathcal{E}_{j-m}^{t_0} \bigcap \mathcal{\tilde E}_{-m}^{t_0,k} \bigg\} \\
%&\qquad + \p\bigg\{ \mathcal{E}_{\mathrm{Good}} \bigcap \mathcal{E}_{2,\infty}^{t_0} \bigcap \mathcal{E}_{j-m}^{t_0} \bigcap \big(\mathcal{\tilde E}_{-m}^{t_0,k}\big)^c \bigg\} \\
%&\leq \p\bigg\{ \xi_k^{j-m} \bigg[ \frac{16\kappa}{\lambda} \mu_0 \sqrt{\frac{r}{p}} \left(  \frac{\dl}{\lambda} + \frac{1}{2^{t_0-1}} \right) + \frac{8 C_1 \sqrt{pr}}{\lambda^2} \bigg] \geq \frac{1}{4} \mu_0 \sqrt{\frac{r}{p}} \bigg( \frac{\dl}{\lambda} + \frac{1}{2^{t_0}} \bigg) \bigcap \mathcal{\tilde E}_{-m}^{t_0,k} \bigg\},
\end{align*}
where we have used Lemma \ref{lem:eventintersection} to show that the intesection of the complement $\mathcal{\tilde E}_{j-m}^{t_0,1}$ with the other events is zero. 

Now we simply observe that $\mathcal{\tilde E}_{j-m}^{t_0,k}$ does not depend on any of the random variables in the matrix $\mathbf{Z}_k^{j-m} - \mathbf{Z}_k$, so we are free to condition on this event.  Recall that 
\begin{align*}
    \xi_k^{t_0,j-m} &= \bigg\| \bigg(\mathbf{Z}_k - \mathbf{Z}_k^{j-m} \bigg) \mathcal{\tilde P}_{1}^{t_0,j-m} \bigg\|.
    \end{align*}
By Lemma \ref{lem:matricizationrowbound}, it holds that
\begin{align*}
    \xi_1^{t_0,j-m} &\leq C \sqrt{p_{-j}\log(p)} \bigg\| \mathcal{\tilde P}_{1}^{t_0,j-m} \bigg\|_{2,\infty}
\end{align*}
with probability at least $1 - O(p^{-30})$.  On the event $\mathcal{\tilde E}_{j-m}^{t_0,k}$ we have that
\begin{align*}
   \bigg\| \mathcal{\tilde P}_{1}^{t_0,j-m} \bigg\|_{2,\infty} &\leq c\Bigg\{ \mu_0^2 \frac{\sqrt{r_2 r_3}}{p_j}  \bigg( \frac{\dl^{(2)}}{\lambda} + \frac{1}{2^{t_0-1}} \bigg)  \bigg( \frac{\dl^{(3)}}{\lambda} + \frac{1}{2^{t_0-1}} \bigg) \\
    &\quad +  \mu_0^2 \frac{\sqrt{r_2 r_3}}{\sqrt{p_jp_3}} \bigg( \frac{\dl^{(2)}}{\lambda} + \frac{1}{2^{t_0-1}} \bigg) + \mu_0^2 \frac{\sqrt{r_2r_3}}{\sqrt{p_jp_2}} \bigg( \frac{\dl^{(3)}}{\lambda} + \frac{1}{2^{t_0-1}} \bigg) \\
    &\quad + \mu_0^2 \frac{\sqrt{r_2 r_3}}{\sqrt{p_2p_3}} \bigg\}.
\end{align*}
Therefore,
\begin{align*}
    &\xi_1^{t_0,j-m} \\
    &\leq C' \sqrt{p_{-j} \log(p)} \bigg( \frac{\mu_0^2 \sqrt{r_2r_3}}{p_j} \frac{\dl^{(2)}}{\lambda} + \frac{\mu_0^2 \sqrt{r_2r_3}}{p_j} \frac{\dl^{(3)}}{\lambda} + 2 \mu_0^2 \frac{\sqrt{r_2 r_3}}{\sqrt{p_jp_3}}  \frac{\dl^{(2)}}{\lambda} +   \mu_0^2 \frac{\sqrt{r_2 r_3}}{\sqrt{p_jp_2}}  \frac{\dl^{(3)}}{\lambda} +  \mu_0^2 \frac{\sqrt{r_2 r_3}}{\sqrt{p_2p_3}}\bigg) \\
    &\quad + C' \sqrt{p_{-j} \log(p)} \bigg( \mu_0^2 \frac{\sqrt{r_2 r_3}}{p_j} \frac{1}{2^{t_0-1}} \frac{1}{2^{t_0-1}} +  \mu_0^2 \frac{\sqrt{r_2 r_3}}{\sqrt{p_jp_3}} \frac{1}{2^{t_0-1}} + \mu_0^2 \frac{\sqrt{r_2r_3}}{\sqrt{p_j p_2}} \frac{1}{2^{t_0-1}} \bigg) \\
    &\leq C'' \mu_0 \sqrt{\frac{r_1}{p_j}} \sqrt{p_1 \log(p)}   \sqrt{p_{-1}}\\
    &\quad \times \Bigg\{ \frac{\mu_0 \sqrt{r_2 r_3}}{p_j \sqrt{r_1}} \bigg(\frac{\dl^{(2)}}{\lambda} + \frac{\dl^{(3)}}{\lambda} \bigg) +  \mu_0 \frac{\sqrt{r_2 r_3}}{\sqrt{p_jp_{3} r_1}} \frac{\dl^{(2)}}{\lambda} +  \mu_0 \frac{\sqrt{r_2 r_3}}{\sqrt{p_jp_{2} r_1}} \frac{\dl^{(3)}}{\lambda} + \mu_0 \frac{\sqrt{r_2r_3}}{\sqrt{p_2p_3}} \Bigg\} \\
    &\quad +  C''\mu_0 \sqrt{\frac{r_1}{p_j}} \sqrt{p_1 \log(p)}   \sqrt{p_{-1}}  \Bigg\{ \frac{\mu_0 \sqrt{r_2 r_3}}{p_j \sqrt{r_1}} \frac{1}{2^{t_0-1}} + 2 \mu_0 \frac{\sqrt{r_2 r_3}}{\sqrt{p_jp_{3} r_1}} \frac{1}{2^{t_0-1}} +  \mu_0 \frac{\sqrt{r_2 r_3}}{\sqrt{p_jp_{2} r_1}} \frac{1}{2^{t_0-1}}\Bigg\},
\end{align*}
where we have absorbed the constants in each term.  Therefore, with probability at least $1 - O(p^{-30})$ it holds that
\begin{align*}
    (I) &\leq  C'' \mu_0 \sqrt{\frac{r_1}{p_j}} \sqrt{p_1 \log(p)}   \sqrt{p_{-1}} \\
    &\quad \times \Bigg\{ \frac{\mu_0 \sqrt{r_2 r_3}}{p_j \sqrt{r_1}} \bigg(\frac{\dl^{(2)}}{\lambda} + \frac{\dl^{(3)}}{\lambda} \bigg) +  \mu_0 \frac{\sqrt{r_2 r_3}}{\sqrt{p_jp_{3} r_1}} \frac{\dl^{(2)}}{\lambda} +  \mu_0 \frac{\sqrt{r_2 r_3}}{\sqrt{p_jp_{2} r_1}} \frac{\dl^{(3)}}{\lambda} + \mu_0 \frac{\sqrt{r_2r_3}}{\sqrt{p_2p_3}} \Bigg\} \\
    &\quad \quad\times \bigg\{ \frac{16 \kappa}{\lambda}  \bigg( \mu_0 \sqrt{\frac{r_2}{p_j}} \bigg( \frac{\dl^{(2)}}{\lambda} + \frac{1}{2^{t_0-1}} \bigg)+ \mu_0 \sqrt{\frac{r_3}{p_j}} \bigg( \frac{\dl^{(3)}}{\lambda} + \frac{1}{2^{t_0-1}} \bigg) \bigg) + \frac{8C_1 \sqrt{pr}}{\lambda^2}  \bigg\} \\
    &\quad + C''\mu_0 \sqrt{\frac{r_1}{p_j}} \sqrt{p_1 \log(p)}   \sqrt{p_{-1}}  \Bigg\{ \frac{\mu_0 \sqrt{r_2 r_3}}{p_j \sqrt{r_1}} \frac{1}{2^{t_0-1}} + 2 \mu_0 \frac{\sqrt{r_2 r_3}}{\sqrt{p_jp_{3} r_1}} \frac{1}{2^{t_0-1}} +  \mu_0 \frac{\sqrt{r_2 r_3}}{\sqrt{p_jp_{2} r_1}} \frac{1}{2^{t_0-1}}\Bigg\} \\
    &\quad \quad \times \bigg\{ \frac{16 \kappa}{\lambda}  \bigg( \mu_0 \sqrt{\frac{r_2}{p_j}} \bigg( \frac{\dl^{(2)}}{\lambda} + \frac{1}{2^{t_0-1}} \bigg)+ \mu_0 \sqrt{\frac{r_3}{p_j}} \bigg( \frac{\dl^{(3)}}{\lambda} + \frac{1}{2^{t_0-1}} \bigg) \bigg) + \frac{8C_1 \sqrt{pr}}{\lambda^2}  \bigg\}.
\end{align*}
We now show the first term is less than $\frac{1}{8} \frac{\dl^{(1)}}{\lambda} \mu_0 \sqrt{\frac{r_1}{p_j}}$ and the second term is less than $\frac{1}{8} \frac{1}{2^{t_0}} \mu_0 \sqrt{\frac{r_1}{p_j}}$.  The first term will be less than this provided that 
\begin{align*}
    \frac{8}{C_0 C''} \sqrt{p_{-1}} &\Bigg\{ \frac{\mu_0 \sqrt{r_2 r_3}}{p_j \sqrt{r_1}} \bigg(\frac{\dl^{(2)}}{\lambda} + \frac{\dl^{(3)}}{\lambda} \bigg) +  \mu_0 \frac{\sqrt{r_2 r_3}}{\sqrt{p_jp_{3} r_1}} \frac{\dl^{(2)}}{\lambda} +  \mu_0 \frac{\sqrt{r_2 r_3}}{\sqrt{p_jp_{2} r_1}} \frac{\dl^{(3)}}{\lambda} + \mu_0 \frac{\sqrt{r_2r_3}}{\sqrt{p_2p_3}} \Bigg\} \\
    &\quad \quad\times \bigg\{ 16 \bigg( \mu_0 \sqrt{\frac{r_2}{p_j}} \bigg( \frac{\dl^{(2)}}{\lambda} + \frac{1}{2^{t_0-1}} \bigg)+ \mu_0 \sqrt{\frac{r_3}{p_j}} \bigg( \frac{\dl^{(3)}}{\lambda} + \frac{1}{2^{t_0-1}} \bigg) \bigg) + \frac{8C_1 \sqrt{pr}}{\lambda \kappa }  \bigg\}
\end{align*}
is less than one.  This follows from basic algebra and the assumptions $\lambda \gtrsim \kappa  \sqrt{\log(p)} p/p_{\min}^{1/4} $, that $\mu_0^2 r \lesssim p_{\min}^{1/2}$, and that $r_k \asymp r$.  A similar argument shows that the second term is smaller than $\frac{1}{8} \frac{1}{2^{t_0}} \mu_0 \sqrt{\frac{r_1}{p_j}}$. Therefore, on the event $\mathcal{\tilde E}_{j-m}^{t_0,1}$, with probability at least $1 - O(p^{-30})$ it holds that
\begin{align*}
    (I) &\leq \frac{1}{8} \bigg( \frac{\dl^{(1)}}{\lambda} + \frac{1}{2^{t_0}} \bigg) \mu_0 \sqrt{\frac{r_1}{p_j}}.
\end{align*}
% where the final inequality holds on the event $\mathcal{\tilde E}_{j-m}^{t_0,1}$ (where we have absorbed the constant in the event).  Consequently, with probability at least $1 - O(p^{-30})$, it holds that
% \begin{align*}
%  (I) &\coloneqq \xi_1^{t_0,j-m} \bigg[ \frac{32 \kappa}{\lambda} \mu_0 \sqrt{\frac{r}{p}} \bigg( \frac{\dl}{\lambda} + \frac{1}{2^{t_0-1}} \bigg) + \frac{8 C_1 \sqrt{pr}}{\lambda^2} \bigg] \\
%  &\leq C p \sqrt{\log(p)} \frac{\mu_0^2 r}{p} \bigg[ \frac{32 \kappa}{\lambda} \mu_0 \sqrt{\frac{r}{p}} \bigg( \frac{\dl}{\lambda} + \frac{1}{2^{t_0-1}} \bigg) + \frac{8 C_1 \sqrt{pr}}{\lambda^2} \bigg] \\
% & \leq \frac{ 32 C \kappa \mu_0^2  r \sqrt{\log(p)}}{\lambda} \mu_0 \sqrt{\frac{r}{p}} \frac{\dl}{\lambda} + \frac{ 32 C \kappa \mu_0^2  r \sqrt{\log(p)}}{\lambda} \mu_0 \sqrt{\frac{r}{p}} \frac{1}{2^{t_0-1}} + \frac{8 C C_1 \mu_0^2 r \sqrt{pr}}{\lambda^2} \\
% &\leq \frac{1}{8} \frac{\dl}{\lambda} \mu_0 \sqrt{\frac{r}{p}} + \frac{1}{8} \frac{1}{2^{t_0-1}} \mu_0 \sqrt{\frac{r}{p}} + \frac{1}{8} \frac{\dl}{\lambda} \mu_0 \sqrt{\frac{r}{p}} \\
% &\leq \frac{1}{4} \bigg( \frac{\dl}{\lambda} + \frac{1}{2^{t_0}} \bigg) \mu_0 \sqrt{\frac{r}{p}},
% \end{align*}
% where the final line follows as long as $C_0$ in the definition of $\dl$ satisfies $C_0 \geq \max\{ 256 C, 64 C C_1\}$ and $\lambda \geq 32 C \kappa r \sqrt{\log(p)} \mu_0^2$ (note that $\lambda \gtrsim r \mu_0^2 \kappa \sqrt{\log(p)}$ is guaranteed when $\lambda \gtrsim \mu_0 \kappa \sqrt{rp\log(p)}$ since $\mu_0 \sqrt{r} \leq \sqrt{p}$ by definition).  
\noindent
\textbf{The term (II):}  By a similar argument, we note that
\begin{align*}
    \mathbb{P}&\Bigg\{\left\{(I I) \geq \frac{1}{4}\left(\frac{\delta_{\mathrm{L}}^{(1)}}{\lambda}+\frac{1}{2^{t_{0}}}\right) \mu_{0} \sqrt{\frac{r_k}{p_j}}\right\} \bigcap \mathcal{E}_{\text {Good }} \bigcap \mathcal{E}_{\text {main }}^{t_{0}-1,1}\Bigg\} \\
    &\leq \p \Bigg\{ \bigg\{ (II) \geq \frac{1}{4}\left(\frac{\delta_{\mathrm{L}}^{(1)}}{\lambda}+\frac{1}{2^{t_{0}}}\right) \mu_{0} \sqrt{\frac{r_k}{p_j}} \bigg\} \bigcap \mathcal{E}_{\text {Good }} \bigcap \mathcal{E}_{\text {main }}^{t_{0}-1,1} \bigcap \mathcal{\tilde E}_{j-m}^{t_0,1}\Bigg\} \\
    &\qquad + \p\bigg\{ \mathcal{E}_{\text {Good }} \bigcap \mathcal{E}_{\text {main }}^{t_{0}-1,1} \bigcap (\mathcal{\tilde E}_{j-m}^{t_0,1} )^c\bigg\} \\
    &\leq \p \Bigg\{ \bigg\{ (II) \geq \frac{1}{4}\left(\frac{\delta_{\mathrm{L}}^{(1)}}{\lambda}+\frac{1}{2^{t_{0}}}\right) \mu_{0} \sqrt{\frac{r_k}{p_j}} \bigg\} \bigcap \mathcal{\tilde E}_{j-m}^{t_0,1}\Bigg\},
\end{align*}
where again we used Lemma \ref{lem:eventintersection}. Conditioning on the event $\mathcal{\tilde E}_{j-m}^{t_0,1}$, by the same argument as in Term $(I)$,  with probability at least $1 - O(p^{-30})$ one has
\begin{align*}
    \xi_1^{j-m} &\leq  \tilde C \mu_0 \sqrt{\frac{r_1}{p_j}} \sqrt{p_1 \log(p)}   \sqrt{p_{-1}} \\
    &\quad \times  \Bigg\{ \frac{\mu_0 \sqrt{r_2 r_3}}{p_j \sqrt{r_1}} \bigg(\frac{\dl^{(2)}}{\lambda} + \frac{\dl^{(3)}}{\lambda} \bigg) +  \mu_0 \frac{\sqrt{r_2 r_3}}{\sqrt{p_jp_{3} r_1}} \frac{\dl^{(2)}}{\lambda} +  \mu_0 \frac{\sqrt{r_2 r_3}}{\sqrt{p_jp_{2} r_1}} \frac{\dl^{(3)}}{\lambda} + \mu_0 \frac{\sqrt{r_2r_3}}{\sqrt{p_2p_3}} \Bigg\} \\
    &\quad + \tilde C\mu_0 \sqrt{\frac{r_1}{p_j}} \sqrt{p_1 \log(p)}   \sqrt{p_{-1}}  \Bigg\{ \frac{\mu_0 \sqrt{r_2 r_3}}{p_j \sqrt{r_1}} \frac{1}{2^{t_0-1}} + 2 \mu_0 \frac{\sqrt{r_2 r_3}}{\sqrt{p_jp_{3} r_1}} \frac{1}{2^{t_0-1}} +  \mu_0 \frac{\sqrt{r_2 r_3}}{\sqrt{p_jp_{2} r_1}} \frac{1}{2^{t_0-1}}\Bigg\},
\end{align*}
where we have once again absorbed the constant.  Therefore, with probability at least $ 1 - O(p^{-30})$, 
\begin{align*}
(II) &= \frac{4}{\lambda^2} \big( \xi_1^{j-m}\big)^2 \\
&\leq \frac{4}{\lambda^2} \tilde C \mu_0 \sqrt{\frac{r_1}{p_j}} \sqrt{p_1 \log(p)}   \sqrt{p_{-1}} \\
&\quad \times \Bigg\{ \frac{\mu_0 \sqrt{r_2 r_3}}{p_j \sqrt{r_1}} \bigg(\frac{\dl^{(2)}}{\lambda} + \frac{\dl^{(3)}}{\lambda} \bigg) +  \mu_0 \frac{\sqrt{r_2 r_3}}{\sqrt{p_jp_{3} r_1}} \frac{\dl^{(2)}}{\lambda} +  \mu_0 \frac{\sqrt{r_2 r_3}}{\sqrt{p_jp_{2} r_1}} \frac{\dl^{(3)}}{\lambda} + \mu_0 \frac{\sqrt{r_2r_3}}{\sqrt{p_2p_3}} \Bigg\} \\
    &\quad + \frac{4}{\lambda^2}  \tilde C\mu_0 \sqrt{\frac{r_1}{p_j}} \sqrt{p_1 \log(p)}   \sqrt{p_{-1}}  \Bigg\{ \frac{\mu_0 \sqrt{r_2 r_3}}{p_j \sqrt{r_1}} \frac{1}{2^{t_0-1}} + 2 \mu_0 \frac{\sqrt{r_2 r_3}}{\sqrt{p_jp_{3} r_1}} \frac{1}{2^{t_0-1}} +  \mu_0 \frac{\sqrt{r_2 r_3}}{\sqrt{p_jp_{2} r_1}} \frac{1}{2^{t_0-1}}\Bigg\}, \\
    &= \frac{1}{8} \frac{\dl^{(1)}}{\lambda} \mu_0 \sqrt{\frac{r_1}{p_j}} \bigg( \frac{4 \tilde C \sqrt{p_{-1}}}{C_0 \lambda} \bigg) \\
    &\quad \times \Bigg\{ \frac{\mu_0 \sqrt{r_2 r_3}}{p_j \sqrt{r_1}} \bigg(\frac{\dl^{(2)}}{\lambda} + \frac{\dl^{(3)}}{\lambda} \bigg) +  \mu_0 \frac{\sqrt{r_2 r_3}}{\sqrt{p_jp_{3} r_1}} \frac{\dl^{(2)}}{\lambda} +  \mu_0 \frac{\sqrt{r_2 r_3}}{\sqrt{p_jp_{2} r_1}} \frac{\dl^{(3)}}{\lambda} + \mu_0 \frac{\sqrt{r_2r_3}}{\sqrt{p_2p_3}} \Bigg\} \\
    &\quad + \frac{1}{8} \frac{1}{2^{t_0}} \mu_0 \sqrt{\frac{r_1}{p_j}} \bigg( \frac{4 \tilde C \sqrt{p_{-1}}}{C_0 \lambda} \bigg) \\
    &\quad \times \Bigg\{ \frac{\mu_0 \sqrt{r_2 r_3}}{p_j \sqrt{r_1}} \bigg(\frac{\dl^{(2)}}{\lambda} + \frac{\dl^{(3)}}{\lambda} \bigg) +  \mu_0 \frac{\sqrt{r_2 r_3}}{\sqrt{p_jp_{3} r_1}} \frac{\dl^{(2)}}{\lambda} +  \mu_0 \frac{\sqrt{r_2 r_3}}{\sqrt{p_jp_{2} r_1}} \frac{\dl^{(3)}}{\lambda} + \mu_0 \frac{\sqrt{r_2r_3}}{\sqrt{p_2p_3}} \Bigg\} \\
&\leq \frac{1}{8}\bigg( \frac{\dl}{\lambda} + \frac{1}{2^{t_0}}\bigg) \mu_0 \sqrt{\frac{r_1}{p_j}},
\end{align*}
where the final inequality holds  when the additional terms are smaller than one, which holds via basic algebra  as long as $C_0 \geq 4 C \tilde C$, $\lambda \gtrsim \kappa \sqrt{\log(p)} p/p_{\min}^{1/4}$, $r_k \asymp r$ and $\mu_0^2 r \lesssim p_{\min}^{1/2}$.
\\ \ \\ 
\textbf{The term (III):} Proceeding similarly again,
  \begin{align*}
    \mathbb{P}&\Bigg\{\left\{(I II) \geq \frac{1}{4}\left(\frac{\delta_{\mathrm{L}}^{(1)}}{\lambda}+\frac{1}{2^{t_{0}}}\right) \mu_{0} \sqrt{\frac{r_1}{p_j}}\right\} \bigcap \mathcal{E}_{\text {Good }} \bigcap \mathcal{E}_{\text {main }}^{t_{0}-1,1}\Bigg\} \\
    &\leq \p \Bigg\{ \bigg\{ (III) \geq \frac{1}{4}\left(\frac{\delta_{\mathrm{L}}^{(1)}}{\lambda}+\frac{1}{2^{t_{0}}}\right) \mu_{0} \sqrt{\frac{r_1}{p_j}} \bigg\} \bigcap \mathcal{E}_{\text {Good }} \bigcap \mathcal{E}_{\text {main }}^{t_{0}-1,1} \bigcap \mathcal{\tilde E}_{j-m}^{t_0,1}\Bigg\} \\
    &\qquad + \p\bigg\{ \mathcal{E}_{\text {Good }} \bigcap \mathcal{E}_{\text {main }}^{t_{0}-1,1} \bigcap (\mathcal{\tilde E}_{j-m}^{t_0,1} )^c\bigg\} \\
    &\leq \p \Bigg\{ \bigg\{ (III) \geq \frac{1}{4}\left(\frac{\delta_{\mathrm{L}}^{(1)}}{\lambda}+\frac{1}{2^{t_{0}}}\right) \mu_{0} \sqrt{\frac{r_1}{p_j}} \bigg\} \bigcap \mathcal{\tilde E}_{j-m}^{t_0,1}\Bigg\},
\end{align*}
where again we used Lemma \ref{lem:eventintersection}. Conditioning on the event $\mathcal{\tilde E}_{j-m}^{t_0,1}$,  by Lemma \ref{lem:matricizationrowbound}, with probability at least $1 - O(p^{-30})$ one has
\begin{align*}
    \tilde\xi_1^{j-m} &= \bigg\| \bigg( \mathbf{Z}_1^{j-m} - \mathbf{Z}_1 \bigg) \mathcal{\tilde P}_{1}^{t_0,j-m} \mathbf{V}_1 \bigg\| \\
    &\leq C  \sqrt{p_{-j}\log(p)} \bigg\| \mathcal{\tilde P}_{1}^{t_0,j-m} \mathbf{V}_k \bigg\|_{2,\infty}. 
\end{align*}
On the event $\mathcal{\tilde E}_{j-m}^{t_0,1}$ it holds that
\begin{align*}
    \| \mathcal{\tilde P}_{k}^{t_0,j-m} \mathbf{V}_1 \|_{2,\infty} &\leq c \Bigg\{ \mu_0^2 \frac{\sqrt{r_2r_3}}{p_j} \bigg( \frac{\dl^{(2)}}{\lambda} + \frac{1}{2^{t_0-1}}\bigg)\bigg( \frac{\dl^{(3)}}{\lambda} + \frac{1}{2^{t_0-1}}\bigg); \\
     &\quad + \mu_0^2 \frac{\sqrt{r_2r_3}}{\sqrt{p_jp_3}} \bigg( \frac{\dl^{(2)}}{\lambda} + \frac{1}{2^{t_0-1}} \bigg) + \mu_0^2 \frac{\sqrt{r_2 r_3}}{\sqrt{p_j p_2}} \bigg( \frac{\dl^{(3)}}{\lambda} + \frac{1}{2^{t_0-1}} \bigg)\\
     &\quad +  \mu_0^2 \frac{\sqrt{r_2r_3}}{\sqrt{p_2p_3}} \bigg( \frac{\dl^{(2)}}{\lambda} + \frac{1}{2^{t_0-1}} \bigg) + \mu_0^2 \frac{\sqrt{r_2r_3}}{\sqrt{p_2p_3}} \bigg( \frac{\dl^{(3)}}{\lambda} + \frac{1}{2^{t_0-1}} \bigg) \\
     &\quad + \mu_0 \sqrt{\frac{r_1}{p_{-1}}} \Bigg\}.
\end{align*}
Therefore, with probability at least $1 - O(p^{-30})$,
\begin{align*}
    (III) &\coloneqq \frac{8 \kappa}{\lambda} \tilde \xi_1^{t_0,j-m} \\
    &\leq \frac{8 C \kappa}{\lambda}  \sqrt{p_{-j}\log(p)} \\
    &\quad \quad \times c \Bigg\{ \mu_0^2 \frac{\sqrt{r_2r_3}}{p_j} \bigg( \frac{\dl^{(2)}}{\lambda} + \frac{1}{2^{t_0-1}}\bigg)\bigg( \frac{\dl^{(3)}}{\lambda} + \frac{1}{2^{t_0-1}}\bigg) \\
     &\quad\quad\quad + \mu_0^2 \frac{\sqrt{r_2r_3}}{\sqrt{p_jp_3}} \bigg( \frac{\dl^{(2)}}{\lambda} + \frac{1}{2^{t_0-1}} \bigg) + \mu_0^2 \frac{\sqrt{r_2 r_3}}{\sqrt{p_j p_2}} \bigg( \frac{\dl^{(3)}}{\lambda} + \frac{1}{2^{t_0-1}} \bigg)\\
     &\quad\quad\quad +  \mu_0^2 \frac{\sqrt{r_2r_3}}{\sqrt{p_2p_3}} \bigg( \frac{\dl^{(2)}}{\lambda} + \frac{1}{2^{t_0-1}} \bigg) + \mu_0^2 \frac{\sqrt{r_2r_3}}{\sqrt{p_2p_3}} \bigg( \frac{\dl^{(3)}}{\lambda} + \frac{1}{2^{t_0-1}} \bigg)  + \mu_0 \sqrt{\frac{r_1}{p_{-1}}} \Bigg\} \\
     &=\frac{1}{8} \frac{\dl^{(1)}}{\lambda} \mu_0 \sqrt{\frac{r_1}{p_j}} \bigg( \frac{64 C'}{C_0} \sqrt{p_{-1}} \bigg) \\
      &\quad \quad \times \Bigg\{ \mu_0 \frac{\sqrt{r_2r_3}}{p_j\sqrt{r_1}} \bigg( \frac{\dl^{(2)}}{\lambda} + \frac{1}{2^{t_0-1}}\bigg)\bigg( \frac{\dl^{(3)}}{\lambda} + \frac{1}{2^{t_0-1}}\bigg) \\
     &\quad\quad\quad + \mu_0 \frac{\sqrt{r_2r_3}}{\sqrt{p_jp_3r_1}} \bigg( \frac{\dl^{(2)}}{\lambda} + \frac{1}{2^{t_0-1}} \bigg) + \mu_0 \frac{\sqrt{r_2 r_3}}{\sqrt{p_j p_2r_1}} \bigg( \frac{\dl^{(3)}}{\lambda} + \frac{1}{2^{t_0-1}} \bigg)\\
     &\quad\quad\quad +  \mu_0 \frac{\sqrt{r_2r_3}}{\sqrt{p_2p_3r_1}} \bigg( \frac{\dl^{(2)}}{\lambda} + \frac{1}{2^{t_0-1}} \bigg) + \mu_0 \frac{\sqrt{r_2r_3}}{\sqrt{p_2p_3r_1}} \bigg( \frac{\dl^{(3)}}{\lambda} + \frac{1}{2^{t_0-1}} \bigg)  +  \sqrt{\frac{1}{p_{-1}}} \Bigg\} \\
     &\leq \frac{1}{8} \bigg( \frac{\dl^{(1)}}{\lambda} + \frac{1}{2^{t_0}} \bigg) \mu_0 \sqrt{\frac{r_1}{p_j}},
\end{align*}
where the final inequality holds by basic algebra as long as $C_0 \geq 64 C C'$ for some other constant $C'$, as well as the assumptions $\lambda \gtrsim \kappa \sqrt{\log(p)} p/p_{\min}^{1/4}$, $r_k \asymp r$ and $\mu_0^2 r \lesssim p_{\min}^{1/2}$.

Consequently, we have shown that the desired bounds on the terms $(I)$, $(II)$, and $(III)$ hold with probability at most $O(p^{-30}) \leq p^{-29}$ as desired.
\end{proof}

\begin{lemma}[Bounding the linear term on a good event] \label{lem:linearterm_goodevent}
Let $t_0$ and $k$ be fixed, and let $m$ be such that $1 \leq m \leq p_k$.  Then 
\begin{align*}
    \p\Bigg\{ \bigg\{ \| e_m\t \mathbf{L}_k^{t_0} \| \geq \frac{1}{4} \bigg( \frac{\dl\ku}{\lambda} + \frac{1}{2^{t_0}} \bigg) \mu_0 \sqrt{\frac{r_k}{p_k}} \bigg\} \bigcap \mathcal{E}_{\mathrm{Good}} \bigcap \mathcal{E}_{\mathrm{main}}^{t_0-1,k}\Bigg\} \leq p^{-29}.\end{align*}
\end{lemma}

\begin{proof}[Proof of Lemma \ref{lem:linearterm_goodevent}]
The proof of this is similar to the proof of Lemma \ref{lem:leaveoneout_goodevent}, only using the deterministic bound in Lemma \ref{lem:linear_deterministic_bd} instead of the deterministic bound in Lemma \ref{lem:leaveoneoutsintheta}.  Once again without loss of generality we prove the result for $k = 1$; the bounds for $k = 2$ and $k = 3$ are similar.

First, on the event $\mathcal{E}_{\mathrm{Good}}$, it holds that $\lambda/2 \leq \lambda_{r_1}(\mathbf{\hat{\Lambda}}_k^{(t_0-1)})$ for $t_0 \geq 1$.  By Lemma \ref{lem:linear_deterministic_bd}, it holds that
\begin{align*}
    \| e_m\t \mathbf{L}_k^{t_0} \| &\leq \frac{8 \kappa}{\lambda} \| \U_1\|_{2,\infty} \bigg( \tau_1 \eta_1^{(t)} + \| \U_1\t \mathbf{Z}_k \mathbf{V}_1 \| \bigg) + \frac{8\kappa}{\lambda} \bigg( \tau_k \eta_k^{(t,k-m)} \bigg) + \frac{4\kappa}{\lambda} \tilde \xi_k^{t_0,k-m}.
\end{align*}
On the event $\mathcal{E}_{\mathrm{main}}^{t_0-1,k} \cap \mathcal{E}_{\mathrm{Good}}$, we have the following bounds: \begin{align*}
    \tau_1 &\leq C_1 \sqrt{pr}; \\
    \|\U_1\t\mathbf{Z}_1 \mathbf{V}_1 \| &\leq C_1 \big( \sqrt{r} + \sqrt{\log(p)} \big);\\
    \eta_1^{(t)} &\leq \frac{\dl^{(2)}}{\lambda} + \frac{\dl^{(3)}}{\lambda} + \frac{2}{2^{t_0-1}}; \\
    \eta_1^{(t,1-m)} &\leq \bigg( \frac{\dl^{(2)}}{\lambda} + \frac{1}{2^{t_0-1}} \bigg) \mu_0 \sqrt{\frac{r_2}{p_1}} + \bigg( \frac{\dl^{(3)}}{\lambda} + \frac{1}{2^{t_0-1}} \bigg) \mu_0 \sqrt{\frac{r_3}{p_1}}.
\end{align*}
 Plugging in these bounds yields
\begin{align*}
      \| e_m\t \mathbf{L}_k^{t_0} \| &\leq  \frac{8C \kappa}{\lambda} \mu_0 \sqrt{\frac{r_1}{p_1}} \bigg(  \sqrt{pr}  \bigg(  \frac{\dl^{(2)}}{\lambda} + \frac{\dl^{(3)}}{\lambda} + \frac{2}{2^{t_0-1}} \bigg) + \sqrt{r} + \sqrt{\log(p)} \bigg) \\
      &\quad + \frac{8C \sqrt{pr} \kappa}{\lambda} \bigg(\bigg( \frac{\dl^{(2)}}{\lambda} + \frac{1}{2^{t_0-1}} \bigg) \mu_0 \sqrt{\frac{r_2}{p_1}} + \bigg( \frac{\dl^{(3)}}{\lambda} + \frac{1}{2^{t_0-1}} \bigg) \mu_0 \sqrt{\frac{r_3}{p_1}}  \bigg) + \frac{4\kappa}{\lambda} \tilde \xi_k^{t_0,k-m} \\
      &\leq \mu_0 \sqrt{\frac{r_1}{p_1}} \frac{8C \kappa}{\lambda} \Bigg[\sqrt{pr} \bigg( \frac{\dl^{(2)}}{\lambda} + \frac{\dl^{(3)}}{\lambda} \bigg) + \sqrt{r} + \sqrt{\log(p)} \Bigg] + \mu_0 \sqrt{\frac{r_1}{p_1}} \frac{8 C \kappa \sqrt{pr}}{\lambda} \frac{2}{2^{t_0-1}} \\
      &\quad + \mu_0 \sqrt{\frac{r_1}{p_1}} \frac{8 C \sqrt{pr} \kappa}{\lambda} \sqrt{\frac{r_2}{r_1}} \bigg( \frac{\dl^{(2)}}{\lambda} + \frac{1}{2^{t_0-1}} \bigg) + \mu_0 \sqrt{\frac{r_1}{p_1}} \frac{8 C \sqrt{pr} \kappa}{\lambda} \sqrt{\frac{r_3}{r_1}} \bigg( \frac{\dl^{(3)}}{\lambda} + \frac{1}{2^{t_0-1}} \bigg) \\
      &\quad + \frac{4 \kappa}{\lambda} \tilde \xi_k^{t_0,k-m} \\
      &\leq \mu_0 \sqrt{\frac{r_1}{p_1}} \frac{\dl^{(1)}}{\lambda} \bigg( \frac{8 C}{C_0 \sqrt{p_1\log(p)}} \bigg) \\
      &\quad \times \Bigg[\sqrt{pr} \bigg( \frac{\dl^{(2)}}{\lambda} + \frac{\dl^{(3)}}{\lambda} \bigg) + \sqrt{r} + \sqrt{\log(p)} + \sqrt{\frac{r_2}{r_1}} \frac{\dl^{(2)}}{\lambda} + \sqrt{\frac{r_3}{r_1}} \frac{\dl^{(3)}}{\lambda} \Bigg] \\
      &\quad + \mu_0 \sqrt{\frac{r_1}{p_1}} \frac{1}{2^{t_0}} \frac{32 C \kappa \sqrt{pr}}{\lambda} + \mu_0 \sqrt{\frac{r_1}{p_1}} \frac{1}{2^{t_0}}\frac{32 C \sqrt{pr}{\kappa}}{\lambda} \sqrt{\frac{r}{r_1}} \\
      &\quad + \frac{4 \kappa}{\lambda} \tilde \xi_1^{t_0,1-m} \\\
    &\leq \frac{1}{8} \bigg( \frac{\dl^{(1)}}{\lambda} + \frac{1}{2^{t_0}} \bigg) + \frac{4\kappa}{\lambda} \tilde \xi_1^{t_0,1-m},
\end{align*}
where the final inequality holds as long as
\begin{align*}
   \bigg( \frac{8 C}{C_0 \sqrt{p_1\log(p)}} \bigg) \Bigg[\sqrt{pr} \bigg( \frac{\dl^{(2)}}{\lambda} + \frac{\dl^{(3)}}{\lambda} \bigg) + \sqrt{r} + \sqrt{\log(p)} + \sqrt{\frac{r_2}{r_1}} \frac{\dl^{(2)}}{\lambda} + \sqrt{\frac{r_3}{r_1}} \frac{\dl^{(3)}}{\lambda} \Bigg] \leq \frac{1}{8}
\end{align*}
and
\begin{align*}
     \frac{32 C \kappa \sqrt{pr}}{\lambda} + \frac{32 C \sqrt{pr}{\kappa}}{\lambda} \sqrt{\frac{r}{r_1}} \leq \frac{1}{8}.
\end{align*}
These two inequalities hold as long as $C_0$ is larger than some fixed constant  and the assumptions $\lambda \gtrsim \kappa \sqrt{\log(p)} p/p_{\min}^{1/4}$, $r_k \asymp r$ and $\mu_0^2 r \lesssim p_{\min}^{1/2}$.
Consequently, 
\begin{align*}
    \p\Bigg\{ &\bigg\{ \| e_m\t \mathbf{L}_1^{t_0} \| \geq \frac{1}{4} \bigg( \frac{\dl^{(1)}}{\lambda} + \frac{1}{2^{t_0}} \bigg) \mu_0 \sqrt{\frac{r_1}{p_1}}\bigg\} \bigcap \mathcal{E}_{\mathrm{Good}} \bigcap \mathcal{E}_{\mathrm{main}}^{t_0,1} \Bigg\} \\
    &\leq \p\Bigg\{ \bigg\{\frac{4\kappa}{\lambda} \tilde \xi_1^{t_0,1-m} \geq \frac{1}{8} \bigg( \frac{\dl^{(1)}}{\lambda} + \frac{1}{2^{t_0}} \bigg) \mu_0 \sqrt{\frac{r_1}{p_1}} \bigg\} \bigcap \mathcal{E}_{\mathrm{Good}} \bigcap \mathcal{E}_{\mathrm{main}}^{t_0,1} \Bigg\} \\
    &\leq \p\Bigg\{ \bigg\{\frac{4\kappa}{\lambda} \tilde \xi_1^{t_0,1-m} \geq \frac{1}{8} \bigg( \frac{\dl^{(1)}}{\lambda} + \frac{1}{2^{t_0}} \bigg) \mu_0 \sqrt{\frac{r_1}{p_1}}\bigg\} \bigcap \mathcal{E}_{\mathrm{Good}} \bigcap \mathcal{E}_{\mathrm{main}}^{t_0,1} \bigcap \mathcal{\tilde E}_{1-m}^{t_0,1} \Bigg\} \\
    &\qquad + \p\bigg\{ \bigcap \mathcal{E}_{\mathrm{Good}} \bigcap \mathcal{E}_{\mathrm{main}}^{t_0,1} \bigcap (\mathcal{\tilde E}_{1-m}^{t_0,1})^c \bigg\} \\
    &\leq \p\Bigg\{ \bigg\{\frac{4\kappa}{\lambda} \tilde \xi_1^{t_0,1-m} \geq \frac{1}{8} \bigg( \frac{\dl^{(1)}}{\lambda} + \frac{1}{2^{t_0}} \bigg)\mu_0 \sqrt{\frac{r_1}{p_1}} \bigg\} \bigcap \mathcal{\tilde E}_{1-m}^{t_0,1} \Bigg\} 
    \end{align*}
where we have used Lemma \ref{lem:eventintersection} to conclude that the event in the penultimate line is empty.  Therefore, it suffices to bound $\tilde \xi_1^{t_0,1-m}$ on the event $\mathcal{\tilde E}_{1-m}^{t_0,1}$.  Since this event is independent from the random variables belonging to $e_m\t \mathbf{Z}_1$, by Lemma \ref{lem:rowbound}, it holds that with probability at least $1 - O(p^{-30})$ that 
\begin{align*}
    \tilde\xi_1^{t_0,1-m} &= \bigg\| \bigg( \mathbf{Z}_1^{1-m} - \mathbf{Z}_1 \bigg) \mathcal{\tilde P}_{1}^{t_0,1-m} \mathbf{V}_1 \bigg\| \\
    &\leq C  \sqrt{p_{-1}\log(p)} \bigg\| \mathcal{\tilde P}_{1}^{t_0,j-m} \mathbf{V}_1 \bigg\|_{2,\infty}. 
\end{align*}
On the event $\mathcal{\tilde E}_{1-m}^{t_0,1}$ it holds that
\begin{align*}
    \| \mathcal{\tilde P}_{1}^{t_0,1-m} \mathbf{V}_1 \|_{2,\infty} &\leq  c \Bigg\{ \mu_0^2 \frac{\sqrt{r_2r_3}}{p_1} \bigg( \frac{\dl^{(2)}}{\lambda} + \frac{1}{2^{t_0-1}}\bigg)\bigg( \frac{\dl^{(3)}}{\lambda} + \frac{1}{2^{t_0-1}}\bigg); \\
     &\quad + \mu_0^2 \frac{\sqrt{r_2r_3}}{\sqrt{p_1p_3}} \bigg( \frac{\dl^{(2)}}{\lambda} + \frac{1}{2^{t_0-1}} \bigg) + \mu_0^2 \frac{\sqrt{r_2 r_3}}{\sqrt{p_1 p_2}} \bigg( \frac{\dl^{(3)}}{\lambda} + \frac{1}{2^{t_0-1}} \bigg)\\
     &\quad +  \mu_0^2 \frac{\sqrt{r_2r_3}}{\sqrt{p_2p_3}} \bigg( \frac{\dl^{(2)}}{\lambda} + \frac{1}{2^{t_0-1}} \bigg) + \mu_0^2 \frac{\sqrt{r_2r_3}}{\sqrt{p_2p_3}} \bigg( \frac{\dl^{(3)}}{\lambda} + \frac{1}{2^{t_0-1}} \bigg) \\
     &\quad + \mu_0 \sqrt{\frac{r_1}{p_{-1}}} \Bigg\}.
\end{align*}
Therefore with probability at least $1 - O(p^{-30})$, one has
\begin{align*}
    \frac{4\kappa}{\lambda} \tilde\xi_1^{t_0,1-m} &\leq \frac{4 C \kappa \sqrt{\log(p)}}{\lambda} \sqrt{p_{-1}} \\
    &\times \Bigg\{ \mu_0^2 \frac{\sqrt{r_2r_3}}{p_1} \bigg( \frac{\dl^{(2)}}{\lambda} + \frac{1}{2^{t_0-1}}\bigg)\bigg( \frac{\dl^{(3)}}{\lambda} + \frac{1}{2^{t_0-1}}\bigg); \\
     &\quad + \mu_0^2 \frac{\sqrt{r_2r_3}}{\sqrt{p_1p_3}} \bigg( \frac{\dl^{(2)}}{\lambda} + \frac{1}{2^{t_0-1}} \bigg) + \mu_0^2 \frac{\sqrt{r_2 r_3}}{\sqrt{p_1 p_2}} \bigg( \frac{\dl^{(3)}}{\lambda} + \frac{1}{2^{t_0-1}} \bigg)\\
     &\quad +  \mu_0^2 \frac{\sqrt{r_2r_3}}{\sqrt{p_2p_3}} \bigg( \frac{\dl^{(2)}}{\lambda} + \frac{1}{2^{t_0-1}} \bigg) + \mu_0^2 \frac{\sqrt{r_2r_3}}{\sqrt{p_2p_3}} \bigg( \frac{\dl^{(3)}}{\lambda} + \frac{1}{2^{t_0-1}} \bigg) + \mu_0 \sqrt{\frac{r_1}{p_{-1}}} \Bigg\} \\
     &\leq \frac{C_0 \kappa \sqrt{p_1\log(p)}}{\lambda} \mu_0 \sqrt{\frac{r_1}{p_1}} \bigg( \frac{4C\sqrt{p_{-1}}}{C_0} \bigg) \\
    &\times \Bigg\{ \mu_0 \frac{\sqrt{r_2r_3}}{\sqrt{r_1}p_1} \bigg( \frac{\dl^{(2)}}{\lambda} + \frac{1}{2^{t_0-1}}\bigg)\bigg( \frac{\dl^{(3)}}{\lambda} + \frac{1}{2^{t_0-1}}\bigg); \\
     &\quad + \mu_0 \frac{\sqrt{r_2r_3}}{\sqrt{r_1p_1p_3}} \bigg( \frac{\dl^{(2)}}{\lambda} + \frac{1}{2^{t_0-1}} \bigg) + \mu_0 \frac{\sqrt{r_2 r_3}}{\sqrt{r_1p_1 p_2}} \bigg( \frac{\dl^{(3)}}{\lambda} + \frac{1}{2^{t_0-1}} \bigg)\\
     &\quad +  \mu_0 \frac{\sqrt{r_2r_3}}{\sqrt{r_1p_2p_3}} \bigg( \frac{\dl^{(2)}}{\lambda} + \frac{1}{2^{t_0-1}} \bigg) + \mu_0 \frac{\sqrt{r_2r_3}}{\sqrt{r_1p_2p_3}} \bigg( \frac{\dl^{(3)}}{\lambda} + \frac{1}{2^{t_0-1}} \bigg) +  \sqrt{\frac{1}{p_{-1}}} \Bigg\} \\
     &\leq \frac{1}{8}\bigg( \frac{\dl^{(1)}}{\lambda} + \frac{1}{2^{t_0}} \bigg) \mu_0 \sqrt{\frac{r_1}{p_1}},
\end{align*}
where the final inequality holds by similar algebraic manipulations as in the previous part of this proof provided that $C_0$ is larger than some fixed constant  together with the assumptions $\lambda \gtrsim \kappa \sqrt{\log(p)} p/p_{\min}^{1/4}$, $r_k \asymp r$ and $\mu_0^2 r \lesssim p_{\min}^{1/2}$.
% \begin{align*}
%     \tilde \xi_1^{t_0,1-m} &\coloneqq \bigg\| \bigg( \mathbf{Z}_1^{1-m} - \mathbf{Z}_1 \bigg) \mathcal{\tilde P}^{t_0,1-m}_{1} \mathbf{V}_1 \bigg\| \\
%     &\leq C p \sqrt{\log(p)} \bigg\| \mathcal{\tilde P}^{t_0,1-m}_{1} \mathbf{V}_1 \bigg\|_{2,\infty} \\
%     &\leq C p \sqrt{\log(p)} \bigg( \mu_0 \frac{\sqrt{r}}{p} + 17 \mu_0^2 \frac{r}{p} \bigg( \frac{\dl}{\lambda} + \frac{1}{2^{t_0-1}} \bigg) \bigg),
% \end{align*}
% where the final inequality holds on the event $\mathcal{\tilde E}_{1-m}^{t_0,1}$.  Therefore, with probability at least $1 - O(p^{-30})$, it holds that
% \begin{align*}
%     \frac{4 \kappa}{\lambda}    \tilde \xi_1^{t_0,1-m} &\leq \frac{4 \kappa}{\lambda} C p \sqrt{\log(p)} \bigg( \mu_0 \frac{\sqrt{r}}{p} + 17 \mu_0^2 \frac{r}{p} \bigg( \frac{\dl}{\lambda} + \frac{1}{2^{t_0-1}} \bigg) \bigg) \\
%     &\leq \frac{4 C \kappa \sqrt{p \log(p)}}{\lambda} \mu_0 \sqrt{\frac{r}{p}} + \frac{68 C \kappa \mu_0 \sqrt{pr \log(p)}}{\lambda} \frac{\dl}{\lambda} \mu_0 \sqrt{\frac{r}{p}}\\
%     &\qquad + \frac{68 C \kappa \sqrt{pr \log(p)}}{\lambda} \mu_0 \sqrt{\frac{r}{p}} \frac{1}{2^{t_0-1}} \\
%     &\leq \frac{1}{8} \bigg( \frac{\dl}{\lambda} + \frac{1}{2^{t_0}} \bigg) \mu_0 \sqrt{\frac{r}{p}}
% \end{align*}
%as long as $C_0 \geq 4 C$ and $\lambda \geq 68 C \kappa \mu_0 \sqrt{pr \log(p)}$.
This completes the proof.
\end{proof}

\begin{lemma}[Bounding the quadratic term on a good event] \label{lem:quadraticterm_goodevent}
The quadratic term satisfies
\begin{align*}
    \p\Bigg\{ \bigg\{ \| e_m\t \mathbf{Q}_k^{(t)} \| \geq \frac{1}{4} \bigg( \frac{\dl\ku}{\lambda} + \frac{1}{2^{t}} \bigg) \mu_0 \sqrt{\frac{r_k}{p_k}} \bigg\}\bigcap \mathcal{E}_{\mathrm{main}}^{t_0-1,1} \bigcap \mathcal{E}_{\mathrm{Good}}\Bigg\} \leq p^{-29}.\end{align*}
\end{lemma}

\begin{proof}
Again without loss of generality we prove the result for $k = 1$; the bounds for $k =2$ and $3$ are similar.  First, on the event $\mathcal{E}_{\mathrm{Good}}$ it holds that $\lambda/2 \leq \lambda_{r_1} ( \mathbf{\hat{\Lambda}}^{(t_0-1)})$, and hence by Lemma \ref{lem:quadratic_deterministic_bd} it holds that
\begin{align*}
    \| e_m\t \mathbf{Q}_1^{t_0} \| &\leq \frac{4}{\lambda^2} \|\U_1\|_{2,\infty} \bigg( \tau_1 \eta_1^{(t)} + \bigg\| \U_1\t \mathbf{Z}_1 \bigg[ \mathcal{P}_{\U_{2}} \otimes  \mathcal{P}_{\U_{3}}  \bigg] \bigg\| \bigg)  + \frac{16}{\lambda^2}\tau_1^2\bigg( \eta_{1}^{(t,1-m)}  \bigg)  \\
%&\qquad + \frac{4}{\lambda^2} \tau_k^2  \|\sin\Theta(\utilde_{k}^{k-m,t-1},\uhat_{k}^{(t-1)}) \| \\
%&\qquad +  \frac{4}{\lambda^2}\bigg\| e_m\t \mathbf{Z}_k\bigg[ \mathcal{P}_{\utilde_{k+1}^{k-m,t}}\otimes \mathcal{P}_{\utilde_{k+2}^{k-m,t}} \bigg]\bigg[\mathcal{P}_{\utilde_{k+1}^{k-m,t}} \otimes \mathcal{P}_{\utilde_{k+2}^{k-m,t}} \bigg] \bigg(\mathbf{Z}_k - \mathbf{Z}_k^{k-m} \bigg) \t \uhat_k^{(t)}\bigg\| \\
%&\qquad 
&+ \frac{4}{\lambda^2}  \xi_1^{t,1-m} \bigg( \tau_1 \| \sin\Theta(\uhat_1^{(t)},\U_1) \| + \tau_1 \eta^{(t-1)}_1 +  \bigg\|\U_1 \U_1\t \mathbf{Z}_1 \mathcal{P}_{\U_{2}}  \otimes   \mathcal{P}_{\U_{3}}  \bigg\| \bigg).
\end{align*}
On the event $\mathcal{E}_{\mathrm{main}}^{t_0-1,1} \cap \mathcal{E}_{\mathrm{Good}}$, one has the following bounds:
\begin{align*}
    \tau_1 &\leq C \sqrt{pr}; \\
    \eta_1^{(t_0,1-m)} &\leq \mu_0 \sqrt{\frac{r_2}{p_1}}\bigg( \frac{\dl^{(2)}}{\lambda} + \frac{1}{2^{t_0-1}} \bigg) + \mu_0 \sqrt{\frac{r_3}{p_1}} \bigg( \frac{\dl^{(3)}}{\lambda} + \frac{1}{2^{t_0-1}} \bigg) ; \\
    \eta_{1}^{(t_0)} &\leq \frac{\dl^{(2)}}{\lambda} + \frac{\dl^{(3)}}{\lambda} + \frac{2}{2^{t_0-1}}; \\
    \|\sin\Theta(\uhat_{1}^{(t_0-1)}, \U_1^{(t_0-1)}) \| &\leq \frac{\dl^{(1)}}{\lambda} + \frac{1}{2^{t_0-1}}; \\
    \| \U_1\t \mathbf{Z}_1 \mathcal{P}_{\U_{2}}  \otimes   \mathcal{P}_{\U_{3}} \| &\leq C( r + \sqrt{\log(p)}). 
\end{align*}
Plugging these in yields
\begin{align*}
     \| e_m\t \mathbf{Q}_1^{t_0} \| &\leq  \frac{4C }{\lambda^2} \mu_0 \sqrt{\frac{r_1}{p_1}} \bigg[  \sqrt{pr} \bigg(  \frac{\dl^{(2)}}{\lambda} + \frac{\dl^{(3)}}{\lambda} + \frac{2}{2^{t_0-1}} \bigg) + r + \sqrt{\log(p)} \bigg] \\
     &\quad + \frac{16C^2 pr }{\lambda^2}\bigg[ \mu_0 \sqrt{\frac{r_2}{p_1}}\bigg( \frac{\dl^{(2)}}{\lambda} + \frac{1}{2^{t_0-1}} \bigg) + \mu_0 \sqrt{\frac{r_3}{p_1}} \bigg( \frac{\dl^{(3)}}{\lambda} + \frac{1}{2^{t_0-1}} \bigg) \bigg]  \\
     &+ \frac{4 C }{\lambda^2}  \xi_1^{t,1-m} \bigg[ \sqrt{pr} \bigg( \frac{\dl^{(1)}}{\lambda}+\frac{\dl^{(2)}}{\lambda} +\frac{\dl^{(3)}}{\lambda}  + \frac{3}{2^{t_0-1}} \bigg) + r + \sqrt{\log(p)}  \bigg] \\
     &\leq \frac{C_0\sqrt{p_1\log(p)}}{\lambda} \mu_0 \sqrt{\frac{r_1}{p_1}}\bigg( \frac{4C}{C_0} \frac{1}{\lambda \sqrt{p_1 \log(p)}} \bigg)  \bigg[ \sqrt{pr} \bigg( \frac{\dl^{(2)}}{\lambda} + \frac{\dl^{(3)}}{\lambda} \bigg) + r + \sqrt{\log(p)} \bigg] \\
     &\quad + \frac{C_0 \sqrt{p_1 \log(p)}}{\lambda}\mu_0 \sqrt{\frac{r_1}{p_1}} \bigg( \frac{16 C^2 pr}{C_0\lambda\sqrt{p_1\log(p)}} \bigg) \bigg[ \sqrt{\frac{r_2}{r_1}}  \frac{\dl^{(2)}}{\lambda} + \sqrt{\frac{r_3}{r_1}}  \frac{\dl^{(3)}}{\lambda} \bigg] \\
     &\quad + \frac{1}{2^{t_0}}\mu_0 \sqrt{\frac{r_1}{p_1}} \bigg[ \frac{32 C \sqrt{pr}}{\lambda} + \sqrt{\frac{r_2}{r_1}} \frac{32 C^2 pr}{\lambda^2} + \sqrt{\frac{r_3}{r_1}} \frac{32 C^2 pr}{\lambda^2} \bigg] \\
     &\quad + \xi_1^{t_0,1-m} \frac{4 C}{\lambda^2} \bigg[ \sqrt{pr} \bigg( \frac{\dl^{(1)}}{\lambda}+\frac{\dl^{(2)}}{\lambda} +\frac{\dl^{(3)}}{\lambda}  + \frac{3}{2^{t_0-1}} \bigg) + r + \sqrt{\log(p)}  \bigg] \\
     &\leq \frac{1}{8} \bigg( \frac{\dl^{(1)}}{\lambda} + \frac{1}{2^{t_0}} \bigg) \mu_0 \sqrt{\frac{r_1}{p_1}} \\
     &\quad + \xi_1^{t_0,1-m} \frac{4 C}{\lambda^2} \bigg[ \sqrt{pr} \bigg( \frac{\dl^{(1)}}{\lambda}+\frac{\dl^{(2)}}{\lambda} +\frac{\dl^{(3)}}{\lambda}  + \frac{3}{2^{t_0-1}} \bigg) + r + \sqrt{\log(p)}  \bigg],
\end{align*}
where the final inequality holds as long as
\begin{align*}
    \bigg( \frac{4C}{C_0} \frac{1}{\lambda \sqrt{p_1 \log(p)}} \bigg)  \bigg[ \sqrt{pr} &\bigg( \frac{\dl^{(2)}}{\lambda} + \frac{\dl^{(3)}}{\lambda} \bigg) + r + \sqrt{\log(p)} \bigg] \\
    &\quad + \bigg( \frac{16 C^2 pr}{C_0\lambda\sqrt{p_1\log(p)}} \bigg) \bigg[ \sqrt{\frac{r_2}{r_1}}  \frac{\dl^{(2)}}{\lambda} + \sqrt{\frac{r_3}{r_1}}  \frac{\dl^{(3)}}{\lambda} \bigg] \leq \frac{1}{8}
\end{align*}
and
\begin{align*}
    \frac{32 C \sqrt{pr}}{\lambda} + \sqrt{\frac{r_2}{r_1}} \frac{32 C^2 pr}{\lambda^2} + \sqrt{\frac{r_3}{r_1}} \frac{32 C^2 pr}{\lambda^2} \leq \frac{1}{8}.
\end{align*}
Both of these conditions hold when $\lambda \gtrsim \kappa \sqrt{\log(p)} p/p_{\min}^{1/4}$, $r_k \asymp r$ and $r \leq C p_{\min}^{1/2}$ provided the constant $C_0$ is larger than some fixed constant. Finally, we note that
\begin{align*}
    \xi_1^{t_0,1-m} &\frac{4 C}{\lambda^2} \bigg[ \sqrt{pr} \bigg( \frac{\dl^{(1)}}{\lambda}+\frac{\dl^{(2)}}{\lambda} +\frac{\dl^{(3)}}{\lambda}  + \frac{3}{2^{t_0-1}} \bigg) + r + \sqrt{\log(p)}  \bigg] \\&\leq \frac{\xi_1^{t_0,1-m}}{\lambda} \bigg( \frac{4 C}{\lambda} \bigg)\bigg[ \sqrt{pr} \bigg( \frac{\dl^{(1)}}{\lambda}+\frac{\dl^{(2)}}{\lambda} +\frac{\dl^{(3)}}{\lambda}  + \frac{3}{2^{t_0-1}} \bigg) + r + \sqrt{\log(p)}  \bigg] \\
    &\leq \frac{\xi_1^{t_0,1-m}}{\lambda} \bigg( \frac{\tilde C\big( \sqrt{pr} + r + \sqrt{\log(p)}\big)} {\lambda} \bigg).
\end{align*}
Define
\begin{align*}
    B_1 &\coloneqq \frac{\xi_1^{t_0,1-m}}{\lambda} \bigg( \frac{\tilde C\big( \sqrt{pr} + r + \sqrt{\log(p)}\big)} {\lambda} \bigg).
\end{align*}
Then
\begin{align*}
    \p\Bigg\{ &\bigg\{ \| e_m\t \mathbf{Q}_1^{t_0} \| \geq \frac{1}{4} \bigg( \frac{\dl^{(1)}}{\lambda} + \frac{1}{2^{t_0}} \bigg) \mu_0 \sqrt{\frac{r_1}{p_1}} \bigg\} \bigcap \mathcal{E}_{\mathrm{main}}^{t_0-1,1} \bigcap \mathcal{E}_{\mathrm{Good}} \Bigg\} \\
    &\leq \p\Bigg\{ \bigg\{ B_1 \geq \frac{1}{8}\bigg( \frac{\dl^{(1)}}{\lambda} + \frac{1}{2^{t_0}} \bigg) \mu_0 \sqrt{\frac{r_1}{p_1}} \bigg\} \mathcal{E}_{\mathrm{main}}^{t_0-1,1} \bigcap \mathcal{E}_{\mathrm{Good}} \Bigg\} \\
    &\leq \p\Bigg\{ \bigg\{ B_1\geq \frac{1}{8} \bigg( \frac{\dl^{(1)}}{\lambda} + \frac{1}{2^{t_0}} \bigg) \mu_0 \sqrt{\frac{r_1}{p_1}}\bigg\} \mathcal{E}_{\mathrm{main}}^{t_0-1,1} \bigcap \mathcal{E}_{\mathrm{Good}} \bigcap \mathcal{\tilde E}_{1-m}^{t_0,1} \Bigg\} \\
    &\qquad + \p\bigg\{ \mathcal{E}_{\mathrm{main}}^{t_0-1,1} \bigcap \mathcal{E}_{\mathrm{Good}}\bigcap (\mathcal{\tilde E}_{1-m}^{t_0,1})^c \bigg\} \\
    &\leq \p\Bigg\{ \bigg\{ B_1 \geq \frac{1}{8} \bigg( \frac{\dl^{(1)}}{\lambda} + \frac{1}{2^{t_0}} \bigg) \mu_0 \sqrt{\frac{r_1}{p_1}} \bigg\} \bigcap \mathcal{\tilde E}_{1-m}^{t_0,1} \Bigg\},
\end{align*}
where we have used Lemma \ref{lem:eventintersection} to conclude that the event in the penultimate line is empty.  Since the event $\mathcal{\tilde E}_{1-m}^{t_0,1} $ is independent from the random variables belonging to $e_m\t \mathbf{Z}_1$, by Lemma \ref{lem:rowbound}, it holds that with probability at least $1 - O(p^{-30})$ that 
\begin{align*}
    \xi_1^{t_0,1-m} &\coloneqq \bigg\| \bigg( \mathbf{Z}_1^{1-m} - \mathbf{Z}_1 \bigg) \mathcal{\tilde P}^{t_0,1-m}_{1} \bigg\| \\
    &\leq C  \sqrt{p_{-1}\log(p)} \bigg\| \mathcal{\tilde P}^{t_0,1-m}_{1}  \bigg\|_{2,\infty}.
\end{align*}
On the event $\mathcal{\tilde E}_{1-m}^{t_0,1}$, we have that
\begin{align*}
    \bigg\| \mathcal{\tilde P}^{t_0,1-m}_{1}  \bigg\|_{2,\infty} &\leq c \Bigg\{ \mu_0^2 \frac{\sqrt{r_2 r_3}}{p_1}  \bigg( \frac{\dl^{(2)}}{\lambda} + \frac{1}{2^{t_0-1}} \bigg)  \bigg( \frac{\dl^{(3)}}{\lambda} + \frac{1}{2^{t_0-1}} \bigg) \\
    &\quad +  \mu_0^2 \frac{\sqrt{r_2 r_3}}{\sqrt{p_1p_3}} \bigg( \frac{\dl^{(2)}}{\lambda} + \frac{1}{2^{t_0-1}} \bigg) + \mu_0^2 \frac{\sqrt{r_2r_3}}{\sqrt{p_1p_2}} \bigg( \frac{\dl^{(3)}}{\lambda} + \frac{1}{2^{t_0-1}} \bigg) \\
    &\quad + \mu_0^2 \frac{\sqrt{r_2 r_3}}{\sqrt{p_2p_3}} \Bigg\}.
\end{align*}
Therefore, with probability at least $1 - O(p^{-30})$, it holds that
\begin{align*}
    B_1 &= \frac{\xi_1^{t_0,1-m}}{\lambda} \bigg( \frac{\tilde C \big( \sqrt{pr} + r + \sqrt{\log(p)} \big)}{\lambda} \bigg) \\
    &\leq \frac{C' \sqrt{p_{-1}\log(p)}}{\lambda} \bigg( \frac{\sqrt{pr} + r + \sqrt{\log(p)} }{\lambda} \bigg) \\
    &\quad \quad \times \bigg\{\mu_0^2 \frac{\sqrt{r_2 r_3}}{p_1}  \bigg( \frac{\dl^{(2)}}{\lambda} + \frac{1}{2^{t_0-1}} \bigg)  \bigg( \frac{\dl^{(3)}}{\lambda} + \frac{1}{2^{t_0-1}} \bigg) \\
    &\quad \quad \quad+  \mu_0^2 \frac{\sqrt{r_2 r_3}}{\sqrt{p_1p_3}} \bigg( \frac{\dl^{(2)}}{\lambda} + \frac{1}{2^{t_0-1}} \bigg) + \mu_0^2 \frac{\sqrt{r_2r_3}}{\sqrt{p_1p_2}} \bigg( \frac{\dl^{(3)}}{\lambda} + \frac{1}{2^{t_0-1}} \bigg) + \mu_0^2 \frac{\sqrt{r_2 r_3}}{\sqrt{p_2p_3}}\bigg\} \\
    &\leq \frac{C' \sqrt{p_{-1}\log(p)}}{\lambda} \bigg( \frac{\sqrt{pr} + r + \sqrt{\log(p)} }{\lambda} \bigg) \\
    &\quad \quad \times \bigg\{\mu_0^2 \frac{\sqrt{r_2 r_3}}{p_1}  \bigg( \frac{\dl^{(2)}}{\lambda} + \frac{1}{2^{t_0-1}} \bigg)  +  \mu_0^2 \frac{\sqrt{r_2 r_3}}{\sqrt{p_1p_3}} \bigg( \frac{\dl^{(2)}}{\lambda} + \frac{1}{2^{t_0-1}} \bigg)  \\
    &\quad \quad \quad+ \mu_0^2 \frac{\sqrt{r_2r_3}}{\sqrt{p_1p_2}} \bigg( \frac{\dl^{(3)}}{\lambda} + \frac{1}{2^{t_0-1}} \bigg) + \mu_0^2 \frac{\sqrt{r_2 r_3}}{\sqrt{p_2p_3}}\bigg\} \\
    &\leq \frac{C_0 \sqrt{p_1\log(p)}}{\lambda} \mu_0 \sqrt{\frac{r_1}{p_1}} \bigg( \sqrt{p_{-1}}\frac{C_0}{C'} \frac{\sqrt{pr}}{\lambda} \bigg) \\
    &\quad \times \bigg( \mu_0 \frac{\sqrt{r_2 r_3}}{p_1 \sqrt{r_1}} \frac{\dl^{(2)}}{\lambda} + 2 \mu_0 \frac{\sqrt{r_2 r_3}}{\sqrt{p_1 p_3 r_1}} \frac{\dl^{(2)}}{\lambda} +  \mu_0 \frac{\sqrt{r_2r_3}}{\sqrt{p_1p_2 r_1}} \frac{\dl^{(3)}}{\lambda} + \mu_0 \frac{\sqrt{r_2 r_3}}{\sqrt{p_2p_3r_1}} \bigg) \\
    &\quad + \frac{1}{2^{t_0}} \mu_0 \sqrt{\frac{r_1}{p_1}} \bigg( \frac{C'' \sqrt{p_1} \sqrt{p_{-1}\log(p)}}{\lambda} \frac{\sqrt{pr}}{\lambda} \bigg) \bigg[ \mu_0 \frac{\sqrt{r_2 r_3}}{p_1 \sqrt{r_1}} + \mu_0\frac{\sqrt{r_2r_3}}{\sqrt{p_1 p_3 r_1}} + \mu_0 \frac{\sqrt{r_2 r_3}}{\sqrt{p_1 p_2 r_1}} \bigg] \\
    &\leq \frac{1}{8} \bigg( \frac{\dl^{(1)}}{\lambda} + \frac{1}{2^{t_0}} \bigg) \mu_0 \sqrt{\frac{r_1}{p_1}},
\end{align*}
where the final inequality holds as long as
\begin{align*}
    \bigg( \sqrt{p_{-1}}\frac{C_0}{C'} \frac{\sqrt{pr}}{\lambda} \bigg) \bigg( \mu_0 \frac{\sqrt{r_2 r_3}}{p_1 \sqrt{r_1}} \frac{\dl^{(2)}}{\lambda} + 2 \mu_0 \frac{\sqrt{r_2 r_3}}{\sqrt{p_1 p_3 r_1}} \frac{\dl^{(2)}}{\lambda} +  \mu_0 \frac{\sqrt{r_2r_3}}{\sqrt{p_1p_2 r_1}} \frac{\dl^{(3)}}{\lambda} + \mu_0 \frac{\sqrt{r_2 r_3}}{\sqrt{p_2p_3r_1}} \bigg) \leq \frac{1}{8}
\end{align*}
and
\begin{align*}
    \bigg( \frac{C'' \sqrt{p_1} \sqrt{p_{-1}\log(p)}}{\lambda} \frac{\sqrt{pr}}{\lambda} \bigg) \bigg[ \mu_0 \frac{\sqrt{r_2 r_3}}{p_1 \sqrt{r_1}} + \mu_0\frac{\sqrt{r_2r_3}}{\sqrt{p_1 p_3 r_1}} + \mu_0 \frac{\sqrt{r_2 r_3}}{\sqrt{p_1 p_2 r_1}}  \bigg] \leq \frac{1}{8},
\end{align*}
both of which hold when $C_0$ is larger than some fixed constant  and $\lambda \gtrsim \kappa \sqrt{\log(p)} p/p_{\min}^{1/4}$, $r_k \asymp r$ and $r \leq C p_{\min}^{1/2}$.
%
%where the final inequality holds on the event $\mathcal{\tilde E}_{1-m}^{t_0,1}$ (where we have absorbed the constant).  Therefore, with probability at least $1 - O(p^{-30})$, it holds that
% \begin{align*}
%     B_1 &= \frac{12 C_1}{\lambda^2} \xi_1^{t,1-m} \bigg( \frac{\dl}{\lambda} + \frac{1}{2^{t_0-1}} + r + \sqrt{\log(p)} \bigg) \\
%     &\leq \frac{12 C_1 C \sqrt{\log(p)} \mu_0^2 r}{\lambda^2}  \bigg( \frac{\dl}{\lambda} + \frac{1}{2^{t_0-1}} + r + \sqrt{\log(p)} \bigg) \\
%     &\leq \frac{1}{8} \bigg( \frac{\dl}{\lambda} + \frac{1}{2^{t_0}} \bigg) \mu_0 \sqrt{\frac{r}{p}},
% \end{align*}
% which holds when $\lambda \geq C \mu_0 \sqrt{pr \log(p)}$ for some sufficiently large $C$.   
This completes the proof.
\end{proof}

%Next we study the event $\mathcal{E}^{t_0}_{k-m}$.  The following lemma shows that $\p( \mathcal{E}^{t_0}_{k-m} ) \geq 1 - (t_0 + 1) p^{-20}$.  

% We will claim that the following two bounds hold simultaneously (and uniformly over $j$ and $m$) with probability at least $1 - (t_0+1)p^{-20}$:
% \begin{align*}
%     \max_k \| \sin\Theta( \utilde_{k}^{t_0,j-m}, \uhat_k^{t_0} ) \| &\leq \frac{\delta_{\mathrm{L}}}{\lambda} \mu_0 \sqrt{\frac{r}{p}} + \frac{1}{2^{t_0}} \mu_0 \sqrt{\frac{r}{p}}; \\
%   \max_k \| \uhat^{(t)}_k - \U_k \mathbf{W}_k^{(t)} \|_{2,\infty} &\leq \frac{\delta_{\mathrm{L}}}{\lambda} \mu_0 \sqrt{\frac{r}{p}} + \frac{1}{2^{t_0}} \mu_0 \sqrt{\frac{r}{p}}.
% \end{align*}
% First, the result on the initialization holds with probability $1 - p^{-20}$ by Theorem \ref{thm:spectralinit_twoinfty}.  Now assume that these bounds hold with probability at least $1 - t_0p^{-20}$ for all $1\leq t \leq t_0 - 1$.  Observe that these bounds are simply $\mathcal{E}_{2,\infty}^{t_0}$ and $\mathcal{E}_{j-m}^{t_0}$ respectively.  First we will study $\mathcal{E}^{t_0+1}_{k-m}$ for fixed $k$. 

 \subsection{Putting it all together: Proof of Theorem \ref{thm:twoinfty}} \label{sec:twoinftyproof}
Recall we define
\begin{align*}
   \mathcal{E}_{\mathrm{Good}} &\coloneqq \bigg\{ \max_k \tau_k \leq C \sqrt{pr} \bigg\} \bigcap \bigg\{ \| \sin\Theta(\uhat_k^{(t)}, \U_k ) \| \leq \frac{\dl\ku}{\lambda} + \frac{1}{2^{t}} \text{ for all $t \leq t_{\max}$ and $1 \leq k \leq 3$ } \bigg\} \\
    &\qquad \bigcap \bigg\{ \max_k \bigg\| \U_k\t \mathbf{Z}_k \mathbf{V}_k \bigg\| \leq C \left( \sqrt{r} + \sqrt{\log(p)} \right) \bigg\}; \\
    &\qquad \bigcap \bigg\{ \max_k  \bigg\| \U_k\t \mathbf{Z}_k \mathcal{P}_{\U_{k+1}}  \otimes   \mathcal{P}_{\U_{k+2}}  \bigg\| \leq C \left( r + \sqrt{\log(p)} \right) \bigg\}; \\
    &\qquad \bigcap \bigg\{ \max_k \bigg\| \mathbf{Z}_k \mathbf{V}_k \bigg\| \leq C \sqrt{p_k} \bigg\}; \\
    \mathcal{E}_{2,\infty}^{t,k} &\coloneqq \bigg\{  \| \uhat_k^{(t)} - \U_k \mathbf{W}_k^{(t)} \|_{2,\infty} \leq \bigg( \frac{\dl\ku}{\lambda} + \frac{1}{2^{t}} \bigg) \mu_0 \sqrt{\frac{r_k}{p_k}} \bigg\}; \\
    %\bigg\{ \max_k \| \uhat_k^{(t)} - \U_k \mathbf{W}_k^{(t)} \|_{2,\infty} \leq \bigg( \frac{\dl}{\lambda} + \frac{1}{2^{t}} \bigg) \mu_0 \sqrt{\frac{r}{p}} \text{ for all $t \leq t_0 - 1$}\bigg\}; \\
    \mathcal{E}_{j-m}^{t,k} &\coloneqq \bigg\{ \| \sin\Theta (\utilde_k^{t,j-m}, \uhat_k^{(t)}) \| \leq  \bigg( \frac{\dl\ku}{\lambda} + \frac{1}{2^{t}} \bigg) \mu_0 \sqrt{\frac{r_k}{p_j}} \bigg\}; \\
    %\bigg\{ \max_j \| \sin\Theta(\utilde_j^{t,k-m}, \uhat_j^{(t)}) \| \leq \bigg( \frac{\dl}{\lambda} + \frac{1}{2^{t}} \bigg) \mu_0 \sqrt{\frac{r}{p}} \text{ for all $t \leq t_0 - 1$} \bigg\}; \\
      \mathcal{\tilde E}_{j-m}^{t,k} &\coloneqq \Bigg\{ \|\mathcal{\tilde P}_k^{t_0,j-m}  \mathbf{V}_k \|_{2,\infty} \leq c \Bigg[ \mu_0^2 \frac{\sqrt{r_{-k}}}{p_j} \bigg( \frac{\dl^{(k+1)}}{\lambda} + \frac{1}{2^{t_0-1}}\bigg)\bigg( \frac{\dl^{(k+2)}}{\lambda} + \frac{1}{2^{t_0-1}}\bigg); \\
     &\quad + \mu_0^2 \frac{\sqrt{r_{-k}}}{\sqrt{p_jp_{k+2}}} \bigg( \frac{\dl^{(k+1)}}{\lambda} + \frac{1}{2^{t_0-1}} \bigg) + \mu_0^2 \frac{\sqrt{r_{-k}}}{\sqrt{p_j p_{k+1}}} \bigg( \frac{\dl^{(k+2)}}{\lambda} + \frac{1}{2^{t_0-1}} \bigg)\\
     &\quad +  \mu_0^2 \frac{\sqrt{r_{-k}}}{\sqrt{p_{-k}}} \bigg( \frac{\dl^{(k+1)}}{\lambda} + \frac{1}{2^{t_0-1}} \bigg) + \mu_0^2 \frac{\sqrt{r_{-k}}}{\sqrt{p_{-k}}} \bigg( \frac{\dl^{(k+2)}}{\lambda} + \frac{1}{2^{t_0-1}} \bigg) + \mu_0 \sqrt{\frac{r_k}{p_{-k}}}\Bigg].\Bigg\} \\
     &\bigcap \Bigg\{ \|\mathcal{\tilde P}_k^{t_0,j-m}   \|_{2,\infty} \leq c \mu_0^2 \frac{\sqrt{r_{-k}}}{p_j}  \bigg( \frac{\dl^{(k+1)}}{\lambda} + \frac{1}{2^{t_0-1}} \bigg)  \bigg( \frac{\dl^{(k+2)}}{\lambda} + \frac{1}{2^{t_0-1}} \bigg) \\
    &\quad +  c\mu_0^2 \frac{\sqrt{r_{-k}}}{\sqrt{p_jp_{k+2}}} \bigg( \frac{\dl^{(k+1)}}{\lambda} + \frac{1}{2^{t_0-1}} \bigg) + c\mu_0^2 \frac{\sqrt{r_{-k}}}{\sqrt{p_jp_{k+1}}} \bigg( \frac{\dl^{(k+2)}}{\lambda} + \frac{1}{2^{t_0-1}} \bigg) + c\mu_0^2 \frac{\sqrt{r_{-k}}}{\sqrt{p_{-k}}}.\Bigg\};\\
  \mathcal{E}_{\mathrm{main}}^{t_0-1,1} &\coloneqq \bigcap_{t=1}^{(t_0-1)} \Bigg\{ \bigcap_{k=1}^{3}  \mathcal{E}^{t,k}_{2,\infty} \cap  \bigcap_{j=1}^{3} \bigcap_{m=1}^{p_j} \mathcal{E}_{k-m}^{t,j} \Bigg\}; \\
       \mathcal{E}_{\mathrm{main}}^{t_0-1,2} &\coloneqq\mathcal{E}_{\mathrm{main}}^{t_0-1,1} \cap \bigg\{ \bigcap_{k=1}^{3} \bigcap_{m = 1}^{p_k} \mathcal{E}_{k-m}^{t_0,1} \bigg\} \cap \mathcal{E}_{2,\infty}^{t_0,1} \\
        \mathcal{E}_{\mathrm{main}}^{t_0-1,3} &\coloneqq\mathcal{E}_{\mathrm{main}}^{t_0-1,2} \cap \bigg\{ \bigcap_{k=1}^{3} \bigcap_{m = 1}^{p_k} \mathcal{E}_{k-m}^{t_0,2} \bigg\} \cap \mathcal{E}_{2,\infty}^{t_0,2}.
  \end{align*}
 
  \begin{proof}[Proof of Theorem \ref{thm:twoinfty}]
   We will show that by induction that with probability at least $1 - 3(t_0+1)p^{-15}$ that simultaneously for all $t \leq t_0$ and each $k$
 \begin{align*}
    \| \uhat_k^{t_0} - \U_k \mathbf{W}_k^{t_0} \|_{2,\infty} \leq \frac{\dl\ku}{\lambda} \mu_0 \sqrt{\frac{r_k}{p_k}} + \frac{1}{2^{t}} \mu_0 \sqrt{\frac{r_k}{p_k}}; \\
     \max_{1 \leq m \leq p_k} \max_{1\leq j\leq 3}\| \sin\Theta(\uhat_j^{t_0}, \utilde_j^{t_0,k-m}) \| \leq \frac{\dl\ku}{\lambda}\mu_0 \sqrt{\frac{r_k}{p_j}} + \frac{1}{2^{t}} \mu_0 \sqrt{\frac{r_k}{p_j}}.
 \end{align*}
 Assuming that for the moment, suppose the algorithm is run for at most $C \log\bigg( \frac{\lambda}{C_0 \kappa \sqrt{p_{\min}\log(p)}} \bigg)$ iterations.  Then %since $\lambda \gtrsim \kappa p\sqrt{\log(p)}/p_{\min}^{1/4}$ 
 it holds that 
  % \begin{align*}
  %     t &\geq \max\{ C \log(p) ,1\}\\
  %     \end{align*}
  % and hence that 
  \begin{align*}
 \log(   2^t ) &=\log\bigg( 2^{C \log( \frac{\lambda}{C_0 \kappa \sqrt{p_{\min}\log(p)}})} \bigg) = C \log(2) \log( \frac{\lambda}{C_0 \kappa \sqrt{p_{\min}\log(p)}}) \geq \log\bigg( \frac{\lambda}{C_0 \kappa \sqrt{p_{\min}\log(p)}}\bigg)
  \end{align*}
  which in particular implies that
  \begin{align*}
      \frac{1}{2^t} \leq \frac{C_0 \sqrt{p_{\min}\log(p)}}{\lambda}.
  \end{align*}
  Moreover, from  the assumption $\lambda \leq \exp(c p)$ for some small constant $c$, it holds that
  \begin{align*}
      t &= C  \log( \frac{\lambda}{C_0 \kappa \sqrt{p_{\min}\log(p)}}) \leq C \log( \lambda) \leq c p,
  \end{align*}
  and hence the event holds with probability at least
  \begin{align*}
      1 - (t-1)p^{-15} \geq 1 - (c p - 1)p^{-15} \geq 1 - p^{-10}
  \end{align*}
  provided $p$ is sufficiently large.  
  Therefore, it remains to show that the result holds by induction.  %We will prove the result for fixed $k$; the same argument goes through for the other two modes, contributing the factor of $3$ to the union bound.   
  \\ \ \\ \noindent
  \textbf{Step 1: Base Case} \\
   By Theorem \ref{thm:spectralinit_twoinfty} it holds that with probability at least $1 - O(p^{-20})$ that
  \begin{align*}
    \| \uhat_k^S - \U_k \mathbf{W}_k^S \|_{2,\infty} &\lesssim \frac{\kappa \mu_0 \sqrt{r_1 \log(p)}}{\lambda} + \frac{\mu_0\sqrt{r _k p_{-k}}\log(p)}{\lambda^2} + \kappa^2 \mu_0^2 \frac{r_k}{p_k} \\ &\leq  \bigg( \frac{C \kappa \sqrt{p_k \log(p)}}{\lambda} + \frac{1}{2} \bigg)\mu_0 \sqrt{\frac{r_k}{p_k}}, % + \frac{C p^{d/2} \log(p)}{\lambda^2} \mu_0 \sqrt{\frac{r}{p}} \\
     % &\leq \frac{\dl}{\lambda} \mu_0 \sqrt{\frac{r}{p}} + \frac{1}{2} \mu_0 \sqrt{\frac{r}{p}}.
  \end{align*}
  where the final inequality holds since $\lambda \gtrsim \kappa p \sqrt{\log(p)} p_{\min}^{1/4}$ and $\mu_0^2 r \leq C p_{\min}^{1/2}$ and that $\kappa^2 \leq c p_{\min}^{1/4}$ as long as $c \times \sqrt{C} \leq \frac{1}{4}$.   In addition, by Lemma \ref{lem:spectralinit_leave_one_out_sintheta} we have the initial bound for each $k$ via
  \begin{align*}
     \max_j \max_m  \|\sin\Theta(\uhat_k^S, \utilde_k^{j-m}) \| &\lesssim \frac{\kappa \sqrt{p_k \log(p)}}{\lambda} \mu_0 \sqrt{\frac{r_1}{p_j}} + \frac{(p_1p_2p_3)^{1/2} \log(p)}{\lambda^2} \mu_0 \sqrt{\frac{r_1}{p_j}} \\
     &\leq \bigg( \frac{C \kappa \sqrt{p_k \log(p)}}{\lambda}  + \frac{1}{2} \bigg)  \mu_0 \sqrt{\frac{r_k}{p_j}},
  \end{align*}
  which holds with probability at least $1 - O(p^{-19})$.  
%  where we used the fact that on the event in Theorem %\ref{thm:spectralinit_twoinfty}, 
%  \begin{align*}
%      \|\uhat_k^S \|_{2,\infty} &\leq \| \U_k\mathbf{W}_k^S - %\uhat_k^S \|_{2,\infty} + \mu_0 \sqrt{\frac{r}{p}} \\
%      &\leq 2 \mu_0 \sqrt{\frac{r}{p}}.
%  \end{align*}
  Therefore, we have established the base case, which holds with probability $1 - O(p^{-19}) \geq 1 - 3p^{-15}$, as long as $C_0$ in the definition of $\dl$ satisfies $C_0 \geq C$, with $C$ as above.  
   \\ \ \\ \noindent
  \textbf{Step 2: Induction Step} \\
  Suppose that for all $t \leq t_0-1$ it holds that with probability at least $1 - 3 t_0 p^{-15}$ that
 \begin{align*}
    \max_k \| \uhat_k^{(t)} - \U_k \mathbf{W}_k^{(t)} \|_{2,\infty} \leq \frac{\dl\ku}{\lambda} \mu_0 \sqrt{\frac{r_k}{p_k}} + \frac{1}{2^{t}} \mu_0 \sqrt{\frac{r_k}{p_k}}; \\
  \max_k    \max_{m} \max_j\| \sin\Theta(\uhat_j^{(t)}, \utilde_j^{t,k-m}) \| \leq \frac{\dl\ku}{\lambda}\mu_0 \sqrt{\frac{r_k}{p_j}} + \frac{1}{2^{t}} \mu_0 \sqrt{\frac{r_k}{p_j}}.
 \end{align*}
%  Define the additional events
%   \begin{align*}
%       \mathcal{E}_{1,\mathrm{main}}^{(t_0-1)} &\coloneqq \bigcap_{t=1}^{(t_0-1)} \bigcap_{k=1}^{3}  \mathcal{E}^{t,k}_{2,\infty} \cap  \bigcap_{m=1}^{p_k} \mathcal{E}_{k-m}^{(t)}; \\
%       \mathcal{E}_{2,\mathrm{main}}^{(t_0-1)} &\coloneqq\mathcal{E}_{1,\mathrm{main}}^{(t_0-1)} \cap\bigg( \bigcap_{m=1}^{p_k} \mathcal{E}_{1-m}^{t_0} \bigg) \cap  \mathcal{E}^{t_0,1}_{2,\infty} ; \\
%         \mathcal{E}_{3,\mathrm{main}}^{(t_0-1)} &\coloneqq\mathcal{E}_{2,\mathrm{main}}^{(t_0-1)} \cap\bigg( \bigcap_{m=1}^{p_k} \mathcal{E}_{2-m}^{t_0} \bigg) \cap  \mathcal{E}^{t_0,2}_{2,\infty},
%   \end{align*}
  Observe that the induction hypothesis is equivalent to stating that $\mathcal{E}_{1,\mathrm{main}}^{(t_0-1)}$ holds with probability at least $1 - 3 t_0 p^{-15}$.  We will now show that with probability at least $1 - 3 t_0 p^{-15} - p^{-15}$ that $\mathcal{E}_{2,\mathrm{main}}^{(t_0-1)}$ holds, which is equivalent to showing that
  \begin{align*}
      \| \uhat_1^{t_0} - \U_1 \mathbf{W}_1^{t_0} \|_{2,\infty} &\leq \frac{\dl^{(1)}}{\lambda} \mu_0 \sqrt{\frac{r_1}{p_1}} + \frac{1}{2^{t_0}} \mu_0 \sqrt{\frac{r_1}{p_1}}; \\
      \max_{m} \max_j \| \sin\Theta( \uhat^{t_0}_j, \utilde_j^{t_0,1-m} ) \| &\leq \frac{\dl^{(1)}}{\lambda} \mu_0 \sqrt{\frac{r_1}{p_j}} + \frac{1}{2^{t_0}} \mu_0 \sqrt{\frac{r_1}{p_j}}.
  \end{align*}
  In other words, we will show that all of the bounds for the first mode hold.
%   which is equivalent to saying that $\mathcal{E}_{2,\mathrm{main}}^{(t_0-1)}$ holds with probability at least $1 - 3 t_0 p^{-15}-p^{-15}$.  
Note that  
  \begin{align*}
      \hat{\U}_1^{t_0} -\U_1\U_1^\top\hat{\U}_1^{t_0} &= \U_{1\perp} \U_{1\perp}\t \mathbf{Z}_1 \bigg[ \mathcal{P}_{\hat{\U}_{2}^{(t_0-1)}} \otimes \mathcal{P}_{\hat{\U}_{3}^{(t_0-1)}} \bigg] \mathbf{Z}_1\t \hat{\U}_1^{t_0}(\mathbf{\hat{\Lambda}}_1^{(t_0)})^{-2} \\
      &\qquad + \U_{1\perp} \U_{1\perp}\t \mathbf{Z}_1 \bigg[ \mathcal{P}_{\hat{\U}_{2}^{(t_0-1)}} \otimes \mathcal{P}_{\hat{\U}_{3}^{(t_0-1)}} \bigg] \mathbf{T}_1\t \hat{\U}_1^{t_0} (\mathbf{\hat{\Lambda}}_1^{(t_0)})^{-2} \\
        &= \mathbf{Q}_{1}^{t_0} + \mathbf{L}_{1}^{t_0}.
  \end{align*}
  Therefore,
  \begin{align*}
      \uhat_1^{t_0} - \U_1 \mathbf{W}_1^{t_0} &= \uhat_1^{t_0} - \U_1 \U_1 \uhat_1^{t_0} + \U_1(\U_1 \uhat_1^{t_0} - \mathbf{W}_1^{t_0} ) \\
      &= \mathbf{Q}_{1}^{t_0} + \mathbf{L}_{1}^{t_0} + \U_1(\U_1 \uhat_1^{t_0} - \mathbf{W}_1^{t_0} ).
  \end{align*}
  Consequently,
  \begin{align*}
      e_m\t \bigg(  \uhat_1^{t_0} - \U_1 \mathbf{W}_1^{t_0}  \bigg) &= e_m\t \mathbf{Q}_{1}^{t_0} + e_m\t \mathbf{L}_{1}^{t_0} + e_m\t \U_1(\U_1 \uhat_1^{t_0} - \mathbf{W}_1^{t_0} ) .
  \end{align*}
We now proceed by bounding probabilistically.   Observe that 
  \begin{align*}
      \p\bigg\{  &\| \uhat_1^{t_0} - \U_1\mathbf{W}_1^{t_0} \|_{2,\infty} \geq \mu_0 \sqrt{\frac{r_1}{p_1}} \bigg( \frac{\dl^{(1)}}{\lambda} + \frac{1}{2^{t_0}} \bigg) \\
      &\qquad \bigcup \max_{m}\max_j \| \sin\Theta(\uhat_j^{t_0},\utilde_{j}^{t_0,1-m} \| \geq \mu_0 \sqrt{\frac{r_1}{p_j}} \bigg( \frac{\dl^{(1)}}{\lambda} + \frac{1}{2^{t_0}} \bigg) \bigg\} \\
      &\leq \p\bigg\{  \| \uhat_1^{t_0} - \U_1\mathbf{W}_1^{t_0}\|_{2,\infty} \geq \mu_0 \sqrt{\frac{r_1}{p_1}} \bigg( \frac{\dl^{(1)}}{\lambda} + \frac{1}{2^{t_0}} \bigg) \\
      &\qquad \bigcup \max_{m}\max_j \| \sin\Theta(\uhat_j^{t_0},\utilde_{j}^{t_0,1-m} \| \geq \mu_0 \sqrt{\frac{r_1}{p_j}} \bigg( \frac{\dl^{(1)}}{\lambda} + \frac{1}{2^{t_0}} \bigg) \bigcap \mathcal{E}_{\mathrm{1,main}}^{(t_0-1)} \bigg\} + 3 t_0 p^{-15} \\ 
      &\leq \p\bigg\{ \| \uhat_1^{t_0} - \U_1 \mathbf{W}_1^{t_0} \|_{2,\infty}  \geq \mu_0 \sqrt{\frac{r_1}{p_1}} \bigg( \frac{\dl^{(1)}}{\lambda} + \frac{1}{2^{t_0}} \bigg) \bigcap \mathcal{E}_{\mathrm{1,main}}^{(t_0-1)} \bigg\} \\
      &\qquad + \p\bigg\{ \max_{m}\max_j \| \sin\Theta(\uhat_j^{t_0},\utilde_{j}^{t_0,1-m} \| \geq \mu_0 \sqrt{\frac{r_1}{p_j}} \bigg( \frac{\dl^{(1)}}{\lambda} + \frac{1}{2^{t_0}} \bigg) \bigcap \mathcal{E}_{\mathrm{1,main}}^{(t_0-1)} \bigg\} + 3 t_0 p^{-15} \\ 
      &\leq \p\bigg\{ \| \uhat_1^{t_0} - \U_1 \mathbf{W}_1^{t_0} \|_{2,\infty} \geq \mu_0 \sqrt{\frac{r_1}{p_1}} \bigg( \frac{\dl^{(1)}}{\lambda} + \frac{1}{2^{t_0}} \bigg) \bigcap \mathcal{E}_{\mathrm{1,main}}^{(t_0-1)}  \bigcap \mathcal{E}_{\mathrm{Good}}\bigg\} \\
      &\qquad + \p\bigg\{ \max_{m}\max_j \| \sin\Theta(\uhat_j^{t_0},\utilde_{j}^{t_0,1-m} \| \geq \mu_0 \sqrt{\frac{r_1}{p_j}} \bigg( \frac{\dl^{(1)}}{\lambda} + \frac{1}{2^{t_0}} \bigg) \bigcap \mathcal{E}_{\mathrm{1,main}}^{(t_0-1)}  \bigcap \mathcal{E}_{\mathrm{Good}}\bigg\} \\
      &\qquad + O(p^{-30})+ 3 t_0 p^{-15} \\
      &\leq p \max_{m} \p\bigg\{ \| e_m\t \big(\uhat_1^{t_0} - \U_1 \mathbf{W}_1^{t_0}\big) \| \geq \mu_0 \sqrt{\frac{r_1}{p_1}} \bigg( \frac{\dl^{(1)}}{\lambda} + \frac{1}{2^{t_0}} \bigg) \bigcap \mathcal{E}_{\mathrm{1,main}}^{(t_0-1)}  \bigcap \mathcal{E}_{\mathrm{Good}}\bigg\} \\
      &\qquad + p \max_{m} \p\bigg\{ \max_j \| \sin\Theta(\uhat_j^{t_0},\utilde_{j}^{t_0,1-m} \| \geq \mu_0 \sqrt{\frac{r_1}{p_j}} \bigg( \frac{\dl^{(1)}}{\lambda} + \frac{1}{2^{t_0}} \bigg) \bigcap \mathcal{E}_{\mathrm{1,main}}^{(t_0-1)}  \bigcap \mathcal{E}_{\mathrm{Good}}\bigg\} \\
      &\qquad + O(p^{-30})+ 3 t_0 p^{-15} \\
      &\leq p \max_{m} \Bigg[ \p\bigg\{ \| e_m\t \mathbf{L}_1^{t_0} \| \geq \frac{1}{4}\mu_0 \sqrt{\frac{r_1}{p_1}} \bigg( \frac{\dl^{(1)}}{\lambda} + \frac{1}{2^{t_0}} \bigg) \bigcap \mathcal{E}_{\mathrm{1,main}}^{(t_0-1)}  \bigcap \mathcal{E}_{\mathrm{Good}}\bigg\} \\
      &\qquad + \p\bigg\{ \| e_m\t \mathbf{Q}_1^{t_0} \| \geq \frac{1}{4}\mu_0 \sqrt{\frac{r_1}{p_1}} \bigg( \frac{\dl^{(1)}}{\lambda} + \frac{1}{2^{t_0}} \bigg) \bigcap \mathcal{E}_{\mathrm{1,main}}^{(t_0-1)} \bigcap \mathcal{E}_{\mathrm{Good}}\bigg\} \\
      &\qquad + \p\bigg\{ \| e_m\t \U_1 \big(\U_1\t \uhat_1^{t_0} - \mathbf{W}_1^{t_0} \big) \| \geq  \frac{1}{4}\mu_0 \sqrt{\frac{r_1}{p_1}} \bigg( \frac{\dl^{(1)}}{\lambda} + \frac{1}{2^{t_0}} \bigg) \bigcap \mathcal{E}_{\mathrm{1,main}}^{(t_0-1)} \bigcap \mathcal{E}_{\mathrm{Good}}\bigg\}  \Bigg]\\
      &\qquad + p \max_{m} \p\bigg\{ \max_j \| \sin\Theta(\uhat_j^{t_0},\utilde_{j}^{t_0,k-m} \| \geq \mu_0 \sqrt{\frac{r_1}{p_j}} \bigg( \frac{\dl^{(1)}}{\lambda} + \frac{1}{2^{t_0}} \bigg) \bigcap \mathcal{E}_{\mathrm{1,main}}^{(t_0-1)}  \bigcap \mathcal{E}_{\mathrm{Good}}\bigg\} \\
      &\qquad + O(p^{-30})+ 3 t_0 p^{-15} \\
      &\leq 3  p^{-28} + p^{-28} + O(p^{-30}) + 3 t_0p^{-15} \\
      &\leq ( 3 t_0 + 1)p^{-15},
  \end{align*}
  for $p$ sufficiently large,   where the penultimate inequality holds by Lemmas \ref{lem:leaveoneout_goodevent}, \ref{lem:linearterm_goodevent}, and \ref{lem:quadraticterm_goodevent}, and the fact that on $\mathcal{E}_{\mathrm{Good}}$,
  \begin{align*}
      \| e_m\t \U_1 \left( \U_1\t \uhat_1^{t_0} - \mathbf{W}_1^{t_0} \right) \| \leq \frac{1}{4} \mu_0 \sqrt{\frac{r_1}{p_1}} \bigg( \frac{\dl^{(1)}}{\lambda} + \frac{1}{2^{t_0}} \bigg)
  \end{align*}
  by Lemma \ref{lem:orthogonalmatrixlemma} since \textcolor{black}{$t_0 \leq \exp(c p)$ by the assumption $\lambda \leq \exp(c p)$}. Therefore, we have shown that the bound holds for $k = 1$.  For $k = 2$, we proceed similarly, only now on the hypothesis that $\mathcal{E}_{\mathrm{main}}^{t_0-1,2}$ holds with probability at least $1 - 3 t_0 p^{-15} - p^{-15}$.  
The exact same argument goes through,accumulating an additional factor of $p^{-15}$.  Finally, for $k = 3$, we proceed again, only now assuming that  $\mathcal{E}^{t_0-1,3}_{\mathrm{main}}$ holds with probability at least $1 - 3 t_0 p^{-15} - 2p^{-15}$.  This accumulates a final factor of $p^{-15}$. Therefore, since this accumulates three factors of $p^{-15}$, it holds that $\mathcal{E}_{\mathrm{main}}^{t_0,1}$ holds with probability at least $1 - 3 (t_0 + 1) p^{-15}$ as desired, which completes the proof.
 \end{proof}

 \subsection{Initialization Bounds} \label{sec:init}
This section contains the proof  the initialization bounds.   \cref{sec:initlemmas} contains preliminary lemmas and their proofs, \cref{sec:inittwoinftyproof} contains the proof of \cref{thm:spectralinit_twoinfty}, and \cref{sec:initlooproof} contains the proof of \cref{lem:spectralinit_leave_one_out_sintheta}.  

\subsubsection{Preliminary Lemmas} \label{sec:initlemmas}
The following result establishes concentration inequalities for the spectral norm of the noise matrices, needed in order to establish sufficient eigengap conditions.  

\begin{lemma}\label{lem:spectralconcentrationinit}
The following bounds hold simultaneously with probability at least $1 - O(p^{-30}):$
\begin{enumerate}
  \item $\| \diag(\mathbf{Z}_1 \mathbf{T}_1\t) \| \lesssim  \lambda_1  \mu_0^2 r  \sqrt{\frac{\log(p)}{p_1}}$; 
    \item $\| \Gamma(\mathbf{Z}_1 \mathbf{Z}_1\t ) \| \lesssim  (p_1p_2p_3)^{1/2}$
    \item $\| \Gamma(\mathbf{T}_1 \mathbf{Z}_1\t ) \| \lesssim  \lambda_1 \sqrt{p_1}$;
  \item $\| \U_1 \mathbf{Z}_1 \mathbf{V}_1 \| \lesssim  \sqrt{r}$.
\end{enumerate}
\end{lemma}

\begin{proof}
Part one follows since 
\begin{align*}
    \| \diag(\mathbf{Z}_1 \mathbf{T}_1\t ) \| &= \max_{i} \big| e_i\t \mathbf{Z}_1 \mathbf{T}_1\t e_i \big| \\
    &\leq \max_{i}\| e_i\t \mathbf{Z}_1 \mathbf{V}_1 \| \| e_i\t \U_1 \mathbf{\Lambda}_1 \| \\
    &\leq \| \mathbf{Z}_1 \mathbf{V}_1 \|_{2,\infty} \mu_0 \sqrt{\frac{r_1}{p_1}} \lambda_1 \\
    &\leq C \sqrt{p_{-1} \log(p)} \mu_0 \sqrt{\frac{r_1}{p_1}} \lambda_1 \| \mathbf{V}_1 \|_{2,\infty} \\
    &\lesssim \lambda_1  \mu_0^2 r  \sqrt{\frac{\log(p)}{p_1}},
\end{align*}
where the final inequality holds with probability at least $1 - O(p^{-30})$ by \cref{lem:matricizationrowbound}.

Part two follows by a slight modification of Lemma 1 of \citet{agterberg_entrywise_2022} (with $M$ in the statement therein taken to be 0), where the higher probability holds by adjusting the constant in the definition of $\delta$ in the proof therein.  We omit the detailed proof for brevity.

Part three follows since
\begin{align*}
    \| \Gamma(\mathbf{Z}_1 \mathbf{T}_1\t ) \| &\leq \| \mathbf{Z}_1 \mathbf{T}_1\t \| + \| \diag(\mathbf{Z}_1 \mathbf{T}_1\t ) \| \\
    &\lesssim \| \mathbf{Z}_1 \mathbf{V}_1 \| \lambda_1 +  \lambda_1  \mu_0^2 r  \sqrt{\frac{\log(p)}{p_1}} \\
    &\lesssim  \sqrt{p_1} \lambda_1 + \lambda_1  \mu_0^2 r  \sqrt{\frac{\log(p)}{p_1}} \\
    &\lesssim \sqrt{p_1} \lambda_1,
\end{align*}
where the penultimate inequality $\| \mathbf{Z}_1 \mathbf{V}_1 \| \lesssim \sqrt{p_1}$ holds by a standard $\varepsilon-$net argument, and the final inequality holds since  $\mu_0^2 r \lesssim \sqrt{p_{\min}}$ by assumption. 

Part four follows via a standard $\varepsilon$-net argument.
\end{proof}

We also have the following result, needed in establishing the concentration of the leave-one-out sequences.
\begin{lemma}\label{lem:spectralinitconcentrationloo}
The following bounds hold with probability at least $1 - O(p^{-30})$:
\begin{enumerate}
    \item $\| \Gamma(\mathbf{T}_1 (\mathbf{Z}_1 - \mathbf{Z}_1^{1-m})\t) \| \lesssim  \lambda_1 \mu_0 \sqrt{r\log(p)}$;
    \item $\| \Gamma(\mathbf{Z}_1 \mathbf{Z}_1\t - \mathbf{Z}_1^{1-m} (\mathbf{Z}_1^{1-m})\t ) \| \lesssim (p_1p_2p_3)^{1/2}$
    \item $\| \Gamma(\mathbf{T}_1 (\mathbf{Z}_1^{1-m} - \mathbf{Z}_1^{1-m,1-l} )\t ) \| \lesssim \mu_0 \lambda_1 \sqrt{r\log(p)} \sqrt{\frac{p_1}{p_{-1}}}$;
    \item $\| \Gamma( \mathbf{Z}_1^{1-m}(\mathbf{Z}_1^{1-m})\t -\mathbf{Z}_1^{1-m,1-l} (\mathbf{Z}_1^{1-m,1-l} )\t ) \| \lesssim p$.
    \item  $\| \Gamma( \mathbf{Z}_1^{1-m}(\mathbf{Z}_1^{1-m}) - \mathbf{Z}_1^{1-m,1-l} (\mathbf{Z}_1^{1-m,1-l})\t ) \utilde_1^{S,1-m,1-l} \| \lesssim  \sqrt{p_1} \log(p) \| \utilde_1^{S,1-m,1-l} \|_{2,\infty}$
\end{enumerate}
Here $\mathbf{Z}_1^{1-m,1-l}$ is the matrix $\mathbf{Z}_1$ with its $m$'th row and $l$'th column removed, and $\utilde_1^{S,1-m,1-l}$ is matrix of leading eigenvectors obtained by initializing with the noise matrix $\mathbf{Z}_1$ replaced with $\mathbf{Z}_1^{1-m,1-l}$.
\end{lemma}

\begin{proof}[Proof of \cref{lem:spectralinitconcentrationloo}]
For part one, we observe that $\mathbf{Z}_1 - \mathbf{Z}_1^{1-m}$ is a zero matrix with only its $m$'th row nonzero.  Therefore,
\begin{align*}
    \| \Gamma\big(\mathbf{Z}_1 - \mathbf{Z}_1^{1-m}) \mathbf{T}_1\t \big) \| &\leq  \| ( \mathbf{Z}_1 - \mathbf{Z}_1^{1-m}) \mathbf{T}_1\t \| + \| \diag\big((\mathbf{Z}_1 - \mathbf{Z}_1^{1-m}) \mathbf{T}_1\t \big) \| \\
    &= \| e_m\t \mathbf{Z}_1 \mathbf{T}_1\t \| + \max_{i} | e_i\t \big( \mathbf{Z}_1 - \mathbf{Z}_1^{1-m} \big) \mathbf{T}_1\t e_i | \\
    &\leq \| \mathbf{Z}_1 \mathbf{V}_1 \|_{2,\infty} \lambda_1 + \| \mathbf{Z}_1 \mathbf{V}_1 \|_{2,\infty} \lambda_1 \mu_0 \sqrt{\frac{r}{p_1}} \\
    &\lesssim \| \mathbf{Z}_1 \mathbf{V}_1 \|_{2,\infty} \lambda_1 \\
    &\lesssim \lambda_1  \sqrt{p_{-1}\log(p)} \|\mathbf{V}_1 \|_{2,\infty} \\
    &\lesssim \lambda_1 \mu_0 \sqrt{r\log(p)},
\end{align*}
with probability at least $1 - O(p^{-30})$, where the penultimate line follows from \cref{lem:rowbound}.

For part 2, we first observe that $\Gamma\big(\mathbf{Z}_1 \mathbf{Z}_1\t - \mathbf{Z}_1^{1-m} (\mathbf{Z}_1^{1-m})\t\big)$ is a matrix with $i,j$ entry equal to
\begin{align*}
    \Gamma\big(\mathbf{Z}_1 \mathbf{Z}_1\t - \mathbf{Z}_1^{1-m} (\mathbf{Z}_1^{1-m})\t\big)_{ij} &= \begin{cases} \langle e_m\t \mathbf{Z}_1, e_j\t\mathbf{Z}_1 \rangle & i =m, j\neq m \\
    \langle e_m\t \mathbf{Z}_1, e_i\t\mathbf{Z}_1 \rangle & j = m, i\neq m \\
    0 &\text{else}. \end{cases}
\end{align*}
Therefore, we can decompose this matrix via
\begin{align*}
    \Gamma\big(\mathbf{Z}_1 \mathbf{Z}_1\t - \mathbf{Z}_1^{1-m} (\mathbf{Z}_1^{1-m})\t\big) &= \mathbf{G}_{\mathrm{row}} + \mathbf{G}_{\mathrm{col}}, 
\end{align*}
where $\mathbf{G}_{\mathrm{row}}$ is the matrix whose only nonzero row is its $m$'th row, in which case it the $m,j$ entry is $\langle e_m\t\mathbf{Z}_1, e_j\t\mathbf{Z}_1\rangle$ for $j \neq m$, and $\mathbf{G}_{\mathrm{col}}$ is defined as the transpose of this matrix.  We then observe that with high probability
\begin{align*}
    \|  \Gamma\big(\mathbf{Z}_1 \mathbf{Z}_1\t - \mathbf{Z}_1^{1-m} (\mathbf{Z}_1^{1-m})\t\big) \| &\leq \| \mathbf{G}_{\mathrm{row}} \| + \| \mathbf{G}_{\mathrm{col}} \| \\
    &\leq 2 \| \mathbf{G}_{\mathrm{row}} \| \\
    &= 2 \| e_m\t \Gamma(\mathbf{Z}_1 \mathbf{Z}_1\t ) \| \\
    &\leq 2 \| \Gamma(\mathbf{Z}_1 \mathbf{Z}_1\t ) \| \\
    &\lesssim (p_1p_2p_3)^{1/2},
\end{align*}
where the final inequality follows from \cref{lem:spectralconcentrationinit}.  

For part three, we first observe that $\mathbf{Z}_1^{1-m} - \mathbf{Z}_1^{1-m,1-l}$  is a matrix with only its $l$'th column nonzero.  Then we note
\begin{align*}
    \|\Gamma\big( \big( \mathbf{Z}_1^{1-m} - \mathbf{Z}_1^{1-m,1-l} \big) \mathbf{T}_1\t \big)\| &\leq \sqrt{p_1}   \|\Gamma\big( \big( \mathbf{Z}_1^{1-m} - \mathbf{Z}_1^{1-m,1-l} \big) \mathbf{T}_1\t \big)\|_{2,\infty} \\
    &\leq \sqrt{p_1} \bigg[  \| \big( \mathbf{Z}_1^{1-m} - \mathbf{Z}_1^{1-m,1-l} \big) \mathbf{T}_1\t \|_{2,\infty} \\
    &\qquad + \| \diag\bigg(\big( \mathbf{Z}_1^{1-m} - \mathbf{Z}_1^{1-m,1-l} \big)\mathbf{T}_1\t\bigg) \|_{2,\infty} \bigg] \\
    &\leq \sqrt{p_1} \bigg[ \max_{i} \| \big(\mathbf{Z}_1\big)_{il}\big( \mathbf{T}_1\t \big)_{l\cdot} \| + \max_{i} | (\mathbf{Z}_1)_{il} (\mathbf{T}_1\t)_{li} | \bigg] \\
    &\lesssim \sqrt{p_1} \bigg[ \sqrt{\log(p)} \| \mathbf{T}_1\t \|_{2,\infty} + \sqrt{\log(p)} \| \mathbf{T}_1\t \|_{\max} \bigg] \\
    &\lesssim \sqrt{p_1\log(p)} \| \mathbf{T}_1\t \|_{2,\infty} \\
    &\lesssim \sqrt{p_1\log(p)}  \| \mathbf{V}_1 \|_{2,\infty} \lambda_1\\
    &\lesssim \mu_0 \lambda_1 \sqrt{r\log(p)} \sqrt{\frac{p_1}{p_{-1}}},
%    \sqrt{\frac{r\log(p)}{p}},
\end{align*}
where we used the fact that $\max_{i,l} |(\mathbf{Z}_1)_{il}| \lesssim \sqrt{\log(p)}$ with probability at least $1 - O(p^{-30})$.

We next note that $\Gamma\big(\mathbf{Z}_1^{1-m}(\mathbf{Z}_1^{1-m})\t -\mathbf{Z}_1^{1-m,1-l} (\mathbf{Z}_1^{1-m,1-l})\t\big)$ is a matrix with entries equal to $(\mathbf{Z}_1)_{il}(\mathbf{Z}_1)_{jl}$ for $i \neq j$ and $i,j \neq m$.  In particular, it is a the $p_1- 1 \times p_1 -1$ dimensional submatrix of the matrix whose entries are simply $(\mathbf{Z}_1)_{il}(\mathbf{Z}_1)_{jl}$ for $i\neq j$.  This is a sample Gram matrix, so by Lemma 1 of \citet{agterberg_entrywise_2022}, it holds that
\begin{align*}
    \| \Gamma\big(\mathbf{Z}_1^{1-m}(\mathbf{Z}_1^{1-m})\t -\mathbf{Z}_1^{1-m,1-l} (\mathbf{Z}_1^{1-m,1-l})\t\big) \| &\lesssim p_1
\end{align*}
with probability at least $1 - O(p^{-30})$ (where as in the proof of \cref{lem:spectralconcentrationinit} the result holds by taking $M = 0$, $d = 1$, and modifying the constant on $\delta$ in the proof of Lemma 1 of \citet{agterberg_entrywise_2022}).  

For the final term, we note that
\begin{align*}
   \|& \Gamma( \mathbf{Z}_1^{1-m}(\mathbf{Z}_1^{1-m}) - \mathbf{Z}_1^{1-m,1-l} (\mathbf{Z}_1^{1-m,1-l})\t ) \utilde_1^{S,1-m,1-l} \|_{2,\infty} \\
   &= \max_{a} \big\| \sum_{j\neq a} (\mathbf{Z}_1)_{al} (\mathbf{Z}_1)_{jl} \big(\utilde^{S,1-m,1-l}_1 \big)_{j\cdot} \big\| \\
   &\leq \max_{a} | (\mathbf{Z}_1) |_{al}  \| \sum_{j\neq a}(\mathbf{Z}_1)_{jl}  \big(\utilde^{S,1-m,1-l}_1 \big)_{j\cdot} \big\| \\
   &\lesssim \sqrt{\log(p)} \max_{a}  \| \sum_{j\neq a}(\mathbf{Z}_1)_{jl}  \big(\utilde^{S,1-m,1-l}_1 \big)_{j\cdot} \big\| \\
   &\lesssim \sqrt{p_1} \log(p) \| \utilde_1^{S,1-m,1-l} \|_{2,\infty},
\end{align*}
where the final inequality follows from \cref{lem:rowbound} applied to $\mathbf{Z}_1\t$.  
\end{proof}

The following result verifies the eigengap conditions that we use repeatedly throughout the proof.  We adopt similar notation to \citet{cai_subspace_2021}.
\begin{lemma}\label{lem:spectraliniteigengapsloo}
Define the matrices
\begin{align*}
    \mathbf{G} &\coloneqq   \Gamma\big( \mathbf{T}_1 \mathbf{T}_1\t + \mathbf{T}_1 \mathbf{Z}_1\t  + \mathbf{Z}_1 \mathbf{T}_1\t + \mathbf{Z}_1 \mathbf{Z}_1\t \big); \\
    \mathbf{G}^{(m)} &\coloneqq   \Gamma\big( \mathbf{T}_1 \mathbf{T}_1\t + \mathbf{T}_1 (\mathbf{Z}_1^{-m})\t  + \mathbf{Z}_1^{-m} \mathbf{T}_1\t + \mathbf{Z}_1^{-m} (\mathbf{Z}_1^{-m})\t)\big); \\
    \mathbf{G}^{(m,l)} &\coloneqq \Gamma\big(   \mathbf{T}_1 \mathbf{T}_1\t + \mathbf{T}_1 (\mathbf{Z}_1^{-m-l})\t  + \mathbf{Z}_1^{-m-l} \mathbf{T}_1\t + \mathbf{Z}_1^{-m-l} (\mathbf{Z}_1^{-m-l})\t) \big).
\end{align*}
Then on the events in \cref{lem:spectralconcentrationinit} and \cref{lem:spectralinitconcentrationloo}, it holds that
\begin{align*}
    \lambda_r^2 - \| \mathbf{G} - \mathbf{T}_1 \mathbf{T}_1\t \| \gtrsim \lambda^2;\\
    \lambda_r \big( \mathbf{G}\big) - \lambda_{r+1}\big( \mathbf{G} \big) - \| \mathbf{G} - \mathbf{G}^{-m} \| \gtrsim \lambda^2; \\
    \lambda_r \big( \mathbf{G}^{-m}\big) - \lambda_{r+1} \big( \mathbf{G}^{-m} \big) - \| \mathbf{G}^{-m} - \mathbf{G}^{-m-l} \| \gtrsim \lambda^2; \\
    \lambda_r \big(\mathbf{G} \big) \gtrsim \lambda^2.
\end{align*}
\end{lemma}

\begin{proof}[Proof of \cref{lem:spectraliniteigengapsloo}]
First, we note that on the event in \cref{lem:spectralconcentrationinit},
\begin{align*}
    \| \mathbf{G} - \mathbf{T}_1 \mathbf{T}_1\t \| &\leq \| \diag( \mathbf{T}_1 \mathbf{T}_1\t ) \| + 2 \| \Gamma(\mathbf{T}_1 \mathbf{Z}_1\t ) \| + \| \Gamma(\mathbf{Z}_1 \mathbf{Z}_1\t ) \| \\
    &\lesssim  \lambda_1^2 \mu_0^2 \frac{r}{p_1} + \lambda_1 \sqrt{p_1} + (p_1p_2p_3)^{1/2} \\
    &\ll \lambda^2; \\
    \| \mathbf{G} - \mathbf{G}^{-m} \| &\leq 2 \| \Gamma(\mathbf{T}_1 (\mathbf{Z}_1 - \mathbf{Z}_1^{1-m})\t ) \| + \| \Gamma(\mathbf{Z}_1 \mathbf{Z}_1\t - \mathbf{Z}_1^{1-m}(\mathbf{Z}_1^{1-m})\t \| \\
    &\lesssim \lambda_1 \mu_0 \sqrt{r\log(p)} + (p_1p_2p_3)^{1/2} \\
    &\ll \lambda^2; \\
    \| \mathbf{G}^{-m} - \mathbf{G}^{-m-l} \| &\leq \| \Gamma\big( \mathbf{T}_1 ( \mathbf{Z}_1^{1-m} - \mathbf{Z}_1^{1-m,1-l})\t \big) \| + \left\|\Gamma\left(\mathbf{Z}_{1}^{1-m}\left(\mathbf{Z}_{1}^{1-m}\right)^{\top}-\mathbf{Z}_{1}^{1-m, 1-l}\left(\mathbf{Z}_{1}^{1-m, 1-l}\right)^{\top}\right)\right\| \\
    &\lesssim \mu_0 \lambda_1 \sqrt{r\log(p)} \sqrt{\frac{p_1}{p_{-1}}} + p_1 \\
    &\ll \lambda^2.
\end{align*}
Therefore, by Weyl's inequality,
\begin{align*}
    \lambda_r\big( \mathbf{G} \big) - \lambda_{r+1}\big( \mathbf{G}\big) - \| \mathbf{G}- \mathbf{T}_1 \mathbf{T}_1\t \|&\geq \lambda_r^2 - 3  \| \mathbf{G}- \mathbf{T}_1 \mathbf{T}_1\t \| \\
    &\gtrsim \lambda^2.
\end{align*}
Similarly,
\begin{align*}
    \lambda_r\big( \mathbf{G}^{-m} \big) - \lambda_{r+1}\big(\mathbf{G}^{-m} \big) - \| \mathbf{G} - \mathbf{G}^{-m} \| &\gtrsim \lambda_r\big( \mathbf{G} \big) - \lambda_{r+1}\big( \mathbf{G}\big) - 2  \| \mathbf{G} - \mathbf{G}^{-m} \| \\
    &\gtrsim \lambda^2. 
\end{align*}
Finally, we note that
\begin{align*}
    \lambda_r\big(\mathbf{G} \big) \geq \lambda_r^2 - \| \mathbf{G} - \mathbf{T}_1 \mathbf{T}_1\t \| \\
    &\geq \lambda^2 - \| \mathbf{G} - \mathbf{T}_1 \mathbf{T}_1\t \| \\
    &\gtrsim \lambda^2.
\end{align*}
This completes the proof.
\end{proof}

With these spectral norm concentration and eigengap conditions fixed, we now consider the $\ell_{2,\infty}$ analysis.  The first step is to show that several terms are negligible with respect to the main bound. For the remainder of the analysis, we implicitly use the eigenvalue bounds in \cref{lem:spectraliniteigengapsloo}, which hold under the events in \cref{lem:spectralconcentrationinit} and \cref{lem:spectralinitconcentrationloo}.

\begin{lemma}\label{lem:negligibleterms}
The following bounds hold with probability at least $1- O(p^{-30})$:
\begin{enumerate}
    \item  $\| \U_1 \U_1\t \mathbf{Z}_1 \mathbf{T}_1\t \uhat_1^S \mathbf{\hat{\Lambda}}_1^{-2} \|_{2,\infty}\lesssim \mu_0 \sqrt{\frac{r_1}{p_1}} \frac{\sqrt{r} \kappa}{\lambda}$ 
    \item $\| \U_1\U_1\t \Gamma( \mathbf{Z}_1 \mathbf{Z}_1\t ) \uhat_1^S \mathbf{\hat{\Lambda}}^{-2}_1 \|_{2,\infty} \lesssim \mu_0 \sqrt{\frac{r_1}{p_1}} \frac{ (p_1p_2p_3)^{1/2}}{\lambda^2}$
    \item $\| \U_1 \U_1\t \diag\bigg( \mathbf{T}_1 \mathbf{T}_1\t + \mathbf{T}_1\mathbf{Z}_1\t + \mathbf{Z}_1\mathbf{T}_1\t \bigg) \uhat_1^S \mathbf{\hat{\Lambda}}^{-2}_1 \|_{2,\infty} \lesssim \mu_0 \sqrt{\frac{r_1}{p_1}} \bigg( \kappa^2 \mu_0^2 \frac{r}{p_1} + \frac{\kappa \mu_0^2 r \sqrt{\log(p)}}{\lambda \sqrt{p_1}} \bigg) $
    \item $\|\diag( \mathbf{T}_1 \mathbf{T}_1\t +\mathbf{T}_1 \mathbf{Z}_1\t  + \mathbf{Z}_1 \mathbf{T}_1\t ) \uhat_1^S \mathbf{\hat{\Lambda}}^{-2}_1 \|_{2,\infty}\lesssim \kappa^2 \mu_0^2 \frac{r_1}{p_1} + \frac{\kappa \mu_0^2 r \sqrt{\log(p)}}{\lambda \sqrt{p_1}}$
    \item $\|\U_1 ( \mathbf{W}_1^S - \U_1\t \uhat_1^S)\|_{2,\infty}\lesssim \mu_0 \sqrt{\frac{r_1}{p_1}} \bigg( \kappa^2 \mu_0^2 \frac{r}{p_1} + \frac{\kappa \sqrt{p_1}}{\lambda} + \frac{(p_1p_2p_3)^{1/2}}{\lambda^2} \bigg)^2$.
\end{enumerate}
Moreover, all of these terms are upper bounded by the quantity
\begin{align*}
    \bigg(\frac{\kappa \sqrt{p_1\log(p)}}{\lambda} + \frac{(p_1p_2p_3)^{1/2}\log(p)}{\lambda^2} + \kappa^2 \mu_0 \sqrt{\frac{r}{p_1}}\bigg) \mu_0\sqrt{\frac{r}{p_1}}.
\end{align*}
\end{lemma}

\begin{proof}[Proof of \cref{lem:negligibleterms}]
For part one, we note that
\begin{align*}
\| \U_1 \U_1\t \mathbf{Z}_1 \mathbf{T}_1\t \uhat_1^S \mathbf{\hat{\Lambda}}_1^{-2} \|_{2,\infty}  &\lesssim  \| \U_1 \|_{2,\infty}  \frac{ \| \U_1\t \mathbf{Z}_1 \mathbf{V}_1 \| \lambda_1}{\lambda^2} \\
&\lesssim \mu_0 \sqrt{\frac{r_1}{p_1}} \frac{\sqrt{r} \kappa}{\lambda}, \end{align*}
since $\| \U_1\t \mathbf{Z}_1 \mathbf{V}_1 \| \lesssim \sqrt{r}$ by \cref{lem:spectralconcentrationinit}.

For part 2, we note
\begin{align*}
    \| \U_1\U_1\t \Gamma( \mathbf{Z}_1 \mathbf{Z}_1\t ) \uhat_1^S \mathbf{\hat{\Lambda}}^{-2}_1 \|_{2,\infty}     &\lesssim \mu_0 \sqrt{\frac{r_1}{p_1}} \frac{ \| \Gamma( \mathbf{Z}_1 \mathbf{Z}_1\t) \|}{\lambda^2} \\
    &\lesssim \mu_0 \sqrt{\frac{r_1}{p_1}} \frac{ (p_1p_2p_3)^{1/2}}{\lambda^2}.% \\
    %&\lesssim \frac{\mu_0 p\sqrt{r_1}}{\lambda^2}
\end{align*}
by \cref{lem:spectralconcentrationinit}. 

For part 3,
\begin{align*}
   \| \U_1 &\U_1\t \diag\bigg( \mathbf{T}_1 \mathbf{T}_1\t + \mathbf{T}_1\mathbf{Z}_1\t + \mathbf{Z}_1\mathbf{T}_1\t \bigg) \uhat_1^S \mathbf{\hat{\Lambda}}^{-2}_1 \|_{2,\infty} \\
   &\lesssim \mu_0 \sqrt{\frac{r_1}{p_1}} \frac{\| \diag\bigg( \mathbf{T}_1 \mathbf{T}_1\t + \mathbf{T}_1\mathbf{Z}_1\t + \mathbf{Z}_1\mathbf{T}_1\t \bigg) \|}{\lambda^2} \\
   &\lesssim \mu_0 \sqrt{\frac{r_1}{p_1}} \bigg( \frac{1}{\lambda^2} \max_{i} | e_i\t \U_1 \mathbf{\Lambda}_1^2 \U_1\t e_i | + \frac{2 \| \diag(\mathbf{Z}_1 \mathbf{T}_1\t ) \|}{\lambda^2} \bigg) \\
   &\lesssim \mu_0 \sqrt{\frac{r_1}{p_1}} \bigg( \kappa^2 \mu_0^2 \frac{r}{p_1} + \frac{ 2 \| \diag( \mathbf{Z}_1 \mathbf{T}_1\t ) \|}{\lambda^2} \bigg) \\
   &\lesssim \mu_0 \sqrt{\frac{r_1}{p_1}} \bigg( \kappa^2 \mu_0^2 \frac{r}{p_1} + \frac{\kappa \mu_0^2 r \sqrt{\log(p)}}{\lambda \sqrt{p_1}} \bigg),
\end{align*}
by \cref{lem:spectralconcentrationinit}.

For part 4, 
\begin{align*}
   \|\diag( \mathbf{T}_1 \mathbf{T}_1\t +\mathbf{T}_1 \mathbf{Z}_1\t  + \mathbf{Z}_1 \mathbf{T}_1\t ) \uhat_1^S \mathbf{\hat{\Lambda}}^{-2}_1 \|_{2,\infty} 
   &\lesssim \frac{1}{\lambda^2} \bigg(  \|\diag( \mathbf{T}_1 \mathbf{T}_1\t ) \| + 2\| \diag( \mathbf{Z}_1 \mathbf{T}_1\t ) \| \bigg) \\
   &\lesssim \kappa^2 \mu_0^2 \frac{r_1}{p_1} + \frac{\| \diag(\mathbf{Z}_1 \mathbf{T}_1\t ) \|}{\lambda^2} \\
   &\lesssim \kappa^2 \mu_0^2 \frac{r_1}{p_1} + \frac{\kappa \mu_0^2 r \sqrt{\log(p)}\|}{\lambda \sqrt{p_1}}
\end{align*}
by \cref{lem:spectralconcentrationinit}.

Finally, by \cref{lem:orthogonalmatrixlemma},
\begin{align*}
\|\U_1 ( \mathbf{W}_1^S - \U_1\t \uhat_1^S)\|_{2,\infty} 
&\leq \mu_0 \sqrt{\frac{r_1}{p_1}} \| \sin\Theta(\uhat_1^S, \U_1 ) \|^2.
\end{align*}
Note that $\U_1$ are the eigenvectors of $\mathbf{T}_1 \mathbf{T}_1\t$ and $\uhat_1^S$ are the eigenvectors of the matrix
\begin{align*}
    \Gamma\bigg(\mathbf{T}_1 \mathbf{T}_1\t + \mathbf{Z}_1 \mathbf{T}_1\t + \mathbf{T}_1 \mathbf{Z}_1\t + \mathbf{Z_1 Z_1}\t \bigg).
\end{align*}
We note that
\begin{align*}
   \bigg\| \mathbf{T}_1 \mathbf{T}_1\t -\Gamma\bigg(\mathbf{T}_1 \mathbf{T}_1\t + \mathbf{Z}_1 \mathbf{T}_1\t + \mathbf{T}_1 \mathbf{Z}_1\t + \mathbf{Z_1 Z_1}\t \bigg)\bigg\| &\leq \| \diag(\mathbf{T}_1 \mathbf{T}_1\t) \| + 2\| \Gamma(\mathbf{Z}_1 \mathbf{T}_1\t ) \| + \| \Gamma(\mathbf{Z}_1 \mathbf{Z}_1\t ) \| \\
   &\lesssim \lambda_1^2 \mu_0^2 \frac{r}{p_1} + \lambda_1 \sqrt{p_1} + (p_1p_2p_3)^{1/2} \\
   &\ll  \lambda^2,
\end{align*}
where we note that we used the fact that $\lambda \gtrsim \kappa p_{\max} \sqrt{\log(p)} p_{\min}^{1/4}$, the fact that $\mu_0^2 r\lesssim \sqrt{p_{\min}}$, and the assumption $\kappa \lesssim p_{\min}^{1/4}$.  Therefore, by the Davis-Kahan Theorem,
\begin{align*}
    \| \sin\Theta(\uhat_1^S, \U_1 ) \| &\lesssim \frac{\lambda_1^2 \mu_0^2 \frac{r}{p_1} + \lambda_1 \sqrt{p_1} + (p_1p_2p_3)^{1/2}}{\lambda^2} \\
    &\lesssim \kappa^2 \mu_0^2 \frac{r}{p_1} + \frac{\kappa \sqrt{p_1}}{\lambda} + \frac{(p_1p_2p_3)^{1/2}}{\lambda^2}.
\end{align*}
Therefore,
\begin{align*}
    \|\U_1 ( \mathbf{W}_1^S - \U_1\t \uhat_1^S)\|_{2,\infty} 
&\lesssim \mu_0 \sqrt{\frac{r_1}{p_1}} \bigg( \kappa^2 \mu_0^2 \frac{r}{p_1} + \frac{\kappa \sqrt{p_1}}{\lambda} + \frac{(p_1p_2p_3)^{1/2}}{\lambda^2} \bigg)^2.
\end{align*}

\end{proof}

\subsubsection{Proof of \cref{thm:spectralinit_twoinfty}} \label{sec:inittwoinftyproof}

\begin{proof}[Proof of \cref{thm:spectralinit_twoinfty}]
Without loss of generality, we consider $k = 1$. We simply decompose
\begin{align*}
    \uhat^S_1 - \U_1 \mathbf{W}_1^S &= \mathbf{Z}_1 \mathbf{T}_1\t \uhat^S_1 \mathbf{\hat{\Lambda}}^{-2}_1 + \Gamma( \mathbf{Z}_1 \mathbf{Z}_1\t) \uhat_1^S \mathbf{\hat{\Lambda}}^{-2}_1 - \U_1 \U_1\t \mathbf{Z}_1 \mathbf{T}_1\t \uhat_1^S \mathbf{\hat{\Lambda}}^{-2}_1 - \U_1\U_1\t \Gamma( \mathbf{Z}_1 \mathbf{Z}_1\t ) \uhat_1^S \mathbf{\hat{\Lambda}}^{-2}_1 \\
&\quad + \U_1 \U_1\t \diag\bigg( \mathbf{T}_1 \mathbf{T}_1\t + \mathbf{T}_1\mathbf{Z}_1\t + \mathbf{Z}_1\mathbf{T}_1\t \bigg) \uhat_1^S \mathbf{\hat{\Lambda}}^{-2}_1 \\
&\quad - \diag( \mathbf{T}_1 \mathbf{T}_1\t +\mathbf{T}_1 \mathbf{Z}_1\t  + \mathbf{Z}_1 \mathbf{T}_1\t ) \uhat_1^S \mathbf{\hat{\Lambda}}^{-2}_1 + \U_1 ( \mathbf{W}_1^S - \U_1\t \uhat_1^S) \\
&= (I) + (II) + (III) + (IV) + (V) + (VI),
\end{align*}
where
\begin{align*}
 (I) &\coloneqq  \mathbf{Z}_1 \mathbf{T}_1\t \uhat^S_1 \mathbf{\hat{\Lambda}}^{-2}_1; \\
 (II) &\coloneqq \Gamma( \mathbf{Z}_1 \mathbf{Z}_1\t) \uhat_1^S \mathbf{\hat{\Lambda}}^{-2}_1 \\
 (III) &= - \U_1 \U_1\t \mathbf{Z}_1 \mathbf{T}_1\t \uhat_1^S \mathbf{\hat{\Lambda}}^{-2}_1; \\
 (IV) &= - \U_1\U_1\t \Gamma( \mathbf{Z}_1 \mathbf{Z}_1\t ) \uhat_1^S \mathbf{\hat{\Lambda}}^{-2}_1 \\
(V) &= \U_1 \U_1\t \diag\bigg( \mathbf{T}_1 \mathbf{T}_1\t + \mathbf{T}_1\mathbf{Z}_1\t + \mathbf{Z}_1\mathbf{T}_1\t \bigg) \uhat_1^S \mathbf{\hat{\Lambda}}^{-2}_1 \\
(VI) &= - \diag( \mathbf{T}_1 \mathbf{T}_1\t +\mathbf{T}_1 \mathbf{Z}_1\t  + \mathbf{Z}_1 \mathbf{T}_1\t ) \uhat_1^S \mathbf{\hat{\Lambda}}^{-2}_1 \\
(VII) &= \U_1 ( \mathbf{W}_1^S - \U_1\t \uhat_1^S) \\
\end{align*}
We note that terms $(III) - (VII)$ are all of smaller order than the bound we desire by \cref{lem:negligibleterms} (with high probability).  With these bounds out of the way, we now turn our attention to terms $(I)$ and $(II)$.  For Term $(I)$, we simply note that by \cref{lem:rowbound}
\begin{align*}
    \| \mathbf{Z}_1 \mathbf{T}_1\t \uhat_1^S \mathbf{\hat{\Lambda}}_1^{-2} \|_{2,\infty} &\lesssim \frac{ \kappa \| \mathbf{Z}_1 \mathbf{V}_1 \|_{2,\infty} }{\lambda} \\
    &\lesssim \frac{\kappa \mu_0 \sqrt{r\log(p)}}{\lambda}. 
\end{align*}
It remains to show that the final term is of smaller order than the bound we desire, which will require the leave-one-out sequences.  Note that
\begin{align*}
    \| e_m\t \Gamma(\mathbf{Z}_1\mathbf{Z}_1\t) \uhat_1^S  \mathbf{\hat{\Lambda}}^{-2} \| &\lesssim \| e_m\t \Gamma(\mathbf{Z}_1\mathbf{Z}_1\t) \uhat_1^S - \utilde_1^{S,1-m} (\utilde_1^{S,1-m})\t \uhat_1^S \| \lambda^{-2}+ \| e_m\t \Gamma(\mathbf{Z}_1\mathbf{Z}_1\t) \utilde_1^{S,1-m} \| \lambda^{-2}\\
    &\lesssim \frac{\| \Gamma(\mathbf{Z}_1\mathbf{Z}_1\t) \|}{\lambda^2} \| \uhat_1^S (\uhat_1^S)\t - \utilde_1^{S,1-m} (\utilde_1^{S,1-m})\t \| + \| e_m\t \Gamma(\mathbf{Z}_1\mathbf{Z}_1\t) \utilde_1^{S,1-m} \|\lambda^{-2} \\
    &\coloneqq A + B.
\end{align*}
\noindent 
\textbf{The term $B$}: For this term, we note that
\begin{align*}
    e_m\t \Gamma(\mathbf{Z}_1\mathbf{Z}_1\t) \utilde_1^{S,1-m} &= \sum_{a\neq m} \langle \mathbf{Z}_{m\cdot}, \mathbf{Z}_{a\cdot} \rangle (\utilde_1^{S,1-m})_{a\cdot} \\
    &= \sum_{l} \mathbf{Z}_{ml} \bigg( \sum_{a\neq m} \mathbf{Z}_{al} (\utilde_{1}^{S,1-m})_{a\cdot} \bigg)
\end{align*}
is a sum of $p_{-1}$ independent random variables (over $l$), and hence satisfies
\begin{align*}
    \| \sum_{l} \mathbf{Z}_{ml} \bigg( \sum_{a\neq m} \mathbf{Z}_{al} (\utilde_{1}^{S,1-m})_{a\cdot} \bigg) \| &\lesssim \sqrt{p_{-1}\log(p)} \max_{l} \| \sum_{a\neq m} (\mathbf{Z}_{al} \utilde_{1}^{S,1-m})_{a\cdot} \|.
\end{align*}
However $\utilde_1^{S,1-m}$ is still dependent on the $a$'th column of $\mathbf{Z}$, so we introduce a leave-two-out estimator $\utilde_1^{S,1-m,1-l}$, obtained by initializing (with diagonal deletion) with the noise matrix $\mathbf{Z}_1$ replaced with $\mathbf{Z}_1^{1-m,1-l}$. For fixed $l$, we observe that
\begin{align*}
    \| \sum_{a\neq m} (\mathbf{Z}_{al} \utilde_{1}^{S,1-m})_{a\cdot} \| &\leq \| \sum_{a\neq m} (\mathbf{Z}_{al}) \big( \utilde_{1}^{S,1-m})_{a\cdot} - \big(\utilde_1^{S,1-m,1-l} (\utilde_1^{S,1-m,1-l})\t \utilde_{1}^{S,1-m} \big)_{a\cdot} \| \\
    &\quad+   \| \sum_{a\neq m} (\mathbf{Z}_{al})\big(\utilde_1^{S,1-m,1-l} (\utilde_1^{S,1-m,1-l})\t \utilde_{1}^{S,1-m} \big)_{a\cdot} \| \\
    &\leq \| (\mathbf{Z}^{-m})\t \| \| \utilde_{1}^{S,1-m} (\utilde_{1}^{S,1-m})\t - \utilde_1^{S,1-m,1-l} (\utilde_1^{S,1-m,1-l})\t \|\\
    &\quad + \| e_l\t (\mathbf{Z}^{-m})\t \utilde_1^{S,1-m,1-l} \|.
\end{align*}
Note that by Lemma \ref{lem:matricizationrowbound}, it holds that
\begin{align}
    \| e_l\t (\mathbf{Z}^{-m})\t \utilde_1^{S,1-m,1-l} \| &\lesssim\sqrt{p_1\log(p)} \| \utilde_1^{S,1-m,1-l} \|_{2,\infty}. \label{eq:alpha3bound}
\end{align}
In addition, by the Davis-Kahan Theorem (using the eigengap condition in \cref{lem:spectraliniteigengapsloo}),
\begin{align}
   \| &\utilde_{1}^{S,1-m} (\utilde_{1}^{S,1-m})\t - \utilde_1^{S,1-m,1-l} (\utilde_1^{S,1-m,1-l})\t \| \nonumber \\
   %&\leq \frac{ \| \Gamma( \mathbf{T}_1 (\mathbf{Z}_1^{1-m,1-l} - \mathbf{Z}_1^{1-m})\t  + (\mathbf{Z}_1^{1-m,1-l}  - \mathbf{Z}_1^{1-m}) \mathbf{T}_1\t + \mathbf{Z}_1^{1-m,1-l}  (\mathbf{Z}_1^{1-m,1-l} )\t - \mathbf{Z}_1^{1-m}(\mathbf{Z}_1^{1-m})\t )\utilde_1^{S,1-m,1-l} \|}{\hat \lambda_{r}^{-m} - \hat \lambda_{r+1}^{-m-l}} \nonumber\\
    &\lesssim \frac{1}{\lambda^2} \bigg( \| \Gamma(\mathbf{Z}_1^{1-m} - \mathbf{Z}_1^{1-m,1-l}) \mathbf{T}_1\t \| \nonumber\\
    &\qquad\qquad + \| \Gamma( \mathbf{Z}_1^{1-m}(\mathbf{Z}_1^{1-m}) - \mathbf{Z}_1^{1-m,1-l} (\mathbf{Z}_1^{1-m,1-l})\t ) \utilde_1^{S,1-m,1-l} \| \bigg) \nonumber\\
    &\lesssim \frac{1}{\lambda^2} \bigg( \mu_0 \lambda_1 \sqrt{r\log(p)} \sqrt{\frac{p_1}{p_{-1}}} + \sqrt{p_1} \log(p) \| \utilde_1^{S,1-m,1-l} \|_{2,\infty} \bigg), \label{eq:leavetwooutsintheta}
\end{align}
where the final inequality holds by \cref{lem:spectralinitconcentrationloo}. 
Consequently, plugging this and \eqref{eq:alpha3bound} into our bound for $B$, we obtain
\begin{align}
    B &\lesssim  \frac{ \sqrt{p_{-1}\log(p)}}{\lambda^2} \bigg\{ \| \mathbf{Z}^{-m} \| \frac{1}{\lambda^2} \bigg(  \lambda_1 \mu_0 \sqrt{r\log(p)} \sqrt{\frac{p_1}{p_{-1}}} + \sqrt{p_1} \log(p) \| \utilde_1^{S,1-m,1-l} \|_{2,\infty} \bigg) \nonumber \\
    &\qquad +  \sqrt{p_1\log(p)} \| \utilde_1^{S,1-m,1-l}\|_{2,\infty} \bigg\}. \nonumber \\
    &\lesssim \frac{ \sqrt{p_{-1}\log(p)}}{\lambda^2} \bigg\{ \frac{p}{\lambda^2} \bigg( \lambda_1 \mu_0 \sqrt{r\log(p)} \sqrt{\frac{p_1}{p_{-1}}} + \sqrt{p_1}\log(p) \| \utilde_1^{S,1-m,1-l} \|_{2,\infty} \bigg) \nonumber \\
    &\qquad + \sqrt{p_1\log(p)} \|  \utilde_1^{S,1-m,1-l} \|_{2,\infty} \bigg\}, \label{BBound}
\end{align}
where we used the fact that $\|\mathbf{Z}^{-m}\| \leq \| \mathbf{Z} \| \lesssim p$ with high probability. The  bound \eqref{BBound} can be improved so as not to depend on $\utilde_1^{S,1-m,1-l}$.  By the bound in \eqref{eq:leavetwooutsintheta}, it holds that
\begin{align*}
    \| \utilde_{1}^{S,1-m,1-l} \|_{2,\infty} &\leq \| \utilde_{1}^{S,1-m} (\utilde_1^{S,1-m})\t \utilde_1^{S,1-m,1-l} \|_{2,\infty} \\
    &\qquad + \| \utilde_{1}^{S,1-m,1-l} - \utilde_{1}^{S,1-m} (\utilde_1^{S,1-m})\t \utilde_1^{S,1-m,1-l} \|_{2,\infty} \\
    &\leq \| \utilde_{1}^{S,1-m}  \|_{2,\infty} + \frac{1}{\lambda^2} \bigg( \mu_0 \lambda_1 \sqrt{r\log(p)} \sqrt{\frac{p_1}{p_{-1}}} + \sqrt{p_1} \log(p) \| \utilde_1^{S,1-m,1-l} \|_{2,\infty} \bigg),\\
    &\leq \| \utilde_1^{S,1-m} \|_{2,\infty} + \mu_0 \sqrt{\frac{r}{p_1}} + o(1) \| \utilde_{1}^{S,1-m,1-l}\|_{2,\infty},
\end{align*}
where we have implicitly observed that
\begin{align*}
    \frac{\kappa p_1\sqrt{\log(p)}}{\lambda \sqrt{p_{-1}}} \mu_0 \sqrt{\frac{r}{p_1}} \leq \mu_0 \sqrt{\frac{r}{p_1}},
\end{align*}
which holds since $\lambda \gtrsim \kappa \sqrt{\log(p)} p/p_{\min}^{1/4}$.  By rearranging, we therefore have that
\begin{align*}
    \| \utilde_1^{S,1-m,1-l} \|_{2,\infty} \lesssim \mu_0 \sqrt{\frac{r}{p_1}} + \| \utilde_1^{S,1-m} \|_{2,\infty}.
\end{align*}
Plugging this into \eqref{BBound}, we obtain
\begin{align*}
     B &\lesssim  \frac{ \sqrt{p_{-1}\log(p)}}{\lambda^2} \bigg\{ \frac{p}{\lambda^2} \bigg( \lambda_1 \mu_0 \sqrt{r\log(p)} \sqrt{\frac{p_1}{p_{-1}}} + \sqrt{p_1}\log(p) \|\bigg[ \|  \utilde_1^{S,1-m} \|_{2,\infty} + \mu_0 \sqrt{\frac{r}{p_1}} \bigg] \bigg) \nonumber \\
    &\qquad + \sqrt{p_1\log(p)} \bigg[ \|  \utilde_1^{S,1-m} \|_{2,\infty} + \mu_0 \sqrt{\frac{r}{p_1}} \bigg]\bigg\} \\
    &\lesssim \frac{p\sqrt{p_1 r} \kappa \mu_0 \log(p)}{\lambda^3} + \frac{p \sqrt{p_1p_2p_3} \log^{3/2}(p)}{\lambda^4} \| \utilde_1^{S,1-m}\|_{2,\infty} + \frac{p \sqrt{p_1p_2p_3} \log^{3/2}(p)}{\lambda^4} \mu_0 \sqrt{\frac{r}{p_1}} \\
    &\quad + \frac{\sqrt{p_1p_2p_3}\log(p)}{\lambda^2} \|\utilde_1^{S,1-m}\|_{2,\infty} + \frac{\sqrt{p_1p_2p_3}\log(p)}{\lambda^2} \mu_0 \sqrt{\frac{r}{p_1}} \\
    &\asymp  \frac{\kappa \sqrt{p_1\log(p)}}{\lambda} \mu_0 \sqrt{\frac{r}{p_1}} \bigg( \frac{p \sqrt{p_1 \log(p)}}{\lambda^2} \bigg) + \frac{\sqrt{p_1p_2p_3} \log(p)}{\lambda^2} \|\utilde_1^{S,1-m}\|_{2,\infty} \bigg( 1 + \frac{p \sqrt{\log(p)}}{\lambda^2} \bigg) \\
    &\quad + \frac{\sqrt{p_1p_2p_3}\log(p)}{\lambda^2}\mu_0 \sqrt{\frac{r}{p_1}} \bigg( 1 + \frac{p \sqrt{\log(p)}}{\lambda^2} \bigg) \\
    &\lesssim  \frac{\kappa \sqrt{p_1\log(p)}}{\lambda} \mu_0 \sqrt{\frac{r}{p_1}} + \frac{\sqrt{p_1p_2p_3} \log(p)}{\lambda^2} \|\utilde_1^{S,1-m}\|_{2,\infty} + \frac{\sqrt{p_1p_2p_3} \log(p)}{\lambda^2} \mu_0 \sqrt{\frac{r}{p_1}},
\end{align*}
which holds whenever $\lambda^2 \gtrsim p\sqrt{p_1 \log(p)}$.  
% Note that since $\lambda \gtrsim \kappa p^{3/4} \sqrt{\log(p)}$, we have that
% \begin{align*}
%     \frac{p^{5/2} \log^{3/2}  \mu_0 \sqrt{r}}{\lambda^4} &\lesssim \frac{p^{3/2} \sqrt{\log(p)}}{\lambda^2} \mu_0 \sqrt{\frac{r}{p}}; \\
%     \frac{p^3 \log^{3/2}(p)}{\lambda^3} &\lesssim \frac{p^{3/2} \sqrt{\log(p)}}{\lambda},
% \end{align*}
% and hence this simplifies to
% \begin{align*}
%     B &\lesssim  \frac{p^{3/2} \sqrt{\log(p)}}{\lambda^2}  \mu_0 \sqrt{\frac{r}{p}}  + \frac{p^{3/2} \sqrt{\log(p)}}{\lambda^2} \| \utilde_1^{S,1-m} \|_{2,\infty}.
% \end{align*}
This bound still depends on the leave-one-out sequence, but we will obtain a bound independent of this sequence shortly upon analyzing term $A$. Note that we have also demonstrated the high-probability upper bound
\begin{align}
    \| e_m\t &\Gamma(\mathbf{Z}_1\mathbf{Z}_1\t) \utilde^{S,1-m} \| \nonumber \\
    &\lesssim \kappa \sqrt{p_1 \log(p)} \mu_0 \sqrt{\frac{r}{p_1}} + \sqrt{p_1p_2p_3}\log(p)\| \utilde_1^{S,1-m}\|_{2,\infty} +\sqrt{p_1p_2p_3}\log(p) \mu_0 \sqrt{\frac{r}{p_1}}. \label{yeree}
\end{align}
\\ \
\\ \noindent
\textbf{The term $A$:} Note that by \cref{lem:spectralconcentrationinit} we have that $\|\Gamma(\mathbf{Z}_1\mathbf{Z}_1\t) \| \lesssim (p_1p_2p_3)^{1/2}$, which yields
\begin{align*}
    A &\lesssim \frac{(p_1p_2p_3)^{1/2}}{\lambda^2}  \| \uhat_1^S (\uhat_1^S)\t - \utilde_1^{S,1-m} (\utilde_1^{S,1-m})\t \|.
\end{align*}
Therefore it suffices to bound the term on the right.  By the Davis-Kahan Theorem, it holds that
\begin{align*}
    \| &\uhat_1^S (\uhat_1^S)\t - \utilde_1^{S,1-m} (\utilde_1^{S,1-m})\t \| \\
    &\lesssim \frac{ \| \Gamma((\mathbf{T}_1 \mathbf{Z}_1 - \mathbf{Z}_1^{1-m})\t + (\mathbf{Z}_1 - \mathbf{Z}_1^{1-m})\mathbf{T}_1\t + \mathbf{Z}_1 \mathbf{Z}_1\t - \mathbf{Z}_1^{1-m} (\mathbf{Z}_1^{1-m})\t \big) \utilde_1^{S,1-m} \|}{\lambda^2} \\
    &\lesssim \frac{1}{\lambda^2} \bigg\{ \| \Gamma \big( (\mathbf{Z}_1 - \mathbf{Z}_1^{1-m})\mathbf{T}_1\t \big) \| + \| \Gamma\big( \mathbf{Z}_1 \mathbf{Z}_1\t - \mathbf{Z}_1^{1-m} (\mathbf{Z}_1^{1-m})\t \big) \utilde_1^{S,1-m} \| \bigg\}.
\end{align*}
By \cref{lem:spectralinitconcentrationloo}, it holds that
\begin{align}
    \| \Gamma \big( (\mathbf{Z}_1 - \mathbf{Z}_1^{1-m})\mathbf{T}_1\t \big) \| &\lesssim \lambda_1 \mu_0 \sqrt{r\log(p)}, \label{eq:A1}
\end{align}
so it suffices to consider the second term.  Note that the matrix
\begin{align*}
    \Gamma\big(\mathbf{Z}_1 \mathbf{Z}_1\t - \mathbf{Z}_1^{1-m} (\mathbf{Z}_1^{1-m})\t\big)
\end{align*}
is rank one symmetric matrix whose $(m,l)$ entry is simply $\langle e_m\t \mathbf{Z}_1, e_l\t \mathbf{Z}_1 \rangle$ for $l \neq m$.  Therefore, define the matrices $\mathbf{G}_{\mathrm{col}}$ and $\mathbf{G}_{\mathrm{row}}$ with $\mathbf{G}_{\mathrm{col}}$ the matrix whose only nonzero entries are in the $m$'th column, in which case they are $\langle e_m\t\mathbf{Z}_1, e_l\mathbf{Z}_1 \rangle$ for $l\neq m$, and $\mathbf{G}_{\mathrm{row}}$ the matrix whose only nonzero entries are in the $m$'th row, with entries defined similarly.  Then
\begin{align*}
     \big\|\Gamma\big(\mathbf{Z}_1 \mathbf{Z}_1\t - \mathbf{Z}_1^{1-m} (\mathbf{Z}_1^{1-m})\t\big) \utilde_{1}^{S,1-m} \| &\leq \| \mathbf{G}_{\mathrm{row}} \utilde_1^{S,1-m} \| + \| \mathbf{G}_{\mathrm{col}} \utilde_1^{S,1-m}\|.
\end{align*}
We consider each term separately.  First, note that
\begin{align*}
    \| \mathbf{G}_{\mathrm{row}} \utilde_1^{S,1-m} \| &= \| e_m\t \Gamma(\mathbf{Z}_1\mathbf{Z}_1\t) \utilde_1^{S,1-m} \|.
\end{align*}
This was already bounded en route to the analysis for term $B$.  In fact, by \eqref{yeree} we already have the upper bound
\begin{align}
    \| e_m\t &\Gamma(\mathbf{Z}_1\mathbf{Z}_1\t) \utilde^{S,1-m} \| \nonumber \\
    &\lesssim \kappa \sqrt{p_1 \log(p)} \mu_0 \sqrt{\frac{r}{p_1}} + \sqrt{p_1p_2p_3}\log(p)\| \utilde_1^{S,1-m}\|_{2,\infty} +\sqrt{p_1p_2p_3}\log(p) \mu_0 \sqrt{\frac{r}{p_1}}. \label{A2}
\end{align}
Next, we argue similarly to the proof of Lemma 4 of \citet{cai_subspace_2021}.  We have
\begin{align*}
    \| \mathbf{G}_{\mathrm{col}} \utilde_{1}^{S,1-m} \| &\leq  \| \mathbf{G}_{\mathrm{col}}\utilde_{1}^{S,1-m} \|_F \\
    &= \bigg( \sum_{j\neq m} \| \langle e_m\t \mathbf{Z}_1 , e_j\t \mathbf{Z}_1 \rangle (\utilde_{1}^{S,1-m})_{m\cdot} \|^2 \bigg)^{1/2} \\
    &\leq \bigg( \sum_{j\neq m} | \langle e_m\t \mathbf{Z}_1, e_j\t \mathbf{Z}_1 \rangle |^2 \| (\utilde_{1}^{S,1-m})_{m\cdot} \|^2 \bigg)^{1/2} \\
    &\leq \| \utilde_1^{S,1-m} \|_{2,\infty} \| \Gamma(\mathbf{Z}_1\mathbf{Z}_1\t) \| \\
    &\lesssim (p_1p_2p_3)^{1/2} \| \utilde_1^{S,1-m} \|_{2,\infty}.
\end{align*}
Therefore,
\begin{align*}
      \| &\uhat_1^S (\uhat_1^S)\t - \utilde_1^{S,1-m} (\utilde_1^{S,1-m})\t \| \\
    &\lesssim \frac{1}{\lambda^2}\bigg\{ \lambda_1 \mu_0 \sqrt{r\log(p)} +  \kappa \sqrt{p_1 \log(p)} \mu_0 \sqrt{\frac{r}{p_1}} + \sqrt{p_1p_2p_3}\log(p)\| \utilde_1^{S,1-m}\|_{2,\infty}\\
    &\quad +\sqrt{p_1p_2p_3}\log(p) \mu_0 \sqrt{\frac{r}{p_1}} + (p_1p_2p_3)^{1/2} \|\utilde_1^{S,1-m} \|_{2,\infty} \bigg\} \\
    &\lesssim \frac{\kappa \sqrt{p_1 \log(p)}}{\lambda} \mu_0 \sqrt{\frac{r}{p_1}} + \frac{\sqrt{p_1p_2p_3}\log(p)}{\lambda^2} \mu_0 \sqrt{\frac{r}{p_1}} + \frac{\sqrt{p_1p_2p_3}\log(p)}{\lambda^2} \| \utilde_1^{S,1-m} \|_{2,\infty}.
\end{align*}
Therefore, it holds that
\begin{align*}
    \|& \utilde_1^{S,1-m} \|_{2,\infty} \\
    &\leq \| \uhat_1^S (\uhat_1^S)\t \utilde_1^{S,1-m}\|_{2,\infty} + \| \utilde_1^{S,1-m} - \uhat_1^S (\uhat_1^S)\t \utilde_1^{S,1-m}\|_{2,\infty}  \\
    &\leq \| \uhat_1^S \|_{2,\infty} +  \frac{\kappa \sqrt{p_1 \log(p)}}{\lambda} \mu_0 \sqrt{\frac{r}{p_1}} + \frac{\sqrt{p_1p_2p_3}\log(p)}{\lambda^2} \mu_0 \sqrt{\frac{r}{p_1}} + \frac{\sqrt{p_1p_2p_3}\log(p)}{\lambda^2} \| \utilde_1^{S,1-m} \|_{2,\infty} \\
    &\leq \| \uhat_1^S \|_{2,\infty}  + \mu_0 \sqrt{\frac{r}{p}} + o(1) \| \utilde_1^{S,1-m} \|_{2,\infty},
\end{align*}
so by rearranging we arrive at
\begin{align*}
     \| \utilde_1^{S,1-m} \|_{2,\infty} &\lesssim  \| \uhat_1^S \|_{2,\infty}  + \mu_0 \sqrt{\frac{r}{p}}.
\end{align*}
Consequently,
\begin{align}
    \| \uhat_1^S& (\uhat_1^S)\t - \utilde_1^{S,1-m} (\utilde_1^{S,1-m})\t \| \nonumber \\
    &\lesssim \frac{\kappa \sqrt{p_1 \log(p)}}{\lambda} \mu_0 \sqrt{\frac{r}{p_1}} + \frac{\sqrt{p_1p_2p_3}\log(p)}{\lambda^2} \mu_0 \sqrt{\frac{r}{p_1}} + \frac{\sqrt{p_1p_2p_3}\log(p)}{\lambda^2} \| \uhat_1^{S} \|_{2,\infty}\label{uhatsintheta}
\end{align}
Therefore, we have that
\begin{align*}
    A %&\lesssim %\frac{p^{3/2}}{\lambda^2} \Bigg[ \frac{1}{\lambda^2} \bigg\{ \lambda_1 \mu_0 \sqrt{r\log(p)} +  p^{3/2} \sqrt{\log(p)} \mu_0 \sqrt{\frac{r}{p}} + p^{3/2} \sqrt{\log(p)} \bigg( \| \uhat_1^S \|_{2,\infty}  + \mu_0 \sqrt{\frac{r}{p}} \bigg) \bigg\} \bigg] \\
     &\lesssim \frac{(p_1p_2p_3)^{1/2}}{\lambda^2}  \| \uhat_1^S (\uhat_1^S)\t - \utilde_1^{S,1-m} (\utilde_1^{S,1-m})\t \| \\
     &\lesssim \frac{(p_1p_2p_3)^{1/2}}{\lambda^2}  \bigg\{ \frac{\kappa \sqrt{p_1 \log(p)}}{\lambda} \mu_0 \sqrt{\frac{r}{p_1}} + \frac{\sqrt{p_1p_2p_3}\log(p)}{\lambda^2} \mu_0 \sqrt{\frac{r}{p_1}} + \frac{\sqrt{p_1p_2p_3}\log(p)}{\lambda^2} \| \uhat_1^{S} \|_{2,\infty} \bigg\} \\
     &\ll  \frac{\kappa \sqrt{p_1 \log(p)}}{\lambda} \mu_0 \sqrt{\frac{r}{p_1}} + \frac{\sqrt{p_1p_2p_3}\log(p)}{\lambda^2} \mu_0 \sqrt{\frac{r}{p_1}} + \frac{\sqrt{p_1p_2p_3}\log(p)}{\lambda^2} \| \uhat_1^{S} \|_{2,\infty}.
%    &\lesssim \frac{p^{3/2}}{\lambda^2} \Bigg[ \frac{\kappa \mu_0 \sqrt{r\log(p)}}{\lambda} + \frac{p^{3/2} \sqrt{\log(p)}}{\lambda^2} \mu_0 \sqrt{\frac{r}{p}} + \frac{p^{3/2} \sqrt{\log(p)}}{\lambda^2} \| \uhat \|_{2,\infty} \Bigg] \\
%    &\ll \bigg(\frac{ \kappa \sqrt{p\log(p)}}{\lambda} + \frac{p^{3/2} \log(p)}{\lambda^2} \bigg) \mu_0 \sqrt{\frac{r}{p}} + \frac{p^{3/2}}{\lambda^2} \|\uhat \|_{2,\infty}.
\end{align*} 
In addition,
\begin{align*}
    B &\lesssim \frac{\kappa \sqrt{p_1\log(p)}}{\lambda} \mu_0 \sqrt{\frac{r}{p_1}} + \frac{\sqrt{p_1p_2p_3} \log(p)}{\lambda^2} \|\utilde_1^{S,1-m}\|_{2,\infty} + \frac{\sqrt{p_1p_2p_3} \log(p)}{\lambda^2} \mu_0 \sqrt{\frac{r}{p_1}}\\
    &\lesssim  \frac{\kappa \sqrt{p_1\log(p)}}{\lambda} \mu_0 \sqrt{\frac{r}{p_1}}  + \frac{\sqrt{p_1p_2p_3} \log(p)}{\lambda^2} \|\uhat_1^S\|_{2,\infty}+ \frac{\sqrt{p_1p_2p_3} \log(p)}{\lambda^2} \mu_0 \sqrt{\frac{r}{p_1}}.
\end{align*}
Combining both of these with the initial bounds in \cref{lem:negligibleterms}, we arrive at
\begin{align*}
    \|e_m\t \big( \uhat_1^S - \U_1 \mathbf{W}_1^S \big) \| &\lesssim  \bigg(\frac{ \kappa \sqrt{p_1\log(p)}}{\lambda} + \frac{(p_1p_2p_3)^{1/2} \log(p)}{\lambda^2} + \kappa^2 \mu_0 \sqrt{\frac{r}{p_1}}\bigg) \mu_0 \sqrt{\frac{r}{p_1}}\\
    &\quad + \frac{(p_1p_2p_3)^{1/2}}{\lambda^2} \|\uhat_1^S \|_{2,\infty}.
\end{align*}
By taking a union bound over all the rows, we have that with probability at least $1 - O(p^{-29})$ that
\begin{align*}
    \| \uhat_1^S - \U_1 \mathbf{W}_1^S \|_{2,\infty} &\lesssim\bigg(\frac{ \kappa \sqrt{p_1\log(p)}}{\lambda} + \frac{(p_1p_2p_3)^{1/2} \log(p)}{\lambda^2} + \kappa^2 \mu_0 \sqrt{\frac{r}{p_1}}\bigg) \mu_0 \sqrt{\frac{r}{p_1}} \\
    &\quad + \frac{(p_1p_2p_3)^{1/2}}{\lambda^2} \|\uhat_1^S \|_{2,\infty}.
\end{align*}
Therefore, 
\begin{align*}
    \|\uhat \|_{2,\infty} &\leq \| \U_1 \|_{2,\infty} + \| \uhat_1^S - \U_1 \mathbf{W}_1^S \|_{2,\infty} \\
    &\leq \mu_0 \sqrt{\frac{r}{p_1}} + o(1) \|\uhat_1^S \|_{2,\infty},
\end{align*}
which, by rearranging, yields
\begin{align*}
    \|\uhat_1^S \|_{2,\infty} &\lesssim \mu_0 \sqrt{\frac{r}{p_1}}.
\end{align*}
Therefore,
\begin{align*}
    \| \uhat_1^S - \U_1 \mathbf{W}_1^S \|_{2\infty} &\lesssim\bigg(\frac{ \kappa \sqrt{p_1\log(p)}}{\lambda} + \frac{(p_1p_2p_3)^{1/2} \log(p)}{\lambda^2} + \kappa^2 \mu_0 \sqrt{\frac{r}{p_1}}\bigg) \mu_0 \sqrt{\frac{r}{p_1}}.
\end{align*}
In addition, \eqref{uhatsintheta} together with the bound above shows that with high probability,
\begin{align*}
     \| \uhat_1^S& (\uhat_1^S)\t - \utilde_1^{S,1-m} (\utilde_1^{S,1-m})\t \| \nonumber \\
    &\lesssim \frac{\kappa \sqrt{p_1 \log(p)}}{\lambda} \mu_0 \sqrt{\frac{r}{p_1}} + \frac{\sqrt{p_1p_2p_3}\log(p)}{\lambda^2} \mu_0 \sqrt{\frac{r}{p_1}} + \frac{\sqrt{p_1p_2p_3}\log(p)}{\lambda^2} \| \uhat_1^{S} \|_{2,\infty} \\
    %&\lesssim \frac{\kappa \mu_0 \sqrt{r\log(p)}}{\lambda} + \frac{p^{3/2} \sqrt{\log(p)}}{\lambda^2} \mu_0 \sqrt{\frac{r}{p}} + \frac{p^{3/2} \sqrt{\log(p)}}{\lambda^2} \bigg\{  \| \uhat_1^S - \U \mathbf{W}_1^S \|_{2,\infty} + \frac{p^{3/2} \sqrt{\log(p)}}{\lambda^2} \mu_0 \sqrt{\frac{r}{p}} \bigg\}\\
    &\lesssim \frac{\kappa \sqrt{p_1 \log(p)}}{\lambda} \mu_0 \sqrt{\frac{r}{p_1}} + \frac{\sqrt{p_1p_2p_3}\log(p)}{\lambda^2} \mu_0 \sqrt{\frac{r}{p_1}}. %\\
    %&\leq \bigg( \frac{\dl}{\lambda} + \frac{1}{2} \bigg) \mu_0 \sqrt{\frac{r}{p}},
\end{align*}
%since $\kappa \mu_0 \sqrt{r} \lesssim \sqrt{p}$.  
Taking another union bound over $m$ shows that this bound holds for all $m$ with probability at least $1 - O(p^{-29})$. Both of these bounds therefore hold with probability at least $1 - p^{-20}$.  
\end{proof}

\subsubsection{Proof of \cref{lem:spectralinit_leave_one_out_sintheta}} \label{sec:initlooproof}
In this section we prove \cref{lem:spectralinit_leave_one_out_sintheta}, which controls the remaining two leave-one-out sequences not bounded in \cref{thm:spectralinit_twoinfty}.

\begin{proof}[Proof of \cref{lem:spectralinit_leave_one_out_sintheta}]
First we provide concentration guarantees, similar to \cref{lem:spectralinitconcentrationloo}.  We will bound the following terms:
\begin{itemize}
    \item $\| \Gamma\big( \mathbf{T}_1 ( \mathbf{Z}_1 - \mathbf{Z}_1^{j-m})\t \big) \|$; 
    \item $\| \Gamma\big( \mathbf{Z}_1 \mathbf{Z}_1\t - \mathbf{Z}_1^{j-m} (\mathbf{Z}_1^{j-m})\t \big) \|$
    \item $\| \Gamma\big(\mathbf{Z}_1 \mathbf{Z}_1\t - \mathbf{Z}_1^{j-m} (\mathbf{Z}_1^{j-m})\t \big) \utilde_1^{S,j-m} \| $.
\end{itemize}
First, note that by \cref{lem:matricizationrowbound}, with probability at least $1 - O(p^{-30})$ it holds that
\begin{align*}
    \| \Gamma\big( (\mathbf{Z}_1 - \mathbf{Z}_1^{j-m}) \mathbf{T}_1\t \big) \| &\leq \| (\mathbf{Z}_1 - \mathbf{Z}_1^{j-m}) \mathbf{T}_1\t \| + \| \diag\big( (\mathbf{Z}_1 - \mathbf{Z}_1^{j-m}) \mathbf{T}_1\t \big) \| \\
    &\leq 2 \| (\mathbf{Z}_1 - \mathbf{Z}_1^{j-m}) \mathbf{T}_1\t \| \\
    &\lesssim  \sqrt{p_{-j}\log(p)} \| \mathbf{T}_1\t \|_{2,\infty} \\
    &\lesssim \lambda_1 \mu_0 \sqrt{r_1 \frac{p_1}{p_j} \log(p)}. %\\
    %&\lesssim \lambda_1 \mu_0 \sqrt{r\log(p)}.
\end{align*}
Next, we consider the matrix $\Gamma\big( \mathbf{Z}_1 \mathbf{Z}_1\t - \mathbf{Z}_1^{j-m} (\mathbf{Z}_1^{j-m})\t \big)$.  First, observe that for $i\neq k$ this matrix has entries of the form
\begin{align*}
    \Gamma\big( \mathbf{Z}_1 \mathbf{Z}_1\t - \mathbf{Z}_1^{j-m} (\mathbf{Z}_1^{j-m})\t \big)_{ik} &= \sum_{l\in \Omega} (\mathbf{Z}_1)_{il} (\mathbf{Z}_1)_{kl},
\end{align*}
where $\Omega$ is the set of indices such that the $l$'th column of $\mathbf{Z}_1$ corresponds to elements belonging to the $m$'th row of $\mathbf{Z}_j$.  A general formula is possible, but not needed for our purposes here; the cardinality of $\Omega$ is equal to the number of nonzero columns of $\mathbf{Z}_1 - \mathbf{Z}_1^{j-m}$, which is $p_{-1-j}$.  Since this matrix is a sample gram matrix, by Lemma 1 of \citet{agterberg_entrywise_2022} it  holds that with probability at least $1 - O(p^{-30})$ (where as in the proof of \cref{lem:spectralconcentrationinit} the higher probability holds by modifying the constant on $\delta$ in the proof of Lemma 1 of \citet{agterberg_entrywise_2022}) that
\begin{align*}
     \| \Gamma\big( \mathbf{Z}_1 \mathbf{Z}_1\t - \mathbf{Z}_1^{j-m} (\mathbf{Z}_1^{j-m})\t \big) \| &\lesssim \sqrt{p_1} + \sqrt{p_1 p_{-1-j}} \\
     &\lesssim \sqrt{p_1} + \sqrt{p_{-j}} \\
     &\ll p^2/p_{\min}^{1/2}
\end{align*}
For the remaining term, we note that
\begin{align*}
    \| \Gamma\big(\mathbf{Z}_1 \mathbf{Z}_1\t - \mathbf{Z}_1^{j-m} (\mathbf{Z}_1^{j-m})\t \big) \utilde_1^{S,j-m} \| &\leq \sqrt{p_1} \| \Gamma\big(\mathbf{Z}_1 \mathbf{Z}_1\t - \mathbf{Z}_1^{j-m} (\mathbf{Z}_1^{j-m})\t \big) \utilde_1^{S,j-m} \|_{2,\infty} \\ 
    &= \sqrt{p_1} \max_i \| \sum_{k\neq i,k=1}^{p_1} \sum_{l\in \Omega} (\mathbf{Z}_1)_{il} (\mathbf{Z}_1)_{kl}\big( \utilde_{1}^{S,j-m}\big)_{k\cdot} \| \\
    &\lesssim \sqrt{p_1}  \sqrt{p_{-1-j} \log(p)}   \max_l \| \sum_{k\neq i,k=1}^{p_1} (\mathbf{Z}_1)_{kl} \big( \utilde_{1}^{S,j-m}\big)_{k\cdot} \| \\
    &\lesssim  p_1 \sqrt{p_{-1-j}}\log(p) \| \utilde_1^{S,j-m} \|_{2,\infty} \\
    &\lesssim \sqrt{p_1} \sqrt{p_{-j}} \log(p) \| \utilde_1^{S,j-m} \|_{2,\infty},
\end{align*}
where we have implicitly used the matrix Hoeffding's inequality twice: once over the summation over $l$ conditional on the collection $(\mathbf{Z}_1)_{kl}$ for $k \neq i$, and then again over the summation in $k$.  

Note that $\uhat_1^S$ are the eigenvectors of the matrix $\Gamma\big(  \mathbf{T}_1 \mathbf{T}_1\t + \mathbf{Z}_1 \mathbf{Z}_1\t  + \mathbf{T}_1 \mathbf{Z}_1\t + \mathbf{Z}_1 \mathbf{T}_1\t\big)$ and $\utilde_1^{S,j-m}$ are the eigenvectors of the matrix $\Gamma\big( \mathbf{T}_1 \mathbf{T}_1\t + \mathbf{Z}_1^{j-m} (\mathbf{Z}_1^{j-m})\t  + \mathbf{T}_1 (\mathbf{Z}_1^{j-m})\t + \mathbf{Z}_1^{j-m} \mathbf{T}_1\t \big)$.  Therefore, the spectral norm of the difference is upper bounded by
\begin{align*}
     2\| \Gamma\big( \mathbf{T}_1 ( \mathbf{Z}_1 - \mathbf{Z}_1^{j-m})\t \big) \| + \|  \Gamma\big( \mathbf{Z}_1 \mathbf{Z}_1\t - \mathbf{Z}_1^{j-m} (\mathbf{Z}_1^{j-m})\t \big) \|     &\lesssim \lambda_1 \mu_0 \sqrt{r_1 \frac{p_1}{p_j} \log(p)} + p^2/p_{\min}^{1/2}\\
    &\ll \lambda^2.
\end{align*}
Moreover, \cref{lem:spectraliniteigengapsloo} shows that
\begin{align*}
    \lambda_r\bigg( \Gamma\big(  \mathbf{T}_1 \mathbf{T}_1\t + \mathbf{Z}_1 \mathbf{Z}_1\t  + \mathbf{T}_1 \mathbf{Z}_1\t + \mathbf{Z}_1 \mathbf{T}_1\t\big) \bigg) - \lambda_{r+1}\bigg( \Gamma\big(  \mathbf{T}_1 \mathbf{T}_1\t + \mathbf{Z}_1 \mathbf{Z}_1\t  + \mathbf{T}_1 \mathbf{Z}_1\t + \mathbf{Z}_1 \mathbf{T}_1\t\big) \bigg) &\gtrsim \lambda^2,
\end{align*}
so by the Davis-Kahan Theorem,
\begin{align}
    \| \utilde_1^{j-m} (\utilde_{1}^{S,j-m})\t - \uhat_1^S (\uhat_1^S)\t \| &\lesssim \frac{1}{\lambda^2} \bigg( \lambda_1 \mu_0 \sqrt{r_1 \frac{p_1}{p_j} \log(p)} + \sqrt{p_1} \sqrt{p_{-j}} \log(p) \| \utilde_1^{S,j-m} \|_{2,\infty} \bigg)\nonumber \\
    &\lesssim \frac{ \kappa \sqrt{p_1\log(p)}}{\lambda} \mu_0 \sqrt{\frac{r_1}{p_j}} + \frac{ \sqrt{p_1 p_{-j}} \log(p)}{\lambda^2} \| \utilde_1^{S,j-m} \|_{2,\infty} \label{toplugin}
\end{align}
In addition, we note that by \cref{thm:spectralinit_twoinfty}, with probability at least $1 - O(p^{-20})$ it holds that
\begin{align*}
    \|  \utilde_1^{S,j-m} \|_{2,\infty} &\leq \| \uhat_1^S (\uhat_1^S)\t \utilde_1^{S,j-m}  \|_{2,\infty} + \| \utilde_1^{S,j-m} - \uhat_1^S (\uhat_1^S)\t \utilde_1^{S,j-m} \|_{2,\infty} \\
    &\leq \| \uhat_1^S \|_{2,\infty} + \| \utilde_1^{S,j-m} (\utilde_1^{S,j-m})\t - \uhat_1^S (\uhat_1^S)\t  \|\\
    &\leq \| \uhat_1^S \|_{2,\infty} + \frac{ \kappa \sqrt{p_1\log(p)}}{\lambda} \mu_0 \sqrt{\frac{r_1}{p_1}} + \frac{(p_1p_2p_3)^{1/2} \log(p)}{\lambda^2} \| \utilde_1^{S,j-m} \|_{2,\infty} \\
    &\leq \mu_0 \sqrt{\frac{r_1}{p_1}} + o(1) \| \utilde_1^{S,j-m} \|_{2,\infty},
\end{align*}
so by rearranging we obtain that
\begin{align*}
    \|  \utilde_1^{S,j-m} \|_{2,\infty} &\lesssim \mu_0 \sqrt{\frac{r_1}{p_1}}.
\end{align*}
Plugging this into \eqref{toplugin} yields 
\begin{align*}
     \| \utilde_1^{j-m} (\utilde_{1}^{S,j-m})\t - \uhat_1^S (\uhat_1^S)\t \| &\lesssim \frac{\kappa \sqrt{p_1 \log(p)}}{\lambda} \mu_0 \sqrt{\frac{r_1}{p_j}} + \frac{\sqrt{p_{-j}} \log(p)}{\lambda^2} \mu_0 \sqrt{r_1} \\
     &\lesssim \frac{\kappa \sqrt{p_1 \log(p)}}{\lambda} \mu_0 \sqrt{\frac{r_1}{p_j}} + \frac{(p_1p_2p_3)^{1/2} \log(p)}{\lambda^2} \mu_0 \sqrt{\frac{r_1}{p_j}}% \\
     %&\leq \bigg( \frac{\dl^{(1)}}{\lambda} + \frac{1}{2} \bigg) \mu_0 \sqrt{\frac{r_1}{p_j}}.  
\end{align*}
with probability at least $1 - O(p^{-20})$. The proof is then completed by taking a union bound over all $p_1$ rows.  
\end{proof}

\section{Proofs of Tensor Mixed-Membership Blockmodel Identifiability and Estimation}
In this section we prove our main results concerning the mixed-membership identifiability and estimation.  First we establish \cref{prop:identifiability} as well as \cref{lem:relationship} relating the properties of the tensor mixed-membership blockmodel to the tensor denoising model.  We then prove our estimation guarantees \cref{thm:estimation}.   Throughout we let $\mathbf{S}_k = \mathcal{M}_k(\mathcal{S})$ and $\mathbf{T}_k$ defined similarly.  

\subsection{Proofs of \cref{prop:identifiability}, \cref{prop:relationship}, and \cref{lem:relationship}}

First we prove \cref{prop:relationship} and \cref{lem:relationship} simultaneously as we will require part of the proof in the proof of \cref{prop:identifiability}.

\begin{proof}[Proof of \cref{prop:relationship} and \cref{lem:relationship}]
For the first part, we follow the proof of Lemma 2.3 of \citet{mao_estimating_2021}.  Without loss of generality we prove the result for mode 1. Let $\mathbf{T}_1$ have singular value decomposition  $\mathbf{T}_1 = \U_1 \mathbf{\Lambda}_1 \mathbf{V}_1\t$ Then since $\U_1 \mathbf{\Lambda}_1 \mathbf{V}_1\t = \mathbf{\Pi}_1 \mathbf{S}_1 (\mathbf{\Pi_2}\otimes \mathbf{\Pi}_3)\t,$ without loss of generality we may assume the first $r_1$ rows of $\mathbf{T}_1$ correspond to pure nodes.   We note that therefore
\begin{align*}
    \U_1\pure \mathbf{\Lambda}_1^2 (\U_1\pure)\t &=  \mathbf{S}_1(\mathbf{\Pi}_2\otimes \mathbf{\Pi}_3)\t(\mathbf{\Pi}_2\otimes \mathbf{\Pi}_3)\mathbf{S}_1\t.
\end{align*}
Since the rank of $\mathbf{\Pi}_2$ and $\mathbf{\Pi}_3$ are $r_2$ and $r_3$ respectively, it holds that the matrix above is rank $r_1$ as long as $r_1 \leq r_2 r_3$ since $\mathbf{S}_1$ is rank $r_1$, which shows that $\U_1\pure$ is rank $r_1$. Furthermore, we have that  $\mathbf{T}_1[1:r_1,\cdot] = \U_1\pure \mathbf{\Lambda}_1 \mathbf{V}_1\t = \mathbf{S}_1(\mathbf{\Pi_2}\otimes \mathbf{\Pi}_3)\t$ which shows that $\U_1\pure = \mathbf{S}_1(\mathbf{\Pi_2}\otimes \mathbf{\Pi}_3)\t \mathbf{V}_1 \mathbf{\Lambda}_1\inv$. Therefore,
\begin{align*}
    \U_1 &= \mathbf{T}_1 \mathbf{V}_1 \mathbf{\Lambda}_1\inv \\
    &= \mathbf{\Pi}_1 \mathbf{S}_1 ( \mathbf{\Pi}_2 \otimes \mathbf{\Pi}_3 )\t \mathbf{V}_1 \mathbf{\Lambda}_1\inv \\
    &= \mathbf{\Pi}_1 \U_1\pure.
\end{align*}
Next, we observe that 
\begin{align*}
    \lambda^2 &= \min_{k} \lambda_{\min} \big( \mathbf{T}_k \mathbf{T}_k\t \big) \\
    &= \min_{k} \lambda_{\min} \big( \mathbf{\Pi}_k \mathbf{S}_k \big( \mathbf{\Pi}_{k+1} \otimes \mathbf{\Pi}_{k+2} \big)\t\big( \mathbf{\Pi}_{k+1} \otimes \mathbf{\Pi}_{k+2} \big) \mathbf{S}_k\t \mathbf{\Pi}_k\t \big) \\
    &\geq \min_{k} \lambda_{\min} ( \mathbf{\Pi}_k\t \mathbf{\Pi}_k ) \lambda_{\min} \big(\mathbf{S}_k \big( \mathbf{\Pi}_{k+1} \otimes \mathbf{\Pi}_{k+2} \big)\t\big( \mathbf{\Pi}_{k+1} \otimes \mathbf{\Pi}_{k+2} \big) \mathbf{S}_k\t \big) \\
    &\geq \min_k \frac{p_k}{r_k} \lambda_{\min} \big(\mathbf{S}_k \big( \mathbf{\Pi}_{k+1} \otimes \mathbf{\Pi}_{k+2} \big)\t\big( \mathbf{\Pi}_{k+1} \otimes \mathbf{\Pi}_{k+2} \big) \mathbf{S}_k\t \big) \\
    &\gtrsim \Delta^2 \min_k \frac{ p_k }{r_k} \lambda_{\min} \bigg( \big( \mathbf{\Pi}_{k+1} \otimes \mathbf{\Pi}_{k+2} \big)\t \big( \mathbf{\Pi}_{k+1} \otimes \mathbf{\Pi}_{k+2} \big)  \bigg) \\
    &\gtrsim \Delta^2 \frac{p_1 p_2 p_3}{r_1 r_2 r_3},
\end{align*}
where the penultimate line follows from the fact that $ \big( \mathbf{\Pi}_{k+1} \otimes \mathbf{\Pi}_{k+2} \big)$ has full column rank.  Therefore, $\lambda \gtrsim \Delta \frac{(p_1 p_2 p_3)^{1/2}}{(r_1 r_2 r_3)^{1/2}}$.  For the reverse direction, by a similar argument,
\begin{align*}
    \lambda^2 &= \min_k \lambda_{\min}( \mathbf{T}_k \mathbf{T}_k\t ) \\
    &= \min_{k} \lambda_{\min} \big( \mathbf{\Pi}_k \mathbf{S}_k \big( \mathbf{\Pi}_{k+1} \otimes \mathbf{\Pi}_{k+2} \big)\t\big( \mathbf{\Pi}_{k+1} \otimes \mathbf{\Pi}_{k+2} \big) \mathbf{S}_k\t \mathbf{\Pi}_k\t \big) \\
    &\leq  \min_{k} \lambda_{\max} ( \mathbf{\Pi}_k\t \mathbf{\Pi}_k ) \lambda_{\min} \big(\mathbf{S}_k \big( \mathbf{\Pi}_{k+1} \otimes \mathbf{\Pi}_{k+2} \big)\t\big( \mathbf{\Pi}_{k+1} \otimes \mathbf{\Pi}_{k+2} \big) \mathbf{S}_k\t \big) \\
    &\leq \min_k \frac{p_k}{r_k} \lambda_{\min} \big( \mathbf{S}_k \big( \mathbf{\Pi}_{k+1} \otimes \mathbf{\Pi}_{k+2} \big)\t\big( \mathbf{\Pi}_{k+1} \otimes \mathbf{\Pi}_{k+2} \big) \mathbf{S}_k\t \big) \\
    &\leq \min_k \frac{p_k}{r_k} \lambda_{\max} \big(  \mathbf{\Pi}_{k+1} \otimes \mathbf{\Pi}_{k+2} \big)\t\big( \mathbf{\Pi}_{k+1} \otimes \mathbf{\Pi}_{k+2} \big) \lambda_{\min} ( \mathbf{S}_k\t \mathbf{S}_k ) \\
    &\leq \min_k \frac{p_k p_{k+1} p_{k+2}}{r_1 r_2 r_3} \lambda_{\min} ( \mathbf{S}_k\t \mathbf{S}_k ) \\
    &\leq \Delta^2 \frac{p_1 p_2 p_3}{r_1 r_2 r_3},
\end{align*}
where we have used the assumption that $\lambda_{\max}(\mathbf{\Pi}_k\t \mathbf{\Pi}_k) \leq \frac{p_k}{r_k}$.  

For the remaining part, we note that by the previous argument, we have
\begin{align*}
    \U_k &= \mathbf{\Pi}_k \U_k\pure.
\end{align*}
Since $\U_k\t \U_k = \mathbf{I}_{r_1}$, it holds that
\begin{align*}
    \U_k\pure( \U_k\pure)\t \mathbf{\Pi}_k\t   \mathbf{\Pi}_k  \U_k\pure( \U_k\pure)\t = \U_k\pure \U_k\t \U_k (\U_k\pure)\t = \U_k\pure (\U_k\pure)\t,
\end{align*}
which demonstrates that
\begin{align*}
    \U_k\pure (\U_k\pure)\t &= (\mathbf{\Pi}_k\t \mathbf{\Pi}_k )\inv.
\end{align*}
Since $\U_k\pure$ is an $r_1 \times r_1$ matrix and $\lambda_{r_1}(\mathbf{\Pi}_k\t \mathbf{\Pi}_k ) \gtrsim \frac{p_k}{r_k}$, it holds that 
\begin{align*}
    \| \U_{k}\pure \|_{2,\infty}^2 &= \max_i \langle \big(\U_k\pure\big)_{i\cdot}, \big(\U_k\pure\big)_{i\cdot} \rangle \\
    &\leq \lambda_{\max} \big( \U_k\pure (\U_k\pure)\t \big) \\
    &\leq \lambda_{\max} \big( \mathbf{\Pi}_k\t \mathbf{\Pi}_k \big) \\
    &\lesssim \frac{r_k}{p_k}.
\end{align*}
Since $\U_k = \mathbf{\Pi}_k \U_k\pure$ has rows that are convex combinations of $\U_k\pure$, it holds that
\begin{align*}
    \| \U_k \|_{2,\infty} &\lesssim \sqrt{\frac{r_k}{p_k}},
\end{align*}
which demonstrates that $\mu_0 = O(1)$.  This completes the proof.
\end{proof}

\begin{proof}[Proof of \cref{prop:identifiability}]
The proof of the first part is similar to Theorem 2.1 of \citet{mao_estimating_2021}. 
Suppose that $\mathbf{T}_k$ has SVD $\U_k \mathbf{\Lambda}_k \mathbf{V}_k\t$. By   \cref{prop:relationship} (which only relies on the assumptions in \cref{prop:identifiability}) it holds that there an invertible matrix $\U_k^{(\mathrm{pure})}$ such that $\U_k = \mathbf{\Pi}_k \U_k^{(\mathrm{pure})}$, where $\U_k\pure$ consists of the rows of $\U_k$ corresponding to pure nodes.  Therefore, for each $i$ it holds that $(\U_k)_{i\cdot}$ is in the convex hull of $\U_k^{(\mathrm{pure})}$.  

Now suppose that there exists other parameters $\mathcal{S}', \mathbf{\Pi}_1'$, $\mathbf{\Pi}_2'$ and $\mathbf{\Pi}_3'$ such that $\mathcal{T} = \mathcal{S}' \times_1 \mathbf{\Pi}_1' \times_2 \mathbf{\Pi}_2' \times_3 \mathbf{\Pi}_3'$, where each  $\mathbf{\Pi}_{k}'$ may have  different pure nodes.  Note that since $\mathcal{T}$ is the same regardless of $\mathbf{\Pi}_k$ and $\mathbf{\Pi}_k'$, its singular value decomposition is fixed (where we arbitrarily specify a choice of sign for unique singular values or basis for repeated singular values).  By the previous argument we have that $\tilde{\U}_k\pure$ must belong to the convex hull of $\U_k\pure$, where $\tilde{\U}_k\pure$ corresponds to the pure nodes associated to $\mathbf{\Pi}_k'$.  By applying \cref{prop:relationship} again to the new decomposition, it must hold that $\U_k = \mathbf{\Pi}_k' \tilde{\U}_k\pure$, which shows that $\U_k\pure$ belongs to the convex hull of $\tilde{\U}_k\pure$.  Since both convex hulls are subsets of each other, it holds that the convex hulls of $\U_k\pure$ and $\tilde{\U}_k\pure$ are the same.  Consequently, it must hold that $\U_k\pure = \mathcal{P}_k\tilde{\U}_k\pure$ for some permutation matrix $\mathcal{P}_k$.  

Now we note that by the identity $\U_k = \mathbf{\Pi}_k \U_k\pure = \mathbf{\Pi}_k' \tilde{\U}_k\pure$, it holds that $\mathbf{\Pi}_k \U_k\pure = \mathbf{\Pi}_k' \mathcal{P}_k \U_k\pure$, which demonstrates that
\begin{align*}
    ( \mathbf{\Pi}_k - \mathbf{\Pi}_k' \mathcal{P}_k ) \U_k\pure = 0.
\end{align*}
Since $\U_k\pure$ is full rank, it must therefore hold that $\mathbf{\Pi}_k = \mathbf{\Pi}_k' \mathcal{P}_k$. % Therefore,
% \begin{align*}
%     \mathbf{T}_k &= \mathbf{\Pi}_k' \mathbf{S}_k' (\mathbf{\Pi}_{k+1}' \otimes \mathbf{\Pi}_{k+2}')\t \\
%     &= \mathbf{\Pi}_k \mathcal{P}_k \mathbf{S}_k' \big( ( \mathbf{\Pi}_{k+1} \mathcal{P}_{k+1}) \otimes (\mathbf{\Pi}_{k+2} \mathcal{P}_{k+2}) \big)\t.
% \end{align*}
Consequently,
\begin{align*}
    \mathcal{T} &= \mathcal{S}' \times_1 \big( \mathbf{\Pi}_1 \mathcal{P}_1 \big) \times_2  \big( \mathbf{\Pi}_{2} \mathcal{P}_{2} \big) \times_3 \big (\mathbf{\Pi}_{3} \mathcal{P}_{3}\big), 
\end{align*}
which shows that $\mathcal{S} = \mathcal{S}' \times_1 \mathcal{P}_1 \times_2 \mathcal{P}_2 \times_3 \mathcal{P}_3$, which completes the proof of the first part of the result.

The second part of the result essentially follows the proof of Theorem 2.2  of \citet{mao_estimating_2021}. Without loss of generality we prove the result for mode $1$.  Assume for contradiction that there is a community without any pure nodes; without loss of generality let it be the first community.  Then there is some $\delta > 0$ such that $(\mathbf{\Pi}_1)_{i1} \leq 1 - \delta$ for all $i$.  Define
\begin{align*}
    \mathbf{H} :&= \left[
\begin{array}{c|c}
1 + (r_1 - 1)\eps^2 & -\eps^2 \mathbf{1}\t_{r_1-1} \\
\hline
0 & \eps \mathbf{1}_{r_1-1}\mathbf{1}\t_{r_1-1} + (1 - (r_1 - 1)\eps) \mathbf{I}_{r_1 - 1}
\end{array}
\right],
\end{align*}
where $0 < \eps < \delta$.  For $\eps$ sufficiently small, $\mathbf{H}$ is full rank, and the rows of $\mathbf{H}$ sum to one.  Consequently, $\mathbf{\tilde \Pi}_1 := \mathbf{\Pi}_1 \mathbf{H}$ also has rows that sum to  one.  Moreover, for any $i$, $(\mathbf{\tilde \Pi}_1)_{i1} = (\mathbf{\Pi}_1)_{i1}( 1 + (K-1)\eps^2) \geq 0$, and for any $2 \leq l \leq r_1$,
\begin{align*}
    (\mathbf{\tilde \Pi}_1)_{il} &= - (\mathbf{\Pi}_1)_{i1} \eps^2 + \sum_{l'=1}^{r_1} (\mathbf{\Pi}_{1})_{il'} \mathbf{H}_{l'l} \\
    &=  - (\mathbf{\Pi}_1)_{i1} \eps^2 + (\mathbf{\Pi}_1)_{il}( 1 - (K-1)\eps) + \sum_{l'=1}^{r_1} (\mathbf{\Pi}_{1})_{il'} \eps \\
    &\geq  - (\mathbf{\Pi}_1)_{i1} \eps^2 + \sum_{l'=1}^{r_1} (\mathbf{\Pi}_{1})_{il'} \eps \\
    &\geq (1-\delta) \eps^2 + \eps \delta \\
    &> 0,
\end{align*}
and hence $\mathbf{\tilde \Pi}_1$ has positive entries.  Therefore, for $\eps$ sufficiently small $\mathbf{\tilde \Pi}_1$ is a valid membership matrix.  In addition, we have that
\begin{align*}
    \mathcal{M}_1(\mathcal{T}) &= \mathbf{\Pi}_1 \mathcal{M}_1(\mathcal{S}) \big( \mathbf{\Pi}_2 \otimes \mathbf{\Pi}_3 \big)\t \\
    &= \mathbf{\tilde \Pi}_1 \mathbf{H}\inv \mathcal{M}_1(\mathcal{S}) \big( \mathbf{\Pi}_2 \otimes \mathbf{\Pi}_3 \big)\t \\
    &= \mathbf{\tilde \Pi}_1  \mathcal{M}_1(\mathcal{S} \times_1 \mathbf{H}\inv ) \big( \mathbf{\Pi}_2 \otimes \mathbf{\Pi}_3 \big)\t,
\end{align*}
which shows that $\mathcal{T} = \mathcal{\tilde S} \times_1 \mathbf{\tilde \Pi}_1 \times_2 \mathbf{\Pi}_2 \times_3 \mathbf{\Pi}_3$ is another representation of $\mathcal{T}$, where $\mathcal{\tilde S} = \mathcal{S} \times_1 \mathbf{H}\inv$.  Since $\mathbf{H}$ is not a permutation matrix, we see that we have a contradiction, which completes the proof.
\end{proof}

\subsection{Proof of \cref{thm:estimation}}
\begin{proof}[Proof of \cref{thm:estimation}]
Our proof is similar to the proof of the main result in \citet{mao_estimating_2021} as well as the proof of Theorem 4.9 in \citet{xie_entrywise_2022}, where we will apply  Theorem 3 of \citet{gillis_fast_2014}.   We first prove the result assuming that $\lambda/\sigma \gtrsim \kappa p \sqrt{\log(p)}/p_{\min}^{1/4}$, that $r \leq p_{\min}^{1/4}$ and that $\mu_0 = O(1)$; the result will then follow by applying \cref{lem:relationship}. Without loss of generality, we prove the result for $k = 1$.

First, observe that by \cref{thm:twoinfty}, with probability at least $1 - p^{-10}$ it holds that there is an orthogonal matrix $\mathbf{W}$ such that for $t$ iterations with $t$ as in \cref{thm:twoinfty} it holds that the output $\uhat$ of HOOI satisfies
\begin{align*}
    \uhat &= \U \mathbf{W} + \mathrm{error},
\end{align*}
with 
\begin{align*}
    \| \mathrm{error} \|_{2,\infty} &\lesssim \frac{\kappa  \sqrt{r_k\log(p)}}{\lambda/\sigma}.
\end{align*}
Since by \cref{lem:relationship} $\U = \mathbf{\Pi} \U\pure$, it holds that
\begin{align*}
    \uhat\t &= \mathbf{W}\t (\U\pure)\t \mathbf{\Pi}\t + \mathrm{error}\t.
\end{align*}
We will apply Theorem 3 of \citet{gillis_fast_2014}, with $\mathbf{M}$, $\mathbf{W}$, $\mathbf{H}$  and $\mathbf{N}$ therein equal to $\U$, $\mathbf{W}\t (\U\pure)\t$, $\mathbf{\Pi}\t$ and $\mathrm{error}\t$ respectively.  Define,  for some sufficiently large constant $C$,
\begin{align*}
    \eps &\coloneqq C \frac{\kappa  \sqrt{r_k\log(p)}}{\lambda/\sigma}.
\end{align*}
It then holds that $\|\mathrm{error}_{i\cdot}\| \leq \eps$ on the event in \cref{thm:twoinfty}.  We also need to check the bound
\begin{align*}
    \eps < \lambda_{\min}( \U\pure ) \min\bigg( \frac{1}{2 \sqrt{r_1 - 1}}, \frac{1}{4} \bigg) \bigg( 1 + 80 \frac{ \sigma_1^2( \U\pure)}{\sigma_r^2(\U\pure)} \bigg)\inv.
\end{align*}
First we note that by the proof of \cref{lem:relationship}, we have that $\lambda_{\min}^2(\U\pure) = \lambda_{\min} \big( \U\pure (\U\pure)\t \big) = \lambda_{\min}\big( ( \mathbf{\Pi\t\Pi} )\inv \big)$.  Since $\lambda_{\max} \big( \mathbf{\Pi\t\Pi}\big) \lesssim \frac{p_1}{r_1}$, we have that $\lambda_{\min}(\U\pure) \gtrsim \frac{\sqrt{r_1}}{\sqrt{p_1}}$.  

We note that
\begin{align*}
   \frac{ \| \U\pure \mathbf{W} \|_{2,\infty}^2}{\lambda_r^2(\U\pure)} &\leq \frac{ \lambda_{\max}^2 (\U\pure) }{\lambda_r^2(\U\pure)} \\
   &=  \frac{ \lambda_{\max}^2 (\U\pure) }{\lambda_r^2(\U\pure)} \\
   &= \frac{\lambda_{\max}(\mathbf{\Pi\t\Pi})}{\lambda_{\min}(\mathbf{\Pi\t\Pi})} \\
   &\asymp C,
\end{align*}
since $\lambda_{\min}(\mathbf{\Pi\t\Pi})\gtrsim p_1/r_1$ by assumption.  Consequently, plugging in these estimates, it suffices to show that
\begin{align*}
    \eps < c \frac{\sqrt{r_1}}{\sqrt{p_1}} \frac{1}{\sqrt{r_1}} =  \frac{c}{\sqrt{p_1}},
\end{align*}
where $c$ is some sufficiently small constant. Plugging in the definition of $\eps$, we see that we require that
\begin{align*}
    C \frac{\kappa \sqrt{r_1 \log(p)}}{\lambda/\sigma} \leq \frac{c}{\sqrt{p_1}},
\end{align*}
which is equivalent to the condition
\begin{align*}
    \lambda/\sigma \gtrsim \kappa \sqrt{p_1 r_1 \log(p)},
\end{align*}
which holds under the condition $\lambda/\sigma \gtrsim \kappa p \sqrt{\log(p)}/p_{\min}^{1/4}$ and $r \leq p_{\min}^{1/4}$.   Therefore, we may apply  Theorem 3 of \citet{gillis_fast_2014} to find that there exists a permutation $\mathcal{P}$ such that
\begin{align*}
    \| \uhat\pure - \mathcal{P}\t \U\pure \mathbf{W} \|_{2,\infty} \leq  C \eps.
\end{align*}
We now use this bound to provide our final bound. 
First, since $\eps \lesssim \frac{1}{\sqrt{p}}$, by Weyl's inequality it holds that
\begin{align*}
    \lambda_{\min}(\uhat\pure) &\geq \lambda_{\min} ( \U_k\pure) - \sqrt{r_1} \eps \\
    &\geq C\frac{\sqrt{r_1}}{\sqrt{p_1}} - \frac{c \sqrt{r_1}}{\sqrt{p_1}} \\
    &\gtrsim \frac{\sqrt{r_1}}{{\sqrt{p_1}}},
\end{align*}
as long as $c$ is sufficiently small. Consequently, $\|(\uhat\pure)\inv \| \lesssim \sqrt{\frac{p_1}{r_1}}$.  Therefore,
\begin{align*}
    \| \mathbf{\widehat \Pi } - \mathbf{\Pi}\mathcal{P} \|_{2,\infty} &= \| \uhat (\uhat\pure)\inv - \U (\U\pure)\inv \mathcal{P} \|_{2,\infty} \\
    &= \| \uhat (\uhat\pure)\inv - \U \mathbf{W} (\mathcal{P}\t\U\pure  \mathbf{W} )\inv \|_{2,\infty} \\
    &\leq \| \big(\uhat - \U \mathbf{W} \big) (\uhat\pure)\inv \|_{2,\infty} + \| \U \mathbf{W} \big( (\uhat\pure)\inv - (\mathcal{P}\t\U\pure  \mathbf{W} )\inv \big) \|_{2,\infty} \\
    &\leq \| \uhat - \U \mathbf{W} \|_{2,\infty} \| (\uhat\pure)\inv \| + \| \mathbf{\Pi} \U\pure \mathbf{W} \big( (\uhat\pure)\inv - (\mathcal{P}\t\U\pure  \mathbf{W} )\inv \big) \|_{2,\infty} \\
    &\leq \eps \| (\uhat\pure)\inv \| + \| \mathbf{\Pi} \|_{\infty \to \infty} \|  \U\pure \mathbf{W} \big( (\uhat\pure)\inv - (\mathcal{P}\t\U\pure  \mathbf{W} )\inv \big) \|_{2,\infty} \\
    &\leq \eps \| (\uhat\pure)\inv \| + \| \U\pure \mathbf{W} \big( (\mathcal{P}\t \uhat\pure)\inv - (\U\pure \mathbf{W} )\inv  \big) \mathcal{P} \|_{2,\infty} \\
    &\leq \eps \| (\uhat\pure)\inv \| + \| \big(\U\pure \mathbf{W}  (\mathcal{P} \uhat\pure)\inv - \mathbf{I}_{r_k}  \big) \mathcal{P} \|_{2,\infty} \\
    &\leq \eps \| (\uhat\pure)\inv \| + \| \big(\U\pure \mathbf{W}  - \mathcal{P} \uhat\pure  \big)   (\mathcal{P} \uhat\pure)\inv \mathcal{P} \|_{2,\infty} \\
    &\leq \eps \| (\uhat\pure)\inv \| + \| \mathcal{P}\t \U\pure \mathbf{W}  -  \uhat\pure \|_{2,\infty} \| (\mathcal{P} \uhat\pure)\inv  \|_{2,\infty} \\
    &\leq 2 \eps\sqrt{p_1/r_1} \\
    &\lesssim  \frac{\kappa \sqrt{r_1 \log(p)}}{\lambda/\sigma} \sqrt{\frac{p_1}{r_1}} \\
    &\asymp \frac{\kappa \sqrt{p_1 \log(p)}}{\lambda/\sigma}.
\end{align*}
Therefore, all that remains is to apply \cref{lem:relationship}.  First, we need to check that the condition
\begin{align*}
    \lambda/\sigma \gtrsim \kappa p \sqrt{\log(p)}/p_{\min}^{1/4}
\end{align*}
holds; by \cref{lem:relationship} this is equivalent to the condition
\begin{align*}
    \Delta/\sigma \gtrsim \frac{\kappa p \sqrt{\log(p)}}{p_{\min}^{1/4}} \frac{\sqrt{r_1 r_2 r_3}}{\sqrt{p_1 p_2 p_3}},
\end{align*}
which is in Assumption \ref{assumption:signalstrength}.  \textcolor{black}{Similarly, the upper bound on the number of iterations applies by substituting $\Delta/\sigma \frac{(p_1p_2p_3)^{1/2}}{(r_1r_2r_3)^{1/2}}$ for the quantity $\lambda/\sigma$ and adjusting the constant $c$ in the exponential.} Finally, by \cref{lem:relationship}, we obtain the final upper bound
\begin{align*}
      \| \mathbf{\widehat \Pi } - \mathbf{\Pi}\mathcal{P} \|_{2,\infty} &\lesssim \frac{\kappa \sigma\sqrt{r_1 r_2 r_3  \log(p)}}{\Delta (p_{-1})^{1/2}},
\end{align*}
as desired. 
\end{proof}
\subsection{Proof of \cref{cor:averagecase}}

\begin{proof}[Proof of \cref{cor:averagecase}]
Fix an index $k$, and let $\mathcal{P}$ denote the permutation  matrix from \cref{thm:estimation}.  Then it holds that
\begin{align*}
    \inf_{\mathrm{Permutations }\ \mathcal{P}}\|  \big(\mathbf{\widehat \Pi}_k - \mathbf{\Pi}_k \mathcal{P} \big)_{i\cdot} )_{i\cdot} \|_1 &\leq \sqrt{r_k} \| \big(\mathbf{\widehat \Pi}_k - \mathbf{\Pi}_k \mathcal{P} \big)_{i\cdot} \|_2 \\
    &\leq \frac{r^2 \kappa \sqrt{\log(p)}}{(\Delta/\sigma) (p_{-k})^{1/2}}.
\end{align*}
Averaging over the rows completes the proof.
\end{proof}

\section{Auxiliary Probabilistic Lemmas} 
 \label{sec:twoinftyaux}
  \begin{lemma}\label{lem:rowbound}
 Let $\mathbf{A}$ be any fixed matrix independent from $e_m\t \mathbf{Z}_k$.  Then there exists an absolute constant $C > 0$ such that with probability at least $1 - O(p_{\max}^{-20})$,
 \begin{align*}
     \| e_m\t \mathbf{Z}_k \mathbf{A} \| &\leq C \sigma \sqrt{p_{-k} \log(p_{\max})}\| \mathbf{A} \|_{2,\infty}.
 \end{align*}
 \end{lemma}

 \begin{proof}
 This follows from \cite{cai_subspace_2021}, Lemma 12.
 \end{proof}
 
 \begin{lemma}\label{lem:matricizationrowbound}
 Let $\mathbf{A}$ be a matrix independent from $\mathbf{Z}_k - \mathbf{Z}_k^{j-m}$, where $\mathbf{Z}_k^{j-m}$ is defined in \cref{sec:fullproof}.  Then there exists an absolute constant $C > 0 $ such that with probability at least $1 - O(p_{\max}^{-30})$,
 \begin{align*}
      \| \bigg(\mathbf{Z}_k - \mathbf{Z}_k^{j-m} \bigg) \mathbf{A} \| &\leq C \sigma \sqrt{p_{-j} \log(p_{\max})}\| \mathbf{A} \|_{2,\infty}.
 \end{align*}
 \end{lemma}
 
 \begin{proof}
 If $j = k$, the result follows by Lemma \ref{lem:rowbound}.  Therefore, we restrict our attention to when $j \neq k$.  First, note that 
 \begin{align*}
     \| \bigg( \mathbf{Z}_k  - \mathbf{Z}_k^{j-m} \bigg) \mathbf{A} \| &\leq \sqrt{p_k}   \| \bigg( \mathbf{Z}_k  - \mathbf{Z}_k^{j-m} \bigg) \mathbf{A} \|_{2,\infty}.
 \end{align*}
 Next, consider any fixed row of $\mathbf{Z}_k - \mathbf{Z}_k^{j-m}$.  Observe that the $q$'th row can be written as
 \begin{align*}
     \sum_{\Omega} \big(\mathbf{Z}_k \big)_{ql} \mathbf{A}_{l\cdot},
 \end{align*}
 where the set $\Omega$ consists of the $p_{-k-j}$ random variables in the $q$'th row of $\mathbf{Z}_k  - \mathbf{Z}_k^{j-m}$.   Note that this is a sum of independent random matrices.  By the matrix Bernstein inequality (Proposition 2 of \cite{koltchinskii2011nuclear}), it holds that with probability at least $1 - p_{\max}^{-31}$ that
 \begin{align*}
     \| \sum_{\Omega} \big(\mathbf{Z}_k \big)_{ql} \mathbf{A}_{l\cdot} \| &\leq C \max\bigg\{ \sigma_{Z} \sqrt{p_{-k-j}\log(p_{\max})}, U_{Z}\log(p)  \bigg\}
 \end{align*}
 where
 \begin{align*}
     \sigma_Z^2 &\coloneqq\max_l \max\bigg\{ \bigg\| \E \big(  \big(\mathbf{Z}_k \big)_{ql} \mathbf{A}_{l\cdot} \big)\big( \big(\mathbf{Z}_k \big)_{ql} \mathbf{A}_{l\cdot} \big)\t \bigg\|, \bigg\|  \E \big( \big(\mathbf{Z}_k \big)_{ql} \mathbf{A}_{l\cdot} \big)\t\big( \big(\mathbf{Z}_k \big)_{ql} \mathbf{A}_{l\cdot} \big) \bigg\| \bigg\};\\
     U_Z &\coloneqq \max_l \| (\mathbf{Z}_k)_{ql} A_{l\cdot} \|_{\psi_2}
 \end{align*}
 (Note that Proposition 2 of \cite{koltchinskii2011nuclear} holds for IID random matrices, but the proof works equally as well if uniform bounds on $\sigma_Z$ and $U_Z$ are obtained).   Observe that
 \begin{align*}
  \bigg\|   \E \bigg[(\mathbf{Z}_k)_{ql} \mathbf{A}_{l\cdot}  \bigg] \bigg[ (\mathbf{Z}_k)_{ql} \mathbf{A}_{l\cdot}  \bigg]\t \bigg\| &\leq \sigma^2  \| \mathbf{A}_{l\cdot} \mathbf{A}_{l\cdot}\t \|\\
  &\leq \sigma^2  \| \mathbf{A} \|_{2,\infty}^2; \\
  \bigg\|   \E \bigg[ (\mathbf{Z}_k)_{ql} \mathbf{A}_{l\cdot}  \bigg]\t \bigg[ (\mathbf{Z}_k)_{ql} \mathbf{A}_{l\cdot}  \bigg] \bigg\| &\leq \sigma^2 \ \| \mathbf{A}_{l\cdot}\t \mathbf{A}_{l\cdot} \| \\
  &\leq \sigma^2  \| \mathbf{A} \|_{2,\infty}^2.
 \end{align*}
 Similarly, by subgaussianity of the entries of $\mathbf{Z}_k$,
 \begin{align*}
     \max_l \| (\mathbf{Z}_k)_{ql} \mathbf{A}_{l\cdot} \|_{\psi_2} &\leq C \sigma \| \mathbf{A} \|_{2,\infty}.
 \end{align*}
 Therefore, with probability at least $1 - p_{\max}^{-31}$, it holds that
 \begin{align*}
     \| \sum_{\Omega} \big(\mathbf{Z}_k \big)_{ql} \mathbf{A}_{l\cdot} \| &\leq C \max\bigg\{ \sigma_{Z} \sqrt{p_{-k-j}\log(p_{\max})}, U_{Z}\log(p_{\max})  \bigg\} \\
     &\leq C \sigma \| \mathbf{A} \|_{2,\infty} \max\bigg\{ \sqrt{p_{-k-j} \log(p_{\max})},\log(p_{\max}) \bigg\} \\
     &\leq C \sigma \| \mathbf{A} \|_{2,\infty} \sqrt{p_{-k-j} \log(p_{\max})}.
 \end{align*}
 Taking a union bound over all $p_k$ rows shows that this holds uniformly with probability at least $1 - O(p_{\max}^{-30})$.  Therefore,
 \begin{align*}
     \bigg\| \bigg( \mathbf{Z}_k  - \mathbf{Z}_k^{j-m} \bigg) \mathbf{A} \| &\leq \sqrt{p_k}  \| \bigg( \mathbf{Z}_k  - \mathbf{Z}_k^{j-m} \bigg) \mathbf{A} \bigg\|_{2,\infty} \\
     &\leq C \sigma \|\mathbf{A} \|_{2,\infty} \sqrt{p_{-j} \log(p_{\max})}
 \end{align*}
 as desired.
 \end{proof}

\begin{lemma}\label{lem:taubound}
Suppose $\mathcal{Z} \in \R^{p_1 \times p_2 \times p_3}$ is a tensor with mean-zero subgaussian entries, each with $\psi_2$ norm bounded by $1$.  Suppose that $r_2 r_3 \leq p_1 r_1$. Then for some universal constant $C$, the following holds with probability at least $1 - c \exp(- c p_{\max})$:
\begin{align*}
    \sup_{\substack{\| \U_1\| = 1, \mathrm{rank}(\U_1) \leq 2r_1 \\, \| \U_2\| = 1, \mathrm{rank}(\U_2) \leq 2r_2}} \| \mathbf{Z} \bigg( \mathcal{P}_{\U_1} \otimes \mathcal{P}_{\U_2} \bigg) \| &\leq C \sqrt{p_{\max}r_{\max}}.
\end{align*}
%with .  %\textcolor{black}{todo: improve this to have precise dependence on $p_1$, $p_2$, and $p_3$, and also prove this?}
\end{lemma}

\begin{proof}
See Lemma 8 of \citet{han_exact_2021} or Lemma 3 of \citet{zhang_optimal_2019}.
\end{proof}

\begin{lemma}\label{lem:Egood}
Define $\mathcal{E}_{\mathrm{Good}}$ as in \eqref{Egood}.  Then under the conditions of \cref{thm:twoinfty}, it holds that $\p(\mathcal{E}_{\mathrm{Good}}) \geq 1 - O(p^{-30})$.      
\end{lemma}

\begin{proof}
Recall the definition of $\mathcal{E}_{\mathrm{Good}}$:
\begin{align*}
    \mathcal{E}_{\mathrm{Good}} &\coloneqq \bigg\{ \max_k \tau_k \leq C \sqrt{pr} \bigg\} \bigcap \bigg\{ \| \sin\Theta(\uhat_k^{(t)}, \U_k ) \| \leq \frac{\dl\ku}{\lambda} + \frac{1}{2^{t}} \text{ for all $t \leq t_{\max}$ and $1\leq k \leq 3$ } \bigg\} \\
    &\qquad \bigcap \bigg\{ \max_k \bigg\| \U_k\t \mathbf{Z}_k \mathbf{V}_k \bigg\| \leq C \left( \sqrt{r} + \sqrt{\log(p)} \right) \bigg\}; \\
    &\qquad \bigcap \bigg\{ \max_k  \bigg\| \U_k\t \mathbf{Z}_k \mathcal{P}_{\U_{k+1}}  \otimes   \mathcal{P}_{\U_{k+2}}  \bigg\| \leq C \left( r + \sqrt{\log(p)} \right) \bigg\}; \\
    &\qquad \bigcap \bigg\{ \max_k \bigg\| \mathbf{Z}_k \mathbf{V}_k \bigg\| \leq C \sqrt{p_k} \bigg\}. \numberthis \label{Egood2}
    \end{align*}
Define the events
\begin{align*}
    \mathcal{E}_1 :&= \bigg\{ \max_k \tau_k \leq C \sqrt{pr} \bigg\}; \\
        \mathcal{E}_2 :&= \bigg\{ \| \sin\Theta(\uhat_k^{(t)}, \U_k ) \| \leq \frac{\dl\ku}{\lambda} + \frac{1}{2^{t}} \text{ for all $t \leq t_{\max}$ and $1\leq k \leq 3$ } \bigg\} ;\\
        \mathcal{E}_3 :&= \bigg\{ \max_k  \bigg\| \U_k\t \mathbf{Z}_k \mathcal{P}_{\U_{k+1}}  \otimes   \mathcal{P}_{\U_{k+2}}  \bigg\| \leq C \left( r + \sqrt{\log(p)} \right) \bigg\};\\
    \mathcal{E}_4 :&= \bigg\{ \max_k \bigg\| \U_k\t \mathbf{Z}_k \mathbf{V}_k \bigg\| \leq C \left( \sqrt{r} + \sqrt{\log(p)} \right) \bigg\}; \\
    \mathcal{E}_5 :&=  \bigg\{ \max_k \bigg\| \mathbf{Z}_k \mathbf{V}_k \bigg\| \leq C \sqrt{p_k} \bigg\}.
\end{align*}
We aim to demonstrate that each event $\mathcal{E}_i$ holds with probability at least $1 - O(p^{-30})$ under the conditions of \cref{thm:twoinfty}, whence the result is complete via a union bound. First we will verify all events except $\mathcal{E}_2$, which we will do last.
\begin{itemize}
    \item \textbf{The event $\mathcal{E}_1$}: We recall that $\tau_k$ is defined via
    \begin{align*}
    \tau_k := \sup_{\substack{ \| \mathbf{U}_1 \| = 1, \mathrm{rank}(\U_1) \leq 2 r_{k+1}\\ \|\mathbf{U}_2\| =1, \mathrm{rank}(\U_2) \leq 2 r_{k+2}} } \|  \mathbf{Z}_k \bigg( \mathcal{P}_{\mathbf{U}_1} \otimes \mathcal{P}_{\mathbf{U}_2} \bigg)\|.
    \end{align*}
    Therefore, the result is implied by \cref{lem:taubound}.   
    \item \textbf{The event $\mathcal{E}_3$:} We note that it suffices to prove the bound for $k = 1$ since the right hand side is invariant to the index $k$.  First, we note that by properties of the Kronecker product and projection  matrices,
    \begin{align*}
        \| \U_1\t \mathbf{Z}_1 \mathcal{P}_{\mathbf{U}_2} \otimes \mathcal{P}_{\mathbf{U}_3} \| &= \| \U_1\t \mathbf{Z}_1 \big( \U_2 \otimes \U_3 \big) \|.
    \end{align*}
Next, let $x$ and $y$ be deterministic unit vectors of dimensions $r_1$ and $r_2 r_3$ respectively.  Observe that
\begin{align*}
    x\t \U_1\t \mathbf{Z}_1  \big( \U_2 \otimes \U_3 \big) y = \sum_{i=1}^{p_1} \sum_{j=1}^{p_2p_3} (\U_1 x)_{i} (\mathbf{Z}_1)_{ij} \bigg(\big( \U_2 \otimes \U_3 \big) y  \bigg)_j,
\end{align*}
which is a sum of independent random variables.  By Hoeffding's inequality for subgaussian random variables, it holds that
\begin{align*}
    \p\bigg\{ \big| x\t \U_1\t \mathbf{Z}_1  \big( \U_2 \otimes \U_3 \big) y \big| \geq t \bigg\} \leq 2 \exp\bigg\{ -c \frac{t^2}{\| \mathbf{U}_1 x\|^2 \| \mathbf{U}_2 \otimes \mathbf{U}_3 y \|^2} \bigg\}.
\end{align*}
By taking $t = C(r + \sqrt{\log(p)}) \| \U_1 x\| \| \U_2\otimes \U_3 y \|$, we have that
\begin{align*}
    \big| x\t \U_1\t \mathbf{Z}_1  \big( \U_2 \otimes \U_3 \big) y \big| &\leq C (r + \sqrt{\log(p)}) \| \U_1 x\| \| \U_2\otimes \U_3 y \| \\
    &\leq C(r + \sqrt{\log(p)})
\end{align*}
with probability at least $1 - 2 \exp\bigg( - C( r+ \sqrt{\log(p)} )^2 \bigg)$,  where we have used the fact that $\U_1$ and $\U_2\otimes \U_3$ are orthonormal matrices.  Since the bound above does not depend on $x$ and $y$, let $\mathcal{N}_1$ and $\mathcal{N}_2$ be $1/4$-nets  for the unit ball in $\mathbb{R}^{r_1}$ $\mathbb{R}^{r_2r_3}$ respectively. By Corollary 4.2.13 of \citet{vershynin_high-dimensional_2018}, we have
\begin{align*}
    | \mathcal{N}_1 | \leq 12^{r_1} \qquad | \mathcal{N}_2 | \leq 12^{r_2r_3}.
\end{align*}
Therefore, by taking a union bound, we see that
\begin{align*}
    \sup_{x \in \mathcal{N}_1, y \in \mathcal{N}_2} \big| x\t \U_1\t \mathbf{Z}_1  \big( \U_2 \otimes \U_3 \big) y \big| &\leq C\big( r + \sqrt{\log(p)}\big)
\end{align*}
with probability at least $1 - 12^{r_1 + r_2r_3} \exp\bigg( - C (r + \sqrt{\log(p)})^2 \bigg) \geq 1 - O(p^{-30})$ as long as $C$ is chosen sufficiently large.  Finally, by exercise 4.4.3 of \citet{vershynin_high-dimensional_2018} it holds that
\begin{align*}
    \| \U_1\t \mathbf{Z}_1 \big( \U_2 \otimes \U_3) \| &= \sup_{\|x\| = 1, \|y\|=1} | \big| x\t \U_1\t \mathbf{Z}_1  \big( \U_2 \otimes \U_3 \big) y \big| \\
    &\leq 2 \sup_{x \in \mathcal{N}_1, y \in \mathcal{N}_2} \big| x\t \U_1\t \mathbf{Z}_1  \big( \U_2 \otimes \U_3 \big) y \big|.
\end{align*}
This shows that $\p\big( \mathcal{E}_3 \big) \geq 1- O(p^{-30})$ as required.
\item \textbf{The events $\mathcal{E}_4$ and $\mathcal{E}_5$:} These bounds follow from the same arguments as the previous bound, only taking $1/4$ nets over different dimensions.  
\item \textbf{The event $\mathcal{E}_2$}: We will replicate the proof of Theorem 1 of \citet{zhang_tensor_2018}.  First, we note that by the proof of \cref{lem:spectraliniteigengapsloo} (which does not depend on the statement of this result), it holds that the initialization satisfies
\begin{align*}
    \| \sin\Theta(\mathbf{\hat U}_k^{(0)}, \U_k) \| \leq \kappa^2 \mu_0^2 \frac{r}{p_k} + \frac{\kappa \sqrt{p_k}}{\lambda} + \frac{(p_1p_2p_3)^{1/2}}{\lambda^2},
\end{align*}
which together with the assumption $\lambda \gtrsim \kappa \sqrt{\log(p)} p/p_{\min}^{1/4}$ implies that for each $k$,
\begin{align*}
    \| \sin\Theta(\mathbf{\hat U}_k^{(0)}, \U_k) \|  \leq \frac{1}{2} \leq \frac{C_0 \kappa\sqrt{p_k \log(p)}}{\lambda} + \frac{1}{2}. \numberthis \label{boundholdsguy}
\end{align*}
These bounds hold with probability at least $1 - O(p_{\max}^{-30})$. Let $\mathcal{E}$ be the event that \eqref{boundholdsguy} holds and that $\mathcal{E}_1$ holds.  We will show that deterministically on this event that our required bound holds.  The proof is by induction on $t$.  Suppose the bound holds up to some $t$.  Define
\begin{align*}
    \mathbf{\hat T}_1^{(t+1)} := \mathcal{M}_1 \big( \mathcal{\hat T} \times_2 (\U_2^{(t)})\t \times_3 (\U_3^{(t)})\t \big); \\
    \mathbf{T}_1^{(t)} := \mathcal{M}_1 \big( \mathcal{ T} \times_2 (\U_2^{(t)})\t \times_3 (\U_3^{(t)})\t \big); \\
     \mathbf{Z}_1^{(t)} := \mathcal{M}_1 \big( \mathcal{ Z} \times_2 (\U_2^{(t)})\t \times_3 (\U_3^{(t)})\t \big).
\end{align*}
It holds that
\begin{align*}
    \lambda_{r_1} \bigg( \mathbf{T}_1^{(t+1)} \bigg) &= \lambda_{r_1} \bigg( \mathbf{T}_1 ( \uhat_2^{(t)} \otimes \uhat_3^{(t)} \bigg) \\
    &= \lambda_{r_1} \bigg( \mathbf{T}_1 \U_2 \otimes \U_3 (\U_2 \otimes \U_3)\t ( \uhat_2^{(t)} \otimes \uhat_3^{(t)} \bigg) \\
    &\geq \lambda_{r_1} \bigg( \mathbf{T}_1 \U_2 \otimes \U_3 \bigg) \lambda_{\min} \bigg( (\U_2 \otimes \U_3)\t ( \uhat_2^{(t)} \otimes \uhat_3^{(t)} \bigg) \\
    &\geq \lambda_{r_1} \bigg( \mathbf{T}_1 \U_2 \otimes \U_3 \bigg) \lambda_{\min} \big( \U_2\t \uhat_2^{(t)} \big)  \lambda_{\min} \big( \U_3\t \uhat_3^{(t)} \big) \\
    &\geq \lambda (1 - \frac{1}{4}),
\end{align*}
where we have used the fact that $\lambda_{\min}(\U_2\t \uhat_2^{(t)}) = \sqrt{ 1 - \| \sin\Theta(\U_2,\uhat_2^{(t)}) \|^2}$.  Furthermore, we have that
\begin{align*}
    \| \mathbf{Z}_1^{(t)} \| &= \| \mathbf{Z}_1( \uhat_2^{(t)} \otimes \uhat_3^{(t)} \| \\
    &\leq \| \mathbf{Z}_1 \big( \U_2 \otimes \U_3) \| + \| \mathbf{Z}_1 ( \mathcal{P}_{\U_{2\perp}} \uhat_2^{(t)}) \otimes \mathcal{P}_{\U_3} \uhat_3^{(t)} \| \\
    &\quad + \| \mathbf{Z}_1 ( \mathcal{P}_{\U_{2}} \uhat_2^{(t)}) \otimes \mathcal{P}_{\U_{3\perp}} \uhat_3^{(t)} \| + \| \mathbf{Z}_1 ( \mathcal{P}_{\U_{2\perp}} \uhat_2^{(t)}) \otimes \mathcal{P}_{\U_{3\perp}} \uhat_3^{(t)} \| \\
    &\leq C \sqrt{p_1} +3 C \sqrt{p r} \| \sin\Theta(\uhat_2^{(t)}, \U_2) \| + C \sqrt{pr} \|\sin\Theta(\uhat_3^{(t)}, \U_3) \| \\
    &\quad +  C \sqrt{pr} \| \sin\Theta(\uhat_2^{(t)}, \U_2) \| \|\sin\Theta(\uhat_3^{(t)},\U_3) \| \\
    &\leq C \sqrt{p_1} + 3 C \sqrt{pr}\bigg( \frac{\delta^{(2)}_{L}}{\lambda} + \frac{\delta^{(3)}_{L}}{\lambda}+ \frac{1}{2^t} \bigg).
\end{align*}
therefore, since $\mathbf{T}_1^{(t)}$ and $\mathbf{T}_1$ have the same left singular vectors, Wedin's Theorem implies that
\begin{align*}
    \|\sin\Theta(\uhat^{(t+1)}_1,\U_1) \| &\leq \frac{C\sqrt{p_1}}{\lambda} + \frac{3 C \sqrt{pr}}{\lambda} \bigg( \frac{\dl^{(2)}}{\lambda} + \frac{\dl^{(3)}}{\lambda} \bigg) + \frac{3 C \sqrt{pr}}{\lambda} \frac{1}{2^t} \\
    &\leq \frac{C \sqrt{p_1}}{\lambda} + \frac{3 C C_0 \kappa \sqrt{pp_2 r \log(p)}}{\lambda} + \frac{3 C C_0 \kappa \sqrt{pp_3 r \log(p)}}{\lambda} + \frac{1}{2^{t+1}} \\
    &\leq \frac{C_0 \kappa \sqrt{p_1 \log(p)}}{\lambda} + \frac{1}{2^{t+1}},
\end{align*}
where the final inequality holds from the assumption that $\lambda \gtrsim \kappa \sqrt{\log(p)} p/p_{\min}^{1/4}$.  Arguing similarly for the other modes, we see that for all $t$ and $k$ it holds that
\begin{align*}
    \|\sin\Theta(\uhat_k^{(t)},\U_k) \| \leq \frac{\dl^{(k)}}{\lambda} + \frac{1}{2^t}
\end{align*}
with probability at least $1 - O(p^{-30})$, which completes the proof.
% apply Theorem 1 of \citet{luo_sharp_2021}. First, we note that by Lemma 9 of \citet{han_exact_2021}, it holds that
% \begin{align*}
%     \sup_{\substack{\mathbf{V}_1 \in \mathbb{R}^{p_1 \times r_1}, \mathbf{V}_2 \in \mathbb{R}^{p_2 \times r_2}, \mathbf{V}_3 \in \mathbb{R}^{p_3 \times r_3}\\ \|\mathbf{V}_i \|\leq 1}} &\leq C \bigg( r^{3/2} + \sqrt{ p_{\max} r} \bigg)
% \end{align*}
% with probability at least $1- C\exp( - c p_{\max})$ (the probability bound is contained in the proof of this result, not its statement).  In addition, by the proof of \cref{lem:spectraliniteigengapsloo} (which does not depend on the statement of this result), it holds that the initialization satisfies
% \begin{align*}
%     \| \sin\Theta(\mathbf{\hat U}_k^{(0)}, \U_k) \| \leq \kappa^2 \mu_0^2 \frac{r}{p_k} + \frac{\kappa \sqrt{p_k}}{\lambda} + \frac{(p_1p_2p_3)^{1/2}}{\lambda^2},
% \end{align*}
% which together with the assumption $\lambda \gtrsim \kappa \sqrt{\log(p)} p/p_{\min}^{1/4}$ implies that
% \begin{align*}
%     \max_{k}\| \sin\Theta(\mathbf{\hat U}_k^{(0)}, \U_k) \|  \leq \frac{1}{2}.
% \end{align*}
% Furthermore, we have that
% \begin{align*}
%     \lambda \geq C \kappa \sqrt{\log(p)} p/p_{\min}^{1/4} \geq  C \bigg( r^{3/2} + \sqrt{ p_{\max} r} \bigg),
% \end{align*}
% which is a high-probability upper bound on the definition of $\xi$ in \citet{luo_sharp_2021}.  As a result, it holds that
% \begin{align*}
%     \| \sin\Theta(\U_k^{(t)}, \U_k) \| \leq \frac{C \sqrt{p_k}}{\lambda} + \frac{1}{2^t}
% \end{align*}
\end{itemize}
 \end{proof}